\documentclass[12pt,amstex]{amsart}

\usepackage{xr}
\usepackage[shortlabels]{enumitem}
\usepackage[english]{babel}
\usepackage{amscd, amsfonts, amsmath, amssymb, amsthm, epsfig, float, geometry, graphicx, pstricks, pst-node, txfonts, enumitem}

\usepackage[all]{xy}
\usepackage{tikz}
\usepackage{multicol}

\usepackage{hyperref}
\hypersetup{
	colorlinks,
	citecolor=black,
	filecolor=black,
	linkcolor=black,
	urlcolor=black
}

\numberwithin{equation}{section}

\theoremstyle{definition}
\newtheorem{thm}{Theorem}[section]
\newtheorem{lem}[thm]{Lemma}
\newtheorem{defn}[thm]{Definition}
\newtheorem{rem}[thm]{Remark}
\newtheorem{prop}[thm]{Proposition}

\newtheorem{coro}[thm]{Corollary}
\newtheorem{ex}[thm]{Example}

\newtheorem{conv}[thm]{Convention}

\def \C {\mathbb{C}}
\def \E {\mathbb{E}}
\def \F {\mathbb{F}}

\def \N {\mathbb{N}}
\def \O {\mathcal{O}}
\def \P {\bold{P}}
\def \Q {\mathbb{Q}}
\def \R {\mathbb{R}}
\def \T {\mathbb{T}}
\def \V {\mathbb{F}_{p}^{d}}

\def \Z {\mathbb{Z}}

\def \+ {\hat{+}}
\def \- {\hat{-}}

\def \d {\delta}
\def \e {\epsilon}

\def \h {\bold{h}}

\def \w {\bold{w}}
\def \x {\bold{x}}
\def \y {\bold{y}}

\def \cor {\text{cor}}
\def \CP {\{\text{pure, nice, independent, consistent}\}}

\def \DR {\text{DR}}

\def \Gow {\Box}

\def \Hor {\text{Horiz}}
\def \ind {\text{ind}}
\def \lin {\text{lin}}
\def \Lip {\text{Lip}}

\def \Nil {\text{Nil}}

\def \ped {\text{ped}}
\def \poly {\text{poly}}
\def \pp {\perp_{M}}

\def \rank {\text{rank}}

\def \SGI {\text{SGI}}
\def \st {\text{HP}}
\def \sp {\text{span}}

\def \Taylor {\text{Taylor}}

\def \Vk {\mathbb{Z}_{K}^{d}}

\def \zp {\mathbb{Z}/p^{\mathbb{N}}}

\title[Spherical higher order Fourier analysis over finite fields III]{Spherical higher order Fourier analysis over finite fields III: a spherical Gowers inverse theorem}
\author{Wenbo Sun}
\address[Wenbo Sun]{Department of Mathematics, Virginia Tech, 225 Stanger Street, Blacksburg, VA, 24061, USA}
\email{swenbo@vt.edu}

\thanks{The author was partially supported by the NSF Grant DMS-2247331}

\subjclass[2020]{11T99, 37A99}

\begin{document}

\maketitle

\begin{abstract}
  This paper is the third part of the series \emph{Spherical higher order Fourier analysis over finite fields}, aiming to develop the higher order Fourier analysis method along spheres over finite fields, and to solve the geometric Ramsey conjecture in the finite field setting.
  
  In this paper, we prove an inverse theorem over finite field for spherical Gowers norms, i.e. a local Gowers norm supported on a sphere. We show that if the $(s+1)$-th spherical Gowers norm of a 1-bounded function $f\colon\mathbb{F}_{p}^{d}\to \mathbb{C}$ is at least $\epsilon$ and if $d$ is sufficiently large depending only on $s$, then $f$ correlates on the sphere with a $p$-periodic $s$-step nilsequence, where the bounds for the complexity and correlation depend only on $d$ and $\epsilon$. This result will be used in later parts of the series to prove the geometric Ramsey conjecture in the finite field setting. 
\end{abstract}

\tableofcontents

\section{Introduction}

\subsection{The spherical Gowers inverse theorem}
This paper is the third part of the series \emph{Spherical higher order Fourier analysis over finite fields} \cite{SunA,SunB,SunD}. The purpose of this paper is to use the tools developed in  \cite{SunA,SunB} to prove a Gowers inverse theorem for spherical sets.
We start by recalling the definition of local Gowers norms. For any finite set $A$ and function $f\colon A\to\C$, denote 
$\E_{n\in A}f(n):=\frac{1}{\vert A\vert}\sum_{n\in A}f(n)$.

\begin{defn}[Local Gowers norms]
	Let $\Omega$ be a subset of $\V$, $s\in\N_{+}$,
	and $f\colon\V\to \C$ be a function. The \emph{$s$-th $\Omega$-Gowers norm} of $f$ is defined by the quantity
	$$\Vert f\Vert_{U^{s}(\Omega)}:=\Bigl\vert\E_{(n,h_{1},\dots,h_{s})\in \Gow_{s}(\Omega)}\prod_{\e=(\e_{1},\dots,\e_{s})\in\{0,1\}^{s}} \mathcal{C}^{\vert\e\vert}f(n+\e_{1}h_{1}+\dots+\e_{s}h_{s})\Bigr\vert^{\frac{1}{2^{s}}},$$
	where 
	$$\Gow_{s}(\Omega):=\{(n,h_{1},\dots,h_{s})\in(\V)^{s+1}\colon n+\e_{1}h_{1}+\dots+\e_{s}h_{s}\in\Omega \text{ for all } (\e_{1},\dots,\e_{s})\in\{0,1\}^{s}\},$$
	$\vert\e\vert:=\e_{1}+\dots,\e_{s}$ and $\mathcal{C}^{2n+1}f=\overline{f}$, $\mathcal{C}^{2n}f=f$ for all $n\in\Z$.\footnote{One can show that $\Vert\cdot\Vert_{U^{s}(\Omega)}$ is indeed a norm when $s\geq 2$.}
\end{defn}

The inverse theorems for Gowers norms is the central part of higher order Fourier analysis, which has been studied extensively in the last two decades, and it has many applications in Semer\'edi-type problems. See \cite{CGS21,CS14,FH17,Gow01,GT08b,GT08,GT10b,GT10,GTZ11,GTZ12,HK05,JST23,Sze12,TZ12b,TZ12} for a (far from being complete) list of related works. For example,
it was shown by Green, Tao and Ziegler \cite{GTZ12} (see also \cite{Man14}) that for $\Omega=\F_{p}$, for any function $f\colon \Omega\to \C$ bounded by 1 with $\Vert f\Vert_{U^{s+1}(\Omega)}\gg 1$, there exists an $s$-step nilsequence $\phi\colon \Omega\to \C$ of low complexity such that $\vert\E_{n\in \Omega}f(n)\phi(n)\vert\gg 1$. This result can be generalized to the case $\Omega=\V$ without much difficulty.

In this paper, we are particularly interested in the case when $$\Omega=V(M):=\{n\in\V\colon M(n)=0\}$$ is the set of zeros of some quadratic form $M\colon\V\to\F_{p}$ (see Section \ref{3:s:defn} for definitions), since such a norm is strongly connected to our study of the geometric Ramsey conjecture in \cite{SunD}. A typical example of $M$ is $M(n)=n\cdot n-r$ for some $r\in\F_{p}$, in which case $V(M)$ is the sphere of square radius $r$ centered at $\bold{0}$.  
For convenience we call $\Vert\cdot\Vert_{U^{s}(V(M))}$ the \emph{$s$-th spherical Gowers norm}. 
Let $\Nil^{s;C,1}_{p}(\V)$ denote the set of all scalar valued $p$-periodic nilsequences of degree at most $s$ and complexity at most $C$ (see Section \ref{3:s:nn} for the precise definition).
Throughout this paper, for $s\in\Z$, denote
\begin{equation}\label{3:thisisns}
N(s):=(2s+16)(15s+453).
\end{equation}

We proof the following spherical Gowers inverse theorem of step $s$ (abbreviated as $\SGI(s)$):

\begin{thm}\label{3:inv}[%Inverse theorem for spherical Gowers norms
	$\SGI(s)$]
		For every $d\in\N_{+},s\in\N$ with $d\geq N(s-1)=(2s+14)(15s+438)$ 	and $\e>0$, there exist $\delta:=\d(d,\e), C:=C(d,\e)>0$, and $p_{0}:=p_{0}(d,\e)\in\N$
		 such that for every prime $p\geq p_{0}$, every
		 non-degenerate quadratic form $M\colon\V\to\F_{p}$,  and every function $f\colon\V\to\C$ bounded in magnitude by 1, if $\Vert f\Vert_{U^{s+1}(V(M))}>\e$, then there exists
		   $\phi\in\Nil^{s;C,1}_{p}(\V)$ 
		such that
		$$\Bigl\vert\E_{n\in V(M)}f(n)\cdot \phi(n)\Bigr\vert>\delta.$$
\end{thm}

\begin{rem}
By considering the cases  $p\geq p_{0}$ and $p<p_{0}$ separately, one can show that the restriction $p\geq p_{0}$ in Theorem \ref{3:inv} is not necessary. We leave the details to the interested readers. 

It is also natural to ask whether one can remove the dependence of $\d$ and $C$ on the dimension $d$ in Theorem \ref{3:inv}. We do not pursuit this imporvement in this paper since Theorem \ref{3:inv} is good enough for us for the study in \cite{SunD}.
\end{rem}

We remark that the converse of $\SGI(s)$ also holds. See Proposition \ref{3:cinv}.

%{\color{} Improve the lower bound of $p$}

%\begin{proof}
%	Suppose we have show Theorem \ref{3:inv} for $p\geq p_{0}(d,\e,s)$.
%	Now let $p<p_{0}(d,\e,s)$ and $f\colon\V\to\C$ be bounded in magnitude by 1 with $\Vert f\Vert_{U^{s+1}(V(M))}>\e$.
%	By Lemma \ref{3:countingh}, there exists $p'(s)>0$ such that $\Gow_{s}(V(M))$ is non-empty as long as $p\geq p'(s)$.
%	By the Pigeonhole Principle, there exists some $n_{0}\in V(M)$ such that $f(n_{0})\geq \e^{2^{s+1}}$.
%	
%	Let $G/\Gamma=\T^{d}$ and $F\colon\T^{d}\to\C$ be a smooth function with $0\leq F\leq 1$ such that $F(x)=0$ for all $d_{\T^{d}}(x,\frac{\tau(n_{0})}{p}\Z^{d})\geq 1/4p$ and that $F(x)=1$ for all $d_{\T^{d}}(x,\frac{\tau(n_{0})}{p}\Z^{d})\leq 1/8p$. It is not hard to see that $\Vert F\Vert_{\Lip(\T^{d})}\leq 10p\leq 10 p_{0}(d,\e,s)$.
%	
%	Let $g\colon\V\to \R^{d}$ be the map given by $g(n)=\frac{\tau(n)}{p}$ and set $\phi(n)=F(g(n)\Z^{d})$. Then $\phi$ belongs to  $\Nil^{s;10 p_{0}(d,\e,s),1}_{p}(V(M))$.
%	Moreover, $\phi(n)$ equals to 1 if $n=n_{0}$ and equals to 0 otherwise.
%	 So by Lemma \ref{3:counting},  
%	$$\Bigl\vert\E_{n\in V(M)}f(n)\cdot \phi(n)\Bigr\vert=\frac{1}{\vert V(M)\vert}\vert f(n_{0})\vert\geq \frac{\e^{2^{s+1}}}{\vert V(M)\vert}=\e^{2^{s+1}}\cdot p^{-(d-1)}(1+O(p^{-1/2}))\gg_{d,\e,s} 1.$$
%\end{proof}	

\subsection{Outline of the proof}\label{3:s:otl}

The outline of our proof of $\SGI(s)$ is similar to the proof Theorem 1.3 of the work of Green, Tao and Ziegler \cite{GTZ12}. However, there are also many significant differences between the two results.
 For convenience we informally say that two sequences $f,g\colon\Omega\to\C$ \emph{correlate} (on $\Omega$) if the average of $fg$ over $\Omega$ is bounded away from zero. We say that $f$ and $g$ \emph{$s$-correlate} (on $\Omega$) if $fg$ correlates (on $\Omega$) with a nilsequence of step $s$ with low complexity.

\textbf{Step 1: some preliminary reductions.} We  first  show that $\SGI(2)$ holds. The case $s=2$ can be proved using pure Fourier analysis method. Our proof is similar to the method used by Lyall, Magyar and Parshall \cite{LMP19}.
Now suppose that $\SGI(s)$ holds and we prove that $\SGI(s+1)$ holds. If $\Vert f\Vert_{U^{s+2}(V(M))}\gg 1$, then for many $h\in \V$, we have that $\Vert \Delta_{h}f\Vert_{U^{s+1}(V(M)^{h})}\gg 1$, where $V(M)^{h}$ is the set of $n\in\V$ with $M(n)=M(n+h)=0$. So it follows from $\SGI(s)+$ (and variation of $\SGI(s)$ which follows immediately from $\SGI(s)$, see Proposition \ref{3:inv+}) that $\Delta_{h}f$ correlates with an $s$-step nilsequence $\chi_{h}$. Using the approximation properties developed in Appendix \ref{3:s:AppB}, we may assume without loss of generality that $\chi_{h}$ are nilcharacters. 
Step 1 is conducted in Section \ref{3:s:b0}.

\textbf{Step 2: a Furstenburg-Weiss type argument.} 
  By using the cocycle identity
 $\Delta_{h+k}f(n)=\Delta_{n}f(n+k)\Delta_{k}f(n),$ 
 we can subtract more information for $\chi_{h}$. 
 Roughly speaking, we can show that for many $(h_{1},h_{2},h_{3},h_{4})\in(\F_{p}^{d})^{4}$ with $h_{1}+h_{2}=h_{3}+h_{4}$, the function
 \begin{equation}\label{3:1234}
 n\mapsto\chi_{h_{1}}(n)\cdot\chi_{h_{2}}(n+h_{1}-h_{4})\cdot\overline{\chi}_{h_{3}}(n)\cdot\overline{\chi}_{h_{4}}(n+h_{1}-h_{4})	
 \end{equation}	
 is correlated with an $(s-1)$-step nilsequence on $V(M)^{h_{1},h_{3},h_{3}-h_{2}}$, the  set of $n\in\V$ with $M(n)=M(n+h_{1})=M(n+h_{3})=M(n+h_{3}-h_{2})=0$. This is done in Section \ref{3:s:b2}. Our method is similar to the one used in Section 8 of \cite{GTZ12}.  However, a new Fubini-type theorem is needed (Theorem \ref{3:ct}) in order to complete this step.
 
 To better understand the idea behind the proof, it is helpful to temporarily think of $\chi_{h}(n)$ as a polynomial function $\exp(P_{h}(n)+P'_{h}(n))$ for some homogeneous polynomial $P_{h}\colon\V\to\F_{p}$ of degree $s-1$ and some polynomial $P'_{h}\colon\V\to\F_{p}$ of degree at most $s-2$. In the setting of \cite{GTZ12}, roughly speaking, the Furstenburg-Weiss type argument implies that for many $h_{1},\dots,h_{4}$ with $h_{1}+h_{2}=h_{3}+h_{4}$, we have that
 $P_{h_{1}}+P_{h_{2}}=P_{h_{3}}+P_{h_{4}}$. 
 One can then use the tools from additive combinatorics to find an explicit expression for $P_{h}$.
 
 However, in our setting,  there is a lost of information  when applying the Furstenburg-Weiss type argument. As a result, we only obtain the weaker conclusion that $P_{h_{1}}+P_{h_{2}}\equiv P_{h_{3}}+P_{h_{4}} \mod J^{M}_{h_{1},h_{2},h_{3}}$ (see Appendix \ref{3:s:AppA8} for the definition) for many $h_{1},\dots,h_{4}$ with $h_{1}+h_{2}=h_{3}+h_{4}$. 
 To overcome this difficulty,    we need to make essential use of the additive combinatorial tools for $M$-ideals developed in the second part of the series \cite{SunB}.
 Using the linearization result for $M$-ideals (Theorem \ref{3:aadd}), we are able to express  $P_{h}  \mod J^{M}_{h}$ as an almost linear Freiman homomorphism.

 \textbf{Step 3: building a nilobject.} Using the information obtained in Step 2 on $\chi_{h}(n)$, we construct a nilobject and express $\chi_{h}(n)$ as $\chi(h,n)$ for some nilcharacter $\chi$ of multi-degree $(1,s)$. This is done in Sections \ref{3:s:b3}, \ref{3:s:b4} and \ref{3:s:b5}. The outline is similar to the one used in Sections 9-12 of \cite{GTZ12}. But our method also have many differences from \cite{GTZ12}. For example, in Section \ref{3:s:b3}, we need to use a more sophisticated sun-flower lemma (Lemma \ref{3:sf2}, based on the work in the second part of the series \cite{SunB}) compared with Lemma 10.10 of \cite{GTZ12}. In Section \ref{3:s:b4}, we need to use a new Ratner-type thoerem (Theorem \ref{3:rat}) based on the factorization theorem developed in  the first part of the series \cite{SunA}. In  Section \ref{3:s:b5}, in the construction of $\chi$, compared with the approach used in Section 12 of \cite{GTZ12}, we need to pay extra attention to ensure that $\chi$ is  $p$-periodic.

 \textbf{Step 4: a symmetric argument}.   
 Using the symmetric identity $\Delta_{h}\Delta_{h'}f(n)=\Delta_{h'}\Delta_{h}f(n)$, we extract further information from $\chi_{h}(n)=\chi(n,h)$ and show that $\chi_{h}(n)=\Theta(n+h)\Theta'(n)$ for some $(s+1)$-step nilsequences $\Theta$ and $\Theta'$. By some standard arguments, this implies that $f$ correlates to $\Theta$  and thus completes the proof of $\SGI(s+1)$.
This step is done in Sections  \ref{3:s:b7} (for the case $s=2$) and  \ref{3:s:b72} (for the case $s\geq 3$), which is in analogous to  Section 13 of \cite{GTZ12}. 

However, unlike in \cite{GTZ12}, the symmetric argument in our setting causes a loss of information. Therefore, in Sections  \ref{3:s:b7} and  \ref{3:s:b72}, we need to have an estimate finer than \cite{GTZ12} in order to keep track of the information for nilcharacters and to minimize the loss of information (see Remark \ref{3:sbr1}). Also as a result of the loss of information, we need to use new methods to improve the results of a significant part of Section 13 of \cite{GTZ12} in order to carry out the symmetric argument. See the discussion at the beginning of Section \ref{3:s:ss02} for more details.

\textbf{Organization of the paper.}
We provide the background material for nilsequences in Section \ref{3:s:bmn}.  Sections \ref{3:s:b0}--\ref{3:s:b72} are devoted to the proof the spherical Gowers inverse theorem, whose roles are explained in the above mentioned outline of the proof. 

In order to prevent the readers from being distracted by technical details, we put the proofs of some results in the appendices. In Appendix \ref{3:s:AppA}, we collect results from the previous two parts of the series \cite{SunA,SunB} which are used in this paper. In Appendix \ref{3:s:AppB}, we prove some approximation properties for nilsequences. In Appendix \ref{3:s:AppC}, we prove some properties on an equivalence property for nilsequences  defined in Definition \ref{3:deneq}.  In Appendix \ref{3:s:AppD}, we prove the converse of the spherical Gowers inverse theorem.

\textbf{Acknowledgements.} We thank James Leng for bringing the manuscript \cite{GTZ24} to our attention, and for pointing out a mistake in an earlier version of this paper on the approximation property for polynomial sequences with degree-rank filtrations.

\subsection{Definitions and notations}\label{3:s:defn}

  \begin{conv}
  Throughout this paper, we use
   $\tau\colon\F_{p}\to \{0,\dots,p-1\}$ to denote the natural bijective embedding, and use $\iota$ to denote the map from $\Z$ (or $\Z_{K}$ for any $K$ divisible by $p$) to $\F_{p}$ given by $\iota(n):=\tau^{-1}(n \mod p\Z)$.
	We also use 
	$\tau$ to denote the map from $\F_{p}^{k}$ to $\Z^{k}$ (or $\Z_{K}^{k}$) given by $\tau(x_{1},\dots,x_{k}):=(\tau(x_{1}),\dots,\tau(x_{k}))$,
	and
	$\iota$ to denote the map from $\Z^{k}$ (or $\Z_{K}^{k}$) to $\F_{p}^{k}$ given by $\iota(x_{1},\dots,x_{k}):=(\iota(x_{1}),\dots,$ $\iota(x_{k}))$. 
	When there is no confusion, we will not state the domain and range of $\tau$ and $\iota$ explicitly.
	
	We may also extend the domain of $\iota$ to all the rational numbers of the form $x/y$ with $(x,y)=1, x\in\Z, y\in\Z\backslash p\Z$ by setting $\iota(x/y):=\iota(xy^{\ast})$, where $y^{\ast}$ is any integer with $yy^{\ast}\equiv 1 \mod p\Z$.
  \end{conv}

Below are the notations we use in this paper:

\begin{itemize}
	\item Let $\N,\N_{+},\Z,\Q,\R,\R+,\C$ denote the set of non-negative integers, positive integers, integers, rational numbers, real numbers, positive real numbers, and complex numbers, respectively. Denote $\T:=\R/\Z$. Let $\F_{p}$ denote the finite field with $p$ elements. Let $\Z_{K}$ denote the cyclic group with $K$ elements.
		\item Throughout this paper, $d$ is a fixed positive integer and $p$ is a prime number.
		\item Throughout this paper, unless otherwise stated, all vectors are assumed to be horizontal vectors.
		\item Let $\mathcal{C}$ be a collection of parameters and $A,B,c\in\R$. We write $A\gg_{\mathcal{C}} B$ if $\vert  A\vert\geq K\vert B\vert$ and $A=O_{\mathcal{C}}(B)$ if $\vert A\vert\leq K\vert B\vert$ for some $K>0$ depending only on the parameters in $\mathcal{C}$.
%		We write $A=o_{c\to\infty;\mathcal{C}}(B)$ if for any $\e>0$, there exists $K>0$ depending only on $\e$ and  the parameters in $\mathcal{C}$ such that $\vert A\vert\leq \e\vert B\vert$ for all $c\geq K$. 
In the above definitions, we allow the set $\mathcal{C}$ to be empty. In this case $K$ will be a universal constant.
\item Let $[N]$ denote the set $\{0,\dots,N-1\}$.
	\item
	For $i=(i_{1},\dots,i_{k})\in\Z^{k}$, denote   $\vert i\vert:=\vert i_{1}\vert+\dots+\vert i_{k}\vert$.
	 For $n=(n_{1},\dots,n_{k})\in\Z^{k}$ and $i=(i_{1},\dots,i_{k})\in\N^{k}$, denote $n^{i}:=n_{1}^{i_{1}}\dots n_{k}^{i_{k}}$ %$\vert i\vert:=i_{1}+\dots+i_{k}$, 
 and $i!:=i_{1}!\dots i_{k}!$. 
For $n=(n_{1},\dots,n_{k})\in\N^{k}$ and $i=(i_{1},\dots,i_{k})\in\N^{k}$, denote $\binom{n}{i}:=\binom{n_{1}}{i_{1}}\dots \binom{n_{k}}{i_{k}}$.	
	\item We say that $\mathcal{F}\colon \R_{+}\to\R_{+}$ is a \emph{growth function} if $\mathcal{F}$ is strictly increasing and $\mathcal{F}(n)\geq n$ for all $n\in \R_{+}$.
%	\item For any set $F$, let $F[x_{1},\dots,x_{k}]$ denote the set of all polynomials in the variables $x_{1},\dots,x_{k}$ whose coefficients are from $F$. 
\item  For any subset $R$ of $\R$, let $\st_{R,d}(s)$ denote the set of all the homogeneous polynomial in $\R[x_{1},\dots,x_{d}]$ of degree $s$ with coefficients in $R$.
	\item For $D\in\N_{+}$, let $\mathbb{S}^{D}$ denote the set of $(z_{1},\dots,z_{D})\in\C^{D}$ with $\vert z_{1}\vert^{2}+\dots+\vert z_{D}\vert^{2}=1$.
	 \item For $D,D'\in\N_{+}$, $v=(v_{1},\dots,v_{D})\in\C^{D}$ and $w=(w_{1},\dots,w_{D'})\in\C^{D'}$. Let $v\otimes w:=(v_{1}w_{1},\dots,v_{D}w_{D'})\in\C^{DD'}$. We remark that if $v\in\mathbb{S}^{D}$ and $w\in\mathbb{S}^{D'}$,  then $v\otimes w\in\mathbb{S}^{DD'}$. Similarly, if $X$ is a set and $f\colon X\to\mathbb{C}^{D}$, $g\colon X\to\mathbb{C}^{D'}$ are functions, then we use $f\otimes g\colon X\to\mathbb{C}^{DD'}$ to denote the function $(f\otimes g)(x):=f(x)\otimes g(x)$.
%	 \item For $N\in\N$, let $[N]:=\{0,\dots,N-1\}$.
    \item For $x\in\R$, let $\lfloor x\rfloor$ denote the largest integer which is not larger than $x$, and $\lceil x\rceil$ denote the smallest integer which is not smaller than $x$. Let $\{x\}:=x-\lfloor x\rfloor$.
		\item For $D\in\N_{+}$, a set $X$ and a function $f\colon X\to\mathbb{C}^{D}$, let $\overline{f}\colon X\to\mathbb{C}^{D}$ denote the function obtained by taking the conjugate of each coordinates of $f$.
	\item Let $D\in\N_{+}$, $X$ be a finite set and $f\colon X\to\C^{D}$ be a function. Denote $\E_{x\in X}f(x):=\frac{1}{\vert X\vert}\sum_{x\in X}f(x)$, the average of $f$ on $X$.
	\item We say that a set $\Omega\subseteq \Z^{k}$ is \emph{$Q$-periodic} if $\Omega=\Omega+Q\Z^{k}$. 
	\item If $\Omega\subseteq \Z^{k}$ is  a $Q$-periodic set and $f\colon \Z^{k}\to\C$ is such that $f(n)=f(n+Qm)$ for all $m,n\in\Z^{k}$ with $n,n+Qm\in\Omega$. Then we denote $\E_{x\in\Omega}f(x):=\E_{x\in\Omega\cap [Q]^{k}}f(x)$.
	\item For $F=\Z^{k}$ or $\F_{p}^{k}$, and $x=(x_{1},\dots,x_{k}), y=(y_{1},\dots,y_{k})\in F$, let $x\cdot y\in \Z$ or $\F_{p}$ denote the dot product given by
	$x\cdot y:=x_{1}y_{1}+\dots+x_{k}y_{k}.$
%	\item Let  $\tau\colon\F_{p}\to \{0,\dots,p-1\}$ denote the natural bijective embedding. Let $\iota\colon \Z\to\F_{p}$ denote the map given by $\iota(n):=\tau^{-1}(n \mod p\Z)$.
%	We also use 
%	$\tau$ to denote the map from $\F_{p}^{k}$ to $\Z^{k}$ given by $\tau(x_{1},\dots,x_{k}):=(\tau(x_{1}),\dots,\tau(x_{k}))$,
%	and
%	$\iota$ to denote the map from $\Z^{k}$ to $\F_{p}^{k}$ given by $\iota(x_{1},\dots,x_{k}):=(\iota(x_{1}),\dots,\iota(x_{k}))$. 
		\item Let $\exp\colon \R\to\C$ denote the function $\exp(x):=e^{2\pi i x}$. 
	\item If $G$ is a connected, simply connected Lie group, then we use  $\log G$ to denote its Lie algebra. Let $\exp\colon \log G\to G$ be the exponential map, and $\log\colon G\to \log G$ be the logarithm map. For $t\in\R$ and $g\in G$, denote $g^{t}:=\exp(t\log g)$.
	\item If $f\colon H\to G$ is a function from an abelian group $H=(H,+)$ to some group $(G,\cdot)$, denote $\Delta_{h} f(n):=f(n+h)\cdot f(n)^{-1}$ for all $n,h\in H$.
	\item If $f\colon H\to \C^{D}$ is a function on an abelian group $H=(H,+)$, denote $\Delta_{h} f(n):=f(n+h)\otimes\overline{f}(n)$ for all $n,h\in H$.
%	\item If $f\colon \Z^{k}\to \C$ or $f\colon \F_{p}^{k}\to\F_{p}$ is a polynomial, then for all $1\leq i\leq k$, let $\partial_{i}f$ denote the formal partial derivative of $f$ in the $i$-th variable.
	\item We write affine subspaces of $\V$ as $V+c$, where $V$ is a subspace of $\V$ passing through $\bold{0}$, and $c\in\V$.
	\item There is a natural correspondence between polynomials taking values in $\F_{p}$ and polynomials  taking values in $\Z/p$.   Let $F\in\poly(\V\to\F_{p}^{d'})$ and $f\in \poly(\Z^{d}\to (\Z/p)^{d'})$ be polynomials of degree at most $s$ for some $s<p$. If  $F=\iota\circ pf\circ\tau$, then we say that $F$ is \emph{induced} by $f$ and $f$ is a \emph{lifting} of $F$.\footnote{As is explained in \cite{SunA},   $\iota\circ pf\circ\tau$ is well defined.}
     We say that $f$ is a  \emph{regular lifting} of $F$ if in addition $f$ has the same degree as $F$ and $f$ has $\{0,\frac{1}{p},\dots,\frac{p-1}{p}\}$-coefficients. 
     \item Let $(X,d_{X})$ be a metric space.
The \emph{Lipschitz norm} of a function $F\colon X\to\C^{D}$ is defined as 
$$\Vert F\Vert_{\Lip(X)}:=\sup_{x\in X}\vert F(x)\vert+\sup_{x,y\in X, x\neq y}\frac{\vert F(x)-F(y)\vert}{d_{X}(x,y)}.$$
For a vector valued function $F=(F_{1},\dots,F_{D})\colon X\to\C^{D}$, we denote $\Vert F\Vert_{\Lip(X)}:=\max_{1\leq i\leq D}\Vert F_{i}\Vert_{\Lip(X)}$.
When there is no confusion, we write $\Vert F\Vert_{\Lip}=\Vert F\Vert_{\Lip(X)}$ for short. 
For a subset $A$  of $\C^{D}$, 
we say that $F\colon X\to A$ is a \emph{Lipschitz function} if $\Vert F\Vert_{\Lip(X)}<\infty$.
Let $\Lip(X\to A)$ denote the set of all Lipschitz functions $F\colon X\to A$.
 \end{itemize}	

 Let $D,D'\in\N_{+}$ and $C>0$. 
Here are some basic notions of complexities:

	\begin{itemize}
		\item \textbf{Real and complex numbers:} a  number $r\in\R$ is of \emph{complexity} at most $C$ if $r=a/b$ for some $a,b\in\Z$ with $-C\leq a,b\leq C$. If $r\notin \Q$, then we say that the complexity of $r$ is infinity. A complex number is of \emph{complexity} at most $C$ if both its real and imaginary parts are of complexity at most $C$.
	   \item \textbf{Vectors and matrices:} a vector or matrix is of \emph{complexity} at most $C$ if all of its entries are of  complexity at most $C$.
	   \item \textbf{Subspaces:} a subspace of $\R^{D}$ is of \emph{complexity} at most $C$ if it is the null space of a matrix of complexity at most $C$.
	  \item \textbf{Linear transformations:} let $L\colon\C^{D}\to\C^{D'}$ be a linear transformation. Then $L$ is associated with an $D\times D'$ matrix $A$ in $\C$. We say that $L$ is of \emph{complexity} at most $C$ if  $A$ is of complexity at most $C$.
	  \item \textbf{Lipschitz function:} the \emph{complexity} of a Lipschitz function is defined to be its Lipschitz norm.
	\end{itemize}

We also need to recall the notations regarding quadratic forms defined in \cite{SunA}.	
	
\begin{defn}	
 We say that a function $M\colon\V\to\F_{p}$ is a \emph{quadratic form} if 
	$$M(n)=(nA)\cdot n+n\cdot u+v$$
	for some $d\times d$ symmetric matrix $A$ in $\F_{p}$, some $u\in \F_{p}^{d}$ and $v\in \F_{p}$.
We say that $A$ is the matrix \emph{associated to} $M$.
We say that $M$ is \emph{pure} if $u=\bold{0}$.
We say that $M$ is \emph{homogeneous} if $u=\bold{0}$ and $v=0$. We say that $M$ is \emph{non-degenerate} if $M$ is of rank $d$, or equivalently, $\det(A)\neq 0$.
\end{defn}

We use $\rank(M):=\rank(A)$ to denote the \emph{rank} of $M$.
Let $V+c$ be an affine subspace of $\V$ of dimension $r$. There exists a (not necessarily unique) bijective linear transformation $\phi\colon \F_{p}^{r}\to V$.
 We define the \emph{rank} $\rank(M\vert_{V+c})$ of $M$  restricted to $V+c$ as the rank of the quadratic form $M(\phi(\cdot)+c)$. It was proved in \cite{SunA} that   $\rank(M\vert_{V+c})$ is independent of the choice of $\phi$.

We also need to use the following notions.

\begin{itemize}
\item For a polynomial $P\in\poly(\F_{p}^{k}\to\F_{p})$, let $V(P)$ denote the set of $n\in\F_{p}^{k}$ such that $P(n)=0$.
\item Let $r\in\N_{+}$, $h_{1},\dots,h_{r}\in \V$ and $M\colon\V\to\F_{p}$ be a quadratic form. Denote $$V(M)^{h_{1},\dots,h_{r}}:=\cap_{i=1}^{r}(V(M(\cdot+h_{i}))\cap V(M).$$
\item Let $\Omega$ be a subset of $\V$ and $s\in\N$.  Let 
$\Gow_{s}(\Omega)$ denote the set of $(n,h_{1},\dots,h_{s})\in(\V)^{s+1}$ such that $n+\e_{1}h_{1}+\dots+\e_{s}h_{s}\in\Omega$ for all $(\e_{1},\dots,\e_{s})\in\{0,1\}^{s}$.
Here we allow $s$ to be 0, in which case $\Gow_{0}(\Omega)=\Omega$.
We say that $\Gow_{s}(\Omega)$ is the \emph{$s$-th Gowers set} of $\Omega$.
\end{itemize}

Quadratic forms can also be defined in the $\Z/p$-setting.

\begin{defn}	
   We say that a function $M\colon\Z^{d}\to\Z/p$ is a \emph{quadratic form} if 
	$$M(n)=\frac{1}{p}((nA)\cdot n+n\cdot u+v)$$
	for some $d\times d$ symmetric matrix $A$ in $\Z$, some $u\in \Z^{d}$ and $v\in \Z$.
We say that $A$ is the matrix \emph{associated to} $M$. 
 \end{defn}

By Lemma %\ref{1:lifting}
A.1 of \cite{SunA},
any quadratic form $\tilde{M}\colon\Z^{d}\to\Z/p$ associated with the matrix $\tilde{A}$ induces a quadratic form $M:=\iota\circ p\tilde{M}\circ\tau\colon\F_{p}^{d}\to\F_{p}$ associated with the matrix $\iota(\tilde{A})$. Conversely, any quadratic form $M\colon\F_{p}^{d}\to\F_{p}$ associated with the matrix $A$ admits a regular lifting $\tilde{M}\colon\Z^{d}\to\Z/p$, which is a quadratic form associated with the matrix $\tau(A)$.

For a quadratic form $\tilde{M}\colon\Z^{d}\to\Z/p$, we say that $\tilde{M}$ is \emph{pure/homogeneous/$p$-non-degenerate} if the  quadratic form $M:=\iota\circ p\tilde{M}\circ\tau$ induced by $\tilde{M}$ is pure/homogeneous/non-degenerate. 
The \emph{$p$-rank} of $\tilde{M}$, denoted by $\rank_{p}(\tilde{M})$, is defined to be the rank of $M$.

%We say that $h_{1},\dots,h_{k}\in\Z^{d}$ are \emph{$p$-linearly independent} if for all $c_{1},\dots,c_{k}\in\Z/p$, $c_{1}h_{1}+\dots+c_{k}h_{k}\in\Z$
% implies that $c_{1},\dots,c_{k}\in\Z$, or equivalently, if $\iota(h_{1}),\dots,\iota(h_{k})$ are linearly independent.

 We also need to use the following notions.

\begin{itemize}
\item For a polynomial $P\in\poly(\Z^{k}\to\R)$, let $V_{p}(P)$ denote the set of $n\in \Z^{k}$ such that $P(n+pm)\in \Z$ for all $m\in\Z^{k}$.
\item For $\Omega\subseteq\Z^{d}$ and $s\in\N$, let $\Gow_{p,s}(\Omega)$ denote the set of $(n,h_{1},\dots,h_{s})\in(\Z^{d})^{s+1}$ such that $n+\e_{1}h_{1}+\dots+\e_{s}h_{s}\in\Omega+p\Z^{d}$ for all $\e_{1},\dots,\e_{s}\in\{0,1\}$. 
We say that $\Gow_{p,s}(\Omega)$ is the \emph{$s$-th $p$-Gowers set} of $\Omega$.
\end{itemize}

For all other definitions and notations which are used but not mentioned in this paper, we refer the readers to Appendix \ref{3:s:AppA}  for details.

\section{Background material for nilsequence}\label{3:s:bmn}

We start with some basic definitions on nilmanifolds.
Many notations we use are similar to the ones used in Section 6 of \cite{GTZ12}. Unlike in the first part of the series \cite{SunA}, in this paper, instead of just working with $\N$-filtered nilmanifold, we also need to work with multi-degree and degree-rank filtrations. As a result, there are overlappings between Sections \ref{3:s:nfn}-\ref{3:s:nnh1} of this paper and Section %\ref{1:s:pp3}
3 of \cite{SunA} since we need to  extend definitions and results in \cite{SunA} to more general filtrations.
On the other hand, the materials in Sections \ref{3:s:nn2}-\ref{3:s:nn} are new and did not appear in \cite{SunA}.

\subsection{Nilmanifolds, filtrations and polynomial sequences}\label{3:s:nfn}

\begin{defn}[Ordering]
	An \emph{ordering} $I=(I,\prec,+,0)$ is a set $I$ equipped with a partial ordering $\prec$, a binary operation $+\colon I\times I\to I$, and a distinguished element $0\in I$ with the following properties:
	\begin{enumerate}[(i)]
		\item the operation $+$ is commutative and associative, and has 0 as the identity element;
		\item the partial ordering $\prec$ has 0 as the minimal element;
		\item if $i,j\prec I$ are such that $i\prec j$, then $i+k\prec j+k$ for all $k\in I$;
		\item for every $j\in I$, the initial segment $\{i\in I\colon i\prec d\}$ is finite.
	\end{enumerate}	
	A \emph{finite downset} in $I$ is a finite subset $J$ of $I$ with the property that $j\in J$ whenever $j\in I$ and $j\prec i$ for some $i\in J$.
\end{defn}	

In this paper, we will use the following special orderings:
	\begin{enumerate}[(i)]
		\item the \emph{degree ordering}, where $I=\N$ with the usual ordering, addition and zero element;
		\item the \emph{multi-degree ordering}, where $I=\N^{k}$ with the usual addition and zero element, and with the product ordering, i.e. $(i'_{1},\dots,i'_{k})\preceq (i_{1},\dots,i_{k})$ if $i_{j}'\leq i_{j}$ for all $1\leq j\leq k$;
		\item the \emph{degree-rank ordering}, where $I$ is the sector $\DR:=\{[s,r]\in\N^{2}\colon 0\leq r\leq s\}$ with the usual addition and zero element, and the lexicographical ordering, i.e. $[s',r']\prec [s,r]$ if $s'<s$ or if $s'=s$ and $r'<r$.\footnote{We write vectors in $\DR$ as $[s,r]$ instead of $(s,r)$ in order to distinguish them from the $\N^{2}$ multi-degree ordering.} 
	\end{enumerate}	
It is not hard to verify that all these three concepts are indeed orderings. We remark that the degree and the degree-rank orderings are \emph{total orderings}, meaning that for $i,j\in I$, either $i\preceq j$ or $j\preceq i$. However, the multi-degree ordering is not a total ordering.

Let $I$ be the degree, multi-degree or degree-rank ordering. We say that the \emph{dimension} of $I$ (denoted as $\dim(I)$) is $1$ if $I=\DR$, and is $k$ if $I=\N^{k}$.

Let $G$ be a group and $g,h\in G$. Denote $[g,h]:=g^{-1}h^{-1}gh$. For subgroups $H,H'$ of $G$, let $[H,H']$ denote the group generated by $[h,h']$ for all $h\in H$ and $h'\in H'$.

\begin{defn}[Filtered group]
	Let $I$ be an ordering and $G$ be a group. An \emph{$I$-pre-filtration} on $G$ is a collection $G_{I}=(G_{i})_{i\in I}$ of subgroups of $G$ indexed by $I$   such that  the following holds:
	\begin{enumerate}[(i)]
		\item for all $i,j\in I$ with $i\prec j$, we have that $G_{i}\supseteq G_{j}$;
		\item  for all $i,j\in I$, we have $[G_{i},G_{j}]\subseteq G_{i+j}$.
	\end{enumerate}	
	
	In addition, we say that  $G_{I}=(G_{i})_{i\in I}$ is an \emph{$I$-filtration} on $G$ if one of the followings hold:
	\begin{enumerate}[(i)]
		\item $I=\N$ is the degree ordering and $G=G_{0}=G_{1}$;
		\item  $I=\N^{k}$ is the multi-degree filtration and $G=G_{\bold{0}}$ is equal to the group generated by $G_{e_{1}},\dots,G_{e_{k}}$;
		\item $I=\N^{2}$ is the degree-rank filtration and $G=G_{[0,0]}$, $G_{[i,0]}=G_{[i,1]}$ for all $i\geq 1$.
	\end{enumerate}
	
	For $s\in I$, we say that $G$ is an \emph{($I$-filtered) (pre-)nilpotent group} of \emph{degree} at most $s$ (or of \emph{degree} $\preceq s$) with respect to some $I$-(pre-)filtration  $(G_{i})_{i\in I}$ if $G_{i}$ is trivial whenever $i\not\preceq s$. For a downset $J$ of $I$, we say that $G$   is an \emph{($I$-filtered) (pre-)nilpotent group} of  \emph{degree} $\subseteq J$ with respect to some $I$-(pre-)filtration  $(G_{i})_{i\in I}$ if $G_{i}$ is trivial whenever $i\notin J$. 
	\end{defn}

In this paper, we will use the following pre-filtrations.
\begin{enumerate}[(i)]
	\item Let $(d_{1},\dots,d_{k})\in\N^{k}$. A \emph{(pre-)nilpotent Lie group of multi-degree $\leq (d_{1},\dots,d_{k})$} (or a \emph{nilpotent Lie group of degree $\leq d_{1}$}, if $k=1$) is a (pre-)nilpotent $I$-filtered Lie group of degree  $\leq (d_{1},\dots,d_{k})$, where $I=\N^{k}$ is the multi-degree ordering. If $J$ is a downset,  we define (pre-)nilpotent Lie group of multi-degree $\subseteq J$ in a similar way. %In order to make it into a filtration, we further require that $G_{0}=G$.
		\item Let $[s,r]\in\DR$, A \emph{(pre-)nilpotent Lie group of degree-rank $\leq [s,r]$} is a (pre-)nilpotent $\DR$-filtered Lie group of degree $\leq [s,r]$. %with the additional assumption that $G_{[0,0]}=G$ and $G_{[i,0]}=G_{[i,1]}$ for all $i\geq 1$.
\end{enumerate}	

    \begin{rem}
    In most part of the paper, we work with genial filtrations instead of pre-filtrations. However, similar to \cite{GTZ12,GTZ24}, there are occasions where the use of pre-filtrations is unavoidable and one needs to pay careful attentions to the distinctions between  filtrations and pre-filtrations.
\end{rem}

\begin{defn}[Nilmanifold]
Let $I$ be an ordering.
	Let $\Gamma$ be a discrete and cocompact subgroup of a connected, simply-connected nilpotent Lie group $G$ with (pre-)filtration $G_{I}=(G_{i})_{i\in I}$ such that $\Gamma_{i}:=\Gamma\cap G_{i}$ is a cocompact subgroup of $G_{i}$\footnote{In some papers, such $\Gamma_{i}$ is called a \emph{rational} subgroup of $G$.} for all $i\in I$.
	Then we say that $G/\Gamma$ is an \emph{($I$-filtered) (pre-)nilmanifold}, and we use $(G/\Gamma)_{I}$ to denote the collection $(G_{i}/\Gamma_{i})_{i\in I}$ (which is called the \emph{$I$-(pre-)filtration} of $G/\Gamma$). We say that $G/\Gamma$ has degree $\preceq s$ or $\subseteq J$ with respect to $(G/\Gamma)_{I}$ if $G$ has degree $\preceq s$ or $\subseteq J$ with respect to $G_{I}$.
\end{defn}

By modifying the top level of a pre-nilmanifold accordingly, one can   convert a pre-nilmanifold to a genial nilmanifold.

  \begin{defn}[Essential components of pre-nilmanifolds]
   Let $I$ be the degree, multi-degree, or degree-rank ordering with $\dim(I)\vert k$. Let $(G/\Gamma)_{I}$ be an $I$-filtered pre-nilmanifold. When $I$ is the degree-rank filtration, we further assume that $G_{[i,0]}=G_{[i,1]}$ for all $i\geq 1$.  We define a subgroup $G'$ of $G$ as follows:
   \begin{itemize}
       \item if $I$ is the degree ordering, then let $G':=G_{1}$;
       \item if $I$ is the multi-degree ordering, then let $G'$ denote the group generated by $G_{(1,0,\dots,0)},$ $\dots,G_{(0,\dots,0,1)}$;
       \item if $I$ is the degree ordering, then let $G':=G_{[1,0]}$.
   \end{itemize} 
   Define $G'_{\bold{0}}:=G'$ and $G'_{i}:=G_{i}$ for all $i\in I\backslash\{0\}$. It is not hard to see that $(G')_{I}$ is an $I$-filtered nilmanifold. Let $\Gamma':=G'\cap\Gamma$. Then $(G'/\Gamma')_{I}$ is an $I$-filtered nilmanifold having the same degree as $(G/\Gamma)_{I}$.  For convenience we call $(G'/\Gamma')_{I}$ the \emph{essential component} of $(G/\Gamma)_{I}$. 
   \end{defn}
   
Next we introduce some special types of nilmanifolds.

\begin{defn}[Sub-nilmanifold]\label{3:id1}
	Let $G/\Gamma$ be an $I$-filtered nilmanifold of degree $\subseteq J$ with filtration $G_{I}$ and $H$ be a rational subgroup of $G$. Then $H/(H\cap \Gamma)$ is also  an $I$-filtered nilmanifold of degree $\subseteq J$ with the filtration $H_{I}$ given by $H_{i}:=G_{i}\cap H$ for all $i\in I$ (see Example 6.14 of \cite{GTZ12}). We say that  $H/(H\cap \Gamma)$ is a \emph{sub-nilmanifold} of $G/\Gamma$, $H_{I}$ (or $(H/(H\cap \Gamma))_{I}$) is the filtration \emph{induced by} $G_{I}$ (or $(G/\Gamma)_{I}$).
\end{defn}	

\begin{defn}[Quotient nilmanifold]\label{3:id2}
	Let $G/\Gamma$ be an $I$-filtered nilmanifold of degree $\subseteq J$ with filtration $G_{I}$ and $H$ be a normal subgroup of $G$. Then $(G/H)/(\Gamma/(\Gamma\cap H))$ is also  an $I$-filtered nilmanifold of most $\subseteq J$ with the filtration $(G/H)_{I}$ given by $(G/H)_{i}:=G_{i}/(H\cap G_{i})$ for all $i\in I$. We say that  $(G/H)/(\Gamma/(\Gamma\cap H))$ is the \emph{quotient nilmanifold} of $G/\Gamma$ by $H$ and that $(G/H)_{I}$ is the filtration \emph{induced by} $G_{I}$.
\end{defn}	

\begin{defn}[Product nilmanifold]\label{3:id4}
	Let $G/\Gamma$ and  $G'/\Gamma'$ be  $I$-filtered nilmanifolds of degree $\subseteq J$ with filtration $G_{I}$ and $G'_{I}$. Then $G\times G'/\Gamma\times\Gamma'$ is also  an $I$-filtered nilmanifold of most $\subseteq J$ with the filtration $(G\times G'/\Gamma\times\Gamma')_{I}$ given by $(G\times G'/\Gamma\times\Gamma')_{i}:=G_{i}\times G'_{i}/\Gamma_{i}\times \Gamma'_{i}$ for all $i\in I$. We say that  $G\times G'/\Gamma\times\Gamma'$ is the \emph{product nilmanifold} of $G/\Gamma$ and $G'/\Gamma'$ and that $(G\times G'/\Gamma\times\Gamma')_{I}$ is the filtration \emph{induced by} $G_{I}$ and $G'_{I}$.
\end{defn}

The following definition is based on from Example 6.11 of \cite{GTZ12}, which explains the connections between different filtrations.

\begin{defn}[Inductions of filtrations]\label{3:id3}
      Let $s\in\N_{+}$ and $G$ be an $\N$-filtered nilmanifold of degree $\leq s$ with filtration $G_{\N}$. Then for any $k\in\N_{+}$, $G_{\N}$ \emph{induces} a  multi-degree filtration $G'_{\N^{k}}$ of degree $\subseteq \{m\in\N^{k}\colon \vert m\vert\leq s\}$ given by $G'_{m}:=G_{\vert m\vert}$ for all $m\in\N^{k}$. Conversely, let $k\in\N_{+}, J\subseteq \N^{k}$, and $G$ be an $\N^{k}$-filtered nilmanifold of degree $\leq J$ with filtration $G_{\N^{k}}$. Then $G_{\N^{k}}$ \emph{induces} a  degree filtration $G'_{\N}$ of degree $\max\{\vert m\vert\colon m\in\N^{k}\}$ given by $G'_{i}:=\bigvee_{m\in\N^{k},\vert m\vert=i}G_{m}$ for all $i\in\N$, where $\vee_{a\in A}G_{a}$ is the group generated by $\cup_{a\in A}G_{a}$.
    
     Let $s\in\N_{+}$ and $G$ be an $\N$-filtered nilmanifold of degree $\leq s$ with filtration $G_{\N}$.
    Since $G=G_{0}$, $G_{\N}$ \emph{induces} a  degree-rank filtration $G'_{\DR}$ of degree $\leq [s,s]$ given by setting $G'_{[t,r]}$ to be the group generated by all the iterated commutators of $g_{i_{1}},\dots,g_{i_{m}}$ with $g_{i_{j}}\in G_{i_{j}}$ for $1\leq j\leq m$ such that either $i_{1}+\dots+i_{m}>t$ or $i_{1}+\dots+i_{m}=t$ and $m\geq \max\{r,1\}$.
     Conversely, let $s,r\in\N_{+}, r\leq s$ and $G$ be an $\DR$-filtered nilmanifold of degree $\leq [s,r]$ with filtration $G_{\DR}$. Then $G_{\DR}$ \emph{induces} a  degree  filtration $G'_{\N}$ of degree $\leq s$ given by setting $G'_{i}:=G_{[i,0]}$ for all $i\in\N$.
     
     A  multi-degree filtration $G_{\N^{k}}$ \emph{induces} a degree-rank filtration $G'_{\DR}$ by first induces a degree filtration $G''_{\N}$ and then induces from $G''_{\N}$ a degree-rank filtration $G'_{\DR}$ in the above mentioned senses. Similarly, a degree-rank filtration $G_{\DR}$   \emph{induces} a multi-degree filtration $G'_{\N^{k}}$ by first induces a degree filtration $G''_{\N}$ and then induces from $G''_{\N}$ a multi-degree filtration $G'_{\N^{k}}$ in the above mentioned senses. 
     
     Let $I,I'$ be the degree, multi-degree or degree-rank filtrations. We say that the $I$-filtration of a nilmanifold $(G/\Gamma)_{I}$ \emph{induces} the $I'$-filtration of a nilmanifold $(G'/\Gamma')_{I'}$ if $G_{I}$ induces $G'_{I'}$.
\end{defn}

We remark that although we used the same terminology ``induce" in Definitions \ref{3:id1}, \ref{3:id2}, \ref{3:id4} 
and \ref{3:id3}, this will not cause any confusion in the paper as the meaning of ``induce" will be clear from the context.

\begin{defn}[Filtered homomorphism]	
	An \emph{($I$-filtered) homomorphism} $\phi\colon G/\Gamma\to G'/\Gamma'$ between two $I$-filtered nilmanifolds is a group homomorphism $\phi\colon G\to G'$  which maps $\Gamma$ to $\Gamma'$ and maps $G_{i}$ to $G'_{i}$ for all $i\in I$.
\end{defn}

Every nilmanifold has an explicit algebraic description by using the Mal'cev basis:

\begin{defn} [Mal'cev basis]\label{3:Mal}
	Let $s\in\N_{+}$, $G/\Gamma$ be a  nilmanifold  of step at most $s$ with the $\N$-filtration  $(G_{i})_{i\in\N}$. Let $\dim(G)=m$ and $\dim(G_{i})=m_{i}$ for all $0\leq i\leq s$. A basis $\mathcal{X}:=\{X_{1},\dots,X_{m}\}$ for the Lie algebra $\log G$ of $G$ (over $\mathbb{R}$) is a \emph{Mal'cev basis} for $G/\Gamma$ adapted to the filtration $G_{\N}$ if
	\begin{itemize}
		\item for all $0\leq j\leq m-1$, $\log H_{j}:=\text{Span}_{\mathbb{R}}\{\xi_{j+1},\dots,\xi_{m}\}$ is a Lie algebra ideal of $\log G$ and so $H_{j}:=\exp(\log H_{j})$  is a normal Lie subgroup of $G;$
		\item $G_{i}=H_{m-m_{i}}$ for all $0\leq i\leq s$;
		\item the map $\psi^{-1}\colon \mathbb{R}^{m}\to G$ given by
		\begin{equation}\nonumber
		\psi^{-1}(t_{1},\dots,t_{m})=\exp(t_{1}X_{1})\dots\exp(t_{m}X_{m})
		\end{equation}	
		is a bijection;
		\item $\Gamma=\psi^{-1}(\Z^{m})$.
	\end{itemize}	
	We call $\psi$ the  \emph{Mal'cev coordinate map} with respect to the Mal'cev basis $\mathcal{X}$.
	If $g=\psi^{-1}(t_{1},\dots,t_{m})$, we say that $(t_{1},\dots,t_{m})$ are the \emph{Mal'cev coordinates} of $g$ with respect to $\mathcal{X}$. 
	
	We say that the Mal'cev basis $\mathcal{X}$  is \emph{$C$-rational} (or of \emph{complexity} at most $C$)  if all the structure constants $c_{i,j,k}$ in the relations
	   	$$[X_{i},X_{j}]=\sum_{k}c_{i,j,k}X_{k}$$
	   	are rational with complexity at most $C$.
\end{defn}

For any $h\in G$, there is a unique way to write $h$ as
  $h=\{h\}[h]$
such that $\psi(\{h\})\in [0,1)^{m}$ and $[h]\in\Gamma$. We adopt this notation throughout this paper.

It is known that for every $\N$-filtration $G_{\N}$ which is rational for $\Gamma$, there exists a Mal'cev basis adapted to it. See for example the discussion on pages 11--12 of \cite{GT12b}.

If $G/\Gamma$ is a nilmanifold with a multi-degree filtration $(G_{i})_{i\in\N^{k}}$. Then there is a  Mal'cev basis $\mathcal{X}$ of $G/\Gamma$ adapted to the  degree filtration $(G'_{i})_{i\in\N}$, where $G'_{i}$ is the group generated by $G_{(i_{1},\dots,i_{k})}, i_{1}+\dots+i_{k}=i$ for all $i\in\N$. We say that $\mathcal{X}$ is the  \emph{Mal'cev basis} of $G/\Gamma$ adapted to $(G_{i})_{i\in\N^{k}}$.

If $G/\Gamma$ is a nilmanifold with a degree-rank filtration $(G_{[s,r]})_{[s,r]\in\DR}$. Then there is a  Mal'cev basis $\mathcal{X}$ of $G/\Gamma$ adapted to the  degree filtration $(G'_{i})_{i\in\N}$, where $G'_{i}=G_{[i,0]}$ for all $i\in\N$. We say that $\mathcal{X}$ is the  \emph{Mal'cev basis} of $G/\Gamma$ adapted to $(G_{[s,r]})_{[s,r]\in\DR}$.

We use the following quantities to describe the complexities of the objected defined above.

 \begin{defn}[Notions of complexities for nilmanifolds]
Let $G/\Gamma$ be a nilmanifold with filtration $G_{I}$ and a Mal'cev basis $\mathcal{X}=\{X_{1},\dots,X_{D}\}$ adapted to it. 
We say that  $G/\Gamma$ is of \emph{complexity} at most $C$ if the  Mal'cev basis $\mathcal{X}$ is $C$-rational and $\dim(G)\leq C$. 

 An element $g\in G$ is of \emph{complexity} at most $C$ (with respect to the Mal'cev coordinate map $\psi\colon G/\Gamma\to\R^{m}$) if $\psi(g)\in [-C,C]^{m}$.

Let $G'/\Gamma'$ be a nilmanifold  endowed with the Mal'cev basis  $\mathcal{X}'=\{X'_{1},\dots,X'_{D'}\}$ respectively. Let $\phi\colon G/\Gamma\to G'/\Gamma'$ be a filtered homomorphism, we say that $\phi$ is of  \emph{complexity} at most $C$ if the map $X_{i}\to\sum_{j}a_{i,j}X'_{j}$ induced by $\phi$ is such that all $a_{i,j}$ are of complexity at most $C$.

 Let $G'\subseteq G$  be a closed connected subgroup. We say that  $G'$ is \emph{$C$-rational} (or of \emph{complexity} at most $C$) relative to $\mathcal{X}$ if the Lie algebra $\log G$ has a basis consisting of linear combinations $\sum_{i}a_{i}X_{i}$ such that $a_{i}$ are rational numbers of complexity at most $C$.
\end{defn}

\begin{conv}
 In the rest of the paper, all nilmanifolds are assumed to have a fixed filtration, Mal'cev basis and a smooth Riemannian metric induced by the Mal'cev basis. Therefore, we will simply say that a nilmanifold, Lipschitz function, sub-nilmanifold etc. is of complexity $C$ without mentioning the reference filtration and Mal'cev basis.
\end{conv}

\begin{defn}[Polynomial sequences]
	Let $d, k\in\N_{+}$ and $G$ be a connected simply-connected (pre-)nilpotent Lie group. 
	\begin{enumerate}[(i)]
		\item Let  $H=\Z^{d}$ or $\R^{d}$\footnote{In this paper, we mainly consider polynomials sequences from $\Z^{d}$ to $G$. The only place we use polynomials sequences from $\R^{d}$ to $G$ is Theorem \ref{3:papprox}.} and $(G_{i})_{i\in\N}$ be a degree (pre-)filtration of $G$. A map $g\colon H\to G$ is a \emph{($\N$-filtered) $d$-integral polynomial sequence} if
		$$\Delta_{h_{m}}\dots \Delta_{h_{1}} g(n)\in G_{m}$$
		 for all $m\in\N$ and $n,h_{1},\dots,h_{m}\in H$.\footnote{Recall that $\Delta_{h}g(n):=g(n+h)g(n)^{-1}$
 for all $n, h\in H$.}
		\item Let $H=(\Z^{d})^{k}$ or $(\R^{d})^{k}$ and $(G_{i})_{i\in\N^{k}}$ be a multi-degree (pre-)filtration of $G$.
		Let 	$e_{i,j}, 1\leq i\leq k, 1 \leq j\leq d$ be the standard unit vectors in $(\Z^{d})^{k}$.
 		 A map $g\colon H\to G$ is a \emph{($\N^{k}$-filtered) $d$-integral polynomial sequence} if
		$$\Delta_{e_{i_{m},j_{m}}}\dots \Delta_{e_{i_{1},j_{1}}} g(n)\in G_{t}$$
		for all $m\in\N$, $1\leq i_{1},\dots,i_{m}\leq k, 1\leq j_{1},\dots,j_{m}\leq d$, and $n\in H$, where $t=(t_{1},\dots,t_{k})\in\N^{k}$ is such that $t_{\ell}$ equals to the number of $i_{1},\dots,i_{m}$ which equals to $\ell$ for all $1\leq \ell\leq k$.
		\item Let $H=\Z^{d}$ or $\R^{d}$ and $(G_{[s,r]})_{[s,r]\in\DR}$ be a degree-rank (pre-)filtration of $G$. A map $g\colon H\to G$ is a \emph{($\DR$-filtered) $d$-integral polynomial sequence} if
		$$\Delta_{h_{m}}\dots \Delta_{h_{1}} g(n)\in G_{[m,0]}$$
		for all $m\in\N$ and $n,h_{1},\dots,h_{m}\in H$.
	\end{enumerate}	
	
    The sets of all $\N$-, $\N^{k}$-, and $\DR$-filtered $d$-integral polynomial sequences are denoted by $\poly(H\to G_{\N})$, $\poly(H^{k}\to G_{\N^{k}})$ and $\poly(H\to G_{\DR})$, respectively.\footnote{Polynomial sequences can be defined for more general groups $H$ and for more general filtrations (see for example \cite{GTZ12}). But we do not need them in this paper.}
  \end{defn}

\begin{rem}
	Our definition of polynomial sequences coincides with the one in \cite{GTZ12} when $d=1$. When $d>1$, the set $\poly(\Z^{d}\to G_{\N})$
	defined in this paper is the set $\poly(\Z^{d}\to G'_{\N^{d}})$ defined in \cite{GTZ12}, where $G'_{(i_{1},\dots,i_{d})}=G_{i_{1}+\dots+i_{d}}$ for all $i_{1},\dots,i_{d}\in\N$, and the set $\poly((\Z^{d})^{k}\to G_{\N})$
	defined in this paper is the set $\poly((\Z^{d})^{k}\to G'_{\N^{dk}})$ defined in \cite{GTZ12}, where $G'_{(i_{1,1},\dots,i_{k,d})}=G_{(i_{1,1}+\dots+i_{1,d},\dots,i_{k,1}+\dots+i_{k,d})}$ for all $i_{1,1},\dots,i_{k,d}\in\N$. The reason we adopt a different notation from  \cite{GTZ12} is that the new notation is more convenient to use since we frequently view vectors in $\Z^{dk}$ as $k$-tuples of elements in $\Z^{d}$ in this paper.	
\end{rem}

 \begin{conv}\label{3:ckk}
 	In the rest of the paper, we will omit the phrase ``$d$-integral" when mentioning a polynomial sequence, since the quantity $d$ will always be clear from the context. For example,
 	let $I$ be the degree, multi-degree or degree-rank ordering, $k\in\N_{+}$ with $\dim(I)\vert k$, and $G/\Gamma$ be an $I$-filtered nilmanifold. Denote $d=k/\dim(I)$. When  writing $\poly(\Omega\to G_{I})$, we regard $\Z^{k}$ as $(\Z^{d})^{\dim(I)}$, and $\poly(\Z^{k}\to G_{I})$ is understood as $\poly((\Z^{d})^{\dim(I)}\to G_{I})$, the set of $d$-integral 
 	polynomial sequences (where $d$ is uniquely determined by $I$ and $k$).
\end{conv}	
 
By Corollary B.4 of \cite{GTZ12}, all of $\poly(\Z^{k}\to G_{\N})$, $\poly((\Z^{k})^{k'}\to G_{\N^{k'}})$, $\poly(\Z^{k}\to G_{\DR})$,  and $\poly((\Z^{k})^{k'}\to G_{\N})$ are groups with respect to the pointwise multiplicative operation, provided that the corresponding $G_{I}$ are genial filtrations.
We refer the readers to Appendix B of \cite{GTZ12} for more properties for   polynomial sequences.

On the other hand, we caution the readers that these sets are not closed under the pointwise multiplicative operation when  $G_{I}$ is not a filtration, and thus Corollary B.4 of \cite{GTZ12} does not apply to pre-filtrations. However, we have the following simple lemma which is good enough for our purposes:
    
  \begin{lem}\label{3:pre3g}
      Let $k\in\N_{+}$, $I$ be the degree, multi-degree or degree-rank ordering with $\dim(I)\vert k$, and $(G/\Gamma)_{I}$ be an $I$-filtered pre-nilmanifold. Let $(G'/\Gamma')_{I}$ be the essential component of $(G/\Gamma)_{I}$.
     Let $g\in \poly(\Z^{k}\to (G'/\Gamma')_{I})$ and $h\in G$. Then $hg$ belongs to  $ \poly(\Z^{k}\to (G/\Gamma)_{I})$.
  \end{lem}
  \begin{proof}
  The conclusion follows from the indentity 
      $$\Delta_{h_{m}}\dots\Delta_{h_{1}}(hg)(n)=\Delta_{h_{m}}\dots\Delta_{h_{1}}h(n)$$
for all $m\geq 1$ and $h_{1},\dots,h_{m}\in \Z^{k}$, and from the fact that $G'_{i}=G_{i}$ for all $i\neq \bold{0}$.
  \end{proof}

We conclude this section with a lemma regarding the expression of a polynomial sequence using Mal'cev basis.

\begin{lem}\label{3:malfil}
	Let $s\in\N_{+}$ and $G$ be an $s$-step connected, simply-connected nilpotent Lie group with the degree filtration $G_{\N}$ and Mal'cev basis $\mathcal{X}=\{X_{1},\dots,X_{m}\}$, where $m=\dim(G)$. Denote $m_{i}:=\dim(G_{i})$.
	Suppose that $G_{i}$ is generated by $X_{m-m_{i}+1},\dots,X_{m}$. Let $g\colon\Z^{d}\to G$ be a map given by
	$$g(n)=\prod_{j=1}^{m}\exp(P_{j}(n)X_{j})$$
	for some polynomials $P_{j}\colon\Z^{d}\to\R$. Then $g\in\poly(\Z^{d}\to G_{\N})$ if and only if $\deg(P_{j})\leq i$ for all $m-m_{i}+1\leq j\leq m-m_{i+1}$.
\end{lem}	
\begin{proof}
	If $\deg(P_{j})\leq i$ for all $m-m_{i}+1\leq j\leq m-m_{i+1}$, then since $\exp((a+b)X_{j})=\exp(aX_{j})\exp(bX_{j})$ for all $a,b\in\mathbb{R}$, it is not hard to see that $n\mapsto \exp(P_{j}(n)X_{j})$ is a map in $\poly(\Z^{d}\to G_{\N})$ for all $1\leq j\leq m$. By Corollary B.4 of \cite{GTZ12}, we have that $g\in \poly(\Z^{d}\to G_{\N})$.
	
	Conversely, suppose that $g\in \poly(\Z^{d}\to G_{\N})$. We say that \emph{Property-$i$} holds if $\deg(P_{j})\leq i'$ for all $m-m_{i'}+1\leq j\leq m-m_{i'+1}$ with $i'\leq i$. Clearly Property-0 holds. Suppose that Property-$(i-1)$ holds for some $1\leq i\leq s$. Then by the converse direction and Corollary B.4 of \cite{GTZ12},
	$$\Bigl(\prod_{j=1}^{m-m_{i}}\exp(P_{j}(n)X_{j})\Bigr)^{-1}g(n)=\prod_{j=m-m_{i}+1}^{m}\exp(P_{j}(n)X_{j})$$
	belongs to $g\in \poly(\Z^{d}\to G_{\N})$. So $\prod_{j=m-m_{i}+1}^{m-m_{i+1}}\exp(P_{j}(n)X_{j}) \mod G_{i+1}$ belongs to $\poly(\Z^{d}\to (G/G_{i+1})_{\N})$. Since $G_{i}/G_{i+1}$ is abelian, by the definition of $\poly(\Z^{d}\to (G/G_{i+1})_{\N})$,
	$$\Delta_{h_{i+1}}\dots\Delta_{h_{1}}\prod_{j=m-m_{i}+1}^{m-m_{i+1}}\exp(P_{j}(n)X_{j}) \equiv\prod_{j=m-m_{i}+1}^{m-m_{i+1}}\exp(\Delta_{h_{i+1}}\dots\Delta_{h_{1}}P_{j}(n)X_{j}) \equiv 0\mod G_{i+1}$$
	for all $n,h_{1},\dots,h_{i+1}\in\Z^{d}$. This means that $\deg(P_{j})\leq i$ for all $m-m_{i}+1\leq j\leq m-m_{i+1}$. So Property-$i$ holds.
	
	Inductively we have that Property-$s$ holds and we are done.
\end{proof}

\subsection{Different notions of polynomial sequences}
 
We may extend the concept of partially periodic polynomial sequences defined in \cite{SunA}  naturally for more general filtrations.

\begin{defn}[null, rational and periodic polynomial sequences]\label{1:defr}
       Let $I$ be the degree, multi-degree or degree-rank ordering, $k\in\N_{+}$ with $\dim(I)\vert k$, $G/\Gamma$ be an $I$-filtered (pre-)nilmanifold,
       and   $g\in\poly(\Z^{k}\to G_{I})$. We say that $g$ is 
       \begin{itemize}
          \item \emph{null} if $g(n)\in\Gamma$ for all $n\in\Z^{k}$;
          \item \emph{rational} if each coordinate of $g(Qn)\in\Gamma$ %is a  polynomial with coefficients in $\Z/Q$ for some $Q\in\N_{+}$, 
          for all $n\in\Z^{k}$ for some $Q\in\N_{+}$, in which case we also say that $g$ is \emph{$Q$-rational};
          \item \emph{periodic} if $g(n+Qm)^{-1}g(n)\in\Gamma$ for all $m,n\in\Z^{k}$ for some $Q\in\N_{+}$, in which case we also say that $g$ is \emph{$Q$-periodic}.
       \end{itemize}
      \end{defn}

\begin{defn}[Partially null, rational and periodic polynomial sequences]
Let $I$ be the degree, multi-degree or degree-rank ordering, $k,Q\in\N_{+}$ with $\dim(I)\vert k$, $G/\Gamma$ be an $I$-filtered (pre-)nilmanifold,  $\Omega$ be a subset of $\Z^{k}$, and 
  $n_{\ast}\in\Omega$. %if $I=\N$ or $\DR$, and let $\Omega$ be a subset of $(\Z^{k})^{k'}$ if $I=\N^{k'}$.
	 \begin{itemize}
          \item We use $\poly(\Omega\to G_{I}\vert\Gamma)$ to denote the set of all $g\in \poly(\Z^{k}\to G_{I})$ which is \emph{partially null on $\Omega$,} meaning that  $g(n)\in \Gamma$ for all $n\in\Omega$.
           \item We use $\poly_{\approx Q,n_{\ast}}(\Omega\to G_{I}\vert\Gamma)$ to denote the set of all $g\in \poly(\Z^{k}\to G_{I})$ which is \emph{partially $Q$-rational on $\Omega$ with base point $n_{\ast}$} meaning that $g(n+Qm)\in \Gamma$ for all $m\in\Z^{k}$ with $n_{\ast}+Qm\in\Omega$. When $n_{\ast}=\bold{0}$, we write $\poly_{\approx Q}(\Omega\to G_{I}\vert\Gamma):=\poly_{\approx Q,n_{\ast}}(\Omega\to G_{I}\vert\Gamma)$ for short.
          \item We use $\poly_{Q}(\Omega\to G_{I}\vert\Gamma)$ to denote the set of all $g\in \poly(\Z^{k}\to G_{I})$ which is \emph{partially $Q$-periodic on $\Omega$,} meaning that $g(n+Qm)^{-1}g(n)\in \Gamma$ for all $m,n\in\Z^{k}$ with $n,n+Qm\in\Omega$.
       \end{itemize}
 \end{defn}

%\begin{defn}[Partially periodic polynomial sequences]
%	Let $I$ be the degree, multi-degree or degree-rank ordering, $k\in\N_{+}$ with $\dim(I)\vert k$, $G/\Gamma$ be an $I$-filtered (pre-)nilmanifold, $p$ be a prime, and $\Omega$ be a subset of $\Z^{k}$. 
%	Let $\poly(\Omega\to G_{I}\vert\Gamma)$ denote the set of all $g\in \poly(\Z^{k}\to G_{I})$ such that $g(n)\in \Gamma$ for all $n\in\Omega$. For $Q\in\N_{+}$, let $\poly_{Q}(\Omega\to G_{I}\vert\Gamma)$ denote the set of all $g\in \poly(\Z^{k}\to G_{I})$ such that $g(n+Qm)^{-1}g(n)\in \Gamma$ for all $m,n\in\Z^{k}$ with $n,n+Qm\in\Omega$.

%Let
 %$\poly(\F_{p}^{k}\to G_{I})$ denote the set of all functions of the form $f\circ\tau$ for some $f\in \poly(\Z^{k}\to G_{I})$.
 %For a subset $\Omega$ of $\F_{p}^{k}$, let
 %$\poly_{p}(\Omega\to G_{I}\vert\Gamma)$ denote the set of all functions of the form $f\circ\tau$ for some $f\in \poly_{p}(\iota^{-1}(\Omega)\to G_{I}\vert\Gamma)$. We define $\poly(\Omega\to G_{I}\vert\Gamma)$  similarly.

 %For $\Omega\subseteq\Z^{k}$
 %or $\F_{p}^{k}$, we call functions in $\poly_{Q}(\Omega\to G_{I}\vert\Gamma)$ \emph{partially $Q$-periodic polynomial sequences on $\Omega$.}
 %\end{defn}	

 As is explain in Example %\ref{1:notagroup2} 
 3.16 of \cite{SunA},	
  	$\poly_{p}(\Omega\to G_{I}\vert \Gamma)$ is not closed under multiplication. However, we have the following:

\begin{lem}\label{3:grg}
	Let $I$ be the degree, multi-degree or degree-rank ordering, $k,Q\in\N_{+}$ with $\dim(I)\vert k$, $G/\Gamma$ be an $I$-filtered nilmanifold, $p$ be a prime, and  $\Omega$ be a subset of $\Z^{k}$. For all $f\in\poly_{Q}(\Omega\to G_{I}\vert \Gamma)$ and $g\in\poly(\Omega\to G_{I}\vert \Gamma)$, we have that $fg\in \poly_{Q}(\Omega\to G_{I}\vert \Gamma)$.
\end{lem}
%\begin{proof}
%By Corollary B.4 of \cite{GTZ12}, $\poly(\Omega\to G_{I})$ is a group. So $fg\in \poly(\Omega\to G_{I})$. So it suffices to show that for all $n\in\Omega+p\Z^{k}$ and $m\in\Z^{k}$, we have that 
%$$(f(n+pm)g(n+pm))^{-1}(f(n)g(n))=g(n+pm)^{-1}(f(n+pm)^{-1}f(n))g(n)$$
%belongs to $\Gamma$. Since $f\in\poly_{p}(\Omega\to G_{I}\vert \Gamma)$, we have that $f(n+pm)^{-1}f(n)\in\Gamma$. Since $g\in\poly(\Omega\to G_{I}\vert \Gamma)$, we have that $g(n+pm)^{-1},g(n)\in\Gamma$. We are done.
%\end{proof}	

The proof of Lemma \ref{3:grg} is almost identical to that of Lemma %\ref{1:grg} 
3.17 in \cite{SunA}. So we omit the details.
	
	We summarize some properties for rational and periodic polynomial sequences as follows.
	
	\begin{lem}\label{3:r2p}
 Let $d,Q\in\N_{+}$, $C>0$, $I$ be the degree, multi-degree or degree-rank ordering, $s\in I$,   $(G/\Gamma)_{I}$ be an $I$-filtered nilmanifold of complexity at most $C$, and $g,g'\in\poly(\Z^{d}\to G_{I})$. 
 There exists $Q':=Q'(C,d,s)\in\N_{+}$ such that
   \begin{enumerate}[(i)]
   	   \item if $g,g'$ are $Q$-rational, then $gg'$ is $Q'Q^{Q'}$-rational;  
           \item if $g$ is $Q$-rational, then $g$ is $Q'Q^{Q'}$-periodic;
           \item if $g$ is $Q$-periodic, then  for any  $n_{\ast}\in\Z^{d}$, the map $g\in\poly(\Z^{d}\to G_{\N})$ given by 
    $$g'(n):=\{g(n_{\ast})\}^{-1}g(n)[g(n_{\ast})]^{-1}$$
     is   $(s!)^{d}Q^{s}$-rational.
      \end{enumerate}
\end{lem}
    \begin{proof}
        Part (i) can be deduced easily from the Baker-Campbell-Hausdorff formula. Part (ii) follows from an argument similar to Lemma A.12 of  \cite{GT12b}. %(see also Proposition 3.19 of \cite{SunA} for a stronger result). 
        Part (iii) can be deduced from  Lemma 3.18 of \cite{SunA}.
    \end{proof}

\iffalse

    \begin{lem}\label{3:r2p}
       Let $d,Q\in\N_{+}$, $C>0$. %$I$ be the degree, multi-degree or degree-rank ordering, $s\in I$,   %$(G/\Gamma)_{I}$ be an $I$-filtered nilmanifold of complexity at most $C$, and $g\in\poly(\Z^{d}\to G_{I})$. 
       There exists $Q':=Q'(C,d,s)\in\N_{+}$ such that every $Q$-rational nilsequence/nilcharacter of complexity at most $C$ defined on a subset of $\Z^{d}$ or $\F_{p}^{d}$ is a    $Q'Q^{Q'}$-periodic nilsequence/nilcharacter  of the same degree and complexity, and that every $Q$-periodic nilsequence/nilcharacter of complexity at most $C$ defined on a subset of $\Z^{d}$ or $\F_{p}^{d}$ is a   $Q'Q^{Q'}$-rational nilsequence/nilcharacter of the same degree and of  complexity $O(C)$.
        %
      %\begin{enumerate}[(i)]
         %  \item if $g$ is $Q$-rational, then $g$ is $Q'Q^{Q'}$-periodic;
           %\item if $g$ is $Q$-periodic, then $g$ is $Q'Q^{Q'}$-rational.
       %\end{enumerate}
    \end{lem}
    \begin{proof}
        The first part of the statement follows from Lemma A.12 of  \cite{GT12b}, and the second part of the statement can be deduced from  Lemma 3.18 of \cite{SunA}.
    \end{proof}

\fi

We may also extend the domains of polynomial sequences to $\F_{p}^{k}$ in a natural way.
 	
		\begin{defn}[Polynomial sequences in $\F_{p}^{k}$]
	Let $I$ be the degree, multi-degree or degree-rank ordering, $k\in\N_{+}$ with $\dim(I)\vert k$, $G/\Gamma$ be an $I$-filtered (pre-)nilmanifold, and $p$ be a prime. Let
 $\poly(\F_{p}^{k}\to G_{I})$ denote the set of all functions of the form $f\circ\tau$ for some $f\in \poly(\Z^{k}\to G_{I})$.
 For a subset $\Omega$ of $\F_{p}^{k}$, let
 $\poly_{p}(\Omega\to G_{I}\vert\Gamma)$ denote the set of all functions of the form $f\circ\tau$ for some $f\in \poly_{p}(\iota^{-1}(\Omega)\to G_{I}\vert\Gamma).$ We define $\poly(\Omega\to G_{I}\vert\Gamma)$  similarly. 
  We call functions in $\poly_{p}(\Omega\to G_{I}\vert\Gamma)$ \emph{partially $p$-periodic polynomial sequences on $\Omega$.}
 \end{defn}

  As is explained in Example %\ref{1:notagroup} 
  3.1823 of \cite{SunA}, partially $p$-periodic polynomial sequences are not closed under translations and transformations. However, the following proposition asserts that partially $p$-periodic  are  closed under translations and transformations modulo $\Gamma$.

       \begin{prop}\label{3:BB}
       	Let $I$ be the degree, multi-degree or degree-rank ordering, $k\in\N_{+}$ with $\dim(I)\vert k$, $G/\Gamma$ be an $I$-filtered nilmanifold,  $p$ be a prime, $\Omega$ be a subset of $\F_{p}^{k}$, and $g\in\poly_{p}(\Omega\to G_{I}\vert \Gamma)$. 
       	\begin{enumerate}[(i)]
       		\item For any $h\in\F_{p}^{k}$, there exists $g'\in\poly_{p}((\Omega-h)\to G_{I}\vert \Gamma)$   such that $g(n+h)\Gamma=g'(n)\Gamma$ for all $n\in\Omega-h$.
       		\item If $I$ is the degree filtration, then for any $k'\in\N_{+}$ and  linear transformation $L\colon \F_{p}^{k'}\to \F_{p}^{k}$, there exists  $g''\in\poly_{p}(L^{-1}(\Omega)\to G_{I}\vert \Gamma)$ such that   $g(L(n))\Gamma=g''(n)\Gamma$ for all $n\in L^{-1}(\Omega)$.	
       	\end{enumerate}	
\end{prop}	

The proof of Proposition \ref{3:BB} is almost identical to that of Proposition %\ref{1:BB} 
3.24 in \cite{SunA}. So we omit the details.

\subsection{The Baker-Campbell-Hausdorff formula}	

The material of this section comes from Appendix C of \cite{GT10b}. We write it down for completeness.

Let $G$ be a group, $t\in\N_{+}$ and $g_{1},\dots,g_{t}\in G$. The \emph{iterated commutator} of $g_{1}$ is defined to be $g_{1}$ itself. Iteratively, we define an \emph{iterated commutator} of $g_{1},\dots,g_{t}$ to be an element of the form $[w,w']$, where $w$ is an iterated commutator of $g_{i_{1}},\dots,g_{i_{r}}$, $w'$ is an iterated commutator of $g_{i'_{1}},\dots,g_{i'_{r'}}$ for some $1\leq r,r'\leq t-1$ with $r+r'=t$ and $\{i_{1},\dots,i_{r}\}\cup \{i'_{1},\dots,i'_{r'}\}=\{1,\dots,t\}$.

Similarly, let $X_{1},\dots,X_{t}$ be elements of a Lie algebra. The \emph{iterated Lie bracket} of $X_{1}$ is defined to be $X_{1}$ itself. Iteratively, we define an \emph{iterated Lie bracket} of $X_{1},\dots,X_{t}$ to be an element of the form $[w,w']$, where $w$ is an iterated Lie bracket of $X_{i_{1}},\dots,X_{i_{r}}$, $w'$ is an iterated Lie bracket of $X_{i'_{1}},\dots,X_{i'_{r'}}$ for some $1\leq r,r'\leq t-1$ with $r+r'=t$ and $\{i_{1},\dots,i_{r}\}\cup \{i'_{1},\dots,i'_{r'}\}=\{1,\dots,t\}$.

Let $G$ be a connected and simply connected nilpotent Lie group. The \emph{Baker-Campbell-Hausdorff formula} asserts that for all $X_{1},X_{2}\in\log G$, we have
$$\exp(X_{1})\exp(X_{2})=\exp\Bigl(X_{1}+X_{2}+\frac{1}{2}[X_{1},X_{2}]+\prod_{\alpha}c_{\alpha}X_{\alpha}\Bigr),$$	
where $\alpha$ is a finite set of labels, $c_{\alpha}$ are real constants, and $X_{\alpha}$ are iterated Lie brackets with $k_{1,\alpha}$ copies of $X_{1}$  and $k_{2,\alpha}$ copies of $X_{2}$ for some $k_{1,\alpha},k_{2,\alpha}\geq 1$ and $k_{1,\alpha}+k_{2,\alpha}\geq 3$. One may use this formula to show that for all $g_{1},g_{2}\in G$ and $x\in \R$, we have that
\begin{equation}\nonumber\label{1:C1}
(g_{1}g_{2})^{x}=g_{1}^{x}g_{2}^{x}\prod_{\alpha}g_{\alpha}^{Q_{\alpha}(x)},
\end{equation}	
where $\alpha$ is a finite set of labels,   $g_{\alpha}$ are iterated commutators with $k_{1,\alpha}$ copies of $g_{1}$  and $k_{2,\alpha}$ copies of $X_{2}$ for some $k_{1,\alpha},k_{2,\alpha}\geq 1$, and $Q_{\alpha}\colon\R\to\R$ are polynomials of degrees at most $k_{1,\alpha}+k_{2,\alpha}$ without constant terms.

Similarly, one can show that for any $g_{1},g_{2}\in G$ and $x_{1},x_{2}\in \R$, we have that	
\begin{equation}\nonumber%\label{1:C2}
[g_{1}^{x_{1}},g_{2}^{x_{2}}]=[g_{1},g_{2}]^{x_{1}x_{2}}\prod_{\alpha}g_{\alpha}^{P_{\alpha}(x_{1},x_{2})},
\end{equation}	
where $\alpha$ is a finite set of labels,   $g_{\alpha}$ are iterated commutators with $k_{1,\alpha}$ copies of $g_{1}$  and $g_{2,\alpha}$ copies of $X_{2}$ for some $k_{1,\alpha},k_{2,\alpha}\geq 1$, $k_{1,\alpha}+k_{2,\alpha}\geq 3$, and $P_{\alpha}\colon\R^{2}\to\R$ are polynomials of degrees at most $k_{1,\alpha}$ in $x_{1}$ and at most $k_{2,\alpha}$ in $x_{2}$ which vanishes when $x_{1}x_{2}=0$.

\subsection{Type-I horizontal  torus and character}\label{3:s:nnh1}

 In this section, we recall the definitions of type-I horizontal  torus and character defined in \cite{SunA}.

\begin{defn}[Type-I horizontal torus and character]
Let $G/\Gamma$ be a nilmanifold endowed with a Mal'cev basis $\mathcal{X}$.  The \emph{type-I horizontal torus} of $G/\Gamma$ is $G/[G,G]\Gamma$.
	A \emph{type-I horizontal character} is a continuous homomorphism $\eta\colon G\to \R$ such that $\eta(\Gamma)\subseteq \Z$.
	When written in the coordinates relative to $\mathcal{X}$, we may write $\eta(g)=k\cdot \psi(g)$ for some unique $k=(k_{1},\dots,k_{m})\in\Z^{m}$, where $\psi\colon G\to \R^{m}$ is the coordinate map with respect to the Mal'cev basis  $\mathcal{X}$. We call the quantity $\Vert\eta\Vert:=\vert k\vert=\vert k_{1}\vert+\dots+\vert k_{m}\vert$ the \emph{complexity} of $\eta$ (with respect to $\mathcal{X}$). 
\end{defn}

It is not hard to see that any type-I horizontal character mod $\Z$ vanishes on $[G,G]\Gamma$ and thus descent to a continuous homomorphism between the type-I horizontal torus $G/[G,G]\Gamma$ and $\R/\Z$. Moreover, $\eta \mod \Z$ is a well defined map from $G/\Gamma$ to $\R/\Z$. 

Type-I horizontal torus and character are used to characterize whether a nisequence is equidistributed on a nilmanifold \cite{GT12b,Lei05}.  The following lemma is a generalization of Lemma %\ref{1:B.9}
 3.21 of \cite{SunA} to multi-degree filtrations.
 Its proof is similar to that of Lemma B.9 of  \cite{GTZ12} and Lemma 2.8 of \cite{CS14}. We omit the details. 
\begin{lem}\label{3:B.9}
	Let $d,k\in\N_{+}$, $H=(\Z^{d})^{k}$ or $(\R^{d})^{k}$,
	 and $G$ be an $\N^{k}$-pre-filtered group of degree $\leq J$ for some finite downset $J$. We complete the partial order on $J$ to a total ordering in an arbitrary fashion. A function $g\colon H\to G$ belongs to $\poly(H\to G_{\N^{k}})$ if and only if 
	 for all $m=(m_{1},\dots,m_{k})\in(\N^{d})^{k}$ with $(\vert m_{1}\vert,\dots,\vert m_{k}\vert)\in J$,
	 there exists  $X_{m}\in \log(G_{(\vert m_{1}\vert,\dots,\vert m_{k}\vert)})$  such that 
	$$g(n)=\prod_{m=(m_{1},\dots,m_{k})\in(\N^{d})^{k}, (\vert m_{1}\vert,\dots,\vert m_{k}\vert)\in J}\exp\Bigl(\binom{n}{m}X_{m}\Bigr).$$
	Moreover, if $g\in\poly(H\to G_{\N^{k}})$, then the choice of $X_{m}$ are unique.
	
	Furthermore, if $G'$ is a subgroup of $G$ and $g$ takes values in $G'$, then $X_{m}\in \log(G')$ for all $m$. 
\end{lem}	

An equivalent way of saying that a function $g\colon H\to G$ belongs to $\poly(\Z^{k}\to G_{\N^{k}})$ is that 
\begin{equation}\label{3:tl}
g(n)=\prod_{m=(m_{1},\dots,m_{k})\in(\N^{d})^{k}, (\vert m_{1}\vert,\dots,\vert m_{k}\vert)\in J}g_{m}^{\binom{n}{m}}
\end{equation}
for some $g_{m}\in G_{(\vert m_{1}\vert,\dots,\vert m_{k}\vert)}$ (with any fixed order in $\N^{k}$).
We call $g_{m}$ the \emph{($m$-th) type-I Taylor  coefficient} of $g$, and (\ref{3:tl}) the \emph{type-I Taylor expansion} of $g$.

\subsection{Type-II horizontal  torus and character}\label{3:s:nn2}

In this paper, in addition the type-I horizontal  torus and character, we also need to use the type-II horizontal  torus and character, which are defined in Definition 9.6 of \cite{GTZ12} and can be used to characterize Ratner-type theorem for nilmanifolds with degree-rank filtrations.

\begin{defn}[Type-II  Taylor coefficients]
	Let $G$ be a degree-rank-filtered connected, simply-connected nilpotent Lie group with filtration $(G_{[s,r]})_{[s,r]\in\DR}$. For every $i\in\N_{+}$, define the \emph{$i$-th type-II horizontal space} $\Hor_{i}(G)$ to be the abelian group
	$$\Hor_{i}(G):=G_{[i,1]}/G_{[i,2]},$$
	with the convention that $G_{[s,r]}:=G_{[s+1,0]}$ if $r>s$ (in particular, $G_{[1,2]}=G_{[2,0]}$).
	If $\Gamma$ is a subgroup of $G$, then define 	$$\Hor_{i}(G/\Gamma):=\Hor_{i}(G)/\Hor_{i}(\Gamma).$$
		It is easy to see that the type-II horizontal spaces $\Hor_{i}(G)$ are abelian Lip groups and that $\Hor_{i}(\Gamma)$ is a sublattice of $\Hor_{i}(G)$. So $\Hor_{i}(G/\Gamma)$ is a torus, which we call the \emph{$i$-th type-II horizontal torus} of $G/\Gamma$.

	For any $d\in \N_{+}$, $g\in\poly(\Z^{d}\to G_{\DR})$ and $m=(m_{1},\dots,m_{d})\in\N^{d}$ with $\vert m\vert=i$, we define the \emph{$m$-th type-II  Taylor coefficient} $\Taylor_{m}(g)$ to be the quantity
	$$\Taylor_{m}(g):=\Delta_{e_{1}}\dots\Delta_{e_{1}}\dots\Delta_{e_{d}}\dots\Delta_{e_{d}} g(n) \mod G_{[i,2]}$$
	for any $n\in\Z^{d}$, where $\Delta_{e_{j}}$ occurs $m_{j}$ times for all $1\leq j\leq d$ and $e_{j}\in\Z^{d}$ is the $j$-th standard unit vector. Note that this map is well defined since $\Taylor_{m}(g)$ takes values in $G_{[i,1]}$ and has first derivatives in $G_{[i+1,1]}\subseteq G_{[i,2]}$ (and thus is independent of the choice of $n$). 
	Denote
	$$\Taylor_{i}(g)(h_{1},\dots,h_{i}):=\Delta_{h_{i}}\dots\Delta_{h_{1}} g(n) \mod G_{[i,2]}$$
	for $h_{1},\dots,h_{i}\in\Z^{d}$ (again this quantity is independent of the choice of $n$). 	By Corollary B.7 of \cite{GTZ12}, it is not hard to see that $\Taylor_{i}(g)(h_{1},\dots,h_{i})$ is a homogeneous polynomial of multi-degree $(1,\dots,1)$ in the variables $h_{1},\dots,h_{i}$, and is symmetric in the variables $h_{1},\dots,h_{i}$.

	If $G$ is a connected, simply-connected nilpotent Lie group with a degree filtration $G_{\N}$, then we may define the type-II  Taylor coefficient and type-II horizontal toruses similarly with $(G_{[s,r]})_{[s,r]\in\DR}$ being the natural degree-rank filtration induced by $G_{\N}$.  
\end{defn}

By Corollary B.7 of \cite{GTZ12} (see also the discussion on page 1282 of \cite{GTZ12}), we have:

\begin{lem}\label{3:magictaylor}
	Let $d\in\N_{+}$ and $G/\Gamma$ be a nilmanifold of degree at most $s$ with a filtration $G_{\N}$.
	For all $g,g'\in\poly(\Z^{d}\to G_{\N})$ and $m\in\N^{d}$, we have that 
	$$\Taylor_{m}(gg')(h_{1},\dots,h_{i})=\Taylor_{m}(g)(h_{1},\dots,h_{i})\cdot\Taylor_{m}(g')(h_{1},\dots,h_{i}) \mod G_{[i,2]}.$$
	In other words,  the map $g\mapsto \Taylor_{m}(g)$ is a homomorphism. 
\end{lem}

\begin{defn}[$i$-th type-II horizontal character]
	For $i\in\N_{+}$ and an $\N$-filtered or $\DR$-filtered nilmanifold $G/\Gamma$. An \emph{$i$-th type-II horizontal character} is a  continuous homomorphisms  $\xi_{i}\colon\Hor_{i}(G)\to\R$ with $\xi_{i}(\Hor_{i}(\Gamma))\subseteq \Z$. We use $\mathfrak{N}_{i}(G/\Gamma)$ to denote the group of all $i$-th type-II horizontal characters. 
	
		The Mal'cev basis $\mathcal{X}$ induces a natural isomorphism $\psi\colon G_{[i,1]}/G_{[i,2]}\to\R^{k}$, and thus we may write $\eta_{i}(g_{i}):=(m_{1},\dots,m_{k})\cdot \psi(g_{i})$ for some $m_{1},\dots,m_{k}\in\Z$, where $k=\dim(G_{[i,1]})-\dim(G_{[i,2]})$. 
		We call the quantity $\Vert\eta_{i}\Vert:=\vert m_{1}\vert+\dots+\vert m_{k}\vert$ the \emph{complexity} of $\eta_{i}$ (with respect to $\mathcal{X}$). 
\end{defn}

We provide  a formula for further uses.

\begin{lem}\label{3:sbeq}
	Let $d,i\in\N_{+}$, $G/\Gamma$ be an $\N$-filtered nilmanifold, $g(n)\in \poly(\Z^{d}\to G_{\N})$, and $\xi_{i}\in\mathfrak{N}_{i}(G/\Gamma)$. Then 
	\begin{equation}\label{3:2is3}
	\begin{split}
	\xi_{i}(\Taylor_{i}(g)(h_{1},\dots,h_{i}))=\Delta_{h_{i}}\dots\Delta_{h_{1}}\Bigl(\sum_{m\in\N^{d},\vert m\vert=i}\xi_{i}(\Taylor_{m}(g))\binom{n}{m}\Bigr).
	\end{split}
	\end{equation}
\end{lem}	
\begin{proof}
	By Corollary B.7 of \cite{GTZ12}, it is not hard to see that both sides of (\ref{3:2is3}) are symmetric and multi-linear in $h_{1},\dots,h_{i}$, and are independent of $n$ (see also the remark on page 1281 of \cite{GTZ12}). On the other hand, for any $m=(m_{1},\dots,m_{d})\in\N^{d}$ with $\vert m\vert=i$, if we take $m_{t}$ many $h_{1},\dots,h_{i}$ to be equal to $e_{t}$ for all $1\leq t\leq d$, then both sides of (\ref{3:2is3}) equals to $\xi_{i}(\Taylor_{m}(g))$. So by multi-linearity, we have that (\ref{3:2is3}) holds.
\end{proof}

For $p$-periodic polynomial sequences, we have some $p$-periodicity properties for their  type-II  Taylor coefficients.

\begin{lem}\label{3:ratispp}
	Let $d,i,r\in\N_{+}$, $p\gg_{d,i} 1$ be a prime, $G/\Gamma$ be an $\N$-filtered nilmanifold, $g\in\poly_{p^{r}}(\Z^{d}\to G_{\N}\vert\Gamma)$, and $\xi_{i}$ be an $i$-th type-II horizontal character of $G/\Gamma$. 
	Then the map $(h_{1},\dots,h_{i})\mapsto\xi_{i}(\Taylor_{i}(g)(h_{1},\dots,h_{i}))$
	is a homogeneous $p^{r}$-periodic symmetric polynomial of multi-degree $(1,\dots,1)$. Moreover, $\Taylor_{m}(g)^{p^{r}}\in G_{[i,2]}\Gamma$ and $\xi_{i}(\Taylor_{m}(g))\in\Z/p^{r}$ for all $m\in\N^{d}$ with $\vert m\vert=i$.
\end{lem}	
\begin{proof}
	It follows from  Corollary B.7 of \cite{GTZ12} that $\xi_{i}(\Taylor_{i}(g)(h_{1},\dots,h_{i}))$  a homogeneous symmetric polynomial of multi-degree $(1,\dots,1)$ (see also the remark on page 1281 of \cite{GTZ12}). We now show that it is $p^{r}$-periodic.
	Note that if $h_{1}\in p^{r}\Z^{d}$, then by the Baker-Campbell-Hausdorff formula, 
	$$\Delta_{h_{i}}\dots\Delta_{h_{1}}g(n)\equiv\Delta_{h_{i}}\dots\Delta_{h_{2}}(g_{h_{1}}(n)) \mod G_{[2,2]},$$
	where $g_{h_{1}}(n):=g(n)^{-1}g(n+h_{1})$ takes values in $\Gamma$. 	
So $\Delta_{h_{i}}\dots\Delta_{h_{1}}g(n)\in G_{[2,2]}\Gamma$
and thus
	\begin{equation}\label{3:cmnn}
	\begin{split}
	\xi_{i}(\Taylor_{i}(g)(h_{1},\dots,h_{i}))\equiv\xi_{i}(\Delta_{h_{i}}\dots\Delta_{h_{1}}g(n) \mod G_{[i,2]})\in\Z \text{ if } h_{1}\in p^{r}\Z^{d}.
	\end{split}
	\end{equation}   
	Since $\xi_{i}(\Taylor_{i}(g)(h_{1},\dots,h_{i}))$ is multi-linear and symmetric, it is not hard to deduce from (\ref{3:cmnn}) that $\xi_{i}(\Taylor_{i}(g)(h_{1},\dots,h_{i}))$ is $p^{r}$-periodic.
	
	 So by Lemma \ref{3:sbeq}, $\xi_{i}(\Taylor_{m}(g))\in\Z/p^{r}$ for all $m\in\N^{d}$ with $\vert m\vert=i$. Since $\xi_{i}$ is an arbitrary $i$-th type-II horizontal character of $G/\Gamma$, we have that $\Taylor_{m}(g)^{p^{r}}$ vanishes at every $i$-th type-II horizontal character of $G/\Gamma$. So $\Taylor_{m}(g)^{p^{r}}\in G_{[i,2]}\Gamma$.
\end{proof}

\subsection{Ratner-type theorem}

Following the idea of \cite{GTZ12},
we use the factorization theorem to obtain a Ratner-type theorem. 
	For an $\N$-filtered nilmanifold $G/\Gamma$ and $i\in\N_{+}$, recall that $\mathfrak{N}_{i}(G/\Gamma)$ is the group of all $i$-th type-II horizontal characters.

\begin{defn} 
Let $g(n)\in \poly(\Z^{d}\to G_{\N})$, $Q\in\N_{+}$
and $\Omega$ be a subset of $\Z^{d}$.
Let $\Xi_{\Omega,i,Q}(g)$ denote the group of all    $\xi_{i}\in \mathfrak{N}_{i}(G/\Gamma)$  such that 
$$\xi_{i}(\Taylor_{i}(g)(h_{1},\dots,h_{i}))=\xi_{i}(\Delta_{h_{i}}\dots\Delta_{h_{1}}g(n) \mod G_{[i,2]})\in\Z/Q$$
for all $(n,h_{1},\dots,h_{i})\in \Gow_{p,i}(\Omega)$.\footnote{Note that $\xi_{i}(\Delta_{h_{i}}\dots\Delta_{h_{1}}g(n) \mod G_{[i,2]})$ is well defined since $\Delta_{h_{i}}\dots\Delta_{h_{1}}g(n)\in G_{i}$.}
 Let
 $$\Xi_{\Omega,i,Q}^{\perp}(g):=\{x\in\Hor_{i}(G)\colon \xi_{i}(x)\in\Z \text{ for all } \xi_{i}\in \Xi_{\Omega,i}(g)\}.$$
 %
 %If $\Omega$ is a subset of $\V$ and $g=g'\circ\tau\in \poly_{p}(\V\to G_{\N}\vert\Gamma)$ for some $g' \in \poly_{p}(\Z^{d}\to G_{\N}\vert\Gamma)$, then denote $\Xi_{\Omega,i}(g):=\Xi_{\tau(\Omega),i}(g')$ and $\Xi_{\Omega,i}^{\perp}(g):=\Xi_{\tau(\Omega),i}^{\perp}(g')$.  
  \end{defn}

%\begin{lem}
% Suppose that $g=g'\circ\tau=g''\circ\tau$ for some $g',g''\in\poly_{p}(\iota^{-1}(\Omega)\to G_{\N}\vert\Gamma)$. Then $\Xi_{\tau(\Omega),i}(g')=\Xi_{\tau(\Omega),i}(g'')$ and $\Xi^{\perp}_{\tau(\Omega),i}(g')=\Xi_{\tau(\Omega),i}^{\perp}(g'')$. In particular, $\Xi_{\Omega,i}(g)$ and $\Xi_{\Omega,i}^{\perp}(g)$ are well defined and are independent of the choice of $g'$.
 %\end{lem}
%\begin{proof}
%It suffices to show that for all $\xi_{i}\in \mathfrak{N}_{i}(G/\Gamma)$ and $(n,h_{1},\dots,h_{i})\in \Gow_{p,i}(\tau(\Omega))$, we have that
%\begin{equation}\label{3:hid2}
%\xi_{i}(\Taylor_{i}(g')(h_{1},\dots,h_{i}))=\xi_{i}(\Taylor_{i}(g'')(h_{1},\dots,h_{i})).
%\end{equation}
%Note that $g'(n)=g''(n)$ for all $n\in \tau(\Omega)$. So for all  $n\in \tau(\Omega)$ and $m,m'\in (\Z^{d})^{k}$, 
%\begin{equation}\label{3:hid3}
%g'(n+pm)^{-1}g''(n+pm')=(g'(n+pm)^{-1}g'(n))(g''(n)^{-1}g''(n+pm'))\in\Gamma.
%\end{equation}
%For all $(n,h_{1},\dots,h_{i})\in \Gow_{p,i}(\tau(\Omega))$, since $n+\e_{1}h_{1}+\dots+\e_{i}h_{i}\in \tau(\Omega)+p(\Z^{d})^{k}$ for all $\e_{1},\dots,\e_{i}\in\{0,1\}$, it follows from (\ref{3:hid3}) that
 %$\Delta_{h_{i}}\dots\Delta_{h_{1}}({g'}^{-1}g'')(n)\in G_{[i,2]}\Gamma$.
%Therefore,  we have that
%$$\xi_{i}(\Taylor_{i}({g'}^{-1}g'')(h_{1},\dots,h_{i}))=\xi_{i}(\Delta_{h_{i}}\dots\Delta_{h_{1}}({g'}^{-1}g'')(n) \mod G_{[i,2]})\in \Z$$
%for all $(n,h_{1},\dots,h_{i})\in \Gow_{p,i}(\tau(\Omega))$. By Lemma \ref{3:magictaylor}, we get (\ref{3:hid2}).
%\end{proof}

We are now ready to state the Ratner-type theorem, in the proof of which we make essential use of the factorization theorem proved in the first part of the series \cite{SunA}.
 For $i\in\N$, let $\pi_{\Hor_{i}(G)}\colon G_{i}\to \Hor_{i}(G)$ be the projection map.

 \begin{thm}[Ratner-type theorem]\label{3:rat}
		Let $0<\d<1/2, C>0$, $d,D\in\N_{+}$, $s\in\N$  and $p\gg_{C,\d,d,D} 1$ be a prime. Let 
		$M\colon\F_{p}^{d}\to\F_{p}$ be a  non-degenerate quadratic form and $V+c$ be an affine subspace of $\V$. %of co-dimension $r$ with $\rank(M\vert_{V+c})=d-r$.
		Denote $\Omega:=V(M)\cap (V+c)$. 
		Let $G/\Gamma$ be an $s$-step $\N$-filtered nilmanifold  of complexity at most $C$, equipped with an $C$-rational Mal'cev basis $\mathcal{X}$, and  $g\in \poly(\Z^{d}\to G_{\N})$ be rational. Let
		$F\in\Lip(G/\Gamma\to\C^{D})$
		be a function with $\Vert F\Vert_{\Lip(G/\Gamma\to\C^{D})}\leq C$ such that 
		$$\vert\E_{n\in \iota^{-1}(\Omega)}F(g(n)\Gamma)\vert\geq \d.\footnote{Note that $g$ is periodic by Lemma \ref{3:r2p}.}$$
		If $\rank(M\vert_{V+c})\geq s+13$ and $p\gg_{C,\d,d,D} 1$, then there exist $Q\in\N_{+}$ with $Q\leq O_{C,\d,d,D}(1)$, $\e\in G$  of complexity $O_{C,\d,d,D}(1)$ and  a rational subgroup $G_{P}$ of $G$ which is $O_{C,\d,d,D}(1)$-rational relative to $\mathcal{X}$ such that for all $i\in\N_{+}$,
		$\pi_{\Hor_{i}(G)}(G_{P}\cap G_{i})\geq \Xi^{\perp}_{\iota^{-1}(\Omega),i,Q}(g)$		
		 and
		$$\Bigl\vert\int_{G_{P}/\Gamma_{P}}F(\e x)\,dm_{G_{P}/\Gamma_{P}}(x)\Bigr\vert\geq \d/2,$$
		where  $m_{G_{P}/\Gamma_{P}}$ is the Haar measure of $G_{P}/\Gamma_{P}$.
\end{thm}	
\begin{proof}  
By taking components we may assume that $F$ is scalar-valued.
	Let $\mathcal{F}\colon\R_{+}\to\R_{+}$ be a growth function to be chosen later depending only on $C,\d,d,D$.
	Fix any $n_{\ast}\in\iota^{-1}(\Omega)$.
	 By Theorem \ref{3:facf3r},  (see Appendix \ref{3:s:AppA7} for definitions) if $p\gg_{C,\d,d,\mathcal{F}} 1$, then we may write $$g=\e g'\gamma,$$ where $\e\in G$ is of complexity $O_{C'}(1)$, $g'\in\poly(\Z^{d}\to (G_{P})_{\N})$ is a rational polynomial sequence with  $(g'(n)\Gamma)_{n\in \iota^{-1}(\Omega)}$ being $\mathcal{F}(C')^{-1}$-totally equidistributed    on some  sub-nilmanifold  $G_{P}/(G_{P}\cap \Gamma)$ of $G/\Gamma$ for some subgroup $G_{P}$ of $G$ which is $C'$-rational relative to $\mathcal{X}$ for some $C\leq C'\leq O_{C,d,\mathcal{F}}(1)$, and	 
	  $\gamma\in\poly_{\approx r,n_{\ast}}(\iota^{-1}(\Omega)\to G_{\N}\vert\Gamma)\cap\poly_{r}(\iota^{-1}(\Omega)\to G_{N}\vert\Gamma)$  for some $r\in\N_{+}$ with $r\leq O_{C,C',d,s}(1)<\mathcal{F}(C')$ (provided that $\mathcal{F}$ grows sufficiently fast).

	 For all $m\in [r]^{d}$,  there exists $\tilde{m}\in\Z^{d}$ with $\tilde{m}-m\in r\Z^{d}$ and $\tilde{m}-n_{\ast}\in p\Z^{d}$. Then $\tilde{m}\in\iota^{-1}(\Omega)$. % and we may write $\tilde{m}=m+rz_{m}$ for some $z_{m}\in\Z^{d}$. 
	 Denote
	 $$g'_{m}(n):=\{\gamma(\tilde{m})\}^{-1}g'(n)\{\gamma(\tilde{m})\}.$$
	  By Lemma \ref{3:r2p}, we may assume that 
	   $g$  and $g'_{m}$ are $prK$-periodic for some $K\in\N_{+}$ for all $m\in [r]^{d}$. 
	Then  
	\begin{equation}\nonumber%\label{3:D61}
	\begin{split}
	&\quad\d\leq \vert\E_{n\in \iota^{-1}(\Omega)}F(g(n)\Gamma)\vert
	=\vert\E_{m\in[r]^{d}}\E_{n\in \iota^{-1}(r^{-1}(\Omega-m))\cap[pK]^{d}}F(\e g'(rn+m)\gamma(rn+m)\Gamma)\vert
	\\&=\vert\E_{m\in[r]^{d}}\E_{n\in \iota^{-1}(r^{-1}(\Omega-m))\cap[pK]^{d}}F(\e g'(rn+m)\gamma(\tilde{m})\Gamma)\vert
	\\&=\vert\E_{m\in[r]^{d}}\E_{n\in \iota^{-1}(r^{-1}(\Omega-m))\cap[pK]^{d}}F(\e \{\gamma(\tilde{m})\}g_{m}'(rn+m)\Gamma)\vert.
	\end{split}
	\end{equation}	
  %where we slightly abuse notation to {\color ???.}

	% denote
	% $$g'_{m}(n):=\{\gamma(\tilde{m})\}^{-1}g'(n)\{\gamma(\tilde{m})\}.$$
 	% By {\color ???} we may assume that 
	%   $g$  and $g'_{m}$ are $prK$-periodic for some $K\in\N_{+}$ for all $m\in [r]^{d}$. Then  
	%\begin{equation}\nonumber%\label{3:D61}
	%\begin{split}
	%&\quad\d\leq \vert\E_{n\in \iota^{-1}(\Omega)}F(g(n)\Gamma)\vert
	%=\vert\E_{m\in[r]^{d}}\E_{n\in \iota^{-1}(r^{-1}(\Omega-m))\cap[pK]^{d}}F(\e g'(rn+m)\gamma(rn+m)\Gamma)\vert
	%\\&=\vert\E_{m\in[r]^{d}}\E_{n\in \iota^{-1}(r^{-1}(\Omega-m))\cap[pK]^{d}}F(\e \{\gamma(m)\}g_{m}'(rn+m)\Gamma)\vert.
	%\end{split}
	%\end{equation}	

	By Theorem \ref{3:ct}, there exists $m\in [r]^{d}$ such that 
	\begin{equation}\label{3:D61}
	\begin{split}
	% \vert\E_{n\in \iota^{-1}(r^{-1}(\Omega-m)-z_{m})\cap[pK]^{d}}F(\e_{m}g_{m}'(rn+\tilde{m})\Gamma)\vert=
	\vert\E_{n\in \iota^{-1}(r^{-1}(\Omega-m))\cap[pK]^{d}}F(\e_{m}g_{m}'(rn+m)\Gamma)\vert\geq \d.
	\end{split}
	\end{equation}	
	where   $\e_{m}:=\e \{\gamma(\tilde{m})\}$.

	Let $\psi\colon G\to\R^{m}$ be the Mal'cev coordinate map.  
	Since $\gamma$ is partially $r$-rational on $\iota^{-1}(\Omega)$ with base point $n_{\ast}$, for all $h\in\Z^{d}$, we have that $\psi(\gamma(n_{\ast}+prh))\in \Z^{\dim(G)}$. Since the map $h\mapsto \psi(\gamma(n_{\ast}+prh))$ is a polynomial of degree at most $s$ by the  Baker-Campbell-Hausdorff formula, we have that $\psi(\gamma(n_{\ast}+ph))\in \Z^{\dim(G)}/r^{s}$ for all $h\in\Z^{d}$. Since $\tilde{m}-n_{\ast}\in p\Z^{d}$, we have that $\psi(\gamma(\tilde{m}))\in \Z^{\dim(G)}/r^{s}$. Since $\psi([\gamma(\tilde{m})])\in  \Z^{\dim(G)}$, by the  Baker-Campbell-Hausdorff formula, enlarging $r$ if necessary,  we may assume without loss of generality that  $\psi(\{\gamma(\tilde{m})\})\in \Z^{\dim(G)}/r$.

	%	For all $n\in\iota^{-1}(\Omega)$, since $\psi(\gamma(n))\in \Z^{\dim(G)}/r$ and $\psi([\gamma(n)])\in \Z^{\dim(G)}$, by the  Baker-Campbell-Hausdorff formula, enlarging $r$ if necessary,  we may assume without loss of generality that  $\psi(\{\gamma(n)\})\in \Z^{\dim(G)}/r$ for all $n\in\iota^{-1}(\Omega)$, where $\psi\colon G\to\R^{m}$ is the Mal'cev coordinate map.  

	Let $G'_{P}:=\{\gamma(\tilde{m})\}^{-1}G_{p}\{\gamma(\tilde{m})\}$. Since $\psi(\{\gamma(\tilde{m})\})\in \Z^{\dim(G)}/r$, by Lemma A.13 of \cite{GT12b}, $G'_{P}$ is a $C'$-rational subgroup  of $G$ relative to $\mathcal{X}$.
	 	Since $(g'(n)\Gamma)_{n\in \iota^{-1}(\Omega)}$ is $\mathcal{F}(C')^{-1}$-totally equidistributed on $G_{p}/(G_{P}\cap\Gamma)$, we have that the sequence $(g'(rn+m)\Gamma)_{n\in \iota^{-1}(r^{-1}(\Omega-m))}$ is $\mathcal{F}(C')^{-1}$-equidistributed on $G_{p}/(G_{P}\cap\Gamma)$ (which is equivalent of saying that the sequence $(g'(n)\Gamma)_{n\in\iota^{-1}(\Omega)\cap(r\Z^{d}+m)}$ is $\mathcal{F}(C')^{-1}$-equidistributed on $G_{p}/(G_{P}\cap\Gamma)$).
%		
%		
%		
%		 {\color   $(g(rn+m)\Gamma)_{n\in\iota^{-1}(r^{-1}(\Omega-m))}$ (or equivalently, the sequence $(g(n)\Gamma)_{n\in\iota^{-1}(\Omega)\cap (r\Z^{d}+m)}$) is $\d$-equidistributed on $G/\Gamma$.}
%
%		
		  By Lemma \ref{3:sftg},  
	the sequence $(g'_{m}(rn+m)\Gamma)_{n\in \iota^{-1}(r^{-1}(\Omega-m))}$ is $C''\mathcal{F}(C')^{-\frac{1}{C''}}$-equidistributed on $G'_{p}/(G'_{P}\cap\Gamma)$ for some $C''=C''(C,d)\geq 1$.  
%	
	%
%	By {\color the bijiection,} (setting $B=r^{-1}(\Omega-m)$, $c=-r^{-1}m$, $L(z):=r^{-1}z$ and $K'=rK$), we have that 
%	\begin{equation}\label{3:D612}
%	\begin{split}
%	&\quad\d\leq \vert\E_{m\in[r]^{d}}\E_{n\in\iota^{-1}(\Omega)\cap[rpK]^{d}}\E_{w\in[K]^{d}}F(\e \{\gamma(m)\}g_{m}'(rr_{\ast}n+rpw+(1-rr^{\ast})m)\Gamma)\vert.
%	\end{split}
%	\end{equation}	
%	
%	Since $(g'(n)\Gamma)_{n\in \iota^{-1}(\Omega)}$ is $\mathcal{F}(C')^{-1}$-equidistributed on $G_{P}/(G_{P}\cap \Gamma)$,
We have that
\begin{equation}\nonumber 
	\begin{split}
	&\quad \Bigl\vert \E_{n\in \iota^{-1}(r^{-1}(\Omega-m))}F(\e_{m} g'_{m}(rn+m)\Gamma)-\int_{G'_{P}/\Gamma'_{P}}F(\e_{m} x)\,dm_{G'_{P}/\Gamma'_{P}}(x)\Bigr\vert
	\\& \leq C''\mathcal{F}(C')^{-\frac{1}{C''}}\Vert F(\e_{m}\cdot)\Vert_{\Lip(G_{P}/\Gamma_{P})}.
	\end{split}
	\end{equation}	
		Since $\Vert F\Vert_{\Lip(G/\Gamma)}\leq C$, we have that $\Vert F\Vert_{\Lip(G'_{P}/\Gamma'_{P})}\leq O_{C,C',\d,s}(1)$. So $\Vert F(\e_{m}\cdot)\Vert_{\Lip(G'_{P}/\Gamma'_{P})}\leq O_{C,C',\d,s}(1)$ by Lemma A.5 of \cite{GT12b}. Therefore,
	   if $\mathcal{F}$ grows sufficiently fast depending on $C,\d,d,D$, then we have that 
		$$\Bigl\vert \E_{n\in \iota^{-1}(r^{-1}(\Omega-m))}F(\e_{m} g'_{m}(rn+m)\Gamma)-\int_{G'_{P}/\Gamma'_{P}}F(\e_{m}x)\,dm_{G'_{P}/\Gamma'_{P}}(x)\Bigr\vert<\d/2$$
	and so it follows
	from (\ref{3:D61})  that $$\Bigl\vert\int_{G'_{P}/\Gamma'_{P}}F(\e_{m} x)\,dm_{G'_{P}/\Gamma'_{P}}(x)\Bigr\vert\geq \d/2.$$
	
	It remains to show that  $$\pi_{\Hor_{i}(G)}(G'_{P}\cap G_{i})\geq \Xi^{\perp}_{\iota^{-1}(\Omega),i,Q}(g)$$ 
	for all $i\in\N_{+}$ for some $Q\in\N_{+}$ with $Q\leq O_{C,\d,d,D}(1)$.  If not, then by duality and the rational nature of $G'_{P}$, there exists   $\xi_{i}\in \mathfrak{N}_{i}(G/\Gamma)$ not belonging to $\Xi_{\iota^{-1}(\Omega),i,Q}(g)$ which annihilates $\pi_{\Hor_{i}(G)}(G'_{P}\cap G_{i})$.
	% 
	% Write $g'=\tilde{g}\circ \tau$ and $\gamma=\tilde{\gamma}\circ\tau$ for some $\tilde{g}\in\poly_{p}(\tau(V(M)\cap(V+c))+p\Z^{d}\to (G_{P})_{\N}\vert G_{P}\cap\Gamma)$ and $\tilde{\gamma}\in\poly(\tau(V(M)\cap(V+c))+p\Z^{d}\to G_{\N}\vert\Gamma)$. Then $g=(\e \tilde{g}\tilde{\gamma})\circ \tau$ and it follows from Lemma \ref{3:grg} that $\e \tilde{g}\tilde{\gamma}\in \poly_{p}(\tau(V(M)\cap(V+c))+p\Z^{d}\to G_{\N}\vert\Gamma)$.
	 By Lemma \ref{3:magictaylor},
	 \begin{equation}\label{3:r1}
	 \begin{split}
	 &\quad \xi_{i}(\Delta_{h_{i}}\dots\Delta_{h_{1}} g(n) \mod G_{[i,2]})	  
	% \\&=\xi_{i}(\Delta_{h_{i}}\dots\Delta_{h_{1}}\e_{m} \mod G_{[i,2]})\xi_{i}(\Delta_{h_{i}}\dots\Delta_{h_{1}}g'(n) \mod G_{[i,2]})
	% \\&\qquad\qquad\xi_{i}(\Delta_{h_{i}}\dots\Delta_{h_{1}}(\{\gamma(\tilde{m})\}^{-1}\gamma(n)) \mod G_{[i,2]})
	% \\&
	\equiv\xi_{i}(\Delta_{h_{i}}\dots\Delta_{h_{1}}\gamma(n) \mod G_{[i,2]}) \mod\Z.
	 \end{split}
	 \end{equation}
	 
	 Since $\gamma$ is partially $r$-periodic on $\iota^{-1}(\Omega)$, for all $(n,h_{1},\dots,h_{i})\in \Gow_{p,i}(\iota^{-1}(\Omega))$ with $h_{1},\dots,h_{i}$ belonging to $r\Z$, we have that $\Delta_{h_{i}}\dots\Delta_{h_{1}} \gamma(n)\in\Gamma$ and so
	  \begin{equation}\nonumber
	 \begin{split}
	\xi_{i}(\Delta_{h_{i}}\dots\Delta_{h_{1}} \gamma(n) \mod G_{[i,2]})\in\Z.   
	\end{split}
	 \end{equation}
	 
	 Now fix any $(n,h_{1},\dots,h_{i})\in \Gow_{p,i}(\iota^{-1}(\Omega))$. We may write $h_{i'}=rv_{i'}+pu_{i'}$ for some $v_{i'},u_{i'}\in\Z^{d}$ for all $1\leq i'\leq i$. For all $y_{1},\dots,y_{i}\in\Z^{d}$, note that $(n,r(v_{1}+py_{1}),\dots,r(v_{i}+py_{i}))\in \Gow_{p,i}(\iota^{-1}(\Omega))$. So by assumption,
	 $$H(y_{1},\dots,y_{i}):=\xi_{i}(\Delta_{r(v_{i}+py_{i})}\dots\Delta_{r(v_{1}+py_{1})} \gamma(n) \mod G_{[i,2]})\in\Z.$$
	 By interpolation, $H$ is a polynomial of degree at most $i$ with coefficients in $\Z/(s!)^{ds}$. So
	 $$\xi_{i}(\Delta_{h_{i}}\dots\Delta_{h_{1}} \gamma(n) \mod G_{[i,2]})=H(u_{1}/r,\dots,u_{i}/r)\in \Z/Q$$
	for some $Q\in\N_{+}$ with $Q\leq O_{C,\d,d,D}(1)$.
%	  \begin{equation}\nonumber
%	 \begin{split}
%	\xi_{i}(\Delta_{h_{i}}\dots\Delta_{h_{1}} \gamma(n) \mod G_{[i,2]})\in\Z/Q   
%	\end{split}
%	 \end{equation}
	 for all $(n,h_{1},\dots,h_{i})\in \Gow_{p,i}(\iota^{-1}(\Omega))$.
	% 	Since $\gamma(n)$ takes values in $\Gamma$ when $n\in \iota^{-1}(\Omega)$,  we have that $$\xi_{i}(\Delta_{h_{i}}\dots\Delta_{h_{1}}\gamma(n) \mod G_{[i,2]})\in\Z$$ for all $(n,h_{1},\dots,h_{i})\in \Gow_{p,i}(\iota^{-1}(\Omega))$.
	 By (\ref{3:r1}), we have that $$\xi_{i}(\Delta_{h_{i}}\dots\Delta_{h_{1}}g(n) \mod G_{[i,2]})\in\Z/Q$$ for all $(n,h_{1},\dots,h_{i})\in \Gow_{p,i}(\iota^{-1}(\Omega))$. So    $\xi_{i}\in\Xi_{\iota^{-1}(\Omega),i,Q}(g)$, a contradiction.
	  This finishes the proof. 
\end{proof}

\subsection{Nilsequences and nilcharacters}\label{3:s:nn}

We start with the definition of vertical torus and character, which plays an important role in the equiditribution properties on nilmanifolds.

\begin{defn}[Vertical torus and character]\label{3:vtc}
	Let $I$ be the degree, multi-degree, or degree-rank ordering, and $J$ be a finite down set of $I$ having a maximum element $s$. Let $G/\Gamma$ be a nilmanifold with an $I$-filtration $(G_{i})_{i\in I}$ with a Mal'cev basis $\mathcal{X}$ adapted to it. Then $G_{s}$ lies in the center of $G$. The \emph{vertical torus} of $G/\Gamma$ is the set $G_{s}/(\Gamma\cap G_{s})$.
	A \emph{vertical character} of $G/\Gamma$ (with respect to the filtration $(G_{i})_{i\in I}$) is a continuous homomorphism $\xi\colon G_{s}\to\R$ such that $\xi(\Gamma\cap G_{s})\subseteq \Z$ (in particular, $\xi$ descents to a continuous homomorphism from the vertical torus $G_{s}/(\Gamma\cap G_{s})$ to $\T$). Let $\psi\colon G\to \R^{m}$ be the Mal'cev coordinate map and denote $m_{s}=\dim(G_{s})$. Then there exists $k\in\Z^{m_{s}}$ such that $\xi(g_{s})=(\bold{0},k)\cdot\psi(g_{s})$ for all $g_{s}\in G_{s}$ (note that the first $\dim(G)-m_{s}$ coordinates of $\psi(g_{s})$ are all zero). We call the quantity $\Vert\xi\Vert:=\vert k\vert$ the \emph{complexity} of $\xi$  (with respect to $\mathcal{X}$). 
\end{defn}

We now define nilsequences and nilcharacters. In addition to the conventional definitions, we also define periodic and partially periodic nilsequences and nilcharacters. Recall that for $D\in\N_{+}$,  $\mathbb{S}^{D}$ denotes the set of $(z_{1},\dots,z_{D})\in\C^{D}$ with $\vert z_{1}\vert^{2}+\dots+\vert z_{D}\vert^{2}=1$.

	\begin{defn}[Nilsequences and nilcharacters]
		Let $D,k\in\N_{+}$, $I$ be the degree, multi-degree, or degree-rank ordering such that $\dim(I)\vert k$, and $J$ be a finite down set of $I$. Let $\Omega$ be a subset $\Z^{k}$. We say that a function $\phi\colon\Omega\to \C^{D}$ is an \emph{($I$-filtered) nilsequence} on $\Omega$ of degree $\subseteq J$ if 
		\begin{equation}\label{3:varr}
		\phi(n)=F(g(n)\Gamma) \text{ for all } n\in\Omega
		\end{equation}
		for some nilmanifold $G/\Gamma$ with an $I$-filtration $(G_{i})_{i\in I}$ of degree $\subseteq J$ equipped with some smooth Riemannian metric $d_{G/\Gamma}$, some $g\in\poly(\Z^{k}\to G_{I})$, and some Lipschitz function $F\colon G/\Gamma\to \C^{D}$. We say that $D$ is the \emph{dimension} of $\phi.$
	    If in addition, 
	    \begin{itemize}
	    	\item $J$ has a maximal element $s$;
	    	\item $F(g_{s}x)=\exp(\eta(g_{s}))F(x)$ for all $x\in G/\Gamma$ and $g_{s}\in G_{s}$ for some vertical character $\eta\colon G_{s}\to \R$ (we say that $\eta$ is the \emph{vertical frequency} of $F$);
	    	\item $F$ takes values in $\mathbb{S}^{D}$,
	    \end{itemize}	
	    then we say that $\phi$ is a \emph{$k'$-integral ($I$-filtered) nilcharacter} on $\Omega$ of degree $\subseteq J$, where $k'=\frac{k}{\dim(I)}$.\footnote{Note that we only define $Q$-periodic nilsequence/nilcharacter on the designated set $\Omega$. However, one can easily extend the definition of such a sequence to the whole space $\Z^{k}$ by requiring (\ref{3:varr}) to hold for all $n\in\Z^{k}$.} 	Following Convention \ref{3:ckk}, we omit the phrase $k'$-integral throughout as the value of $k'$ is clear given $I$ and $\Omega$.
	 A similar remark applies to the degree and the degree-rank filtrations.

	    For any $s\in I$, we say that a function $\phi\colon\Omega\to \C^{D}$ is a ($I$-filtered) nilsequence/nilcharacter on $\Omega$ of degree $\leq s$ if it is a ($I$-filtered) nilsequence/nilcharacter on $\Omega$ of degree $\subseteq \{i\in I\colon i\leq s\}$.
	    
	      If in addition, the polynomial sequence $g$ in (\ref{3:varr}) % is $Q$-rational, 
	       belongs to $\poly_{\approx Q}(\Z^{k}\to G_{I}\vert\Gamma)$, then we say that $\phi$ is a \emph{$Q$-rational} ($I$-filtered)  nilsequence/nilcharacter on $\Omega$.%\footnote{One can defined partially rational  nilsequence/nilcharacter on $\Omega$ with some base point $n_{\ast}\in\Omega$ in a similar way, but we do not need such a concept in this paper.} 

	%     If in addition, the polynomial sequence $g$ in (\ref{3:varr}) belongs to $\poly_{Q}(\Omega\to G_{I}\vert\Gamma)$, then we say that $\phi$ is a \emph{partially $Q$-periodic} ($I$-filtered)  nilsequence/nilcharacter on $\Omega$. 

	    If in addition, the polynomial sequence $g$ in (\ref{3:varr}) belongs to $\poly_{Q}(\Z^{k}\to G_{I}\vert\Gamma)$, then we say that $\phi$ is a \emph{$Q$-periodic} ($I$-filtered)  nilsequence/nilcharacter on $\Omega$.

%		Let $\Omega$ be a subset of $\F_{p}^{k}$. %Denote $\Omega'=\iota^{-1}(\Omega)$. 
%	   A function $\phi\colon\Omega\to \C^{D}$ is an \emph{($I$-filtered) nilsequence/nilcharacter} on $\Omega$
%	of degree $\leq s$ (or $\subseteq J$) if $\phi(n)=\phi'\circ\tau(n)$ for all $n\in\Omega$ for some ($I$-filtered) nilsequence/nilcharacter $\phi'\colon \iota^{-1}(\Omega)\to\C^{D}$  on $\Omega$
%	of degree $\leq s$ (or $\subseteq J$). We say that $\phi$ is \emph{partially $p$-periodic/$p$-periodic
%	} (or simply \emph{partially perodic/periodic}) on $\Omega$ if we may require $\phi'$ to be partially $p$-perodic/$p$-periodic on $\iota^{-1}(\Omega)$.	

		Let $\Omega$ be a subset of $\F_{p}^{k}$. %Denote $\Omega'=\iota^{-1}(\Omega)$. 
	   A function $\phi\colon\Omega\to \C^{D}$ is an \emph{($I$-filtered) nilsequence/nilcharacter} on $\Omega$
	of degree $\leq s$ (or $\subseteq J$) if $\phi(n)=\phi'\circ\tau(n)$ for all $n\in\Omega$ for some ($I$-filtered) nilsequence/nilcharacter $\phi'\colon \iota^{-1}(\Omega)\to\C^{D}$  on $\Omega$
	of degree $\leq s$ (or $\subseteq J$). We say that $\phi$ is \emph{$p$-periodic/$p$-periodic
	} (or simply \emph{perodic/periodic}) on $\Omega$ if we may require $\phi'$ to be   $p$-perodic/$p$-periodic on $\iota^{-1}(\Omega)$.	
	\end{defn}
	
Note that we allow nilsequences and nilcharacters to be vector valued (i.e. $D\geq 2$) in order to avoid certain topological obstructions in the constructions of nilsequences. See \cite{GTZ12} pages 1252--1255  for further discussions.

It is natural to extend nilsequences to \emph{pre-nilsequences} by allowing the underlying space $G/\Gamma$ to be pre-nilmanifolds. However, it turns out that such an extension is not necessary, as every pre-nilsequence can be expressed as a nilsequence with comparable complexities.

  \begin{lem}[Pre-nilsequences are nilsequences]\label{3:pre2g}
      Let the notation be as above. Let $D,Q\in\N_{+},C>0$, $J$ be a finite downset of $I$ and $s\in I$. Let $g\in\poly(\Z^{k}\to (G/\Gamma)_{I})$, and $F\colon G/\Gamma\to \C^{D}$ be a Lipschitz function. Then there exist   $g'\in\poly(\Z^{k}\to (G'/\Gamma')_{I})$ with $g'(\bold{0})=id_{G'}$, and  a Lipschitz function $F'\colon G'/\Gamma'\to \C^{D}$  such that
 $$F(g(n)\Gamma)=F'(g'(n)\Gamma')$$
 for all $n\in \Z^{k}$, and that the followings hold:
    \begin{itemize}
       \item if $G/\Gamma$ is of complexity at most $C$, then  $G'/\Gamma'$ is of complexity at most $O(C)$;
       \item if $G/\Gamma$ is of degree $\subseteq J$, then so is  $G'/\Gamma'$;
       \item if $g\in\poly_{Q}(\Z^{k}\to G_{I}\vert\Gamma)$, then $g'\in\poly_{Q}(\Z^{k}\to G'_{I}\vert\Gamma')$;
        \item if $g\in\poly_{\approx Q}(\Z^{k}\to G_{I}\vert\Gamma)$, then $g'\in\poly_{\approx rQ^{r}}(\Z^{k}\to G'_{I}\vert\Gamma')$ for some $r\in\N_{+}$, $r\leq O_{C,k}(1)$;  
       \item if $F$ is of Lipschitz norm at most $C$, then  $F'$ is of of Lipschitz norm at most $O(C)$;
       \item if $F$ takes values in $\mathbb{S}^{D}$, then so does $F'$;
       \item if some vertical character $\eta\colon G_{s}\to \R$ is the vertical frequency of $F$ and if $G_{s}$ is in the center of $G$, then $\eta$ is also the vertical frequency of $F'$.  
   \end{itemize} 
%     Let $D,k\in\N_{+},C>0$, $p$ be a prime,  $I$ be the degree, multi-degree, or degree-rank ordering with $\dim(I)\vert k$, and let $J$ be a finite down set of $I$. %Let $(\mathcal{A},\mathcal{B})$ be one of the following pairs:
%%      $$(\poly(\Omega'\to (G/\Gamma)_{\N^{k'}}),\poly(\Omega'\to (G'/\Gamma')_{\N^{k'}}))$$
%%      {\color{} here $\Omega'$ is a subset of $\Z^{k}$, }   
 %Let $(G/\Gamma)_{I}$ be an $I$-filtered pre-nilmanifold of degree $\subseteq J$ and complexity at most $C$, $g\in\poly(\Z^{k}\to (G/\Gamma)_{I})$, and $F\colon G/\Gamma\to \C^{D}$ be a Lipschitz function of Lipschitz norm at most $C$. When $I$ is the degree-rank filtration, we further assume that $G_{[i,0]}=G_{[i,1]}$ for all $i\geq 1$. Then there exist an  $I$-filtered nilmanifold $(G'/\Gamma')_{I}$  of degree $\subseteq J$ and complexity at most $O(C)$, $g'\in\poly(\Z^{k}\to (G'/\Gamma')_{I})$ with $g'(\bold{0})=id_{G'}$, and $F'\colon G'/\Gamma'\to \C^{D}$ be a Lipschitz function of Lipschitz norm at most $O(C)$ such that
% $$F(g(n)\Gamma)=F'(g'(n)\Gamma')$$
% for all $n\in \Z^{k}$. Moreover, if $\vert F\vert=1$
  \end{lem}
  \begin{proof}
   %    If $I$ is the degree ordering, then denote $G':=G_{1}$. If $I$ is the multi-degree ordering, then let $G'$ denote the group generated by $G_{e_{1}},\dots,G_{e_{\dim(I)}}$.  {\color{} What about the degree-rank filtration???}
  %
 % Define $G'_{i}:=G'\cap G_{i}$ for all $i\in I$ and let $\Gamma':=G'\cap\Gamma$. It is clear to see that $(G'/\Gamma')_{I}$ is an $I$-filtered nilmanifold of degree $\subseteq J$ and complexity at most $C$. 
 We assume that $s\neq \bold{0}$ since otherwise there is nothing to prove.
%Let $V$ be a fundamental domain of $G/\Gamma$ which is contained in a ball of radius $O(C)$ centered at $id_{G}$. Then there is a unique way to write every $h\in G$ in the form $h=\{h\}[h]$ for some $\{h\}\in V$ and $[h]\in \Gamma$.
  Let $F'\colon G'/\Gamma'\to\C$ be the map given by $$F'(h\Gamma'):=F(\{g(\bold{0})\}h\Gamma)$$ for all $h\in G'$. Let $g'\colon \Z^{k}\to G'_{I}$ be the map given by $$g'(n):=\{g(\bold{0})\}^{-1}g(n)g(\bold{0})^{-1}\{g(\bold{0})\}.$$
 
By Lemma \ref{3:r2p},
if $g\in\poly_{\approx Q}(\Z^{k}\to G_{I}\vert\Gamma)$, then $g'\in\poly_{\approx rQ^{r}}(\Z^{k}\to G'_{I}\vert\Gamma')$ for some $r\in\N_{+}$, $r\leq O_{C,k}(1)$.
      Since $\{g(\bold{0})\}\Gamma=g(\bold{0})\Gamma$, we have that 
      $$F'(g'(n)\Gamma')=F(g(n)g(\bold{0})^{-1}\{g(\bold{0})\}\Gamma)=F(g(n)\Gamma)$$
      for all $n\in \Z^{k}$.
%Since $\{g(0)\}$ belongs to $V$, we have that
Note that  $F'$ is a Lipschitz function of Lipschitz norm at most $O(C)$.
Also note that $g'(\bold{0})=id_{G'}$ and that
$$\Delta_{h_{m}}\dots\Delta_{h_{1}}g'(n)=\{g(\bold{0})\}^{-1}(\Delta_{h_{m}}\dots\Delta_{h_{1}}g(n))\{g(\bold{0})\}$$
for all $m\geq 1$ and $h_{1},\dots,h_{m}\in \Z^{k}$. So it follows from the assumption $g\in \poly(\Z^{k}\to(G/\Gamma)_{I})$ and the fact $[G_{0},G_{i}]\subseteq G_{i}, i\in I$ that $g'\in \poly(\Z^{k}\to(G'/\Gamma')_{I})$. 
%
%
%Finally, since  $g\in \poly_{p}(\F_{p}^{k}\to G_{I})$, we have that $g'\in \poly_{p}(\F_{p}^{k}\to G_{I})$ by ???. On the other hand, we have 
%$$g'(n)=[\{g(\bold{0})\},g(0)g(n)^{-1}]g(n)g(0)^{-1}\in G'.$$

The remaining parts of this lemma can be checked easily. 
  \end{proof}

We next define the spaces of nilsequences and nilcharacters that we will work with in this paper.

\begin{defn}[Spaces of nilsequences and nilcharacters]\label{3:snn}
	Let   $I$ be the degree, multi-degree, or degree-rank ordering, $k,Q\in\N_{+}$ with $\dim(I)\vert k$,  $J$ be a finite down set of $I$, $p$ be a prime, and $\Omega$ be a subset of  $\Z^{k}$ or $\F_{p}^{k}$. We set $Q=p$ if $\Omega$ is a subset of $\F_{p}^{k}$.
	\begin{itemize}
		\item Let $\Nil^{J}(\Omega)$ (resp. $\Nil^{J}_{Q}(\Omega)$, $\Nil^{J}_{\approx Q}(\Omega)$) denote the set of all nilsequences (resp. $Q$-periodic nilsequences, $Q$-rational nilsequences) on $\Omega$ of degree $\subseteq J$.
%		\item Let $\Nil^{J}_{Q}(\Omega)$ denote the set of all $Q$-periodic nilsequences on $\Omega$ of degree $\subseteq J$. 
%		\item Let $\Nil^{J}_{\approx Q}(\Omega)$ denote the set of all $Q$-rational nilsequences on $\Omega$ of degree $\subseteq J$. 
\item Let $\Xi^{J}(\Omega)$ (resp. $\Xi^{J}_{Q}(\Omega)$, $\Xi^{J}_{\approx Q}(\Omega)$) denote the set of all nilcharacters (resp. $Q$-periodic nilcharacters, $Q$-rational nilcharacters) on $\Omega$ of degree $\subseteq J$, if $J$ has a maximum element.
%		\item Let $\Xi^{J}_{Q}(\Omega)$ denote the set of  $Q$-periodic nilcharacters on $\Omega$, if $J$ has a maximum element.
		\item We define $\Nil^{\prec J}(\Omega):=\Nil^{J'}(\Omega)$, %$\Nil^{\prec J}_{Q}(\Omega):=\Nil^{J'}_{Q}(\Omega)$ and $\Nil^{\prec J}_{\approx Q}(\Omega):=\Nil^{J'}_{\approx Q}(\Omega)$, 
		where $J'$ is the largest down set of $I$ which is contained in but not equals to $J$. 
		Since $I$ is the degree, multi-degree or the degree-rank ordering, it is not hard to see that such $J'$ exists and is unique. We adopt similar notations for $\Nil^{\prec J}_{Q}(\Omega)$ and $\Nil^{\prec J}_{\approx Q}(\Omega)$, 
	\item For $s\in I$, we write $\Nil^{s}(\Omega):=\Nil^{J}(\Omega)$ and $\Nil^{\prec s}(\Omega):=\Nil^{J'}(\Omega)$,   where $J$ is the down set consisting of all $i\in I$ which does not exceed $s$ and $J'$ is the down set consisting of all $i\in I$ which does not exceed $s'$, where $s'$ is the largest element which is smaller than $s$ (if such $s'$ exists). We use similar notations for $\Nil^{s}_{Q}(\Omega)$, $\Nil^{\prec s}_{Q}(\Omega)$, $\Nil^{s}_{\approx Q}(\Omega)$, $\Nil^{\prec s}_{\approx Q}(\Omega)$, $\Xi^{s}(\Omega)$, $\Xi^{\prec s}(\Omega)$, $\Xi^{s}_{Q}(\Omega)$, $\Xi^{\prec s}_{Q}(\Omega)$, $\Xi^{s}_{\approx Q}(\Omega)$ and $\Xi^{\prec s}_{\approx Q}(\Omega)$.
 	\end{itemize}		 
\end{defn}

For convenience, we introduce the following definition.

	\begin{defn}[Representations of  nilsequences and nilcharacters]
		Let   $I$ be the degree, multi-degree, or degree-rank ordering, $k,Q\in\N_{+}$ with $\dim(I)\vert k$,  $J$ be a finite down set of $I$, $p$ be a prime, and $\Omega$ be a subset of  $\Z^{k}$. 
\begin{itemize}
\item For any $\phi\in\Nil^{J}(\Omega)$ (resp. $\Nil^{J}_{Q}(\Omega)$, $\Nil^{J}_{\approx Q}(\Omega)$), we may write $\phi(n)=F(g(n)\Gamma), n\in\Omega$ for some degree-$J$ nilmanifold $(G/\Gamma)_{I}$, $g\in \poly(\Z^{k}\to G_{I})$ (resp. $\poly_{Q}(\Z^{k}\to G_{I}\vert\Gamma)$, $\poly_{\approx Q}(\Z^{k}\to G_{I}\vert\Gamma)$), and Lipschitz function $F\colon G/\Gamma\to\C^{D}$. We say that $((G/\Gamma)_{I},g,F)$ is a \emph{$\Nil^{J}(\Omega)$ (resp. $\Nil^{J}_{Q}(\Omega)$, $\Nil^{J}_{\approx Q}(\Omega)$)-representation} of $\phi$. 
%\item For any $\phi\in\Nil^{J}_{Q}(\Omega)$, we may write $\phi(n)=F(g(n)\Gamma), n\in\Omega$ for some degree-$J$ nilmanifold $(G/\Gamma)_{I}$, $g\in \poly_{Q}(\Z^{k}\to G_{I}\vert\Gamma)$, and Lipschitz function $F\colon G/\Gamma\to\C^{D}$. We say that $((G/\Gamma)_{I},g,F)$ is a \emph{$\Nil^{J}_{Q}(\Omega)$-representation} of $\phi$. 
\item For any $\xi\in\Xi^{J}(\Omega)$ (resp. $\Xi^{J}_{Q}(\Omega), \Xi^{J}_{\approx Q}(\Omega)$), we may write $\chi(n)=F(g(n)\Gamma), n\in\Omega$ for some degree-$J$ nilmanifold $(G/\Gamma)_{I}$, $g\in \poly(\Z^{k}\to G_{I}\vert\Gamma)$ (resp. $\poly_{Q}(\Z^{k}\to G_{I}\vert\Gamma), \poly_{\approx Q}(\Z^{k}\to G_{I}\vert\Gamma)$), and Lipschitz function $F\colon G/\Gamma\to\C^{D}$ having a vertical frequency $\eta$. We say that $((G/\Gamma)_{I},g,F,\eta)$ is a \emph{$\Xi^{J}(\Omega)$ (resp. $\Xi^{J}_{Q}(\Omega), \Xi^{J}_{\approx Q}(\Omega)$)-representation} of $\chi$. 
\item The complexity of a $\Nil^{J}(\Omega)$ (resp. $\Nil^{J}_{Q}(\Omega)$, $\Nil^{J}_{\approx Q}(\Omega)$, $\Xi^{J}(\Omega)$, $\Xi^{J}_{Q}(\Omega)$, $\Xi^{J}_{\approx Q}(\Omega)$)-representation is the the maximum of the complexities of $G/\Gamma$, $F$ and $\eta$ (if $\eta$ appears in the corresponding definitions). We say that $\phi$ is of \emph{complexity}  at most $C$ as an element of $\Nil^{J}(\Omega)$ (resp. $\Nil^{J}_{Q}(\Omega)$, $\Nil^{J}_{\approx Q}(\Omega)$, $\Xi^{J}(\Omega)$, $\Xi^{J}_{Q}(\Omega)$, $\Xi^{J}_{\approx Q}(\Omega)$) if  $\phi$ admits a  $\Nil^{J}(\Omega)$ (resp. $\Nil^{J}_{Q}(\Omega)$, $\Nil^{J}_{\approx Q}(\Omega)$, $\Xi^{J}_{Q}(\Omega)$, $\Xi^{J}_{\approx Q}(\Omega)$)-representation of complexity at most $C$.
\item Let $\Omega\subseteq \F_{p}^{k}$ and $\phi\in \Nil^{J}(\Omega)$. We say that $((G/\Gamma)_{I},g,F)$ is a \emph{$\Nil^{J}(\Omega)$-representation} of $\phi$ if it is a $\Nil^{J}(\iota^{-1}(\Omega))$-representation  of some $\phi'\in \Nil^{J}(\iota^{-1}(\Omega'))$ such that $\phi(n)=\phi'\circ \tau(n)$ for all $n\in\Omega$. We say that $\phi$  is of \emph{complexity}  at most $C$ as an element of $\Nil^{J}(\Omega)$ if  $\phi$ admits a  $\Nil^{J}(\Omega)$-representation of complexity at most $C$. We adopt similar notations for $\Nil^{J}_{p}(\Omega)$- and $\Xi^{J}_{p}(\Omega)$-representations and complexities.  
\item We define similar notations for $\Nil^{\prec J}(\Omega)$,  $\Nil^{\prec s}(\Omega)$, ect.
  \end{itemize}
\end{defn}

\begin{rem}
In certain situations, a sequence can belong to multiple sets defined above, and have different complexities in different sets. For example,
a sequence $\phi$ can simultaneously belong to $\Xi^{s+1}_{p}(\Omega)$ and  $\Nil^{s}(\Omega)$. In this case, the complexity of $\phi$ as an element in $\Xi^{s+1}_{p}(\Omega)$ and the complexity of $\phi$ as an element in $\Nil^{s}(\Omega)$ are different. We will specify which complexity we are referring to when such ambiguity happens.
\end{rem}

\begin{defn}[Complexity of  nilsequences and nilcharacters]
	Let $\Nil^{J;C}(\Omega)$ denote the set of all nilsequences in $\Nil^{J}(\Omega)$ of complexity at most $C$ (as an element in $\Nil^{J;C}(\Omega)$), and  $\Nil^{J;C,D}(\Omega)$ denote the set of all nilsequences in $\Nil^{J;C}(\Omega)$ of complexity at most $C$ and dimension at most $D$.
	We adopt similar notations for $\Nil^{J}_{Q}(\Omega)$, $\Xi^{J}_{Q}(\Omega)$, $\Nil^{\prec J}(\Omega)$, $\Nil^{s}(\Omega)$, ect.
\end{defn}

We refer the readers to Appendix \ref{3:s:AppB} for some approximation properties for nilsequeneces which will be used in this paper.

\section{Some preliminary results}\label{3:s:b0}

In this section, we provide some  preliminary results and  setup the initial steps for the proof of $\SGI(s)$.

\subsection{The proof of $\SGI(1)$}\label{3:s:b1}
Note that $\SGI(0)$ holds trivially by taking $\d=\e$ and $\phi\equiv 1$.
We now prove $\SGI(1)$.
It turns out that the second local Gowers norm is connected to the Fourier coefficients.
 	 For a function $f\colon\V\to \C$, let $\widehat{f}\colon \V\to\C$ denote the Fourier transform of $f$ given by $\widehat{f}(\xi):=\E_{x\in\V}f(x)\exp(-\frac{1}{p}\tau(\xi\cdot x)).$ We have

\begin{lem}\label{3:pla}
	Let $d\in\N$ with $d\geq 9$, $p$ be a prime, $M\colon\V\to\F_{p}$ be a non-degenerate quadratic form, and $f\colon\V\to\C$ with $\vert f\vert\leq 1$. Then
	$$\Vert f\Vert^{4}_{U^{2}(V(M))}\leq\sup_{\xi\in\V}p^{1-d}\vert\widehat{\bold{1}_{V(M)}f}(\xi)\vert(1+O(p^{-\frac{1}{8}})).$$
\end{lem}	
\begin{proof}
The approach we use  is similar to the proof of Theorem 10 of \cite{LMP19}.
	By Example %\ref{1:rreepp} 
	9.5 of \cite{SunA}, $\Gow_{2}(V(M))$ is  an $M$-set 	of total co-dimension 4.
	By Lemma \ref{3:countingh}, since $d\geq 9$, we have that  
	$$\vert \Gow_{2}(V(M))\vert=p^{3d-4}(1+O(p^{-1/2})).$$
	For convenience denote $\delta:=\bold{1}_{V(M)}$. Then
	\begin{equation}\nonumber
	\begin{split}
	&\quad\Vert f\Vert^{4}_{U^{2}(V(M))}
	=\E_{(n,h_{1},h_{2})\in \Gow_{2}(V(M))}\prod_{\e=(\e_{1},\e_{2})\in\{0,1\}^{2}}\mathcal{C}^{\vert\e\vert}f(n+\e_{1}h_{1}+\e_{2}h_{2})
	\\&=\frac{p^{3d}}{\vert \Gow_{2}(V(M))\vert}\cdot\E_{(n,h_{1},h_{2})\in (\V)^{3}}\prod_{\e=(\e_{1},\e_{2})\in\{0,1\}^{2}}\mathcal{C}^{\vert\e\vert}(\delta f)(n+\e_{1}h_{1}+\e_{2}h_{2})
	\\&\leq  p^{4}\E_{n\in\V}\d(n)\Bigl\vert\E_{y,z,w\in\V}\d\overline{f}(y)\d\overline{f}(z)
	\d f(w)\bold{1}_{n=y+z-w}\Bigr\vert\Bigl(1+O(p^{-\frac{1}{2}})).
	\end{split}
	\end{equation}		
	By the Cauchy-Schwartz inequality,
	\begin{equation}\nonumber
	\begin{split}
	\Vert f\Vert^{8}_{U^{2}(V(M))}\leq p^{8} \Bigl(\E_{n\in\V}\d(n)\Bigr)\Bigl(\E_{n\in\V}\d(n)\Bigl\vert \E_{y,z,w\in\V}\d\overline{f}(y)\d\overline{f}(z)\d f(w)\bold{1}_{n=y+z-w}\Bigr\vert^{2}\Bigr)\Bigl(1+O(p^{-1})\Bigr).
	\end{split}
	\end{equation}	
	By Lemma \ref{3:countingh},
	\begin{equation}\nonumber
	\begin{split}
	&\quad\Vert f\Vert^{8}_{U^{2}(V(M))}
	\\&\leq p^{7} \E_{n\in\V}\d(n) \E_{y_{1},z_{1},w_{1},y_{2},z_{2},w_{2}\in\V}\prod_{j=1,2}\mathcal{C}^{j}\Bigl(\d\overline{f}(y_{j})\d\overline{f}(z_{j})\d f(w_{j})
	\bold{1}_{n=y_{j}+z_{j}-w_{j}}\Bigr)
	\Bigl(1+O(p^{-\frac{1}{2}})\Bigr)
	\\&=p^{7-d} \E_{y_{1},z_{1},w_{1},y_{2},z_{2},w_{2}\in\V}
	\\&\qquad\quad\prod_{j=1,2}\mathcal{C}^{j}\Bigl(\d\overline{f}(y_{j})\d\overline{f}(z_{j})\d f(w_{j})\Bigr)		
	\d(y_{1}+z_{1}-w_{1})
	\bold{1}_{y_{1}+z_{1}-w_{1}=y_{2}+z_{2}-w_{2}}\Bigl(1+O(p^{-\frac{1}{2}})\Bigr)
	\\&\leq p^{7-d}\E_{w_{1},w_{2}\in\V}\d(w_{1})\d(w_{2})
	\\&\qquad\quad\cdot\Bigl\vert\E_{y_{1},z_{1},y_{2},z_{2}\in\V}\d(y_{1}+z_{1}-w_{1})\bold{1}_{y_{1}+z_{1}-w_{1}=y_{2}+z_{2}-w_{2}}\prod_{j=1,2}\mathcal{C}^{j}(\d f(y_{j})\d\overline{f}(z_{j}))\Bigr\vert\Bigl(1+O(p^{-\frac{1}{2}})\Bigr).
	\end{split}
	\end{equation}
	Again by a similar argument using the Cauchy-Schwartz inequality and Lemma \ref{3:countingh}, we may bound $\Vert f\Vert^{16}_{U^{2}(V(M))}$ by
	\begin{equation}\nonumber
	\begin{split}
	p^{12-3d}\E_{y_{j},z_{j}\in\V, 1\leq j\leq 4}\prod_{j=1}^{4}\mathcal{C}^{j}(\d f(y_{j})\d\overline{f}(z_{j}))\bold{1}_{y_{2}-y_{1}+z_{2}-z_{1}=y_{4}-y_{3}+z_{4}-z_{3}}W_{y_{1},\dots,z_{4}}\Bigl(1+O(p^{-\frac{1}{2}})\Bigr),
	\end{split}
	\end{equation}
	where $W_{y_{1},\dots,z_{4}}$ is the qunatitiy
	\begin{equation}\nonumber
	\begin{split}
	\E_{w\in\V}\d (w)\d(y_{2}-y_{1}+z_{2}-z_{1}+w)\d(y_{1}+z_{1}-w)\d(y_{3}+z_{3}-w)
	=\frac{\vert V(M)\cap V_{y_{1},\dots,z_{4}}\vert}{p^{d}},
	\end{split}
	\end{equation}
	with $V_{y_{1},\dots,z_{4}}$ being the set of $w\in\V$ such that $M(w)=M(y_{2}-y_{1}+z_{2}-z_{1}+w)=M(y_{1}+z_{1}-w)=M(y_{3}+z_{3}-w)$.
	
	If
	$y_{2}-y_{1}+z_{2}-z_{1}, y_{1}+z_{1}, y_{3}+z_{3}$ are linearly independent, then $V_{y_{1},\dots,z_{4}}$ is an affine subspace of $\V$ of co-dimension 3. Since $d\geq 9$, by Lemma \ref{3:countingh}, we have that
	$W_{y_{1},\dots,z_{4}}=p^{-4}(1+O(p^{-\frac{1}{2}})).$
	On the other hand, by Lemma \ref{3:iiddpp}, it is not hard to compute that  the number of tuples $(y_{1},y_{2},y_{3},z_{1},z_{2},z_{3})\in(\V)^{6}$ such that $y_{2}-y_{1}+z_{2}-z_{1}, y_{1}+z_{1}, y_{3}+z_{3}, V$ are not linearly independent is at most $3p^{5d+2}$. 
	So $\Vert f\Vert^{16}_{U^{2}(V(M))}$ is bounded by
	\begin{equation}\label{3:temp1}
	\begin{split}
	 p^{8-3d}\E_{y_{j},z_{j}\in\V, 1\leq j\leq 4}\prod_{j=1}^{4}\mathcal{C}^{j}(\d f(y_{j})\d\overline{f}(z_{j}))\bold{1}_{y_{2}-y_{1}+z_{2}-z_{1}=y_{4}-y_{3}+z_{4}-z_{3}}\Bigl(1+O(p^{-\frac{1}{2}})\Bigr).
	\end{split}
	\end{equation}
	Since
	$$\bold{1}_{y_{2}-y_{1}+z_{2}-z_{1}=y_{4}-y_{3}+z_{4}-z_{3}}=p^{-d}\sum_{\xi\in\V}\exp\Bigl(-\frac{1}{p}\tau(\xi\cdot (y_{4}-y_{3}-y_{2}+y_{1}+z_{4}-z_{3}-z_{2}+z_{1}))\Bigr),$$
	we may rewrite (\ref{3:temp1}) as $p^{8-4d}\sum_{\xi\in\V}\vert\widehat{\d f}(\xi)\vert^{8}(1+O(p^{-\frac{1}{2}})).$
	So
	\begin{equation}\nonumber
	\begin{split}
	&\quad\Vert f\Vert^{16}_{U^{2}(V(M))}
	\leq p^{8-4d}\sum_{\xi\in\V}\vert\widehat{\d f}(\xi)\vert^{8}(1+O(p^{-\frac{1}{2}}))
	\\&\leq  p^{8-4d}\sup_{\xi\in\V}\vert\widehat{\d f}(\xi)\vert^{4}\sum_{\xi\in\V}\vert\widehat{\d f}(\xi)\vert^{4}(1+O(p^{-\frac{1}{2}}))
	\\&=p^{8-4d}\sup_{\xi\in\V}\vert\widehat{\d f}(\xi)\vert^{4}\Vert\d f\Vert^{4}_{U^{2}(\V)}(1+O(p^{-\frac{1}{2}})) \text{ (see page 14 of \cite{LMP19})}
	\\&=p^{4-4d}\sup_{\xi\in\V}\vert\widehat{\d f}(\xi)\vert^{4}\Vert f\Vert^{4}_{U^{2}(V(M))}(1+O(p^{-\frac{1}{2}}))
	\\&\leq p^{4-4d}\sup_{\xi\in\V}\vert\widehat{\d f}(\xi)\vert^{4}(1+O(p^{-\frac{1}{2}})).
	\end{split}
	\end{equation}
	This finishes the proof.
\end{proof}

We are now ready to  prove $\SGI(1)$.
Suppose that $\Vert f\Vert_{U^{2}(V(M))}>\e$.  
By Lemma \ref{3:pla}, if $d\geq 9$ and $p\gg_{\e}1$,  then there exists $\xi\in \V$ such that
$p^{1-d}\Bigl\vert\widehat{\bold{1}_{V(M)}f}(\xi)\Bigr\vert\geq (\e/2)^{4}.$
By Lemma \ref{3:countingh},
 \begin{equation}\nonumber
 \begin{split}
 &\quad(\e/2)^{4}\leq p^{1-d}\Bigl\vert\widehat{\bold{1}_{V(M) }f}(\xi)\Bigr\vert=p^{1-d}\Bigl\vert\E_{n\in \V}\bold{1}_{V(M)}f(n)\cdot \exp(-\frac{\tau(\xi\cdot n)}{p})\Bigr\vert
 \\&=\Bigl\vert\E_{n\in V(M)}f(n)\cdot \exp(-\frac{\tau(\xi\cdot n)}{p})\Bigr\vert+O(p^{-\frac{1}{2}}).
 \end{split}
 \end{equation}
It is clear that the map $n\mapsto \exp(-\frac{\tau(\xi\cdot n)}{p})$ is a 1-step $p$-periodic nilsequence of complexity $O(1)$. This completes the proof of $\SGI(1)$.

\subsection{A preliminary reduction for $\SGI(s)$}

For our applications, we also need to use the following improved version of the inverse theorem, which we denote as $\SGI(s)+$:

\begin{prop}\label{3:inv+}[%Inverse theorem for spherical Gowers norms
	$\SGI(s)+$]
	Let $d\in\N_{+},r,s\in\N$, with $d-r\geq N(s-1)$ and $\e>0$. There exist $\delta:=\d(d-r,\e), C:=C(d-r,\e)>0$, and $p_{0}:=p_{0}(d-r,\e)\in\N$
	such that for every prime $p\geq p_{0}$, every
	non-degenerate quadratic form $M\colon\V\to\F_{p}$, every  affine subspace $V+c$ of $\V$ of co-dimension $r$, and every function $f\colon\V\to\C$ bounded in magnitude by 1, if $\rank(M\vert_{V+c})=d-r$, and $\Vert f\Vert_{U^{s+1}(V(M)\cap(V+c))}>\e$, then there exists 
     $\phi\in\Nil^{s;C,1}_{p}(\V)$
	such that
	$$\Bigl\vert\E_{n\in V(M)\cap (V+c)}f(n)\cdot \phi(n)\Bigr\vert>\delta.$$
\end{prop}
\begin{lem}\label{3:sgi++}  
	We have that $\SGI(s)\Rightarrow \SGI(s)+$.
\end{lem}	
\begin{proof} 
	Let $L\colon\F_{p}^{d-r}\to V$ be a bijective linear transformation and $\tilde{M}\colon \F_{p}^{d-r}\to \F_{p}$ be the quadratic form given by $\tilde{M}(m):=M(L(m)+c)$. It is not hard to see that
	$$\Vert f\Vert_{U^{s+1}(V(M)\cap(V+c))}=\Vert f(L(\cdot)+c)\Vert_{U^{s+1}(V(\tilde{M}))}.$$
	Since $\rank(M\vert_{V+c})=d-r$, $\tilde{M}$ is non-degenerate.
	By $\SGI(s)$, there exist $\delta:=\d(d-r,\e), C:=C(d-r,\e)>0$,   $p_{0}:=p_{0}(d-r,\e)\in\N$, 
	and   $\phi\in \Nil^{s;C,1}_{p}(\F_{p}^{d-r})$   
	such that if $p\geq p_{0}$, then
	$$\Bigl\vert\E_{m\in V(\tilde{M})}f(L(m)+c)\cdot \phi(m)\Bigr\vert=\Bigl\vert\E_{n\in V(M)\cap (V+c)}f(n)\cdot \phi(L^{-1}(n-c))\Bigr\vert>\delta.$$
	Assume that $\phi(m)=F(g(m)\Gamma)$ for some $s$-step nilmanifold $G/\Gamma$ of complexity at most $C$, Lipschitz function $F\colon G/\Gamma\to\C$ of  complexity at most $C$, and some $g\in \poly_{p}(\V\to G_{\N}\vert\Gamma)$. 
	By Proposition \ref{3:BB}, 
	there exists $g'\in \poly_{p}(\V\to G_{\N}\vert\Gamma)$ such that $g'(n)\Gamma=g(L^{-1}(n-c))\Gamma$ for all $n\in\V$.
So
		$$\Bigl\vert\E_{n\in V(M)\cap (V+c)}f(n)\cdot F(g'(n)\Gamma)\Bigr\vert=\Bigl\vert\E_{n\in V(M)\cap (V+c)}f(n)\cdot \phi(L^{-1}(n-c))\Bigr\vert>\delta.$$
\end{proof}

Now suppose that $\SGI(s)$ holds for some $s\geq 1$ and our goal is to show $\SGI(s+1)$.
By Lemma \ref{3:sgi++}, $\SGI(s)+$ holds.
 Suppose that $\Vert f\Vert_{U^{s+2}(V(M))}>\e$. 
 By Example %\ref{1:rreepp} 
 9.5 of \cite{SunA},  $\Gow_{s+2}(V(M))$ is an $M$-set  of total co-dimension $(s^{2}+5s+8)/2$.
  Note that for all $h_{s+2}\in \V$, $(n,h_{1},\dots,h_{s+2})\in \Gow_{s+2}(V(M))$ if and only if $(n,h_{1},\dots,h_{s+1})\in \Gow_{s+1}(V(M)^{h_{s+2}})$ (recall Section \ref{3:s:defn} for definitions).
Since  $d\geq s^{2}+5s+9$, by Theorem \ref{3:ct}, we have that 
 \begin{equation}\nonumber
 \begin{split}
 	&\quad\Vert f\Vert^{2^{s+2}}_{U^{s+2}(V(M))}=\E_{(n,h_{1},\dots,h_{s+2})\in \Gow_{s+2}(V(M))}\prod_{\e=(\e_{1},\dots,\e_{s+2})\in\{0,1\}^{s+2}}\mathcal{C}^{\vert\e\vert}f(n+\e_{1}h_{1}+\dots+\e_{s+2}h_{s+2})
 	\\&=\E_{h_{s+2}\in \V}\E_{(n,h_{1},\dots,h_{s+1})\in \Gow_{s+1}(V(M)^{h_{s+2}})}\prod_{\e=(\e_{1},\dots,\e_{s+1})\in\{0,1\}^{s+1}}\mathcal{C}^{\vert\e\vert}\overline{\Delta_{h}f}(n+\e_{1}h_{1}+\dots+\e_{s+1}h_{s+1})+O_{s}(p^{-\frac{1}{2}})
 	\\&=\E_{h_{s+2}\in \V}\Vert \Delta_{h_{s+2}}f\Vert^{2^{s+1}}_{U^{s+1}(V(M)^{h_{s+2}})}+O_{s}(p^{-\frac{1}{2}}).
 \end{split}
\end{equation}

 By the Pigeonhole Principle,  if $p\gg_{\e,s} 1$, then there exists a subset $H$ of $\V$ of cardinality $\gg_{d,\e}p^{d}$ such that for all $h\in H$, we have that $$\Vert \Delta_{h}f\Vert^{2^{s+1}}_{U^{s+1}(V(M)^{h})}\gg_{d,\e}1.$$
 Since $d\geq 3$, by Lemma \ref{3:countingh} and passing to a subset if necessary, we may further assume that $(hA)\cdot h\neq 0$ for all $h\in H$, where $A$ is the matrix associated with $M$.
 
 Fix $h\in H$.
 Note that $V(M)^{h}$ is the intersection of $V(M)$ with an affine subspace $W_{h}$ of $\V$ of co-dimension 1. Since $(hA)\cdot h\neq 0$, $\sp_{\F_{p}}\{h\}\cap \sp_{\F_{p}}\{h\}^{\pp}=\{\bold{0}\}$. By Proposition \ref{3:iissoo}, $\rank(M\vert_{W_{h}})=d-1$.  
  By $\SGI(s)+$, 	
 Corollary \ref{3:LE.6}  and the Pigeonhole Principle, for all  $h\in H$, there exists an $s$-step nilcharacter $\chi_{h}\in\Xi^{s;O_{d,\e}(1),O_{d,\e}(1)}_{p}(\V)$ such that
\begin{equation}\label{3:ini10.58}
\Bigl\vert\E_{n\in V(M)^{h}}f(n+h)\overline{f}(n)\otimes \chi_{h}(n)\Bigr\vert\gg_{d,\e} 1.
\end{equation}

We now need to translate (\ref{3:ini10.58}) into the $\Z$-setting. We may extend $f$ periodically to a function defined on $\Z^{d}$ by setting $\tilde{f}(n):=f(\iota(n))$ for all $n\in\Z^{d}$. Then $\tilde{f}(n+pm)=\tilde{f}(n)$ for all $m,n\in\Z^{d}$.
With a slight abuse of notation, we also denote $\tilde{f}$ by $f$. Let $\tilde{M}\colon\Z^{d}\to\Z/p$ be the regular lifting of $M$. 
Then for any $h\in\Z^{d}$, we have that $\iota(V_{p}(\tilde{M})^{h})=V(M)^{\iota(h)}$. Since $\chi_{h}$ is $p$-periodic, it follows from (\ref{3:ini10.58}) that 
\begin{equation}\label{3:ini10.5}
\Bigl\vert\E_{n\in V_{p}(\tilde{M})^{h}}f(n+h)\overline{f}(n)\otimes \chi_{h}(n)\Bigr\vert\gg_{d,\e} 1
\end{equation}
for some $\chi_{h}\in\Xi^{s;O_{d,\e}(1),O_{d,\e}(1)}_{p}(\Z^{d})$ for a subset $H$ of $[p]^{d}$ of cardinality $\gg_{d,\e} p^{d}$.
This completes Step 1 described in Section \ref{3:s:otl}.

\section{A Furstenberg-Weiss type argument}\label{3:s:b2}

In this section, we conduct Step 2 described in Section \ref{3:s:otl} by showing the following intermediate result:

\begin{prop}\label{3:inductioni1}
Let $d,D,r,s\in\N_{+}, d\geq 9, C,\e>0$,  $p\gg_{C,d,D,\e,s} 1$ be a prime, $M\colon\V\to\F_{p}$ be a non-degenerate quadratic form with $\tilde{M}\colon\Z^{d}\to\Z/p$ being its regular lifting, and $f\colon\Z^{d}_{p^{r}}\to\C$ be a sequence bounded by 1 with $f(n+pm)=f(n)$ for all $m,n\in\Z_{p^{r}}^{d}$.
Suppose that for some $1\leq r_{\ast}\leq s$,  there exist 
a subset $H\subseteq \Z^{d}_{p^{r}}$ with $\vert H\vert\geq \e p^{rd}$, some $\chi_{0}\in\Xi_{p}^{(1,s);C,D}((\Z^{d})^{2})$,  some $\chi_{h}\in \Xi_{p^{r}}^{[s,r_{\ast}];C,D}(\Z^{d})$, and some $\psi_{h}\in\Nil_{p}^{s-1;C,1}(\Z^{d})$ for all $h\in H$ such that
\begin{equation}\label{3:longlong2} 
	\Bigl\vert\E_{n\in V_{p}(\tilde{M})^{h}}f(n+h)\overline{f}(n)\chi_{0}(h,n)\otimes \chi_{h}(n)\psi_{h}(n)\Bigr\vert>\e
	\end{equation}
for all $h\in H$.\footnote{In (\ref{3:longlong2}), $h$ is understood as any vector $\tilde{h}$ in $\Z^{d}$ whose projection to $\Z_{p^{r}}^{d}$ is $h$. It is clear that (\ref{3:longlong2}) is independent of the choices of  $\tilde{h}$.}
Then there exist a subset $U\subseteq H^{3}$ with $\vert U\vert\gg_{C,d,D,\e,s} p^{3rd}$ 
	   such that writing
	   $$\chi_{h_{1},h_{2},h_{3}}(n):= \chi_{h_{1}}(n)\otimes\chi_{h_{2}}(n+h_{3}-h_{2})\otimes\overline{\chi}_{h_{3}}(n)\otimes\overline{\chi}_{h_{1}+h_{2}-h_{3}}(n+h_{3}-h_{2}),$$
	   we have that 
	   	\begin{equation}\label{3:notlongago}
	   \Bigl\vert\E_{n\in V_{p}(\tilde{M})^{h_{1},h_{3},h_{3}-h_{2}}}\chi_{h_{1},h_{2},h_{3}}(n) \psi_{h_{1},h_{2},h_{3}}(n)\Bigr\vert\gg_{C,d,D,\e,s} 1
	   	\end{equation}
	for some $\psi_{h_{1},h_{2},h_{3}}\in \Nil_{p}^{s-1;O_{C,d,D,\e,s}(1),1}(\Z^{d})$ for all $(h_{1},h_{2},h_{3})\in U$. Moreover, for all $(h_{1},h_{2},h_{3})$ in $U$, $\iota(h_{1}),\iota(h_{2}),\iota(h_{3})$ are linearly independent and $M$-non-isotropic (see Appendix \ref{3:s:AppA2} for the definition).
\end{prop}	
\begin{proof}
Our strategy is similar to the one used in \cite{Gow98,GTZ12,GTZ24}. An additional difficulty is that the Fubini's theorem in our setting (Theorem \ref{3:ct}) is more delicate, which makes the computations more complicated.

 Throughout the proof we assume that $p\gg_{C,d,D,\e,s} 1$.
By (\ref{3:longlong2}) and by taking an average over $H$, we have that 
$$\Bigl\vert\E_{h\in [p^{r}]^{d}}\E_{n\in V_{p}(\tilde{M})^{h}\cap [p^{r}]^{d}}\Delta_{h}f(n)\chi_{0}(h,n)\otimes \chi_{h}(n)\psi_{h}(n)\Bigr\vert\gg_{C,d,D,\e,s} 1$$
(for $h\in\Z^{d}_{p^{r}}\backslash H$, set $\chi_{h}=\psi_{h}\equiv 0$\footnote{Note that the function $\chi_{h}$ defined this way is not a nilcharacter as its modulo is not 1. However, this does not affect our proof.}).
%Since $\vert \chi_{0}(h,n)\otimes \chi_{h}(n)\vert=1$, 
By restricting to the component with the largest absolute value and by
  switching $\psi_{h}(n)$ to $a_{h}\psi_{h}(n)$ for some $a_{h}\in \mathbb{S}$ if necessary, we may assume without loss of generality that 
\begin{equation}\label{3:443}
\vert\E_{h\in [p^{r}]^{d}}\E_{n\in V_{p}(\tilde{M})^{h}\cap [p^{r}]^{d}}\Delta_{h}f(n)\chi_{0}(h,n)\otimes \chi_{h}(n)\psi_{h}(n)\vert\gg_{C,d,D,\e,s} 1.
\end{equation}
It is not hard to see that $\Gow_{1}(V(M))$ is a consistent $M$-set of total co-dimension 2 and that the set 
$$\Omega:=\{(n,h,k)\in(\V)^{2}\colon n, n+h, n+h+k\in V(M)\}$$
is a consistent $M$-set of total co-dimension 3. Since  $d\geq 7$,   it follows from Theorem \ref{3:ct} and the Cauchy-Schwartz inequality that
 \begin{equation}\label{3:444}
\begin{split}
&\quad\Bigl\vert\E_{h\in [p^{r}]^{d}}\E_{n\in V_{p}(\tilde{M})^{h}\cap [p^{r}]^{d}}\Delta_{h}f(n)\chi_{0}(h,n)\otimes \chi_{h}(n)\psi_{h}(n)\Bigr\vert^{2}
\\&=\Bigl\vert\E_{n\in V_{p}(\tilde{M})\cap [p^{r}]^{d}}\E_{h\in (V_{p}(\tilde{M})-n)\cap [p^{r}]^{d}}\Delta_{h}f(n)\chi_{0}(h,n)\otimes \chi_{h}(n)\psi_{h}(n)\Bigr\vert^{2}+O(p^{-1/2})
\\&\ll_{C,d,D,\e,s}\E_{n\in V_{p}(\tilde{M})\cap [p^{r}]^{d}}\Bigl\vert\E_{h\in (V_{p}(\tilde{M})-n)\cap [p^{r}]^{d}}\Delta_{h}f(n)\chi_{0}(h,n)\otimes \chi_{h}(n)\psi_{h}(n)\Bigr\vert^{2}+O(p^{-1/2})
\\&=\Bigl\vert\E_{n\in V_{p}(\tilde{M})}\E_{h,h'\in V_{p}(\tilde{M})-n}f(n+h)\overline{f}(n+h')
\chi_{0}(h,n)\otimes\overline{\chi}_{0}(h',n)\otimes \chi_{h}(n)\otimes \overline{\chi}_{h'}(n)\psi_{h,h'}(n)\Bigr\vert+O(p^{-1/2})
\\&=\Bigl\vert\E_{(n,h,k)\in\iota^{-1}(\Omega)}
 f(n+h)\overline{f}(n+h+k)\chi_{0}(h,n)\otimes\overline{\chi}_{0}(h+k,n)\otimes \chi_{h}(n)\otimes \overline{\chi}_{h+k}(n)\psi_{h,h+k}(n)\Bigr\vert+O(p^{-1/2})
\end{split}
\end{equation}
where $\psi_{h,h'}$ is some element in $\Nil_{p}^{s-1;O_{C,d,D,\e,s}(1),1}(\Z^{d})$.  
By Proposition \ref{3:newappr}, (\ref{3:443}) and (\ref{3:444}) and the Pigeonhole Principle, there exist some $\psi_{k}\in\Xi_{p}^{s;O_{C,d,D,\e,s}(1),O_{C,d,D,\e,s}(1)}(\Z^{d})$ and  some $\psi'_{k}\in\Nil_{p}^{(1,s-1);O_{C,d,D,\e,s}(1),1}((\Z^{d})^{2})$ such that 
\begin{equation}\label{3:new323}
\begin{split}
&\quad\Bigl\vert\E_{(n,h,k)\in\iota^{-1}(\Omega)} f(n+h)\overline{f}(n+h+k)\psi_{k}(n)\otimes\chi_{h}(n)\otimes \overline{\chi}_{h+k}(n)\psi'_{k}(h,n)\psi_{h,h+k}(n)\Bigr\vert
\gg_{C,d,D,\e,s} 1.
\end{split}
\end{equation}
%By Lemma \ref{3:LE.5} and the Pigeonhole Principle, we may   assume without loss of generality that the term $\psi_{k}$ in (\ref{3:new323}) belongs to $\Nil_{p}^{s;O_{C,D,\e,s}(1),1}(\V)$. ???
%
Since $\psi'_{k}(h,\cdot)\in \Nil_{p}^{s-1;O_{C,d,D,\e,s}(1),1}(\Z^{d})$, we may absorb the term $\psi'_{k}(h,n)$ by $\psi_{h,h+k}(n)$ in (\ref{3:new323}). 
Since $\psi_{k}\in\Xi_{p}^{s;O_{C,D,\e,s}(1),O_{C,d,D,\e,s}(1)}(\Z^{d})$, it follows from Lemma \ref{3:LE8} (viii) that
  the map $n\mapsto \psi_{k}(n+h)\otimes\overline{\psi}_{k}(n)$ belongs to $\Xi_{p}^{s-1;O_{C,d,D,\e,s}(1),O_{C,d,D,\e,s}(1)}(\Z^{d})$. So we may replace the term $\psi_{k}(n)$ by $\psi_{k}(n+h)$ in (\ref{3:new323}). 
Therefore, we may assume without loss of generality that 
\begin{equation}\label{3:new324}
\begin{split}
\Bigl\vert\E_{(n,h,k)\in\iota^{-1}(\Omega)\cap [p^{r}]^{d}}f(n+h)\overline{f}(n+h+k)\psi_{k}(n+h)\otimes\chi_{h}(n)\otimes \overline{\chi}_{h+k}(n)\psi_{h,h+k}(n)\Bigr\vert
\gg_{C,d,D,\e,s} 1.
\end{split}
\end{equation}
Finally, by the Pigeonhole Principle and by decomposing $\psi_{k}$ into its components, we may assume that the term $\psi_{k}(n+h)$ in (\ref{3:new324}) belongs to $\Nil_{p}^{s;O_{C,d,D,\e,s}(1),1}(\Z^{d})$ instead of $\Xi_{p}^{s;O_{C,d,D,\e,s}(1),O_{C,d,D,\e,s}(1)}(\Z^{d})$.

It is not hard to see that the set
$$\Omega':=\{(m,h,h',k)\in(\V)^{4}\colon m,m+k,m-h,m-h'\in V(M)\}$$
is a consistent $M$-set of total co-dimension 4, since $d\geq 9$,
by Theorem \ref{3:ct}, (\ref{3:new324}) and the Cauchy-Schwartz inequality, we have that 
\begin{equation}\nonumber
\begin{split}
&\quad1\ll_{C,d,D,\e,s} 
\Bigl\vert\E_{m-h,m,m+k\in V_{p}(\tilde{M})}f(m)\overline{f}(m+k)\psi_{k}(m)\chi_{h}(m-h)\otimes \overline{\chi}_{h+k}(m-h)\psi_{h,h+k}(m-h)\Bigr\vert^{2}
\\&=\Bigl\vert\E_{m,m+k\in V_{p}(\tilde{M})\cap[p^{r}]^{d}}f(m)\overline{f}(m+k)\psi_{k}(m)
\\&\qquad\qquad\E_{h\in (m-V_{p}(\tilde{M}))\cap[p^{r}]^{d}}\chi_{h}(m-h)\otimes \overline{\chi}_{h+k}(m-h)\psi_{h,h+k}(m-h)\Bigr\vert^{2}+O(p^{-1/2})
\\&\ll_{C,d,D,\e,s} \E_{m,m+k\in V_{p}(\tilde{M})}\Bigl\vert\E_{h\in (m-V_{p}(\tilde{M}))\cap[p^{r}]^{d}}\chi_{h}(m-h)\otimes \overline{\chi}_{h+k}(m-h)\psi_{h,h+k}(m-h)\Bigr\vert^{2}+O(p^{-1/2})
\\&=\Bigl\vert\E_{(m,h,h',k)\in\iota^{-1}(\Omega')}\chi_{h}(m-h)\otimes \overline{\chi}_{h+k}(m-h)\otimes\overline{\chi}_{h'}(m-h')\otimes \chi_{h'+k}(m-h')\psi'_{h,h',k}(m-h)\Bigr\vert+O(p^{-1/2}),
\end{split}
\end{equation}
where $\psi'_{h,h',k}$ is some element in $\Nil_{p}^{s-1;O_{C,d,D,\e,s}(1),1}(\Z^{d})$.  
Setting $h_{1}:=h, h_{2}:=h'+k, h_{3}:=h+k,$ and $n:=m-h$, we have that the average of 
\begin{equation}\nonumber
\begin{split}
\Bigl\vert\E_{n\in V_{p}(\tilde{M})^{h_{1},h_{3},h_{3}-h_{2}}\cap [p^{r}]^{d}}\chi_{h_{1}}(n)\otimes \overline{\chi}_{h_{3}}(n)\otimes\overline{\chi}_{h_{1}+h_{2}-h_{3}}(n+h_{3}-h_{2})\otimes\chi_{h_{2}}(n+h_{3}-h_{2})\psi_{h_{1},h_{2},h_{3}}(n)\Bigr\vert
\end{split}
\end{equation}
over all $h_{1},h_{2},h_{3}\in\Z_{p^{r}}^{d}$ is at least $\gg_{C,d,D,\e,s} 1$,
where $\psi_{h_{1},h_{2},h_{3}}$ is some element in the set $\Nil_{p}^{s-1;O_{C,d,D,\e,s}(1),1}(\Z^{d})$.  
By the Pigeonhole Principle, there exists a subset $U\subseteq H^{3}$ with $\vert U\vert\gg_{C,d,D,\e,s} p^{3rd}$  such that 
(\ref{3:notlongago}) holds for all $(h_{1},h_{2},h_{3})\in U$ (note that (\ref{3:notlongago}) fails if $(h_{1},h_{2},h_{3})\notin H^{3}$ by our construction). Finally, since $d\geq 3$, by Lemma \ref{3:iiddpp} we may remove from $U$ the tuples  $(h_{1},h_{2},h_{3})\in U$ with $\iota(h_{1}),\iota(h_{2}),\iota(h_{3})$ being linearly dependent or $M$-isotropic while keeping the cardinality of $U$ to be $\gg_{C,d,D,\e,s} p^{3rd}$. This completes the proof.
\end{proof}

We also have the following simpler version of Proposition \ref{3:inductioni1}.

\begin{prop}\label{3:inductioni2}
	Let $d,D,s\in\N_{+}$ with $d\geq 9$, $C,\e>0$, $p\gg_{C,d,D,\e,s} 1$ be a prime,
	 $M\colon\V\to\F_{p}$ be a non-degenerate quadratic form, and
	 $f\colon\V\to\C$ be bounded by 1.
	Suppose that there exist  a subset $H\subseteq \V$ with $\vert H\vert\geq \e p^{d}$,  and some functions $g_{h}\colon\V\to\C^{D}$ bounded by $C$ for all $h\in H$ such that 
		\begin{equation}\nonumber
		\Bigl\vert\E_{n\in V(M)^{h}}f(n+h)\overline{f}(n)g_{h}(n)\Bigr\vert>\e
		\end{equation}
	for all $h\in H$.  Denote $g_{h}\equiv 0$ if $h\notin H$.
	Then for all $h_{1},h_{2},h_{3}\in H$, writing 
		$$g_{h_{1},h_{2},h_{3}}(n):= g_{h_{1}}(n)\otimes g_{h_{2}}(n+h_{3}-h_{2})\otimes\overline{g}_{h_{3}}(n)\otimes\overline{g}_{h_{1}+h_{2}-h_{3}}(n+h_{3}-h_{2}),$$
	we have that 
	\begin{equation}\nonumber
	\Bigl\vert\E_{h_{1},h_{2},h_{3}\in\V}\E_{n\in V(M)^{h_{1},h_{3},h_{3}-h_{2}}}g_{h_{1},h_{2},h_{3}}(n) \Bigr\vert\gg_{C,d,D,\e,s} 1.
	\end{equation}
	In particular,
	there exist a subset $U\subseteq H^{3}$ with $\vert U\vert\gg_{C,d,D,\e,s}p^{3d}$ 
	such that for all $(h_{1},h_{2},h_{3})\in U$, $h_{1},h_{2},h_{3}$ are linearly independent and $M$-non-isotropic,  and that 
	\begin{equation}\nonumber
	\Bigl\vert\E_{n\in V(M)^{h_{1},h_{3},h_{3}-h_{2}}}g_{h_{1},h_{2},h_{3}}(n) \Bigr\vert\gg_{C,d,D,\e,s} 1.
	\end{equation}
\end{prop}	

The proof of Proposition \ref{3:inductioni2} is almost the same as that of Proposition \ref{3:inductioni1}. We leave its proof to the interested readers.

\section{The sunflower lemma}\label{3:s:b3}

We will use Proposition \ref{3:inductioni1} to analysis the ``frequency" of nilcharacters. It is helpful for the readers to imagine each nilcharacter $\chi_{h}$ as a function of the form $\exp(P_{h}(n))$ for some polynomial $P_{h}\colon\Z^{d}\to \Z/p^{r}$, and to imagine its ``frequency" as the leading terms of $P_{h}$.
In Section 10 of \cite{GTZ12},  Green, Tao and Ziegler proved a sunflower lemma (Lemma 10.10 of \cite{GTZ12}) explaining how to decompose a sequence of real number as linear combinations of generators with some ``sunflower-type" independence properties. In our setting, the ``frequency" of a nilcharacter is a homogeneous polynomial over $d$-variables (which can be viewed as a ``string" of real numbers) instead of a real number. Therefore, we need to extend many concepts in  \cite{GTZ12} from real numbers to strings of real numbers.
For convenience, throughout this section, we use $m$ to denote elements in $\N^{d}$.

\begin{defn}[Strings]
Let $d,p,r,s\in\N_{+}$,
$\vec{D}=(D_{j})_{1\leq j\leq s}\in\N^{s}$ and  $\mathcal{T}=(\xi_{m,k})_{1\leq \vert m\vert\leq s, 1\leq k\leq D_{\vert m\vert}}$ be a tuple of elements of the form  $\xi_{m,k}\in\Z/p^{r}$ (here $m\in\N^{d}$).  We say that $(\vec{D},\mathcal{T})$ (or $\mathcal{T}$) is a \emph{$(d,p^{r})$-tower} (or simply a \emph{tower} when there is no confusion). We can also write $\mathcal{T}=(\mathcal{F}_{j})_{1\leq j\leq s}$, where $\mathcal{F}_{j}:=(\xi_{m,k})_{\vert m\vert=j, 1\leq k\leq D_{j}}$ is the \emph{$j$-th floor} of $\mathcal{T}$. We can further write $\mathcal{F}_{j}=(\vec{\xi}_{j,k})_{1\leq k\leq D_{j}}$, where $\vec{\xi}_{j,k}:=(\xi_{m,k})_{\vert m\vert=j}$ is a \emph{$(j,d,p^{r})$-string} (when there is no confusion, we call it to be a \emph{$(j,d)$-string}, \emph{$j$-string} or simply a \emph{string}).
We may define the addition and scalar multiplication operators for strings, floors and towers by treating them as vectors.
\end{defn}

It is helpful to think of a string as a homogeneous polynomial. Indeed,
%for any $a\in\Z\backslash\{0\}$ with $\vert a\vert<p$, let $a^{\ast}$ denote the unique integer in $\{1,\dots,p-1\}$ such that $aa^{\ast}\equiv 1 \mod p\Z$.  
we may \emph{associate} every $(j,d,p^{r})$-string $\vec{\xi}=(\xi_{m})_{\vert m\vert=j}$ with %a homogeneous $\Z/p^{r}$-coefficient polynomial
some  $f\in\st_{\Z/p^{r},s}(d)$  (recall the definition in Section \ref{3:s:defn})
%$f\in\poly(\Z^{d}\to\Z/p^{r})$ 
given by $f(n)=\sum_{\vert m\vert=j}(m!)^{\ast}\xi_{m}n^{m}$, where $(m!)^{\ast}$ is the unique integer in $\{0,\dots,p^{r}-1\}$ with $(m!)^{\ast}m!\equiv 1 \mod p^{r}\Z$, and conversely \emph{associate} every %$\Z/p^{r}$-coefficient polynomial $f\in\poly(\Z^{d}\to\Z/p^{r})$ 
 $f\in\st_{\Z/p^{r},s}(d)$ given by $f(n)=\sum_{\vert m\vert=j}a_{m}n^{m}, a_{m}\in\frac{1}{p^{r}}\Z$ with a $j$-string $\vec{\xi}=(m!a_{m})_{\vert m\vert=j}$.

In our setting, it is convenient to identify two polynomials if their difference vanishes on a subset of $\Z^{d}$ of our interest. This leads to the following definition:

\begin{defn}[Reducible strings]\label{3:redst}
Let  $\Omega$ be a subset of $\Z^{d}$. We say that a polynomial $f\colon\Z^{d}\to \R$ of degree $j$ is  \emph{$(\Omega,p)$-reducible} if $\Delta_{h_{j}}\dots\Delta_{h_{1}}f(n)\in\Z$
for all $(n,h_{1},\dots,h_{j})\in \Gow_{p,j}(\Omega)$ (recall the definition in Section \ref{3:s:defn}). 
A $j$-string $\vec{\xi}=(\xi_{m})_{\vert m\vert=j}$ is \emph{$(\Omega,p)$-reducible} if the polynomial associated to it is  $(\Omega,p)$-reducible.
 We say that a tower is \emph{$(\Omega,p)$-reducible} if every string of the tower is $(\Omega,p)$-reducible. 
 \end{defn}
 
 We have the following characterization for $(V_{p}(\tilde{M})^{h_{1},\dots,h_{s}},p)$-reducible strings, using the language developed in \cite{SunB} (see Appendix \ref{3:s:AppA8} for definitions).

 \begin{lem}\label{3:att55}
 	Let $d,j,k,r\in\N_{+}, d\geq k+j^{2}+j+3$, $p\gg_{d} 1$ be a prime, $M\colon\V\to\F_{p}$ be a non-degenerate quadratic form with $\tilde{M}\colon\Z^{d}\to\Z/p$ being its regular lifting, and $h_{1},\dots,h_{k}\in \Z^{d}$ with $\iota(h_{1}),\dots,\iota(h_{k})$ being linearly independent and $M$-non-isotropic.  Let  $\vec{\xi}=(\xi_{m})_{\vert m\vert=j}$ be a $(j,d,p^{r})$-string and  $f\in\st_{\Z/p^{r},s}(d)$ be the   polynomial associated to $\vec{\xi}$. %and $F\colon\V\to\F_{p}$ be the map induced by $f$. 
	Then $\vec{\xi}$ (or $f$) is  $(V_{p}(\tilde{M})^{h_{1},\dots,h_{k}},p)$-reducible if and only if $f\in J^{M}_{\iota(h_{1}),\dots,\iota(h_{k})}$.\footnote{Here $V_{p}(\tilde{M})^{h_{1},\dots,h_{k}}$ is understood as $V_{p}(\tilde{M})^{\tilde{h}_{1},\dots,\tilde{h}_{k}}$ for any $\tilde{h}_{i}$ whose projection to $\Z_{p^{r}}^{d}$ is $h_{i}$. Cleary the set $V_{p}(\tilde{M})^{h_{1},\dots,h_{k}}$ is independent of the choices of $\tilde{h}_{1},\dots,\tilde{h}_{k}$.}
 \end{lem}
 \begin{proof}
 Let $V$ be the span of $\iota(h_{1}),\dots,\iota(h_{k})$.
  Since $\iota(h_{1}),\dots,\iota(h_{k})$ are linearly independent, we have that $\iota(V_{p}(\tilde{M})^{h_{1},\dots,h_{k}})=V(M)^{\iota(h_{1}),\dots,\iota(h_{k})}=V(M)\cap (V^{\pp}+c)$ for some $c\in\V$. Since $\iota(h_{1}),\dots,\iota(h_{k})$ are $M$-non-isotropic, we have that $\rank(M\vert_{V^{\pp}})=\dim(V^{\pp})=d-k\geq j+3$.    Let $A$ be the matrix associated to $M$ and let $M_{0}\colon\V\to\F_{p}$ be the quadratic form given by $M_{0}(n):=(nA)\cdot n$.
 
 Since  $f$ is homogeneous,
   it follows from Proposition \ref{3:att3} that 
 $f$ is $(\Omega,p)$-reducible if and only if
      $qf\in \poly(\iota^{-1}(V(M_{0})\cap V^{\pp})\to \R\vert\Z)$ for some $q\in\N, p\nmid q$, or equivalently, $qf\in J^{M}_{\iota(h_{1}),\dots,\iota(h_{k})}$.  
      
      If $f\in  J^{M}_{\iota(h_{1}),\dots,\iota(h_{k})}$, then clearly $qf\in  J^{M}_{\iota(h_{1}),\dots,\iota(h_{k})}$.
     Conversely, if $qf\in J^{M}_{\iota(h_{1}),\dots,\iota(h_{k})}$, then since all the coefficients of $f$ are in $\Z/p^{r}$, we have that $(1-q^{\ast}q)f\in J^{M}$ and thus  
       $f=(1-q^{\ast}q)f+q^{\ast}(qf)\in J^{M}_{\iota(h_{1}),\dots,\iota(h_{k})}$,where  $q^{\ast}\in\Z$ is any integer such that $q^{\ast}q\equiv 1\mod p^{r}\Z$.
%		
%		  	By definition, $\vec{\xi}$ is \ $(\tau(V(M)^{h_{1},\dots,h_{s}}),p)$-reducible if and only if $\Delta_{n_{j}}\dots\Delta_{n_{1}}f(n)\in\Z$
 %	for all $(n,n_{1},\dots,n_{j})\in \Gow_{p,j}(\tau(V(M)^{h_{1},\dots,h_{s}}))$.
%	By Proposition \ref{3:att3} (the part that (i) $\Leftrightarrow$ (iii)) and the fact that $f$ is homogeneous, this is equivalent of saying that $F\in J^{M}_{h_{1},\dots,h_{s}}$.
 \end{proof}

 We may define linear independent/combination for strings in the natural way (modulo reducible strings):

\begin{defn}[Linear independent/combination for strings]
	Let  $\Omega$ be a subset of $\Z^{d}$ and  $c\in\N$. We say that $j$-strings $\vec{\xi_{(1)}},\dots, \vec{\xi_{(\ell)}}$ are \emph{$(\Omega,c,p)$-independent} if for every $1\leq j\leq s$ and $a_{1},\dots,a_{\ell}\in\Z$ with $-c\leq a_{k}\leq c$, if $\sum_{k=1}^{\ell}a_{k}\vec{\xi_{(k)}}$ is $(\Omega,p)$-reducible, then $a_{1}=\dots=a_{k}=0$.
A string which can be written in the form $\vec{\xi}+\sum_{k=1}^{\ell}a_{k}\vec{\xi_{(k)}}$ for some $-c\leq a_{k}\leq c$ and some $(\Omega,p)$-reducible string $\vec{\xi}$ is called an \emph{$(\Omega,c,p)$-linear combination}  of $\vec{\xi_{(1)}},\dots, \vec{\xi_{(\ell)}}$.
	
Let $\mathcal{T}=(\vec{\xi}_{j,k})_{1\leq j\leq s, 1\leq k\leq D_{j}}$ be a $(d,p)$-tower. We say that $\mathcal{T}$ is \emph{$(\Omega,c,p)$-independent} if for every $1\leq j\leq s$ and the $j$-strings $\vec{\xi}_{j,1},\dots,\vec{\xi}_{j,D_{j}}$ are $(\Omega,c,p)$-independent.
 \end{defn}

 We may also extend the definition of $p$-almost linear function/Freiman  homomorphism defined in \cite{SunB} to strings (see Appendix \ref{3:s:AppA8} for definitions):
 
 \begin{defn}[$p$-almost linear function/Freiman homomorphism for strings]
 	Let $d,j,L$, $r\in\N_{+}$, $p$ be a prime, and $H$ be a subset of $\Z_{p^{r}}^{d}$. 
 	We say that a map $h\mapsto \vec{\xi}_{h}=(\xi_{h,m})_{m\in\N^{d},\vert m\vert=j}$ from $H$ to the set of $(j,d,p^{r})$-strings is a \emph{$\Z/p^{r}$-almost linear function/Freiman homomorphism} (of complexity at most $L$)    if the map $h\mapsto \xi_{h,m}$ is a $\Z/p^{r}$-almost linear function/Freiman homomorphism (of complexity at most $L$) for all $m\in\N^{d},\vert m\vert=j$.
 \end{defn}	

We are now ready to prove the main result for this section, which says that any sequences of strings can be written as linear combinations of generators with some sunflower-type independence properties. This is where we make essential use of the  additive combinatorics result obtained in the previous part of the series \cite{SunB} (i.e. Theorem \ref{3:aadd} of this paper).

\begin{lem}[Sunflower lemma for frequencies]\label{3:sf1}
	Let $d,j,\ell,L,r\in\N_{+}$, $\d,\e>0$, $p\gg_{\d,d,\e,\ell,L,r} 1$ be a prime, $M\colon \V\to\F_{p}$ be a non-degenerate quadratic form with $\tilde{M}\colon\Z^{d}\to\Z/p$ being its regular lifting, and 
	  $H$ be a subset of $\Z^{d}_{p^{r}}$ with $\vert H\vert>\e p^{rd}$. For all $h\in H$, let $\vec{\xi}_{h,1},\dots,\vec{\xi}_{h,\ell}$ be $(j,d,p^{r})$-strings.  If $d\geq N(j)$, 	then there exist $a\in\N_{+}$, $a=O_{\ell}(1)$, a subset $H'$ of $H$ with $\vert H'\vert\gg_{\d,d,\e,\ell,L,r} p^{rd}$, integers
	 $\ell_{\cor},\ell_{\ind},\ell_{\lin}=O_{\ell}(1)$,
a collection of ``core" $(j,d,p^{r})$-strings $$\mathcal{S}^{\cor}=\{\vec{\xi}^{\cor}_{1},\dots,\vec{\xi}^{\cor}_{\ell_{\cor}}\}$$ 
	 collections of  ``petal" $(j,d,p^{r})$-strings
	$$\mathcal{S}^{\ind}_{h}=\{\vec{\xi}^{\ind}_{h,1},\dots,\vec{\xi}^{\ind}_{h,\ell_{\ind}}\},  \mathcal{S}^{\lin}_{h}=\{\vec{\xi}^{\lin}_{h,1},\dots,\vec{\xi}^{\lin}_{h,\ell_{\lin}}\}$$
	for all $h\in H'$ such that the following holds:
	\begin{enumerate}[(i)]
		\item For all $h\in H'$ and $1\leq k\leq \ell$, $\vec{\xi}_{h,k}$ is a $(V_{p}(\tilde{M})^{h},L^{a},p)$-linear combination of 
		$$\mathcal{S}^{\cor}\cup\mathcal{S}^{\ind}_{h}\cup \mathcal{S}^{\lin}_{h}.$$
		\item There exists a subset $\Gamma'$ of $\Gamma:=\{(h_{1},\dots,h_{4})\in {H'}^{4}\colon h_{1}+h_{2}=h_{3}+h_{4}\}$ with $\vert\Gamma'\vert>(1-\d)\vert \Gamma\vert$ such that for all $(h_{1},\dots,h_{4})\in \Gamma'$,		the $(j,d,p^{r})$-strings in $$\mathcal{S}^{\cor}\cup\cup_{i=1}^{4}\mathcal{S}^{\ind}_{h_{i}}\cup\cup_{i=1}^{3} \mathcal{S}^{\lin}_{h_{i}}$$
	are $(V_{p}(\tilde{M})^{h_{1},h_{2},h_{3}},L,p)$-independent. 		
	\item For all $1\leq k\leq \ell_{\lin}$, the map $h\mapsto\vec{\xi}^{\lin}_{h,k}, h\in H'$ is a $\Z/p^{r}$-almost linear Freiman homomorphism of complexity $O_{\d,d,\e,\ell,L,r}(1)$. 
	\end{enumerate}	
\end{lem}

\begin{proof}%[Proof of Lemma \ref{3:sf1}]
The outline of the proof of Lemma \ref{3:sf1} is similar to that of Lemma 10.5 of \cite{GTZ12}.  
	We say that a choice of $H',\ell_{\ind},\ell_{\lin},\ell_{\cor},a$, $\mathcal{S}^{\ind}_{h},\mathcal{S}^{\lin}_{h},\mathcal{S}^{\cor},h\in H'$ is a \emph{partial solution with weight} $(\ell_{\ind},\ell_{\lin},\ell_{\cor})$ if
	 all the stated conditions are satisfied expect Condition (ii). Define the lexicographical ordering  on the weights by setting $(\ell_{\ind},\ell_{\lin},\ell_{\cor})>(\ell'_{\ind},\ell'_{\lin},\ell'_{\cor})$ if $\ell_{\ind}>\ell'_{\ind}$; or $\ell_{\ind}=\ell'_{\ind}$, $\ell_{\lin}>\ell'_{\lin}$; or $\ell_{\ind}=\ell'_{\ind}, \ell_{\lin}=\ell'_{\lin}, \ell_{\cor}>\ell'_{\cor}$.
	
	By setting $H'=H$, $\ell_{\ind}=\ell,\ell_{\lin}=\ell_{\cor}=0, a=1$ and $\vec{\xi}^{\ind}_{h,k}:=\vec{\xi}_{h,k}$ for all $h\in H, 1\leq k\leq \ell$, we get a  partial solution with weight $(\ell,0,0)$. 
	
	\textbf{Claim:} If  
	$$H',\ell_{\ind},\ell_{\lin},\ell_{\cor},a,\mathcal{S}^{\ind}_{h},\mathcal{S}^{\lin}_{h},\mathcal{S}^{\cor}, h\in H'$$
	is a partial solution with $\vert H'\vert\gg_{\d,d,\e,\ell,L,r} p^{rd}$ such that Condition (ii) is not satisfied and that $2\ell_{\ind}+\ell_{\lin}+\ell_{\cor}\leq 2\ell$, then there exists a partial solution 
	$$H'',\ell'_{\ind},\ell'_{\lin},\ell'_{\cor},a+1,{\mathcal{S}'}^{\ind}_{h},{\mathcal{S}'}^{\lin}_{h},{\mathcal{S}'}^{\cor}, h\in H''$$
	with 
	\begin{equation}\label{3:llddll}
	(\ell'_{\ind},\ell'_{\lin},\ell'_{\cor})<(\ell_{\ind},\ell_{\lin},\ell_{\cor}), \text{ and } 2\ell'_{\ind}+\ell'_{\lin}+\ell'_{\cor}\leq 2\ell_{\ind}+\ell_{\lin}+\ell_{\cor},
	\end{equation}
	 and that $\vert H''\vert\gg_{\d,d,\e,\ell,L,r} p^{rd}$.
	
	Throughout the proof, for $a\in\Z\backslash p\Z$, we use $a^{\ast}$ to denote the unique integer in $\{0,\dots,p^{r}-1\}$ with $a^{\ast}a\equiv 1\mod p^{r}\Z$.
	By Lemma \ref{3:iiddpp},  we may assume without loss of generality that the image of all the vectors in $H'$ under $\iota$ are $M$-non-isotropic.
	We may write 
	$$\mathcal{S}^{\ind}_{h}=\{\vec{\xi}^{\ind}_{h,k}\colon 1\leq k\leq\ell_{\ind}\},\mathcal{S}^{\lin}_{h}=\{\vec{\xi}^{\lin}_{h,k}\colon 1\leq k\leq\ell_{\lin}\},\mathcal{S}^{\cor}=\{\vec{\xi}^{\cor}_{k}\colon 1\leq k\leq\ell_{\cor}\}$$
	for all $h\in H'$.
	For $h\in\Z_{p^{r}}^{d}$, let $H'_{h}$ denote the set of $(h_{1},h_{2})\in {H'}^{2}$ such that $h_{1}+h_{2}=h$. Then   
		\begin{equation}\label{3:thisisr}
		\vert\{(h_{1},h_{2},h_{3},h_{4})\in H'^{4}\colon h_{1}+h_{2}=h_{3}+h_{4}\}\vert=\sum_{h\in\Z_{p^{r}}^{d}}\vert H'_{h}\vert^{2}\geq \frac{1}{p^{rd}}\Bigl(\sum_{h\in\Z_{p^{r}}^{d}}\vert H'_{h}\vert\Bigr)^{2}=\frac{\vert H'\vert^{4}}{p^{rd}}.
		\end{equation}
		Since
$\vert H'\vert\gg_{\d,d,\e,\ell,L,r} p^{rd}$,
	by assumption, (\ref{3:thisisr}) and the Pigeonhole Principle,  there exist $-L\leq a^{\cor}_{k}, a^{\ind}_{i,k}, a^{\lin}_{i,k}\leq L$ not all equal to zero and a subset $W$ of $\{(h_{1},h_{2},h_{3},h_{4})\in H'^{4}\colon h_{1}+h_{2}=h_{3}+h_{4}\}$ with $\vert W\vert\gg_{\d,d,\e,\ell,L,r} p^{3rd}$   such that for all $(h_{1},h_{2},h_{3},h_{4})\in W$, we have that
	\begin{equation}\label{3:mm1}
		\sum_{k=1}^{\ell_{\cor}}a^{\cor}_{k}\vec{\xi}^{\cor}_{k}+\sum_{i=1}^{4}\sum_{k=1}^{\ell_{\ind}}a^{\ind}_{i,k}\vec{\xi}^{\ind}_{h_{i},k}+\sum_{i=1}^{3}\sum_{k=1}^{\ell_{\lin}}a^{\lin}_{i,k}\vec{\xi}^{\lin}_{h_{i},k}
		\text{ is $(V_{p}(\tilde{M})^{h_{1},h_{2},h_{3}},p)$-reducible.}
	\end{equation}
	
	Let $Z$ be the set of $h_{1},h_{2},h_{3},h_{4}\in\Z_{p^{r}}^{d}$ such that $h_{1}+h_{2}=h_{3}+h_{4}$ and that $\iota(h_{1}),\iota(h_{2}),\iota(h_{3})$ are either linearly dependent or $M$-non-isotropic. Let $W':=W\backslash Z$. 
	Since $d\geq 3$, by Lemma \ref{3:iiddpp}, we have that 
	  $\vert W'\vert\gg_{\d,d,\e,\ell,L,r} p^{3rd}$. We consider three different cases.

	\textbf{Case 1:} all of $a^{\ind}_{i,k}, a^{\lin}_{i,k}$ are zero. 
	 We may assume without loss of generality that $a^{\cor}_{1}\neq 0$. 
	 Let $\vec{\xi'}^{\cor}_{k}:=(a^{\cor}_{1})^{\ast}\vec{\xi}^{\cor}_{k}$, $\vec{\xi'}^{\ind}_{h,k}:=(a^{\cor}_{1})^{\ast}\vec{\xi}^{\ind}_{h,k}$
	 and $\vec{\xi'}^{\lin}_{h,k}:=(a^{\cor}_{1})^{\ast}\vec{\xi}^{\ind}_{h,k}$. 
		 Note that (\ref{3:mm1}) implies that $\vec{\xi}^{\cor}_{1}+\sum_{k=2}^{\ell_{\cor}}a^{\cor}_{k}\vec{\xi'}^{\cor}_{k}$
	 is $(V_{p}(\tilde{M})^{h_{1},h_{2},h_{3}},p)$-reducible for all $(h_{1},h_{2},h_{3},h_{4})\in W'$.
	 We show that $\vec{\xi}^{\cor}_{1}+\sum_{k=2}^{\ell_{\cor}}a^{\cor}_{k}\vec{\xi'}^{\cor}_{k}$
	 is $(V_{p}(\tilde{M}),p)$-reducible by the intersection method introduced in \cite{SunB}.
	 
	 Let $F\in\st_{\Z/p^{r},d}(j)$ be the polynomial  associated to the $j$-string $$\vec{\xi}^{\cor}_{1}+\sum_{k=2}^{\ell_{\cor}}a^{\cor}_{k}\vec{\xi'}^{\cor}_{k}.$$
	 Since $\vert W'\vert\gg_{\d,d,\e,\ell,L,r}p^{3rd}$ and $d\geq N(j)$,  there exist $(h_{i,1},h_{i,2},h_{i,3},h_{i,4})\in W'', 1\leq i\leq j+1$ such that $\iota(h_{i,1}),\iota(h_{i,2}),\iota(h_{i,3}), 1\leq i\leq j+1$ are linearly independent.
	So for all $1\leq i\leq j+1$,
	 the $j$-string in $\vec{\xi}^{\cor}_{1}+\sum_{k=2}^{\ell_{\cor}}a^{\cor}_{k}\vec{\xi'}^{\cor}_{k}$ is $(V_{p}(\tilde{M})^{h_{i,1},h_{i,2},h_{i,3}},p)$-reducible.  Since $\iota(h_{i,1}),\iota(h_{i,2}),\iota(h_{i,3})$ are not $M$-isotropic, by Lemma \ref{3:att55}, $F$ belongs to $J^{M}_{\iota(h_{i,1}),\iota(h_{i,2}),\iota(h_{i,3})}$.  
	 By Proposition    \ref{3:gr0} (setting $m=0, r=3$ and $L=j$), we have that $F$ belongs to $J^{M}$.
So by Lemma \ref{3:att55},
     $\vec{\xi}^{\cor}_{1}+\sum_{k=2}^{\ell_{\cor}}a^{\cor}_{k}\vec{\xi'}^{\cor}_{k}$
	 is $(V_{p}(\tilde{M}),p)$-reducible.

	 This means that  $\vec{\xi}^{\cor}_{1}$ is a $(V_{p}(\tilde{M}),L,p)$-linear combination of $\vec{\xi'}^{\cor}_{2},\dots,\vec{\xi'}^{\cor}_{\ell_{\cor}}$.
	 Since $\vec{\xi}_{h,k}$ is a $(V_{p}(\tilde{M}),L^{a},p)$-linear combination of the strings in $\mathcal{S}^{\cor}\cup\mathcal{S}^{\ind}_{h}\cup \mathcal{S}^{\lin}_{h}$, it is also a $(V_{p}(\tilde{M}),L^{a+1},p)$-linear combination of the strings in
	  ${\mathcal{S}'}^{\ind}_{h},{\mathcal{S}'}^{\lin}_{h},{\mathcal{S}'}^{\cor}$, where 
	 $${\mathcal{S}'}^{\ind}_{h}=\{\vec{\xi'}^{\ind}_{h,k}\colon 1\leq k\leq \ell_{\ind}\},
	 {\mathcal{S}'}^{\lin}_{h}=\{\vec{\xi'}^{\lin}_{h,k}\colon 1\leq k\leq \ell_{\lin}\},
	 {\mathcal{S}'}^{\cor}=\{\vec{\xi'}^{\cor}_{k}\colon 2\leq k\leq \ell_{\cor}\}.$$
	 This means that 
	 $$H',\ell_{\ind},\ell_{\lin},\ell_{\cor}-1,a+1,  
	 {\mathcal{S}'}^{\ind}_{h},{\mathcal{S}'}^{\lin}_{h},{\mathcal{S}'}^{\cor}, h\in H'$$
	  is a partial solution satisfying (\ref{3:llddll}).

	 \textbf{Case 2:} all of $a^{\ind}_{i,k}$ are zero, but not all of $a^{\lin}_{i,k}$ are zero. We may assume without loss of generality that $a^{\lin}_{1,1}\neq 0$. 
	 For $h\in H'$, let $U_{h}$ denote the set of $(h_{2},h_{3})\in {H'}^{2}$ such that $(h,h_{2},h_{3},h+h_{2}-h_{3})\in W'$ for all $1\leq i\leq j+3$. and $U'_{h}$ denote the set of $(h_{1,2},h_{1,3},\dots,h_{j+3,2},h_{j+3,3})\in {H'}^{2(j+3)}$ such that $(h,h_{i,2},h_{i,3},h+h_{i,2}-h_{i,3})\in W'$ for all $1\leq i\leq j+3$.
	  For $\bold{h}=(h_{1,2},h_{1,3},\dots,h_{j+3,2},h_{j+3,3})\in {H'}^{2(j+3)}$, let $U''_{\h}$ denote the set of $h\in H'$ such that  $\h\in U'_{h}$.
	 Then
	 $$\sum_{\h\in {H'}^{2(j+3)}}\vert U''_{\h}\vert=\sum_{h\in H'}\vert U'_{h}\vert=\sum_{h\in H'}\vert U_{h}\vert^{j+3}\geq \frac{1}{\vert H'\vert^{j+2}}\Bigl(\sum_{h\in H'}\vert U_{h}\vert\Bigr)^{j+3}=\frac{\vert W'\vert^{j+3}}{\vert H'\vert^{j+2}}\gg_{\d,d,\e,\ell,L,r} p^{(2j+7)rd}.$$	 
	 By the Pigeonhole Principle, there exists $\tilde{H}\subseteq {H'}^{2(j+3)}$ with $\vert \tilde{H}\vert\gg_{\d,d,\e,\ell,L,r} p^{2(j+3)rd}$ such that $\vert U''_{\h}\vert\gg_{\d,d,\e,\ell,L,r} p^{rd}$ for all $\h\in \tilde{H}$. By Lemma \ref{3:iiddpp}, we may pick some $$\h=(h_{1,2},h_{1,3},\dots,h_{j+3,2},h_{j+3,3})\in \tilde{H}$$ with $\iota(h_{1,2}),\iota(h_{1,3}),\dots,\iota(h_{j+3,2}),\iota(h_{j+3,3})$ being linearly independent. 
	 
	 Fix any such $\h$ and set $H''=U''_{\h}$. For all $h\in H''$,
	 let $\vec{\xi'}^{\cor}_{k}:=(a^{\lin}_{1,1})^{\ast}\vec{\xi}^{\cor}_{k}$, $\vec{\xi'}^{\ind}_{h,k}:=(a^{\lin}_{1,1})^{\ast}\vec{\xi}^{\ind}_{h,k}$ 
	 and $\vec{\xi'}^{\lin}_{h,k}:=(a^{\lin}_{1,1})^{\ast}\vec{\xi}^{\ind}_{h,k}$.
	 By (\ref{3:mm1}) and the construction of $H''$, for all $h\in H''$, we have that
	 \begin{equation}\label{3:neww22}
	 \sum_{k=1}^{\ell_{\cor}}a^{\cor}_{k}\vec{\xi'}^{\cor}_{k}+\sum_{j=2}^{3}\sum_{k=1}^{\ell_{\lin}}a^{\lin}_{j,k}\vec{\xi'}^{\lin}_{h_{i,j},k}+\sum_{k=2}^{\ell_{\lin}}a^{\lin}_{1,k}\vec{\xi'}^{\lin}_{h,k}+\vec{\xi}^{\lin}_{h,1}
	 \text{ is $(V_{p}(\tilde{M})^{h,h_{i,2},h_{i,3}},p)$-reducible.}
	 \end{equation}

	 Let $F_{i}\in\st_{\Z/p^{r},d}(j)$ be the polynomial  associated to the $j$-string $$ \sum_{k=1}^{\ell_{\cor}}a^{\cor}_{k}\vec{\xi'}^{\cor}_{k}+\sum_{j=2}^{3}\sum_{k=1}^{\ell_{\lin}}a^{\lin}_{j,k}\vec{\xi'}^{\lin}_{h_{i,j},k},$$
	 and $G_{h}\in\st_{\Z/p^{r},d}(j)$ be the polynomial  associated to the $j$-string $$ \sum_{k=2}^{\ell_{\lin}}a^{\lin}_{1,k}\vec{\xi'}^{\lin}_{h,k}+\vec{\xi}^{\lin}_{h,1}.$$
	  Since $\iota(h),\iota(h_{i,2}),\iota(h_{i,3})$ are linearly independent and $M$-non-isotropic, by Lemma \ref{3:att55} and (\ref{3:neww22}), $F_{i}+G_{h}$ belongs to $J^{M}_{\iota(h),\iota(h_{i,2}),\iota(h_{i,3})}$ for all $1\leq i\leq j+3$ and $h\in H''$. So for all $h,h'\in H''$ and $1\leq i\leq j+3$, we have that 
	  $$G_{h}-G_{h'}\equiv (G_{h}+F_{i})-(G_{h'}+F_{i})\equiv 0 \mod J^{M}_{\iota(h),\iota(h'),\iota(h_{i,2}),\iota(h_{i,3})}.$$
	  Since $\iota(h_{1,2}),\iota(h_{1,3}),\dots,\iota(h_{j+3,2}),\iota(h_{j+3,3})$ are linearly independent,
	  by Proposition \ref{3:gr0} (setting $m\leq 2, s=j, r=2$),  
	  we have that 
	  $$G_{h}\equiv G_{h'} \mod J^{M}_{\iota(h),\iota(h')}$$
	  for all $h,h'\in H''$. Since $\bold{0}\notin\iota(H'')$, by Proposition \ref{3:coco1p},   there exists $G\in\st_{\Z/p^{r},d}(j)$ (recall the definition in Section \ref{3:s:defn}) such that 
	   $$G_{h}\equiv G \mod J^{M}_{\iota(h)}$$
	   for all $h\in H''$. In other words, letting $\vec{\xi}_{0}$ denote the $j$-string associated to %any regular lifting of
	    $G$, by Lemma \ref{3:att55}, we have that 
	   $$-\vec{\xi}_{0}+\sum_{k=2}^{\ell_{\lin}}a^{\lin}_{1,k}\vec{\xi'}^{\lin}_{h,k}+\vec{\xi}^{\lin}_{h,1}$$
	   is $(V_{p}(\tilde{M})^{h},p)$-reducible for all $h\in H''$.  

	 This mean that 
	  $\vec{\xi}^{\lin}_{h,1}$ is a $(V_{p}(\tilde{M})^{h},L,p)$-linear combination of $\vec{\xi}_{0},\vec{\xi'}^{\lin}_{h,k}, 2\leq k\leq \ell_{\lin}$. Since $\vec{\xi}_{h,k}$ is a $(V_{p}(\tilde{M})^{h},L^{a},p)$-linear combination of the strings in
	   $\mathcal{S}^{\cor}\cup\mathcal{S}^{\ind}_{h}\cup \mathcal{S}^{\lin}_{h}$, it is also a $(V_{p}(\tilde{M})^{h},L^{a+1},p)$-linear combination of the strings in
	  ${\mathcal{S}'}^{\ind}_{h},{\mathcal{S}'}^{\lin}_{h},{\mathcal{S}'}^{\cor}$, where 
	  $${\mathcal{S}'}^{\ind}_{h}=\{\vec{\xi'}^{\ind}_{h,k}\colon 1\leq k\leq \ell_{\ind}\},
	  {\mathcal{S}'}^{\lin}_{h}=\{\vec{\xi'}^{\lin}_{h,k}\colon 2\leq k\leq \ell_{\lin}\},
	  {\mathcal{S}'}^{\cor}=\{\vec{\xi}_{0},\vec{\xi'}^{\cor}_{k}\colon 1\leq k\leq \ell_{\cor}\}.$$
	  This means that 
	  $$H'',\ell_{\ind},\ell_{\lin}-1,\ell_{\cor}+1,a+1,  
	  {\mathcal{S}'}^{\ind}_{h},{\mathcal{S}'}^{\lin}_{h},{\mathcal{S}'}^{\cor}, h\in H''$$
	  is a partial solution satisfying (\ref{3:llddll}).
 
	 \textbf{Case 3:}  
	  not all of $a^{\ind}_{i,k}$ are zero. We may assume without loss of generality that $a^{\lin}_{4,1}\neq 0$. Let $\Xi_{1},\dots,\Xi_{4}\colon \Vk\to (\Z/p^{r})^{\binom{j+d-1}{d-1}}$ be functions given by
	 $$\Xi_{i}(h):=\sum_{k=1}^{\ell_{\ind}}a^{\ind}_{i,k}\vec{\xi}^{\ind}_{h,k}+\sum_{k=1}^{\ell_{\lin}}a^{\lin}_{i,k}\vec{\xi}^{\lin}_{h,k}$$
	 for $i=1,2,3$, and
	  $$\Xi_{4}(h):=\sum_{k=1}^{\ell_{\cor}}a^{\cor}_{k}\vec{\xi}^{\cor}_{k}+\sum_{k=1}^{\ell_{\ind}}a^{\ind}_{4,k}\vec{\xi}^{\ind}_{h,k}.$$
	  Then (\ref{3:mm1}) implies that
	  \begin{equation}\label{3:1x2x3x4x} 
	  \text{$\Xi_{1}(h_{1})+\Xi_{2}(h_{2})+\Xi_{3}(h_{3})+\Xi_{4}(h_{4})$ is $(V_{p}(\tilde{M})^{h_{1},h_{2},h_{3}},p)$-reducible}
	  \end{equation} 
	  for all $(h_{1},h_{2},h_{3},h_{1}+h_{2}-h_{3})\in W'$. 
%	   For any $(j,d)$-string $\vec{\xi}:=(\xi_{m})_{\vert m\vert=j}$, let $\omega(\vec{\xi})$ the polynomial in $\st_{d}(j)$ given by
%	   $$\omega(\vec{\xi})(n):=\sum_{m\in\N^{d},\vert m\vert=j}\iota(p\xi_{m})(m!)^{\ast}n^{m}, n\in\V,$$ 
%	   i.e. $\omega(\vec{\xi})$ is the polynomial induced by the polynomial associated to $\vec{\xi}$.
	  For $1\leq i\leq 4$, let $\omega_{i}\colon H'\to \st_{\Z/p^{r},d}(j)$ be the map which maps $h\in H'$ to the $\Z/p^{r}$-coefficient polynomial associated to $\Xi_{i}(h)$.
%  given by $\omega_{i}(h):=\omega(\Xi_{i}(h))$.
	  Since $d\geq N(j)$ and $\iota(h_{1}),\iota(h_{2}),\iota(h_{3})$ are linearly independent and $M$-non-isotropic, by (\ref{3:1x2x3x4x}) and Lemma \ref{3:att55},  we have that
	  \begin{equation}\nonumber
	  \omega_{1}(h_{1})+\omega_{2}(h_{2})+\omega_{3}(h_{3})+\omega_{4}(h_{4})\in J^{M}_{\iota(h_{1}),\iota(h_{2}),\iota(h_{3})}.
	  \end{equation}
	 Since $d\geq N(j)$, by Theorem \ref{3:aadd}, 	 there exists
	   a subset $H''\subseteq H'$ with $\vert H''\vert\gg_{\d,d,\e,\ell,L,r} p^{rd}$, some $g\in \st_{\Z/p^{r},d}(j)$  and
	  a $\Z/p^{r}$-almost linear Freiman  homomorphism $T\colon H\to\st_{\Z/p^{r},d}(j)$ of complexity $O_{\d,d,\e,\ell,L,r}(1)$ such that $$\omega_{4}(h)-T(h)-g\in J^{M}_{\iota(h)}$$ for all $h\in H''$. By Lemma \ref{3:iiddpp}, we may assume without loss of generality that the image of all elements in $H''$ under $\iota$ are non-zero and $M$-non-isotropic.	  	  
By definition and Lemma \ref{3:att55}, there exist  a $\Z/p^{r}$-almost linear Freiman homomorphism $h\mapsto \vec{\xi}_{h}$ of complexity $O_{\d,d,\e,\ell,L,r}(1)$ and some $(j,d,p^{r})$-string $\vec{\xi}_{0}$ such that
	  $$\Xi_{4}(h)-\vec{\xi}_{h}-\vec{\xi}_{0}=a^{\ind}_{4,1}\vec{\xi}^{\ind}_{h,1}-\vec{\xi}_{h}-\vec{\xi}_{0}+\sum_{k=1}^{\ell_{\cor}}a^{\cor}_{k}\vec{\xi}^{\cor}_{k}+\sum_{k=2}^{\ell_{\ind}}a^{\ind}_{4,k}\vec{\xi}^{\ind}_{h,k}$$
	  is $(V_{p}(\tilde{M})^{h},p)$-reducible for all $h\in H''$. Let  $\vec{\xi'}_{0}:=(a^{\ind}_{4,1})^{\ast}\vec{\xi}_{0}$,
	  $\vec{\xi'}^{\cor}_{k}:=(a^{\ind}_{4,1})^{\ast}\vec{\xi}^{\cor}_{k}$,
	  $\vec{\xi'}^{\ind}_{h,k}:=(a^{\ind}_{4,1})^{\ast}\vec{\xi}^{\ind}_{h,k}$,  and	  $\vec{\xi}'_{h}:=(a^{\ind}_{4,1})^{\ast}\vec{\xi}_{h}$.
	  
	   This mean that 
	   $\vec{\xi}^{\ind}_{h,1}$ is a $(V_{p}(\tilde{M})^{h},L,p)$-linear combination of $\vec{\xi'}_{0}$, $\vec{\xi'}_{h}$, $\vec{\xi'}_{k}^{\cor}$, $\vec{\xi'}_{h,k}^{\ind}$. Since $\vec{\xi}_{h,k}$ is a $(V_{p}(\tilde{M})^{h},L^{a},p)$-linear combination of the strings in
	   $\mathcal{S}^{\cor}\cup\mathcal{S}^{\ind}_{h}\cup \mathcal{S}^{\lin}_{h}$, it is also a $(V_{p}(\tilde{M})^{h},L^{a+1},p)$-linear combination of the strings in
	   ${\mathcal{S}'}^{\ind}_{h},{\mathcal{S}'}^{\lin}_{h},{\mathcal{S}'}^{\cor}$, where 
	   $${\mathcal{S}'}^{\ind}_{h}=\{\vec{\xi'}^{\ind}_{h,k}\colon 2\leq k\leq \ell_{\ind}\},
	   {\mathcal{S}'}^{\lin}_{h}=\{\vec{\xi'}_{h},\vec{\xi'}^{\lin}_{h,k}\colon 1\leq k\leq \ell_{\lin}\},
	   {\mathcal{S}'}^{\cor}=\{\vec{\xi'}_{0},\vec{\xi'}^{\cor}_{k}\colon 1\leq k\leq \ell_{\cor}\}.$$
	   This means that 
	   $$H'',\ell_{\ind}-1,\ell_{\lin}+1,\ell_{\cor}+1,a+1,  
	   {\mathcal{S}'}^{\ind}_{h},{\mathcal{S}'}^{\lin}_{h},{\mathcal{S}'}^{\cor}, h\in H''$$
	   is a partial solution satisfying (\ref{3:llddll}).
	  This finishes the proof of the claim.
	  
	 \
	  
	  From the claim, we can find a series partial solutions with decreasing weights. Moreover, by (\ref{3:llddll}), it is not hard to see that this process terminates after at most $O_{\ell}(1)$ steps. By the claim, this partial solution satisfies Condition (ii), and we are done.
\end{proof}

Next we use the decomposition of strings in Lemma \ref{3:sf1} to obtain a decomposition of nilcharacters.
We need to recall some definitions introduced in \cite{GTZ12}:

\begin{defn}[Universal nilmanifold]
	A \emph{dimension vector} is a tuple
	$$\vec{D}=(D_{1},\dots,D_{s})\in\N^{s}.$$
	Given a dimension vector, we defined the \emph{universal nilpotent group} $G^{\vec{D},d}$ of degree-rank $[s,r_{\ast}]$ to be the Lip group generated by formal generators $e_{m,j}$ for $m\in\N^{d}, 1\leq \vert m\vert\leq s$ and $1\leq j\leq D_{i}$ such that any iterated commutator of $e_{m_{1},j_{1}},\dots,e_{m_{\ell},j_{\ell}}$ is trivial if either 
	$\vert m_{1}\vert+\dots+\vert m_{\ell}\vert>s$ or $\vert m_{1}\vert+\dots+\vert m_{\ell}\vert=s$ and $\ell\geq r_{\ast}+1$.
	
	We endow this group a degree-rank filtration $(G^{\vec{D},d}_{[s,r]})_{[s,r]\in\DR}$ by letting $G^{\vec{D},d}_{[s,r]}$ be the Lie group generated by all the elements of the  iterated commutators of $e_{m_{1},j_{1}},\dots,e_{m_{\ell},j_{\ell}}$ with $1\leq \vert m_{1}\vert,\dots,\vert m_{\ell}\vert\leq s$ and $1\leq j_{i}\leq D_{\vert m_{i}\vert}, 1\leq i\leq \ell$ such that either $\vert m_{1}\vert+\dots+\vert m_{\ell}\vert>s$ or $\vert m_{1}\vert+\dots+\vert m_{\ell}\vert=s$ and $\ell\geq r$. It is not hard to verify that this is indeed a filtration of degree-rank $\leq [s,r_{\ast}]$. We then let $\Gamma^{\vec{D},d}$ be the discrete group generated by $e_{m,j}, 1\leq \vert m\vert\leq s, 1\leq j\leq D_{\vert m\vert}$, and refer to $G^{\vec{D},d}/\Gamma^{\vec{D},d}$ as the \emph{universal nilmanifold} with dimension vector $(\vec{D},d)$ of degree-rank $[s,r_{\ast}]$.
	
	A \emph{universal vertical frequency} on the universal nilmanifold $G^{\vec{D},d}/\Gamma^{\vec{D},d}$  of degree-rank $[s,r_{\ast}]$ is a continuous homomorphism $\eta\colon G^{\vec{D},d}_{[s,r_{\ast}]}\to\R$ which sends $\Gamma^{\vec{D},d}_{[s,r_{\ast}]}$ to the integers. The \emph{complexity} of a universal vertical frequency can be defined similar to Definition \ref{3:vtc}.
 \end{defn}	
 
 Let $(g_{m})_{\vert m\vert=j}$ be a sequence of elements with $g_{m}\in G_{(j,1)}/G_{(j,2)}$ (for convenience we also call $(g_{m})_{\vert m\vert=j}$ a \emph{$j$-string}), and $\Omega$ be a subset of $\Z^{d}$. We say that $(g_{m})_{\vert m\vert=j}$ is \emph{$(\Omega,p)$-reducible} if $(\xi_{j}(g_{m}))_{\vert m\vert=j}$  is $(\Omega,p)$-reducible for all $\xi_{j}\in\mathfrak{N}_{j}(G/\Gamma)$.

\begin{defn}[Universal representation]
 Let $d,D,r\in\N_{+}, C>0$, $[s,r_{\ast}]\in\DR$, $p$ be a prime, $\Omega_{1},\Omega_{2}$ be  subsets of $\Z^{d}$, and $\chi\in\Xi^{[s,r_{\ast}]}_{p^{r}}(\Omega_{1})$ be a nilcharacter of degree-rank $\leq [s,r_{\ast}]$ of dimension $D$. A \emph{universal representation} of $\chi$ is a collection of the following data:
\begin{itemize}
	\item a filtered nilmanifold $G/\Gamma$ of degree-rank $\leq [s,r_{\ast}]$ and complexity at most $C$;
	\item a filtered nilmanifold $G_{0}/\Gamma_{0}$ of degree-rank $<[s,r_{\ast}]$ and complexity at most $C$;
	\item a function $F\in\Lip((G/\Gamma\times G_{0}/\Gamma_{0})\to \mathbb{S}^{D})$ of Lipschitz norm at most $C$;
	\item  $p^{r}$-periodic polynomial sequences $g\in \poly_{p^{r}}(\Z^{d}\to G_{\N^{d}}\vert\Gamma)$ and  $g_{0}\in \poly_{p^{r}}(\Z^{d}\to (G_{0})_{\N^{d}}\vert\Gamma_{0})$; 
	\item a dimension vector $\vec{D}=(D_{1},\dots,D_{s})\in\N^{s}$ with $\vert\vec{D}\vert\leq C$;
	\item a universal vertical frequency $\eta\colon G^{\vec{D},d}_{[s,r_{\ast}]}\to \R$ of complexity at most $C$ on the universal nilmanifold $G^{\vec{D},d}/\Gamma^{\vec{D},d}$ of degree-rank $[s,r_{\ast}]$;
	\item a filtered homomorphism $\phi\colon G^{\vec{D},d}/\Gamma^{\vec{D},d}\to G/\Gamma$ of complexity at most $C$;
	\item a frequency tower $\mathcal{T}=(\xi_{m,k})_{1\leq \vert m\vert\leq s, 1\leq k\leq D_{\vert m\vert}}$ with $\xi_{m,k}\in \Z/p^{r}$; 
	\end{itemize}		
such that the followings hold:
\begin{itemize}
	\item For all $n\in \Omega_{1}$, we have that
	\begin{equation}\label{3:9.4}
	     \chi(n)=F(g(n)\Gamma,g_{0}(n)\Gamma_{0}).
	\end{equation}
	\item For all $t\in G^{\vec{D},d}_{[s,r_{\ast}]}$ and $(x,x_{0})\in G/\Gamma\times G_{0}/\Gamma_{0}$, we have that
	\begin{equation}\label{3:9.5}
		F(\phi(t)x,x_{0})=\exp(\eta(t))F(x,x_{0}).
	\end{equation}
	\item For all $1\leq j\leq s$, the   $(j,d,p^{r})$-string 
	\begin{equation}\nonumber
	\Bigl(\Taylor_{m}(g)\cdot\Bigl(\pi_{\Hor_{j}(G/\Gamma)}\circ\phi(\prod_{k=1}^{D_{j}}e_{m,k}^{\xi_{m,k}})\Bigr)^{-1}\Bigr)_{\vert m\vert=j}\footnote{This is a $(j,d,p^{r})$-string by Lemma \ref{3:ratispp}.}
	\end{equation}
	is $(\Omega_{2},p)$-reducible, where $\pi_{\Hor_{j}(G/\Gamma)}\colon G_{[j,0]}\to \Hor_{i}(G/\Gamma)$ is the projection map.   
	\end{itemize}	
We say that the tuple $(\vec{D},\mathcal{T},\eta)$, or  $(\vec{D},\mathcal{T},\eta,G/\Gamma\times G_{0}/\Gamma_{0})$ or $(\vec{D},\mathcal{T},\eta,F,\phi,G/\Gamma\times G_{0}/\Gamma_{0})$ is an \emph{$(\Omega_{1},\Omega_{2},C)$-universal representation} of $\chi$.
\end{defn}

\begin{rem}
	Note that if $(\vec{D},\mathcal{T},\eta,F,\phi,G/\Gamma\times G_{0}/\Gamma_{0})$ is a universal representation  of $\chi$ such that (\ref{3:9.4}) holds, then by lifting $g$ to a polynomial sequence on the universal nilmanifold $G^{\vec{D},d},$ we may write $\chi$ as 
	$$\chi(n)=F(\phi\circ\tilde{g}(n)\Gamma^{\vec{D},d},g_{0}(n)\Gamma_{0})$$
	for some $\tilde{g}\in\poly(\Z^{d}\to G^{\vec{D},d}_{\N})$. However, we caution the readers that $\tilde{g}$ is not necessarily $p^{r}$-periodic. 
\end{rem}

We collect some basic properties on the universal representation of nilcharcters. The following lemma says that every nilcharcter admits a universal representation, which 
 is an extension of Lemma 9.12 of \cite{GTZ12}:
\begin{lem}[Existence of universal representation]\label{3:10.2}
	Let $d,r\in\N_{+}$, $\Omega_{1},\Omega_{2}$ be  subsets of $\Z^{d}$, $C>0$ and $[s,r_{\ast}]\in\DR$.
	Every $\chi\in\Xi^{[s,r_{\ast}];C}_{p^{r}}(\Omega_{1})$ has an $(\Omega_{1},\Omega_{2}, O_{C,d,s}(1))$-universal representation $(\vec{D},\mathcal{T},\eta,G/\Gamma)$ if  $p\gg_{C,d,s} 1$.\footnote{We remark that $G_{0}/\Gamma_{0}$ can be taken to be the trivial system, and $G/\Gamma$ can be taken to be the nilmanifold on which $\chi$ is defined.} 
	\end{lem}		
\begin{proof}%[Proof of Lemma \ref{3:10.2}]
	Suppose that $\chi(n)=F(g(n)\Gamma), n\in \Omega_{1}$ for some nilmanifold $G/\Gamma$ of degree-rank $\leq [s,r_{\ast}]$ and complexity at most $C$, some $g\in\poly_{p^{r}}(\Z^{d}\to G_{\N^{d}}\vert\Gamma)$ and some $F\in \Lip(G/\Gamma\to \mathbb{S}^{D})$ of Lipschitz norm at most $C$ with a vertical frequency of complexity at most $C$. For each $1\leq i\leq s$, let $f_{i,1},\dots,f_{i,D_{i}}$ be a basis of generators of $\Gamma_{i}$. Set $\vec{D}:=(D_{1},\dots,D_{s})$. Then we have a filtered homomorphism $\phi\colon G^{\vec{D},d}\to G$ which maps $e_{m,k}$ to $f_{\vert m\vert,k}$.
	Since $G/\Gamma$ is of complexity at most $C$, $ G^{\vec{D},d}/\Gamma^{\vec{D},d}$ is of complexity at most $O_{C,d,s}(1)$ and $\phi$ is $O_{C,d,s}(1)$-rational.

	 Fix $m\in\N^{d}$ with $1\leq\vert m\vert\leq s$. 
	 Since $g\in\poly_{p^{r}}(\Z^{d}\to G_{\N^{d}}\vert\Gamma)$, by Lemma \ref{3:ratispp},   $\Taylor_{m}(g)^{p^{r}}\in G_{[i,2]}\Gamma$. Since $G_{[i,1]}/G_{[i,2]}$ is abelian and $f_{i,1},\dots,f_{i,D_{i}}$ is a basis of generators of $\Gamma_{i}$, we have that $$\Taylor_{m}(g)\equiv\prod_{k=1}^{D_{\vert m\vert}}f_{\vert m\vert,k}^{\xi_{m,k}}\equiv\Bigl(\pi_{\Hor_{\vert m\vert}(G/\Gamma)}(\phi(\prod_{k=1}^{D_{\vert m\vert}}e_{m,k}^{\xi_{m,k}}))\Bigr) \mod G_{[i,2]}\Gamma$$ for some $\xi_{m,k}\in\Z/p^{r}$.
Therefore,
	 \begin{equation}\nonumber
	 \Bigl(\Taylor_{m}(g)\cdot\Bigl(\pi_{\Hor_{j}(G/\Gamma)}\circ\phi(\prod_{k=1}^{D_{j}}e_{m,k}^{\xi_{m,k}})\Bigr)^{-1}\Bigr)_{\vert m\vert=j}
	 \end{equation}
	 is a $G_{[i,2]}\Gamma$-valued string and thus is $(\Omega_{2},p)$-reducible.
    
    Since $\phi$ is $O_{C,d,s}(1)$-rational, condition (\ref{3:9.5}) can be derived by pulling back the vertical frequency of $F$ by $\phi$ (and thus $\eta$ is of complexity at most $O_{C,d,s}(1)$).  This finishes the proof by setting $G_{0}/\Gamma_{0}$ to be trivial.
 \end{proof}	

Let  $\Omega$ be a subset of $\Z^{d}$ and $c\in\N$.
 Let $\mathcal{T}=(\vec{\xi}_{j,k})_{1\leq j\leq s, 1\leq k\leq D_{j}}$ and $\mathcal{T}'=(\vec{\xi}'_{j,k})_{1\leq j\leq s, 1\leq k\leq D'_{j}}$ be two $(d,p)$-towers. We say that $\mathcal{T}$ is \emph{$(\Omega,c,p)$-represented} by $\mathcal{T}'$ if for all $1\leq j\leq s$ and $1\leq k\leq D_{j}$, $\vec{\xi}_{j,k}$ is an $(\Omega,c,p)$-linear combination of $\vec{\xi}'_{j,1},\dots,\vec{\xi}'_{j,D'_{j}}$.
The following lemma explain how different ``basis" affects the representation of a nilcharacter, which is similar to Lemma 10.8 of \cite{GTZ12}:
\begin{lem}[Change of basis]\label{3:10.8}
	Let $d,r\in\N_{+}$, $\Omega_{1},\Omega_{2}$ be  subsets of $\Z^{d}$ and $[s,r_{\ast}]\in\DR$.
	Let  $C>0$ and $\chi\in \Xi^{[s,r_{\ast}];C}_{p^{r}}(\Omega_{1})$ be a degree-rank $[s,r_{\ast}]$ nilcharacter with an $(\Omega_{1},\Omega_{2},C)$-universal representation $(\vec{D},\mathcal{T},\eta,G/\Gamma\times G_{0}/\Gamma_{0})$. Suppose that the frequency tower $(\vec{D},\mathcal{T})$ is $(\Omega_{2},C,p)$-represented by another frequency tower $(\vec{D}',\mathcal{T}')$. Then there exists a vertical frequency $\eta'\colon G^{\vec{D}',d}_{[s,r_{\ast}]}\to\R$ such that $\chi$ has an $(\Omega_{1},\Omega_{2},O_{C,d,\vert\vec{D}'\vert}(1))$-universal representation $(\vec{D}',\mathcal{T}',\eta',G/\Gamma\times G_{0}/\Gamma_{0})$. 
\end{lem}	
\begin{proof}
	Let
	$(\vec{D},\mathcal{T},\eta,F,\phi,G/\Gamma\times G_{0}/\Gamma_{0})$ be an $(\Omega_{1},\Omega_{2},C)$-universal representation of $\chi$.
	Suppose that $\vec{D}=(D_{1},\dots,D_{s})$, $\vec{D}'=(D'_{1},\dots,D'_{s})$, $\mathcal{T}=(\vec{\xi}_{j,k})_{1\leq j\leq s,1\leq k\leq D_{j}}$,  and $\mathcal{T}'=(\vec{\xi}'_{j,k})_{1\leq j\leq s,1\leq k\leq D'_{j}}$.
By assumption, each $j$-string $\vec{\xi}_{j,k}$ of $\mathcal{T}$ can be written as
\begin{equation}\label{3:1081}
\vec{\xi}_{j,k}=\vec{\zeta}_{j,k}+\sum_{k'=1}^{D'_{j}}a_{j,k,k'}\vec{\xi}'_{j,k'}
\end{equation} 	
for some $(\Omega_{2},p)$-reducible $j$-string $\vec{\zeta}_{j,k}$ and $-C\leq a_{j,k,k'}\leq C$.
Let $\psi\colon G^{\vec{D}',d}\to G^{\vec{D},d}$ be the unique filtered homomorphism that maps $e'_{m,k'}$ to $\prod_{k=1}^{D_{j}}e_{m,k}^{a_{j,k,k'}}$. 
Write  $\vec{\zeta}_{j,k}=(\zeta_{m,k})_{\vert m\vert=j}$, $\vec{\xi}_{j,k}=(\xi_{m,k})_{\vert m\vert=j}$ and $\vec{\xi}'_{j,k}=(\xi'_{m,k})_{\vert m\vert=j}$.
By assumption, $\chi$ can be written as $$\chi(n)=F(g(n)\Gamma,g_{0}(n)\Gamma_{0}), n\in \Omega_{1}$$ with
\begin{equation}\nonumber
\Bigl(\Taylor_{m}(g)\cdot\Bigl(\pi_{\Hor_{j}(G/\Gamma)}\circ\phi(\prod_{k=1}^{D_{j}}e_{m,k}^{\xi_{m,k}})\Bigr)^{-1}\Bigr)_{\vert m\vert=j}
\end{equation}
being $(\Omega_{2},p)$-reducible
for some filtered homomorphism $\phi\colon G^{\vec{D},d}\to G$ and some tower $(\xi_{m,k})_{1\leq \vert m\vert\leq s, 1\leq k\leq D_{\vert m\vert}}$ for all $1\leq j\leq s$.
Then for all $\vert m\vert=j$, by (\ref{3:1081}),
\begin{equation}\label{3:1082}
\begin{split}
&\quad \Taylor_{m}(g)\cdot\Bigl(\pi_{\Hor_{j}(G/\Gamma)}\circ\phi\circ\psi(\prod_{k'=1}^{D'_{j}}{e'}_{m,k'}^{\xi_{m,k'}})\Bigr)^{-1}
\\&=\Taylor_{m}(g)\cdot\Bigl(\pi_{\Hor_{j}(G/\Gamma)}\circ\phi(\prod_{k=1}^{D_{j}}{e}_{m,k}^{\sum_{k'=1}^{D'_{j}}a_{j,k,k'}\xi_{m,k'}})\Bigr)^{-1}
\\&=\Taylor_{m}(g)\cdot\Bigl(\pi_{\Hor_{j}(G/\Gamma)}\circ\phi(\prod_{k=1}^{D_{j}}{e}_{m,k}^{\xi_{m,k}})\Bigr)^{-1}\cdot\prod_{k=1}^{D_{j}}\Bigl(\pi_{\Hor_{j}(G/\Gamma)}\circ\phi({e}_{m,k}^{\zeta_{m,k}})\Bigr).
\end{split}
\end{equation}
Since both $\Bigl(\Taylor_{m}(g)\cdot\Bigl(\pi_{\Hor_{j}(G/\Gamma)}\circ\phi(\prod_{k=1}^{D_{j}}{e}_{m,k}^{\xi_{m,k}})\Bigr)^{-1}\Bigr)_{\vert m\vert=j}$ and $\vec{\zeta}_{j,1},\dots,\vec{\zeta}_{j,D_{j}}$ are $(\Omega_{2},p)$-reducible, the left hand side of (\ref{3:1082}) 
 is also  $(\tau(\Omega_{2}),p)$-reducible.

On the other hand,
 $\phi\circ\psi\colon G^{\vec{D}',d}\to G$ is a filtered homomorphism and $\eta\circ \psi\colon G^{\vec{D}',d}_{[s,r_{\ast}]}\to\R$ is a vertical frequency. 
 Since $\eta,\phi$ and $G/\Gamma$ are   of complexities at most $C$ and $-C\leq a_{j,k,k'}\leq C$, it is not hard to see that $\phi\circ\psi$ is $O_{C,d,\vert\vec{D}'\vert}(1)$-rational and that $\eta\circ \psi$ is of complexity at most $O_{C,d,\vert\vec{D}'\vert}(1)$. 
 This finishes the proof.
\end{proof}	

We refer the readers to Section 9 of \cite{GTZ12} for further properties of universal representations of nilcharacters.

Let $\mathcal{T}=(\xi_{m,k})_{1\leq \vert m\vert\leq s,1\leq k\leq D_{\vert m\vert}}$  and $\mathcal{T}'=(\xi'_{m,k})_{1\leq \vert m\vert\leq s,1\leq k\leq D'_{\vert m\vert}}$ be two towers. Let $\mathcal{T}\uplus \mathcal{T}':=(\xi''_{m,k})_{1\leq \vert m\vert\leq s, 1\leq k\leq D_{\vert m\vert}+D'_{\vert m\vert}}$ be the tower defined by $\xi''_{m,k}=\xi_{m,k}$ for $1\leq k\leq D_{\vert m\vert}$ and $\xi''_{m,k'+D_{\vert m\vert}}=\xi'_{m,k'}$ for $1\leq k'\leq D'_{\vert m\vert}$. We have:

\begin{lem}[Sunflower lemma for nilcharacters]\label{3:sf2}  
	Let $[s,r_{\ast}]\in\DR$, $d,D,r\in\N_{+}$, with $d\geq N(s)$, $C,\d,\e,L>0$, $p\gg_{C,\d,d,\e,L,r} 1$ be a prime, and $\tilde{M}\colon\Z^{d}\to\Z/p$ be a non-degenerate quadratic form. Let $H$ be a subset of $\Z^{d}_{p^{r}}$ with $\vert H\vert>\e p^{rd}$
and $(\chi_{h})_{h\in H}$ be a family of nilcharacters in $\Xi^{[s,r_{\ast}];C,D}_{p^{r}}(V_{p}(\tilde{M}))$. If   $d\geq N(s)$, then there exist
	\begin{itemize}
		\item a subset $H'\subseteq H$ with $\vert H'\vert\gg_{C,\d,d,\e,L,r}p^{rd}$;
		\item a dimension vector $\vec{D}=(D_{j})_{1\leq j\leq s}$ with $\vert \vec{D}\vert=O_{C,\d,d,\e,L,r}(1)$ which can be further decomposed as $\vec{D}=\vec{D^{\cor}}+\vec{D^{\ped}}$ and $\vec{D^{\ped}}=\vec{D^{\lin}}+\vec{D^{\ind}}$ for some dimension vectors
		 $\vec{D^{\cor}}=(D^{\cor}_{j})_{1\leq j\leq s}$, $\vec{D^{\lin}}=(D^{\lin}_{j})_{1\leq j\leq s}$ and $\vec{D^{\ind}}=(D^{\ind}_{j})_{1\leq j\leq s}$;
		\item a core $(d,p^{r})$-frequency tower $(\vec{D^{\cor}},\mathcal{T}^{\cor})$;
		\item for each $h\in H'$ petal $(d,p^{r})$-frequency towers $(\vec{D^{\lin}},\mathcal{T}_{h}^{\lin}=(\vec{\xi}^{\lin}_{h,j,k})_{1\leq j\leq s, 1\leq k\leq D^{\lin}_{j}})$,  $(\vec{D^{\ind}},\mathcal{T}_{h}^{\ind}=(\vec{\xi}^{\lin}_{h,j,k})_{1\leq j\leq s, 1\leq k\leq D^{\ind}_{j}})$ and  $(\vec{D^{\ped}},\mathcal{T}_{h}^{\ped}=(\vec{\xi}^{\lin}_{h,j,k})_{1\leq j\leq s, 1\leq k\leq D^{\ped}_{j}})$ such that  $\mathcal{T}_{h}^{\ped}=\mathcal{T}_{h}^{\ind}\uplus \mathcal{T}_{h}^{\lin}$;
		\item a filtered nilmanifold $G/\Gamma$ of degree-rank $\leq [s,r_{\ast}]$ of complexity $O_{C,\d,d,\e,L,r}(1)$ and a filtered nilmanifold $G_{0}/\Gamma_{0}$ of degree-rank $\leq [s-1,r_{\ast}-1]$ of complexity $O_{C,\d,d,\e,L,r}(1)$;
			\item for each $h\in H'$ a function $F_{h}\in\Lip((G/\Gamma\times G_{0}/\Gamma_{0})\to \mathbb{S}^{D})$ of Lipschitz norm $O_{C,\d,d,\e}(1)$;	
	    \item a vertical frequency $\eta\colon G^{\vec{D},d}_{[s,r_{\ast}]}\to\R$ of complexity  $O_{C,d}(1)$ with dimension vector $\vec{D}$ on the universal nilmanifold $G^{\vec{D},d}/\Gamma^{\vec{D},d}$ of degree-rank $[s,r_{\ast}]$;
	    \item a $O_{C,\d,d,\e,L,r}(1)$-rational filtered homomorphism $\phi\colon G^{\vec{D},d}/\Gamma^{\vec{D},d}\to G/\Gamma$
	\end{itemize}	
	such that the followings hold:
		\begin{enumerate}[(i)]
			\item For all $h\in H'$, $\chi_{h}$ has an $(V_{p}(\tilde{M}),V_{p}(\tilde{M})^{h},O_{C,\d,d,\e,L,r}(1))$-universal representation  of the form 
			$$(\vec{D},\mathcal{T}^{\cor}\uplus\mathcal{T}^{\ped}_{h},\eta,F_{h},\phi,G/\Gamma\times G_{0}/\Gamma_{0}).$$
				\item For all but $\d p^{3rd}$ additive quadruples $h_{1}+h_{2}=h_{3}+h_{4}$ with $h_{1},\dots,h_{4}\in H'$, the tower
				$$\mathcal{T}^{\cor}\uplus \uplus_{i=1}^{3}\mathcal{T}^{\lin}_{h_{i}}\uplus\uplus_{i=1}^{4} \mathcal{T}^{\ind}_{h_{i}}$$
				is $(V_{p}(\tilde{M})^{h_{1},h_{2},h_{3}}, L,p)$-independent.  
			\item For all $1\leq j\leq s$ and $1\leq k\leq D^{\lin}_{j}$, the map $h\mapsto\vec{\xi}^{\lin}_{h,j,k}, h\in H'$ is a $\Z/p^{r}$-almost linear Freiman homomorphism of complexity $O_{C,\d,d,\e,L,r}(1)$.	
 		\end{enumerate}
\end{lem}	
\begin{proof}
	 First we may raise the dimension of $\chi_{h}$ to $D$ for all $h\in H$ by adding constant zero functions to the new coordinates. 
	By Lemma \ref{3:10.2},  each $\chi_{h}$ has a $(V_{p}(\tilde{M}),\Z^{d},O_{C,d}(1))$-universal representation $(\vec{D}_{h}=(D_{h,j})_{1\leq j\leq s},\mathcal{T}_{h}=(\vec{\xi}_{h,j,k})_{1\leq j\leq s, 1\leq k\leq D_{h,j}},\eta_{h})$.  By the Pigeonhole Principle, there exists a subset $H''\subseteq H$ with $\vert H''\vert\gg_{C,d,\e}p^{rd}$ such that $\vec{D}_{h}=\vec{D}=(D_{1},\dots,D_{s})$
	for all $h\in H''$. 
	Since $\vert\vec{D}\vert=O_{C,d}(1)$,
	applying Lemma \ref{3:sf1} to the $j$-strings $\vec{\xi}_{h,j,1},\dots,\vec{\xi}_{h,j,D_{j}}, h\in H''$ for each $1\leq j\leq s$, we have that if $p\gg_{C,\d,d,\e,L,r} 1$, then there exists a subset $H'\subseteq H''$ with $\vert H'\vert\gg_{C,\d,d,\e,L,r}  p^{rd}$ such that for all $h\in H'$, the tower $(\vec{D},\mathcal{T}_{h})$ is $(V_{p}(\tilde{M})^{h},O_{C,\d,d,\e,L,r}(1),p)$-represented by a  tower of the form  $(\vec{D^{\cor}}+\vec{D^{\ped}},\mathcal{T}^{\cor}\uplus\mathcal{T}^{\ped}_{h})$ satisfying Property (i) of Lemma \ref{3:sf1}, and Properties (ii) and (iii) in Lemma \ref{3:sf2}. By Lemma \ref{3:10.8}, we see that for all $h\in H'$, $\chi_{h}$ has a  $(V_{p}(\tilde{M}),V_{p}(\tilde{M})^{h},O_{C,\d,d,\e,L,r}(1))$-universal representation of the form $$(\vec{D^{\cor}}+\vec{D^{\ped}},\mathcal{T}^{\cor}\uplus\mathcal{T}^{\ped}_{h},\eta'_{h},F_{h},\phi_{h},G_{h}/\Gamma_{h}\times G_{0,h}/\Gamma_{0,h}).$$
	By the Pigeonhole Principle and passing to another subset,
    we may require 
    $\eta'_{h}=\eta$, $\phi_{h}=\phi$, $G_{h}/\Gamma_{h}=G/\Gamma$ and   $G_{0,h}/\Gamma_{0,h}=G_{0}/\Gamma_{0}$ to be independent of $h$.
    We are done.
\end{proof}	

\section{Information on the frequencies of $\chi_{h}$}\label{3:s:b4}

In this section, we use the sunflower lemma (Lemma \ref{3:sf2}) obtained in Section \ref{3:s:b3} to analysis the inequality (\ref{3:notlongago}). Our goal is to understand the relation between the frequencies of the characters $\chi_{h_{1}},\chi_{h_{2}},\chi_{h_{3}}$ and $\chi_{h_{1}+h_{2}-h_{3}}$ when $\chi_{h_{1},h_{2},h_{3}}$ correlates with an $(s-1)$-step nilsequence. 
Throughout this section, we assume that the conditions in Proposition \ref{3:inductioni1} are satisfies, i.e.
  there exist a subset $H\subseteq \Z^{d}_{p^{r}}$ with $\vert H\vert\geq \e p^{rd}$, some $\chi_{0}\in\Xi_{p}^{(1,s);C,D}((\Z^{d})^{2})$,  some $\chi_{h}\in \Xi_{p^{r}}^{[s,r_{\ast}];C,D}(\Z^{d})$, and some $\psi_{h}\in\Nil_{p}^{s-1;C,1}(\Z^{d})$ for all $h\in H$ such that 
(\ref{3:longlong2})
holds for all $h\in H$. 
We may 
apply Lemma \ref{3:sf2} to $(\chi_{h})_{h\in H}$ for some $\d,L>0$ depending only on $C,d,D,\e,r$ to be chosen later, and for $d\geq N(s), p\gg_{C,d,D,\e,r} 1$. For convenience we keep the same notions as in Lemma \ref{3:sf2}. 
Then we have a subset  $H'$ of $H$ with $\vert H'\vert\gg_{C,d,D,\e,r} p^{rd}$  such that
for all $h\in H'$, $\chi_{h}$ has a $(V_{p}(\tilde{M}),V_{p}(\tilde{M})^{h},O_{C,d,D,\e,r}(1))$-universal representation
$$(\vec{D^{\cor}}+\vec{D^{\ped}},\mathcal{T}^{\cor}\uplus\mathcal{T}^{\ped}_{h},\eta).$$

The main effort in this section is to understand the property of the common vertical frequency $\eta$, with the main result being Theorem \ref{3:vns} to be stated later. 
By Proposition \ref{3:inductioni1}, there exists a subset $U\subseteq (H')^{3}$ with $\vert U\vert\gg_{C,d,D,\e} p^{3rd}$ 
such that for all $(h_{1},h_{2},h_{3})\in U$, $\iota(h_{1}),\iota(h_{2}),\iota(h_{3})$ are linearly independent and $M$-non-isotropic, and  	\begin{equation}\label{3:recall1}
	\Bigl\vert\E_{n\in V_{p}(\tilde{M})^{h_{1},h_{3},h_{3}-h_{2}}}\chi_{h_{1},h_{2},h_{3}}(n) \psi_{h_{1},h_{2},h_{3}}(n)\Bigr\vert\gg_{C,d,D,\e} 1
	\end{equation}
	for some $\psi_{h_{1},h_{2},h_{3}}\in \Nil_{p}^{s-1;O_{C,d,D,\e}(1),1}(\Z^{d})$, where
$$\chi_{h_{1},h_{2},h_{3}}(n):= \chi_{h_{1}}(n)\otimes\chi_{h_{2}}(n+h_{3}-h_{2})\otimes\overline{\chi}_{h_{3}}(n)\otimes\overline{\chi}_{h_{1}+h_{2}-h_{3}}(n+h_{3}-h_{2}).$$

Let $U'$ be the set of $(h_{1},h_{2},h_{3})\in (H')^{3}$ such that $(h_{1},h_{2},h_{3},h_{1}+h_{2}-h_{3})$ does not fall in the exceptional set in which Condition (ii) in Lemma \ref{3:sf2} fails. Since $\vert U\vert\gg_{C,d,D,\e} p^{3rd}$, we may choose $\d$ sufficiently small depending on $C,d,D,\e$ such that $\vert U'\vert\geq \vert U\vert/2$. In particular, there exists at least one tuple $(h_{1},h_{2},h_{3})\in (H')^{3}$ with $\iota(h_{1}),\iota(h_{2}),\iota(h_{3})$ being linearly independent and $M$-non-isotropic such that (\ref{3:recall1}) holds and that
 $$\mathcal{T}^{\cor}\uplus \uplus_{i=1}^{3}\mathcal{T}^{\lin}_{h_{i}}\uplus\uplus_{i=1}^{4} \mathcal{T}^{\ind}_{h_{i}}$$
is $(V_{p}(\tilde{M})^{h_{1},h_{3},h_{3}-h_{2}},O_{C,d,D,\e,r}(1),p)$-independent, where $h_{4}=h_{1}+h_{2}-h_{3}$.

We now fix such a choice of $h_{1},h_{2},h_{3}$. Then
$\chi_{h_{1},h_{2},h_{3}}(n)\psi_{h_{1},h_{2},h_{3}}(n)$ can be written as a degree-rank $\leq [s,r_{\ast}]$ nilsequence $n\mapsto F(g(n)\Gamma)$. Here
$$G/\Gamma:=\Bigl(\prod_{i=1}^{4}G_{(i)}/\Gamma_{(i)}\Bigr)\times G_{(0)}/\Gamma_{(0)}$$
for some filtered nilmanifold $G_{(0)}/\Gamma_{(0)}$ of degree-rank $<[s,r_{\ast}]$ and complexity $O_{C,d,D,\e,r}(1)$, some  filtered nilmanifolds $G_{(i)}/\Gamma_{(i)}$ of degree-rank $\leq [s,r_{\ast}]$ and complexity $O_{C,d,D,\e,r}(1)$ for $1\leq i\leq 4$. 
Lemma \ref{3:ratispp}, the polynomial sequence $g$ can be written as
$$g(n)=(g_{1}(n),g_{2}(n),g_{3}(n),g_{4}(n),g_{0}(n))\in\poly_{p^{r}}(\Z^{d}\to G_{\DR}\vert\Gamma)$$
for some $g_{i}\in \poly_{p^{r}}(\Z^{d}\to (G_{(i)})_{\DR}\vert \Gamma_{(i)}), 1\leq i\leq 4$
 such that for all $1\leq i\leq 4$, we may write $g_{i}=g'_{i}g''_{i}$
for some   $g'_{i},g''_{i}\in \poly(\Z^{d}\to (G_{(i)})_{\DR})$,
where $\Taylor_{m}(g'_{i})$ is of the form
\begin{equation}\label{3:11.3}
\Taylor_{m}(g'_{i})=\Bigl(\pi_{\Hor_{\vert m\vert}(G_{(i)}/\Gamma_{(i)})}\circ\phi_{i}(\prod_{k=1}^{D_{\vert m\vert}}e_{m,k}^{\xi_{h_{i},m,k}})\Bigr),\footnote{For a core frequency, i.e. for $k\leq D_{\vert m\vert}^{\cor}$, we write $\xi_{h_{i},m,k}=\xi_{m,k}$ for convenience.}
\end{equation}
and 
$(\Taylor_{m}(g''_{i}))_{\vert m\vert=j}$
is  a $(V_{p}(\tilde{M})^{h_{i}},p)$-reducible $(j,d,p^{r})$-string for all $1\leq j\leq s$,
where $\vec{D}=(D_{1},\dots,D_{s}), \phi_{i}\colon G^{\vec{D}}/\Gamma^{\vec{D}}\to G_{(i)}/\Gamma_{(i)}$ is a filtered homomorphism and $\pi_{\Hor_{j}(G_{(i)}/\Gamma_{(i)})}\colon$ $(G_{(i)})_{j}\to \Hor_{j}(G_{(i)}/\Gamma_{(i)})$ is the projection to the $j$-th type-II horizontal torus. Finally, $F\in\Lip(G/\Gamma\to\mathbb{S}^{D^{4}})$ has Lipschitz norm $O_{C,d,D,\e}(1)$, and
\begin{equation}\label{3:11.4}
\begin{split}
  F(\phi_{1}(t_{1})x_{1},\phi_{2}(t_{2})x_{2},\phi_{3}(t_{3})x_{3},\phi_{4}(t_{4})x_{4},y)
=\exp(\eta(t_{1}t_{2}t_{3}^{-1}t_{4}^{-1}))F(x_{1},x_{2},x_{3},x_{4},y)
\end{split}
\end{equation} 
for all $(x_{1},x_{2},x_{3},x_{4},y)\in G/\Gamma$ and $t_{1},\dots,t_{4}\in G^{\vec{D},d}_{[s,r_{\ast}]}$ (note that shifting $\chi_{h_{i}}$ does not affect the type-II Taylor coefficients of $g_{i}$, see the remark following Definition 9.6 of \cite{GTZ12}).

By (\ref{3:recall1}), we have that,  
$$\vert\E_{n\in V_{p}(\tilde{M})^{h_{1},h_{3},h_{3}-h_{2}}}F(g(n)\Gamma)\vert\gg_{C,D,d,\e} 1.$$
Since $\iota(h_{1}),\iota(h_{2}),\iota(h_{3})$ are linearly independent, there exists an affine subspace $V+c$ of $\V$ of co-dimension 3 such that $V_{p}(\tilde{M})^{h_{1},h_{3},h_{3}-h_{2}}=\iota^{-1}(V(M)\cap (V+c))$.
Moreover, since   $\iota(h_{1}),\iota(h_{2}),\iota(h_{3})$ are  $M$-non-isotropic, 
it follows from Proposition \ref{3:iissoo} that $\rank(M\vert_{V+c})=d-3$.
By  Theorem \ref{3:rat} (applied to the $\N$-filtration induced by the degree-rank filtration of $G$),  since $p\gg_{C,d,D,\e} 1$ and $d\geq s+16$, we have that
\begin{equation}\label{3:11.5}
\begin{split}
\Bigl\vert\int_{G_{P}/\Gamma_{P}}F(\e_{0} x)\,dm_{G_{P}/\Gamma_{P}}(x)\Bigr\vert\gg_{C,D,d,\e} 1
\end{split}
\end{equation}
for some $\e_{0}\in G$ of complexity $O_{C,d,D,\e,r}(1)$, some rational subgroup $G_{P}$ of $G$ which is $O_{C,d,D,\e,r}(1)$-rational relative to the $O_{C,d,D,\e,r}(1)$-rational Mal'cev basis $\mathcal{X}$ of $G/\Gamma$, and some $\Gamma_{P}=G_{P}\cap\Gamma$ with  $m_{G_{P}/\Gamma_{P}}$ being the Haar measure of $G_{P}/\Gamma_{P}$  such that for all $1\leq j\leq s$,
\begin{equation}\label{3:11.6}
\begin{split}
\pi_{\Hor_{j}(G)}(G_{P}\cap G_{j})\geq \Xi^{\perp}_{V_{p}(\tilde{M})^{h_{1},h_{3},h_{3}-h_{2}},j,Q}(g)
\end{split}
\end{equation} 
for some $Q\in\N_{+}, Q\leq O_{C,d,D,\e}(1)$,
where
$\Xi_{V_{p}(\tilde{M})^{h_{1},h_{3},h_{3}-h_{2}},j,Q}(g)$ is the group of all continuous homomorphisms $\xi_{j}\colon\Hor_{j}(G/\Gamma)\to\R$  such that 
such that 
$$\xi_{j}(\Taylor_{j}(g)(m_{1},\dots,m_{j}))=\xi_{j}(\Delta_{m_{j}}\dots\Delta_{m_{1}}g(n) \mod G_{[j,2]})\in\Z/Q$$
for all $(n,m_{1},\dots,m_{j})\in \Gow_{p,j}(V_{p}(\tilde{M})^{h_{1},h_{3},h_{3}-h_{2}})$,
and
$$\Xi^{\perp}_{V_{p}(\tilde{M})^{h_{1},h_{3},h_{3}-h_{2}},j}(g):=\{x\in \Hor_{j}(G)\colon \xi_{j}(x)\in\Z \text{ for all } \xi_{j}\in\Xi_{V_{p}(\tilde{M})^{h_{1},h_{3},h_{3}-h_{2}},j,Q}(g)\}.$$

\begin{lem}\label{3:changeh2h3}
	We have that $\Xi_{V_{p}(\tilde{M})^{h_{1},h_{3},h_{3}-h_{2}},j,Q}(g)=\Xi_{V_{p}(\tilde{M})^{h_{1},h_{2},h_{3}},j,Q}(g)$. 
	\end{lem}
\begin{proof}
Fix any $\xi_{j}\in\Xi_{V_{p}(\tilde{M})^{h_{1},h_{3},h_{3}-h_{2}},j,Q}(g)$. 
Denote
 $$f(n):=\sum_{m\in\N^{d},\vert m\vert=j}\xi_{j}(\Taylor_{m}(g))(m!)^{\ast}n^{m}.$$
By  Lemma  \ref{3:ratispp}, $\Taylor_{m}(g)\in\Z/p^{r}$ for all $m\in\N^{d}$ with $\vert m\vert=j$.
So $f$ belongs to $\st_{\Z/p^{r},d}(j)$. 
  By Lemma \ref{3:sbeq}, we have 
 \begin{equation}\nonumber
 \begin{split}
 &\quad 0\equiv Q\xi_{j}(\Taylor_{j}(g)(t_{1},\dots,t_{j}))\equiv Q\Delta_{t_{j}}\dots\Delta_{t_{1}}\Bigl(\sum_{m\in\N^{d},\vert m\vert=j}\xi_{j}(\Taylor_{m}(g))\binom{n}{m}\Bigr)
  \\&\equiv Q\Delta_{t_{j}}\dots\Delta_{t_{1}}f(n) \mod\Z
 \end{split} 
 \end{equation} 
  for all  $(n,m_{1},\dots,m_{j})\in \Gow_{p,j}(V_{p}(\tilde{M})^{h_{1},h_{3},h_{3}-h_{2}})$.
 This means that $Qf$ is $(V_{p}(\tilde{M})^{h_{1},h_{3},h_{3}-h_{2}},p)$-reducible. By Lemma \ref{3:att55}, the map induced by $Qf$ belongs to $J^{M}_{\iota(h_{1}),\iota(h_{3}),\iota(h_{3}-h_{2})}=J^{M}_{\iota(h_{1}),\iota(h_{2}),\iota(h_{3})}$. Again by  Lemma  \ref{3:att55}, 
 we have that
 $Qf$ is $(V_{p}(\tilde{M})^{h_{1},h_{2},h_{3}},p)$-reducible. In other words, $\xi_{j}\in\Xi_{V_{p}(\tilde{M})^{h_{1},h_{2},h_{3}},j,Q}(g)$ and thus $\Xi_{V_{p}(\tilde{M})^{h_{1},h_{3},h_{3}-h_{2}},j,Q}(g)\subseteq \Xi_{V_{p}(\tilde{M})^{h_{1},h_{2},h_{3}},j,Q}(g)$. For similar reasons, we have that $\Xi_{V_{p}(\tilde{M})^{h_{1},h_{3},h_{3}-h_{2}},j,Q}(g)\supseteq \Xi_{V_{p}(\tilde{M})^{h_{1},h_{2},h_{3}},j,Q}(g)$ and so $\Xi_{V_{p}(\tilde{M})^{h_{1},h_{3},h_{3}-h_{2}},j,Q}(g)=\Xi_{V_{p}(\tilde{M})^{h_{1},h_{2},h_{3}},j,Q}(g)$.
\end{proof}

\begin{lem}\label{3:L11.2}
For all $h_{1}\in (G_{(1)})_{[s,r_{\ast}]}$ with 
$(h_{1},id,id,id,id)\in G_{P}$, we have that $\eta(h_{1})\in\Z$.	
\end{lem}
\begin{proof}
	Denote $h:=(h_{1},id,id,id,id)\in G_{P}$. Then $h$ is in the center of $G$ and so 
	\begin{equation}\nonumber
	\begin{split}
	&\quad\int_{G_{P}/\Gamma_{P}}F(\e_{0} x)\,dm_{G_{P}/\Gamma_{P}}(x)=\int_{G_{P}/\Gamma_{P}}F(h\e_{0} x)\,dm_{G_{P}/\Gamma_{P}}(x)
	\\&=\exp(\eta(h_{1}))\int_{G_{P}/\Gamma_{P}}F(\e_{0} x)\,dm_{G_{P}/\Gamma_{P}}(x),
	\end{split}
	\end{equation}
	 	where we used (\ref{3:11.4}) in the last equality. We may then deduce from (\ref{3:11.5}) that $\eta(h_{1})\in\Z$.
\end{proof}	

For $1\leq j\leq s$, let $V_{123,j}$ denote the subgroup of $\Hor_{j}(G_{(1)})\times\Hor_{j}(G_{(2)})\times \Hor_{j}(G_{(3)})$ generated by 
$$(\phi_{1}(e_{m,k}),\phi_{2}(e_{m,k}),\phi_{3}(e_{m,k})), \vert m\vert=j, 1\leq k\leq D^{\cor}_{j}$$
and
$$(\phi_{1}(e_{m,k}),id,id),(id,\phi_{2}(e_{m,k}),id),(id,id,\phi_{3}(e_{m,k})), \vert m\vert=j, D^{\cor}_{j}+1\leq k\leq D_{j}.$$
Define $V_{124,j}$, $V_{134,j}$ and $V_{234,j}$ in a similar way.
\begin{lem}\label{3:L11.3}
	There exists $L_{0}=L_{0}(c,d,D,\e,r)\in\N$ such that 
if $L\geq L_{0}$,  then
for $1\leq j\leq s$, the projection of $G_{P}\cap G_{j}$ to  $\Hor_{j}(G_{(1)})\times\Hor_{j}(G_{(2)})\times \Hor_{j}(G_{(3)})$ contains $V_{123,j}$, and similar results  hold with $V_{123,j}$ replaced by $V_{124,j}$, $V_{134,j}$ and $V_{234,j}$.	
\end{lem}
\begin{proof}
  We only prove for the case $V_{123,j}$ since the proof of other cases are similar.
   Suppose that the statement fails for some $1\leq j\leq s$.
   Let $$\pi_{j}\colon G_{P}\cap G_{j}\to \Hor_{j}(G_{(1)})\times\Hor_{j}(G_{(2)})\times \Hor_{j}(G_{(3)})$$ and 	$$\pi'_{j}\colon \Hor_{j}(G)\to \Hor_{j}(G_{(1)})\times\Hor_{j}(G_{(2)})\times \Hor_{j}(G_{(3)})$$ denote the projection maps, and let $Z$ denote the set of $\tilde{\xi}_{j}\in\mathfrak{N}_{j}(G/\Gamma)$ which annihilates  ${\pi'}^{-1}_{j}\circ\pi_{j}(G_{P}\cap G_{j})$ (recall that $\mathfrak{N}_{j}(G/\Gamma)$ is the group of all $j$-th type-II horizontal characters).  
   By duality, there exists $\tilde{\xi}_{j}\in Z$ which is nontrivial on ${\pi'}^{-1}_{j}(V_{123,j})$.  Since $G_{P}$ is $O_{C,d,D,\e}(1)$-rational relative to the Mal'cev basis $\mathcal{X}$ of $G/\Gamma$,   there exist $t=O_{C,d,D,\e,r}(1)\in \N$ and $\xi_{j,1},\dots,\xi_{j,t}\in Z$ of complexity at most $O_{C,d,D,\e}(1)$
   such that every $\tilde{\xi_{j}}\in Z$  is a linear combination of $\xi_{j,1},\dots,\xi_{j,t}$. So there exists $\xi_{j}\in\{\xi_{j,1},\dots,\xi_{j,t}\}$ which is nontrivial on ${\pi'}^{-1}_{j}(V_{123,j})$.
    By (\ref{3:11.6}), Lemma \ref{3:changeh2h3} and duality, $\xi_{j}\in\Xi_{V_{p}(\tilde{M})^{h_{1},h_{2},h_{3}},j,Q}(g)$.

    Since 
    $\xi_{j}$ annihilates  ${\pi'}^{-1}_{j}\circ\pi_{j}(G_{P}\cap G_{j})$, we have that 
    \begin{equation}\label{3:11101}
    \xi_{j}(x_{1},x_{2},x_{3},x_{4},x_{0})=\xi_{(1),j}(x_{1})+\xi_{(2),j}(x_{2})+\xi_{(3),j}(x_{3})
    \end{equation}
    for all $x_{i}\in \Hor_{j}(G_{(i)}), 0\leq i\leq 4$ for some 
    $\xi_{(i),j}\in\mathfrak{N}_{j}(G_{(i)}/\Gamma_{(i)})$.
  Fix any $(n,t_{1},\dots,t_{j})\in \Gow_{p,j}(V_{p}(\tilde{M})^{h_{1},h_{2},h_{3}})$. Since
  the $j$-string
  $(\Taylor_{m}(g''_{i}))_{\vert m\vert=j}$
  is  $(V_{p}(\tilde{M})^{h_{i}},p)$-reducible, we have that 
    $$\sum_{\vert m\vert=j}\xi_{(i),j}(\Taylor_{m}(g''_{i}))(m!)^{\ast}n^{m}$$
    is $(V_{p}(\tilde{M})^{h_{i}},p)$-reducible for all $1\leq i\leq 3$. 
    Since $(n,t_{1},\dots,t_{j})\in \Gow_{p,j}(V_{p}(\tilde{M})^{h_{i}})$ for $1\leq i\leq 3$,    
     it follows from Lemma \ref{3:sbeq} that  
     \begin{equation}\nonumber
     \begin{split}
     &\quad\sum_{i=1}^{3}\xi_{(i),j}(\Taylor_{j}(g''_{i})(t_{1},\dots,t_{j}))\equiv\sum_{i=1}^{3}\Delta_{t_{j}}\dots\Delta_{t_{1}}\Bigl(\sum_{m\in\N^{d},\vert m\vert=j}\xi_{(i),j}(\Taylor_{m}(g''_{i}))\binom{n}{m}\Bigr)
     \\&\equiv\sum_{i=1}^{3}\Delta_{t_{j}}\dots\Delta_{t_{1}}\Bigl(\sum_{m\in\N^{d},\vert m\vert=j}\xi_{(i),j}(\Taylor_{m}(g''_{i}))(m!)^{\ast}n^{m}\Bigr)
     \equiv 0 \mod\Z.
     \end{split}
     \end{equation}
     On the other hand,  By the definition of $\Xi_{V_{p}(\tilde{M})^{h_{1},h_{2},h_{3}},j,Q}(g)$ and Lemma \ref{3:sbeq}, we have 
     \begin{equation}\nonumber
     \begin{split}
      \sum_{i=1}^{3}\xi_{(i),j}(\Taylor_{j}(g_{i})(t_{1},\dots,t_{j}))=\sum_{i=1}^{3}\Delta_{t_{j}}\dots\Delta_{t_{1}}\Bigl(\sum_{m\in\N^{d},\vert m\vert=j}\xi_{(i),j}(\Taylor_{m}(g_{i}))\binom{n}{m}\Bigr)\in\Z/Q.
     \end{split}
     \end{equation}
    So by (\ref{3:11.3}), 
   \begin{equation}\nonumber
   \begin{split}
   &\quad 
    \Delta_{t_{j}}\dots\Delta_{t_{1}}\Bigl(\sum_{m\in\N^{d},\vert m\vert=j}\Bigl(\sum_{i=1}^{3}\sum_{k=1}^{D_{j}}\xi_{(i),j}\circ \phi_{i}(e_{m,k})\Bigr)\xi_{h_{i},m,k}(m!)^{\ast}n^{m}\Bigr)
   \\&\equiv\sum_{i=1}^{3}\sum_{m\in\N^{d},\vert m\vert=j}\sum_{k=1}^{D_{j}}\xi_{(i),j}\circ \phi_{i}(e_{m,k})\xi_{h_{i},m,k}\Bigl(\Delta_{t_{j}}\dots\Delta_{t_{1}}\binom{n}{m}\Bigr)
   \\&\equiv\xi_{j}(\Taylor_{j}(g')(t_{1},\dots,t_{j}))
   \equiv\sum_{i=1}^{3}\xi_{(i),j}(\Taylor_{j}(g'_{i})(t_{1},\dots,t_{j}))
   \\&\equiv\sum_{i=1}^{3}\xi_{(i),j}(\Taylor_{j}(g_{i})(t_{1},\dots,t_{j}))-\sum_{i=1}^{3}\xi_{(i),j}(\Taylor_{j}(g''_{i})(t_{1},\dots,t_{j}))\equiv 0 \mod\Z/Q.
   \end{split}
   \end{equation}
    Therefore
     \begin{equation}\label{3:lili}
     \sum_{m\in\N^{d},\vert m\vert=j}\Bigl(Q\sum_{i=1}^{3}\sum_{k=1}^{D_{j}}\xi_{(i),j}\circ \phi_{i}(e_{m,k})\Bigr)\xi_{h_{i},m,k}(m!)^{\ast}n^{m} \text{ is $(V_{p}(\tilde{M})^{h_{1},h_{2},h_{3}},p)$-reducible.}
     \end{equation}
     Recall that $Q\leq O_{C,d,D,\e}(1)$.
   Since  $\xi_{j}$ is
   of complexity at most $O_{C,d,D,\e}(1)$, so are $\xi_{(i),j}, 1\leq i\leq 3$.  Since   $\phi_{i}$ is $O_{C,d,D,\e,r}(1)$-rational, we have that for all $1\leq i\leq 3$ and $D_{j}^{\cor}+1\leq k\leq D_{j}$,
   $\xi_{(i),j}\circ \phi_{i}(e_{m,k})$ is of complexity $O_{C,d,D,\e,r}(1)$.
   So if the tower
   $$\mathcal{T}^{\cor}\uplus \uplus_{i=1}^{3}\mathcal{T}^{\lin}_{h_{i}}\uplus\uplus_{i=1}^{4} \mathcal{T}^{\ind}_{h_{i}}$$
   is $(V_{p}(\tilde{M})^{h_{1},h_{2},h_{3}}, L,p)$-independent  for some $L=O_{C,d,D,\e,r}(1)$, then (\ref{3:lili}) implies that $\xi_{(i),j}\circ \phi_{i}(e_{m,k})$ vanishes for all $1\leq i\leq 3$, $\vert m\vert=j$ and $D_{j}^{\cor}+1\leq k\leq D_{j}$, and that $\sum_{i=1}^{3}\xi_{(i),j}\circ \phi_{i}(e_{m,k})$ vanishes for all $\vert m\vert=j$ and $1\leq k\leq D_{j}^{\cor}$. By (\ref{3:11101}), this implies that $\xi_{j}$ vanishes on ${\pi'}^{-1}_{j}(V_{123,j})$, a contradiction.
    \end{proof}

For $1\leq j\leq s$, let $V_{\ind,j}$ be the subspace of $\Hor_{j}(G_{(1)})\times\Hor_{j}(G_{(2)})\times\Hor_{j}(G_{(3)})\times\Hor_{j}(G_{(4)})$ generated by
$$(\phi_{1}(e_{m,k}),\phi_{2}(e_{m,k}),\phi_{3}(e_{m,k}),\phi_{4}(e_{m,k})), \vert m\vert=j, 1\leq k\leq D_{j}^{\cor}$$
and
\begin{equation}\nonumber
\begin{split}
 (\phi_{1}(e_{m,k}),id,id,id),  (id,\phi_{2}(e_{m,k}),id,id), (id,id,\phi_{3}(e_{m,k}),id),  (id,id,id,\phi_{4}(e_{m,k}))
\end{split}
\end{equation}
for $\vert m\vert=j$ and $D^{\cor}_{j}+1\leq k\leq D^{\cor}_{j}+D^{\ind}_{j}.$

\begin{lem}\label{3:L11.5}
		There exists $L_{0}=L_{0}(c,d,D,\e,r)\in\N$ such that 
	if $L\geq L_{0}$, then
	for $1\leq j\leq s$, the projection of $G_{P}\cap G_{j}$ to  $\Hor_{j}(G_{(1)})\times\Hor_{j}(G_{(2)})\times \Hor_{j}(G_{(3)})\times \Hor_{j}(G_{(4)})$ contains $V_{\ind,j}$.
\end{lem}	
\begin{proof}
	  Suppose that the statement fails for some $1\leq j\leq s$.
	  Let $$\pi_{j}\colon G_{P}\cap G_{j}\to \Hor_{j}(G_{(1)})\times\Hor_{j}(G_{(2)})\times \Hor_{j}(G_{(3)})\times \Hor_{j}(G_{(4)})$$ and $$\pi'_{j}\colon \Hor_{j}(G)\to \Hor_{j}(G_{(1)})\times\Hor_{j}(G_{(2)})\times \Hor_{j}(G_{(3)})\times \Hor_{j}(G_{(4)})$$ denote the projection maps.
	  By (\ref{3:11.6}) and duality,  similar to the argument in Lemma \ref{3:L11.3}, we may find some $\xi_{j}\in \Xi_{V(M)^{h_{1},h_{2},h_{3}},j}(g)$ of complexity at most $O_{C,d,D,\e,r}(1)$ which annihilates  ${\pi'}^{-1}_{j}\circ\pi_{j}(G_{P}\cap G_{j})$ and  is nontrivial on ${\pi'}^{-1}_{j}(V_{\ind,j})$.

	  Since $\xi_{j}$ annihilates  ${\pi'}^{-1}_{j}\circ\pi_{j}(G_{P}\cap G_{j})$,
	 We may write
	  \begin{equation}\label{3:11102}
	  \xi_{j}(x_{1},x_{2},x_{3},x_{4},x_{0})=\xi_{(1),j}(x_{1})+\xi_{(2),j}(x_{2})+\xi_{(3),j}(x_{3})+\xi_{(4),j}(x_{4})
	  \end{equation}
	  for all $x_{i}\in \Hor_{j}(G_{(i)}), 0\leq i\leq 4$ for some characters $\xi_{(i),j}\in\mathfrak{N}_{j}(G_{(i)}/\Gamma_{(i)})$. 
	  Fix any $(n,t_{1},\dots,t_{j})\in \Gow_{p,j}(V_{p}(\tilde{M})^{h_{1},h_{2},h_{3}})$.
	  Since
	  the $j$-string
	  $(\Taylor_{m}(g''_{i}))_{\vert m\vert=j}$
	  is  $(V_{p}(\tilde{M})^{h_{i}},p)$-reducible, we have that
	  $$\sum_{\vert m\vert=j}\xi_{(i),j}(\Taylor_{m}(g''_{i}))(m!)^{\ast}n^{m}$$
	  is $(V_{p}(\tilde{M})^{h_{i}},p)$-reducible
	  for all $1\leq i\leq 4$.
	  By Lemma \ref{3:sbeq},
	   \begin{equation}\nonumber
	   \begin{split}
	   &\quad\sum_{i=1}^{4}\xi_{(i),j}(\Taylor_{j}(g''_{i})(t_{1},\dots,t_{j}))
	   \equiv
	   \sum_{i=1}^{4}\Delta_{t_{j}}\dots\Delta_{t_{1}}\Bigl(\sum_{m\in\N^{d},\vert m\vert=j}\xi_{(i),j}(\Taylor_{m}(g''_{i}))\binom{n}{m}\Bigr)
	   \\&\equiv
	   \sum_{i=1}^{4}\Delta_{t_{j}}\dots\Delta_{t_{1}}\Bigl(\sum_{m\in\N^{d},\vert m\vert=j}\xi_{(i),j}(\Taylor_{m}(g''_{i}))(m!)^{\ast}n^{m}\Bigr)
	  \equiv 0 \mod\Z/Q.
	   \end{split}
	   \end{equation}
	   On the other hand,  By the definition of $\Xi_{V(M)^{h_{1},h_{2},h_{3}},j,Q}(g)$ and Lemma \ref{3:sbeq}, we have 
	   \begin{equation}\nonumber
	   \begin{split}
	   \sum_{i=1}^{4}\xi_{(i),j}(\Taylor_{j}(g_{i})(t_{1},\dots,t_{j}))=\sum_{i=1}^{4}\Delta_{t_{j}}\dots\Delta_{t_{1}}\Bigl(\sum_{m\in\N^{d},\vert m\vert=j}\xi_{(i),j}(\Taylor_{m}(g_{i}))\binom{n}{m}\Bigr)\in\Z/Q.
	   \end{split}
	   \end{equation}
	   So by (\ref{3:11102}), 
	   \begin{equation}\nonumber
	   \begin{split}
	   &\quad 
	   \Delta_{t_{j}}\dots\Delta_{t_{1}}\Bigl(\sum_{m\in\N^{d},\vert m\vert=j}\Bigl(\sum_{i=1}^{4}\sum_{k=1}^{D_{j}}\xi_{(i),j}\circ \phi_{i}(e_{m,k})\Bigr)\xi_{h_{i},m,k}(m!)^{\ast}n^{m}\Bigr)
	   \\&\equiv\sum_{i=1}^{4}\sum_{m\in\N^{d},\vert m\vert=j}\sum_{k=1}^{D_{j}}\xi_{(i),j}\circ \phi_{i}(e_{m,k})\xi_{h_{i},m,k}\Bigl(\Delta_{t_{j}}\dots\Delta_{t_{1}}\binom{n}{m}\Bigr)
	   \\&\equiv\xi_{j}(\Taylor_{j}(g')(t_{1},\dots,t_{j}))
	   \equiv\sum_{i=1}^{4}\xi_{(i),j}(\Taylor_{j}(g'_{i})(t_{1},\dots,t_{j}))
	   \\&\equiv\sum_{i=1}^{4}\xi_{(i),j}(\Taylor_{j}(g_{i})(t_{1},\dots,t_{j}))-\sum_{i=1}^{4}\xi_{(i),j}(\Taylor_{j}(g''_{i})(t_{1},\dots,t_{j}))\equiv 0 \mod\Z.
	   \end{split}
	   \end{equation}
	 Therefore,
     \begin{equation}\label{3:lili2}
     \sum_{m\in\N^{d},\vert m\vert=j}\Bigl(Q\sum_{i=1}^{4}\sum_{k=1}^{D_{j}}\xi_{(i),j}\circ \phi_{i}(e_{m,k})\Bigr)\xi_{h_{i},m,k}(m!)^{\ast}n^{m} \text{ is $(V_{p}(\tilde{M})^{h_{1},h_{2},h_{3}},p)$-reducible.}
     \end{equation}
     Recall that $Q\leq O_{C,d,D,\e}(1)$.
	    Since the map $h\mapsto \xi_{h,m,k}, h\in H'$ is a $\Z/p^{r}$-almost linear Freiman homomorphism for all $D_{j}^{\cor}+D_{j}^{\ind}+1\leq k\leq D_{j}$ and $\vert m\vert=j$, we have that  
	    \begin{equation}\label{3:572}
	    \xi_{h_{4},m,k}\equiv\xi_{h_{1},m,k}+\xi_{h_{2},m,k}-\xi_{h_{3},m,k} \mod\Z.
	    \end{equation}
	    Since  $\xi_{j}$ is
	    of complexity at most $O_{C,d,D,\e,r}(1)$, so are $\xi_{(i),j}, 1\leq i\leq 4$.  Since   $\phi_{i}$ is $O_{C,d,D,\e,r}(1)$-rational, we have that for all $1\leq i\leq 4$ and $D_{j}^{\cor}+1\leq k\leq D_{j}$,
	    $\xi_{(i),j}\circ \phi_{i}(e_{m,k})$ is of complexity $O_{C,d,D,\e,r}(1)$.
	    So if
	    the tower
	    $$\mathcal{T}^{\cor}\uplus \uplus_{i=1}^{3}\mathcal{T}^{\lin}_{h_{i}}\uplus\uplus_{i=1}^{4} \mathcal{T}^{\ind}_{h_{i}}$$
	    is $(V_{p}(\tilde{M})^{h_{1},h_{2},h_{3}},L,p)$-independent for some $L=O_{C,d,D,\e,r}(1)$, then we may deduce from (\ref{3:lili2}) and (\ref{3:572}) that $\xi_{(i),j}\circ \phi_{i}(e_{m,k})$ vanishes for all $1\leq i\leq 4$, $\vert m\vert=j$ and $D_{j}^{\cor}+1\leq k\leq D_{j}^{\cor}+D_{j}^{\ind}$, and that $\sum_{i=1}^{4}\xi_{(i),j}\circ \phi_{i}(e_{m,k})$ vanishes for $\vert m\vert=j$ and $1\leq k\leq D_{j}^{\cor}$. By (\ref{3:11102}), this implies that $\xi_{j}$ vanishes on ${\pi'}^{-1}_{j}(V_{\ind,j})$, a contradiction.
\end{proof}

We may now summarize the discussion in this section by the following theorem:

\begin{thm}\label{3:vns}
	Let the notations be the same as in Section \ref{3:s:b4} and let $\d=\d(C,d,D,\e,r)$, $L=L_{0}(C,d,D,\e,r)$.  
	Let $w\in G^{\vec{D},d}$ be any iterated commutator of $e_{m_{1},j_{1}},\dots,e_{m_{r_{\ast}},j_{r_{\ast}}}$ for some $1\leq \vert m_{1}\vert,\dots,\vert m_{r_{\ast}}\vert\leq s$ with $\vert m_{1}\vert+\dots+\vert m_{r_{\ast}}\vert=s$ and for some $1\leq j_{\ell}\leq D_{\vert m_{\ell}\vert}$ for all $1\leq \ell\leq r_{\ast}$.
	\begin{enumerate}[(i)]
		\item If $j_{\ell}>D^{\cor}_{\vert m_{\ell}\vert}$ for at least two values of $\ell$, then $\eta(w)\in\Z$;
		\item If $D^{\cor}_{\vert m_{\ell}\vert}<j_{\ell}\leq D^{\cor}_{\vert m_{\ell}\vert}+D^{\ind}_{\vert m_{\ell}\vert}$ for at least one value of $\ell$, then $\eta(w)\in\Z$.
	\end{enumerate}	
\end{thm}	
\begin{proof}	
Let $h_{1},h_{2},h_{3}\in H'$ be chosen as in this section.	
Suppose that $j_{\ell}>D^{\cor}_{\vert m_{\ell}\vert}$ for at least two values of $\ell$. By Corollary 11.4 of \cite{GTZ12}, we have that Lemma \ref{3:L11.3} implies that $(\phi_{1}(w),id,id,id,id)\in G_{P}$. So $\eta(w)\in\Z$ by Lemma \ref{3:L11.2}.

Similarly, 	suppose that $D^{\cor}_{\vert m_{\ell}\vert}<j_{\ell}\leq D^{\cor}_{\vert m_{\ell}\vert}+D^{\ind}_{\vert m_{\ell}\vert}$ for at least one value of $\ell$. By Corollary 11.6 of \cite{GTZ12}, we have that Lemma \ref{3:L11.5} implies that $(\phi_{1}(w),id,id,id,id)\in G_{P}$. So $\eta(w)\in\Z$ by Lemma \ref{3:L11.2}.
\end{proof}

\section{The construction of a degree-$(1,s)$ nilcharacter}\label{3:s:b5}

In this section, we complete Step 3 described in Section \ref{3:s:otl} by showing that the nilcharacters $\chi_{h}(n)$ described in Proposition \ref{3:inductioni1} can be expressed in the form $\chi(h,n)$ for some nilcharacters $\chi$ of  multi-degree $(1,s)$.
To be more precise, we show:

\begin{prop}\label{3:srthm}
	Let $d,D,r\in\N_{+}$, $C,\e>0$, $[s,r_{\ast}]\in\DR$, $p\gg_{C,d,D,\e,r} 1$ be a prime,   $M\colon\V\to\F_{p}$ be a non-degenerate quadratic form with $\tilde{M}\colon\Z^{d}\to\Z/p$ being its regular lifting, and
	 $f\colon\Z^{d}\to\C$ be a function bounded by 1 with $f(n+pm)=f(n)$ for all $m,n\in\Z^{d}$. Suppose that there exist a subset $H$  of $\Z^{d}_{p^{r}}$ with $\vert H\vert\geq \e p^{rd}$, some $\chi_{0}\in\Xi_{p}^{(1,s);C,D}((\Z^{d})^{2})$,  some $\chi_{h}\in \Xi_{p^{r}}^{[s,r_{\ast}];C,D}(\Z^{d})$, and some $\psi_{h}\in\Nil_{p}^{s-1;C,1}(\Z^{d})$ for all $h\in H$ such that 
	\begin{equation}\label{3:longlongago}
	\Bigl\vert\E_{n\in V_{p}(\tilde{M})^{h}}f(n+h)\overline{f}(n)\chi_{0}(h,n)\otimes \chi_{h}(n)\psi_{h}(n)\Bigr\vert>\e
	\end{equation}
	for all $h\in H$. 
	If $d\geq N(s)$,
	then there exist a subset $H'\subseteq H$ with $\vert H'\vert\gg_{C,d,D,\e,r} p^{rd}$,  some $Q\in\N_{+}$ with $Q\leq O_{C,d,D,\e,r}(1)$  some  $\chi\in\Xi_{p}^{(1,s);O_{C,d,D,\e,r}(1),O_{C,d,D,\e,r}(1)}((\Z^{d})^{2})$, some $\chi'_{h}\in \Xi_{p^{Q}}^{<[s,r_{\ast}];O_{C,d,D,\e,r}(1),O_{C,d,D,\e,r}(1)}(\Z^{d})$ and some $\psi'_{h}\in\Nil_{p}^{s-1;O_{C,d,D,\e,r}(1),1}(\Z^{d})$ for all $h\in H'$ such that
	\begin{equation}\nonumber
	\Bigl\vert\E_{n\in V_{p}(\tilde{M})^{h}}f(n+h)\overline{f}(n)\chi(h,n)\otimes \chi'_{h}(n)\psi'_{h}(n)\Bigr\vert\gg_{C,d,D,\e,r} 1
	\end{equation}
	for all $h\in H'$.
\end{prop}

The rest of the section is devoted to the proof of Proposition \ref{3:srthm}. 

\textbf{Step 1.} We first use the sunflower lemma to represent all of $\chi_{h}$ in a uniform way.
Passing to a subset if necessary, by Lemma \ref{3:countingh}, we may assume without loss of generality that $\iota(h)$ is $M$-non-isotropic  for all $h\in H$. 
Let $\d=\d(C,d,D,\e), L=L(C,d,D,\e)>0$ to be chosen later.
Since $p\gg_{C,d,D,\e,r} 1$ and $d\geq N(s)$, there exists a subset $H'$ of $H$ with $\vert H'\vert\gg_{C,d,D,\e,r} p^{rd}$ such that all the requirements of Lemma \ref{3:sf2} are satisfied for $(\chi_{h})_{h\in H'}$.  For convenience we use the same notations as in Lemma \ref{3:sf2}.

Note that if we replace the function $F_{h}$ which defines $\chi_{h}$ in Lemma \ref{3:sf2} by a function $F'$ such that $\Vert F_{h}-F'\Vert_{L^{\infty}}<\e/2$. Then  (\ref{3:longlongago}) still holds with the right hand side replaced by $\e/2$. On the other hand, there exists $K=O_{C,d,D,\e,r}(1)$ such that all of $F_{h}, h\in H'$ belong to the set $S$ of all Lipschitz functions of Lipschitz norm at most $K$. Note that $S$
is  equicontinuous and pointwise bounded, and thus is relatively compact by Arzel\`a-Ascoli Theorem. So there exist open balls $B_{1},\dots,B_{N}$ of radius $\e/2$ centered at $F_{1},\dots,F_{N}$ respectively for some $N=O_{C,d,D,\e,r}(1)$ which covers $S$. By the Pigeonhole Principle, there exists a subset $H''$ of $H'$ of cardinality  $\gg_{C,d,D,\e,r}p^{rd}$ such that all of $F_{h},h\in H''$ lie in the same ball $B$ centered at $F$. Then $\Vert F-F_{h}\Vert_{L^{\infty}}<\e/2$. 
In conclusion, replacing the right hand side of (\ref{3:longlongago}) by $\e/2$ if necessary and denoting $H''$ as $H'$ for convenience, we may assume without loss of generality that 
  (\ref{3:longlongago}) holds and that  $F_{h}=F$ for all $h\in H'$.
We now choose $\d=\d(C,d,D,\e)$ to be sufficiently small and $L=L(C,d,D,\e)$ to be sufficiently large such that Theorem \ref{3:vns} applies to the family $(\chi_{h})_{h\in H'}$.

For all  $h\in H'$, we may write $\chi_{h}$ as %$\chi_{h}=\chi'_{h}\circ\tau$, where 
$$\chi_{h}(n)=F(g_{h}(n)\Gamma,g_{0,h}(n)\Gamma_{0}), \text{ for all } n\in\Z^{d}, h\in  H',$$
where
$g_{h}\in\poly_{p^{r}}(\Z^{d}\to G_{\N}\vert\Gamma)$, $g_{0,h}\in\poly_{p^{r}}(\Z^{d}\to (G_{0})_{\N}\vert\Gamma_{0})$,
 $F$ is a Lipschitz function of complexity $O_{C,d,D,\e,r}(1)$ with a vertical frequency $\eta\colon G^{\vec{D},d}_{[s,r_{\ast}]}\to\R$ of complexity $O_{C,d,D,\e,r}(1)$ in the sense that 
\begin{equation}\label{3:12.1}
F(\phi(t)x,x_{0})=\exp(\eta(t))F(x,x_{0})
\end{equation}
for all $t\in G^{\vec{D},d}_{[s,r_{\ast}]}, x\in G/\Gamma$ and $x_{0}\in G_{0}/\Gamma_{0}$,
 and that the $(j,d,p^{r})$-string
\begin{equation}\label{3:tlh}
\Bigl(\Taylor_{m}(g_{h})\cdot\Bigl(\pi_{\Hor_{j}(G/\Gamma)}\circ\phi(\prod_{k=1}^{D^{\cor}_{j}}e_{m,k}^{\xi_{m,k}}\prod_{k=D^{\cor}_{j}+1}^{D_{j}}e_{m,k}^{\xi_{h,m,k}})\Bigr)^{-1}\Bigr)_{\vert m\vert=j}
\end{equation}
is $(V_{p}(\tilde{M})^{h},p)$-reducible for all $1\leq j\leq s$.

 \textbf{Step 2.} We next show that by modifying the polynomial sequence $g_{h}$ appropriately, we can make (\ref{3:tlh}) take values in $\Gamma$.

	 Let $\mathcal{X}$ be a Mal'cev basis adapted to the degree filtration $(G_{[j,1]})_{j\in\N}$. 
	 Assume that  $\mathcal{X}$ can be partitioned into $\cup_{1\leq x\leq s}\mathcal{X}_{x}$ such that $G_{[j,1]}$ is generated by $\cup_{x\geq j}\mathcal{X}_{x}$. Since $G_{[j,1]}/G_{[j+1,1]}$ is abelian, 	 We may further partition $\mathcal{X}=\cup_{[x,y]\in \DR, [x,y]\leq [s,r_{\ast}]}\mathcal{X}_{x,y}$, where $\mathcal{X}_{x,y}=\{X_{x,y,1},\dots,X_{x,y,\ell_{x,y}}\}$ is such that $G_{[x,y]}$ is generated by $\cup_{[x',y']\in \DR, [x',y']\geq [x,y]}\mathcal{X}_{x',y'}$.
	By Lemma \ref{3:malfil}, we may write
	 $$g_{h}(n)=\prod_{1\leq x\leq s}\prod_{1\leq y\leq x}\prod_{z=1}^{\ell_{x,y}}\exp(P^{h}_{x,y,z}(n)X_{x,y,z})$$
	 for some polynomial $P^{h}_{x,y,z}\colon\Z^{d}\to\R$ of degree at most $x$. 
	 Assume that 
	   \begin{equation}\label{3:talor0}
	   P^{h}_{x,1,z}(n)=\sum_{m\in\N^{d},\vert m\vert\leq x}a^{h}_{x,z,m}\binom{n}{m}
	   \end{equation}
for some $a^{h}_{x,z,m}\in\R$ for all $1\leq x\leq s$, $1\leq z\leq \ell_{x,1}$.
 By Lemma \ref{3:magictaylor}, for all $m\in\N^{d}$, we have that 
 \begin{equation}\label{3:talora}
 \Taylor_{m}(g_{h})=\prod_{z=1}^{\ell_{\vert m\vert,1}}\exp(a^{h}_{\vert m\vert,z,m}X_{\vert m\vert,1,z})  \mod G_{[\vert m\vert,2]}.
 \end{equation}
 Assume that
 \begin{equation}\label{3:talorb}
 \phi\Bigl(\prod_{k=1}^{D^{\cor}_{j}}e_{m,k}^{\xi_{m,k}}\prod_{k=D^{\cor}_{j}+1}^{D_{j}}e_{m,k}^{\xi_{h,m,k}}\Bigr)=\prod_{z=1}^{\ell_{\vert m\vert,1}}\exp(b^{h}_{z,m}X_{\vert m\vert,1,z})  \mod G_{[\vert m\vert,2]}
 \end{equation}
for some $b^{h}_{z,m}\in\R$	for all $m\in\N^{d}$. Since
$g_{h}\in\poly_{p^{r}}(\Z^{d}\to G_{\N}\vert\Gamma)$, by Lemma \ref{3:ratispp}, we have that $a^{h}_{\vert m\vert,z,m}\in\Z/p^{r}$. Since 
  $\xi_{m,k},\xi_{h,m,k}\in\Z/p^{r}$, we have that $b^{h}_{z,m}\in\Z/p^{r}$.
	 Since (\ref{3:tlh}) is $(V_{p}(\tilde{M})^{h},p)$-reducible for all $1\leq x\leq s$, it follows from (\ref{3:talora}) and (\ref{3:talorb}) that the polynomial
	 $$f^{h}_{x,z}(n):=\sum_{m\in\N^{d},\vert m\vert=x}(a^{h}_{x,z,m}-b^{h}_{z,m})(m!)^{\ast}n^{m}$$
	 is $(V_{p}(\tilde{M})^{h},p)$-reducible for all $1\leq x\leq s$ and $1\leq z\leq \ell_{x,1}$.
	  Since $d\geq N(s)$ and $\iota(h)$ is $M$-non-isotropic, by Lemma \ref{3:att55},  we have that $f^{h}_{x,z}\in J^{M}_{\iota(h)}$.
	  % we may write 
	%   \begin{equation}\label{3:abxy}
	%   \sum_{m\in\N^{d},\vert m\vert=x}(a^{h}_{x,z,m}-b^{h}_{z,m})(m!)^{\ast}n^{m}=f^{h}_{x,z}(n)+g^{h}_{x,z}(n)
	%   \end{equation}
	%   for some $\Z/p^{x}$-valued polynomial $f^{h}_{x,z}$ of degree at most $x$ 
%	with $f^{h}_{x,z}(n)\in\Z$ for all $n\in V_{p}(\tilde{M})^{h}$, and some $\Z/p^{x}$-valued polynomial $g^{h}_{x,z}$  of degree at most $x-1$. By Lemma \ref{3:ivie}, {\color need and extension???} modifying $f^{h}_{x,z}$ and $g^{h}_{x,z}$ if necessary, we may further require that $g^{h}_{x,z}$ has   $\Z/p^{x}$-coefficients.

	Denote 
	$$g''_{h}(n):=\prod_{1\leq x\leq s}\prod_{z=1}^{\ell_{x,1}}\exp((P^{h}_{x,1,z}(n)+f^{h}_{x,z}(n))X_{x,1,z}).$$
%	where
%	 \begin{equation}\label{3:abxy2}
%	 Q^{h}_{x,1,z}(n):=P^{h}_{x,1,z}(n)+\sum_{m\in\N^{d},\vert m\vert=x}(b^{h}_{z,m}-a^{h}_{x,z,m})(m!)^{\ast}n^{m}%+g^{h}_{x,z}(n)
%	 \end{equation}
%	is a polynomial with $\Z/p^{x}$-coefficients. 
	Since  $P^{h}_{x,1,z}+f^{h}_{x,z}$ is a polynomial of degree at most $x$, we have that the map $n\mapsto \exp((P^{h}_{x,1,z}(n)+f^{h}_{x,z}(n))X_{x,1,z})$ belongs to $\poly(\Z^{d}\to G_{\N})$, and thus $g''_{h}\in\poly(\Z^{d}\to G_{\N})$ by Corollary B.4 of \cite{GTZ12}. %(however, we caution the readers that $g''_{h}$ needs not necessarily belong to $\poly_{p}(\Z^{d}\to G_{\N}\vert\Gamma)$). 
	By Lemma \ref{3:magictaylor}, (\ref{3:talor0}) and (\ref{3:talorb}),
	\begin{equation}\label{3:tlh20}
	\begin{split}
	&\quad\Taylor_{m}(g''_{h})\equiv\prod_{z=1}^{\ell_{\vert m\vert,1}}\Taylor_{m}(\exp((P^{h}_{x,1,z}(n)+f^{h}_{x,z}(n))X_{\vert m\vert,1,z}))
	\\&\equiv\prod_{z=1}^{\ell_{\vert m\vert,1}}\exp\Bigl(((1-m!(m!)^{\ast})a^{h}_{\vert m\vert,z,m}+m!(m!)^{\ast}b^{h}_{z,m})X_{\vert m\vert,1,z}\Bigr)
	\\&\equiv\prod_{z=1}^{\ell_{\vert m\vert,1}}\exp(b^{h}_{z,m}X_{\vert m\vert,1,z})
	\equiv \phi\Bigl(\prod_{k=1}^{D^{\cor}_{j}}e_{m,k}^{\xi_{m,k}}\prod_{k=D^{\cor}_{j}+1}^{D_{j}}e_{m,k}^{\xi_{h,m,k}}\Bigr) \mod G_{[j,2]}
	\end{split}
	\end{equation}
	for all $1\leq j\leq s$ and $m\in\N^{d}$ with $\vert m\vert=j$.	
	%
	%On the other hand, by (\ref{3:abxy}) and (\ref{3:abxy2}),
	% $P^{h}_{x,1,z}(n)-Q^{h}_{x,1,z}(n)=f^{h}_{x,z}(n)$.
	So by the Baker-Campbell-Hausdorff formula, there exists $g'''_{h}\in\poly(\Z^{d}\to [G,G]_{\N})$ (with the filtration induced by $G_{\N}$) such that 
	$$g_{h}(n)=g''_{h}(n)g'''_{h}(n)\prod_{1\leq x\leq s}\prod_{z=1}^{\ell_{x,1}}\exp(f^{h}_{x,z}(n)X_{x,1,z})$$
	for all $n\in\Z^{d}$. Write $g'_{h}:=g''_{h}g'''_{h}\in\poly(\Z^{d}\to G_{\N})$. Since $g'''_{h}$ takes values in $G_{[0,2]}$, it follows from Lemma \ref{3:magictaylor} and (\ref{3:tlh20}) that 
		\begin{equation}\label{3:tlh2}
		\Taylor_{m}(g'_{h})=\Taylor_{m}(g''_{h})= \phi\Bigl(\prod_{k=1}^{D^{\cor}_{j}}e_{m,k}^{\xi_{m,k}}\prod_{k=D^{\cor}_{j}+1}^{D_{j}}e_{m,k}^{\xi_{h,m,k}}\Bigr) \mod G_{[j,2]}
		\end{equation}
		for all $1\leq j\leq s$ and $m\in\N^{d}$ with $\vert m\vert=j$.
		Finally, for all $n\in V_{p}(\tilde{M})^{h}$, since  $f^{h}_{x,z}(n)\in\Z$, we have that $g_{h}(n)\Gamma=g'_{h}(n)\Gamma$. So 
	$$\chi_{h}(n)=F(g'_{h}(n)\Gamma,g_{0,h}(n)\Gamma_{0}), \text{ for all } n\in V_{p}(\tilde{M})^{h}, h\in  H'.$$

 \textbf{Step 3.} 
 Our next step is to show that we can remove the core frequencies $\xi_{m,k}$ in the expression (\ref{3:tlh2}) by studying a nilcharacter which differs from $\chi'_{h}$ by a nilsequence of lower degree. We use a variation of the method in Section 12 of \cite{GTZ12}.
Let $\tilde{G}$ be the free Lie group generated by the generators $\tilde{e}_{m,k}, m\in\N^{d}, 1\leq \vert m\vert\leq s, k\in [1,D^{\cor}_{\vert m\vert}]\cup[D^{\cor}_{\vert m\vert}+D^{\ind}_{\vert m\vert}+1,D_{\vert m\vert}]$
subject to the following relations:
\begin{itemize}
	\item any iterated commutators $\tilde{e}_{m_{1},k_{1}},\dots,\tilde{e}_{m_{r},k_{r}}$ vanishes if $\vert m_{1}\vert+\dots+\vert m_{r}\vert>s$;
	\item any iterated commutators $\tilde{e}_{m_{1},k_{1}},\dots,\tilde{e}_{m_{r},k_{r}}$ vanishes if $\vert m_{1}\vert+\dots+\vert m_{r}\vert=s$ and $r>r_{\ast}$;
 	\item any iterated commutators $\tilde{e}_{m_{1},k_{1}},\dots,\tilde{e}_{m_{r},k_{r}}$ vanishes if $k_{\ell}>D^{\cor}_{\vert m_{\ell}\vert}$ for at least two values of $\ell$.
\end{itemize}	
We may endow a $\DR$-filtration $\tilde{G}_{\DR}$ on $\tilde{G}$ by setting $\tilde{G}_{[t,r]}, r\neq 0$ to be the group generated by iterated commutators $\tilde{e}_{m_{1},k_{1}},\dots,\tilde{e}_{m_{r'},k_{r'}}$ with $\vert m_{1}\vert+\dots+\vert m_{r}\vert\geq t$ and $r'\geq r$, or with  $\vert m_{1}\vert+\dots+\vert m_{r}\vert>t$, and by setting $\tilde{G}_{[0,0]}:=\tilde{G}$, $\tilde{G}_{[t,0]}:=\tilde{G}_{[t,1]}$ for all $t\in\N$. . Let $\tilde{\Gamma}$ be the discrete group generated by all of $\tilde{e}_{m,k}$. Then $\tilde{G}/\tilde{\Gamma}$ is a nilmanifold of degree-rank at most $[s,r_{\ast}]$.

Let $G^{\ast}$ be the subgroup of $G^{\vec{D},d}$ generated by iterated commutators of $e_{m_{1},k_{1}},\dots,e_{m_{r},k_{r}}$ for some $r\in\N_{+}$ 
such that either $k_{\ell}>D^{\cor}_{\vert m_{\ell}\vert}$ for at least two values of $\ell$, or  $D^{\cor}_{\vert m_{\ell}\vert}<k_{\ell}\leq D^{\cor}_{\vert m_{\ell}\vert}+D^{\ind}_{\vert m_{\ell}\vert}$ for at least one value of $\ell$.
Then $\tilde{G}$ is isomorphic to the quotient of $G^{\vec{D},d}$ by $G^{\ast}$.
 Let $\tilde{\phi}\colon G^{\vec{D},d}\to \tilde{G}$ denote the quotient map.
It is clear that $\tilde{G}_{[s,r_{\ast}]}$ is isomorphic to the quotient of $G^{\vec{D},d}_{[s,r_{\ast}]}$ by $G^{\ast}\cap G^{\vec{D},d}_{[s,r_{\ast}]}$ under the map $\phi$.
 
Let $G^{\ast\ast}$ be the subgroup of $G^{\vec{D},d}$ generated by iterated commutators of $e_{m_{1},j_{1}},\dots,e_{m_{r_{\ast}},j_{r_{\ast}}}$ for some $1\leq \vert m_{1}\vert,\dots,\vert m_{r_{\ast}}\vert\leq s$ with $\vert m_{1}\vert+\dots+\vert m_{r_{\ast}}\vert=s$ and for some $1\leq j_{\ell}\leq D_{\vert m_{\ell}\vert}$ for all $1\leq \ell\leq r_{\ast}$ such that either
 $j_{\ell}>D^{\cor}_{\vert m_{\ell}\vert}$ for at least two values of $\ell$ or   $D^{\cor}_{\vert m_{\ell}\vert}<j_{\ell}\leq D^{\cor}_{\vert m_{\ell}\vert}+D^{\ind}_{\vert m_{\ell}\vert}$ for at least one value of $\ell$.
	It is clear that 
	 $G^{\ast}\cap G^{\vec{D},d}_{[s,r_{\ast}]}$ is a subgroup of the central group $G^{\vec{D},d}_{[s,r_{\ast}]}$ containing $G^{\ast\ast}$.

By Theorem \ref{3:vns}, $\eta\colon G^{\vec{D},d}_{[s,r_{\ast}]}\to\R$ annihilates $G^{\ast\ast}$ and thus descends to a group homomorphism $\eta'\colon G^{\vec{D},d}_{[s,r_{\ast}]}/G^{\ast\ast}\to \R$. Since $G^{\ast}\cap G^{\vec{D},d}_{[s,r_{\ast}]}$ contains $G^{\ast\ast}$, $\eta'$ induces a group homomorphism $\eta''\colon G^{\vec{D},d}_{[s,r_{\ast}]}/(G\cap G^{\vec{D},d}_{[s,r_{\ast}]})\to\R$.
Since  $\tilde{G}_{[s,r_{\ast}]}$ is isomorphic to $G^{\vec{D},d}_{[s,r_{\ast}]}/(G\cap G^{\vec{D},d}_{[s,r_{\ast}]})$ under the quotient map $\tilde{\phi}$, $\eta''$ induces a group homomorphism $\tilde{\eta}\colon \tilde{G}_{[s,r_{\ast}]}\to\R$. Finally, since $\eta$ is a vertical frequency of $G^{\vec{D},d}/\Gamma^{\vec{D},d}$ of complexity $O_{C,d,D,\e,r}(1)$, we have that $\tilde{\eta}$ is a vertical frequency of $\tilde{G}/\tilde{\Gamma}$ of complexity $O_{C,d,D,\e,r}(1)$.
By (\ref{3:12.1}) and similar to the construction in (6.3) of \cite{GTZ12}, it is not hard to construct a function $\tilde{F}\in\Lip(\tilde{G}/\tilde{\Gamma}\to \mathbb{S}^{O_{C,d,D,\e,r}(1)})$ of complexity $O_{C,d,D,\e,r}(1)$ with vertical frequency $\tilde{\eta}$. 

 For all $D^{\cor}_{\vert m\vert}+D^{\ind}_{\vert m\vert}+1\leq k\leq D_{\vert m\vert}$, we may write
  $\xi_{h,m,k}=\xi_{m,k}(h)$ for some  almost $\Z/p^{r}$-linear Freiman homomorphism $\xi_{m,k}$ defined on $H'$ given by
 $$\xi_{h,m,k}:=\xi_{m,k}(h)=\sum_{i=1}^{L}\{\alpha_{m,k,i}\cdot h\}\beta_{m,k,i}$$
 for some $L=O_{C,d,D,\e,r}(1)$, $\alpha_{m,k,i}\in (\Z/p^{r})^{d},\beta_{m,k,i}\in\mathbb{Z}/p^{r}$. 
 Passing to another subset of $H'$ if necessary, we may assume without loss of generality that for all $m,k$ and $i$, all of  $\{\alpha_{m,k,i}\cdot h\}, h\in H'$ lie in an interval $I_{m,k,i}$ of length at most $1/10$.

Denote
\begin{equation}\label{3:12.3}
g_{0}(n):=\prod_{1\leq \vert m\vert\leq s}\prod_{k=1}^{D^{\cor}_{\vert m\vert}}\tilde{\phi}(e_{m,k})^{\xi_{m,k}\binom{n}{m}},
\end{equation}
and
\begin{equation}\label{3:12.4}
\tilde{g}_{h}(n):=\prod_{1\leq \vert m\vert\leq s}\prod_{k=D^{\cor}_{\vert m\vert}+D^{\ind}_{\vert m\vert}+1}^{D_{\vert m\vert}}\tilde{\phi}(e_{m,k})^{\xi_{m,k}(h)\binom{n}{m}}=\prod_{1\leq \vert m\vert\leq s}\prod_{k=D^{\cor}_{\vert m\vert}+D^{\ind}_{\vert m\vert}+1}^{D_{\vert m\vert}}\prod_{i=1}^{K}\tilde{\phi}(e_{m,k})^{\{\alpha_{m,k,i}\cdot \tau(h)\}\beta_{m,k,i}\binom{n}{m}}.
\end{equation}
Consider the nilcharacter 
\begin{equation}\label{3:121e1}
\tilde{\chi}_{h}(n):=\tilde{F}(g_{0}(n)\tilde{g}_{h}(n)\tilde{\Gamma}).
\end{equation}
Clearly, $\tilde{\chi}_{h}\in \Xi^{[s,r_{\ast}];O_{C,d,D,\e,r}(1),O_{C,d,D,\e,r}(1)}(\Z^{d})$.
%Again we remark that $g_{0}\tilde{g}_{h}$ is not necessarily $p$-periodic.
%However, 
Similar to Lemma 12.1 of \cite{GTZ12}, we have

\begin{lem}\label{3:L12.1}
	For all $h\in H'$, we have that  the map
	$$n\mapsto \chi_{h}(n)\otimes\overline{\tilde{\chi}}_{h}(n)$$ belongs to $\Nil^{<[s;r_{\ast}];O_{C,d,D,\e,r}(1)}_{\approx Qp^{Q}}(V_{p}(\tilde{M})^{h})$ for some $Q\in\N_{+}$ with $Q\leq O_{C,d,D,\e,r}(1)$.
	Moreover,
	$\chi_{h}$ is an $O_{C,d,D,\e,r}(1)$-complexity linear combination of the components $\tilde{\chi}_{h}\otimes \psi'_{h}$ for some $\psi'_{h}\in\Nil^{<[s,r_{\ast}];O_{C,d,D,\e,r},O_{C,d,D,\e,r}}_{\approx Qp^{Q}}(V_{p}(\tilde{M})^{h})$.
\end{lem}

\begin{proof}
		We may rewrite the sequence $n\mapsto \chi_{h}(n)\otimes\overline{\tilde{\chi}}_{h}(n)$ as 
	\begin{equation}\label{3:12.6}
	n\mapsto F'_{h}(f_{h}(n)\Gamma') \text{ for all } n\in V_{p}(\tilde{M})^{h},
	\end{equation}
	where $G':=G\times G_{0}\times \tilde{G}$, $\Gamma':=\Gamma\times\Gamma_{0}\times\tilde{\Gamma}$, $$f_{h}(n):=(g'_{h}(n),g_{0,h}(n),g_{0}(n)\tilde{g}_{h}(n))$$	
	and $F'_{h}$ is the Lipschitz function on $G'/\Gamma'$ given by
	$$F'_{h}(x,x_{0},y):=F_{h}(x,x_{0})\otimes \overline{F}(y).$$	
	We define a $\DR$-filtration $G'_{\DR}$ on $G'$ as follows.
	 For $[i,r]\in\DR$ with $i,r\geq 1$, let $G'_{[i,r]}$  be the Lie group generated by $G_{[i,r+1]}\times (G_{0})_{[i,r]}\times \tilde{G}_{[i,r+1]}$ and $\{(\phi(g),id, \tilde{\phi}(g))\colon g\in G_{[s,r]}^{\vec{D},d}\}$, where we take the convention that $[i,i+1]=[i+1,0]$. 
	 Set also  $G'_{[i,0]}:=G'_{[i,1]}$ for $i\geq 1$ and $G'_{[0,0]}=G'$. 
	 It is not hard to see that this defines a degree-rank pre-filtration. 
	
	We first show that $f_{h}$ is a polynomial with respect to the filtration $G'_{\DR}$. 
	Note that the sequence $n\mapsto (id,g_{0,h}(n),id)$ is already a polynomial with respect to the pre-filtration $G'_{\DR}$. %By Corollary B.4 of \cite{GTZ12},  
	Since this sequence is commutative with
 the sequence 	
	\begin{equation}\label{3:12.7}
	n\mapsto (g'_{h}(n),id,g_{0}(n)\tilde{g}_{h}(n)),
	\end{equation}	
	it suffices to show that (\ref{3:12.7})
	is a polynomial sequence with respect to  the pre-filtration $G'_{\DR}$ as well. 
		Let $(G''/\Gamma'')_{\DR}$ be the essential component of $(G'/\Gamma')_{\DR}$.
	 By Lemma \ref{3:B.9}, we may write $g'_{h}(n)=\prod_{0\leq \vert m\vert\leq s}{g'_{h,m}}^{\binom{n}{m}}$ for some $g'_{h,m}\in G_{[\vert m\vert,0]}$. By (\ref{3:tlh2}) and Lemma \ref{3:magictaylor}, we have
	$$g'_{h,m}\equiv\Taylor_{m}(g'_{h})\equiv\phi\Bigl(\prod_{k=1}^{D^{\cor}_{\vert m\vert}}e_{m,k}^{\xi_{m,k}}\prod_{k=D^{\cor}_{\vert m\vert}+1}^{D_{\vert m\vert}}e_{m,k}^{\xi_{h,m,k}}\Bigr) \mod G_{[\vert m\vert,2]}.$$
	By the construction of the filtration $G'_{\DR}$, this implies that 
	$$\Bigl(g'_{h,m},id, \prod_{k=1}^{D^{\cor}_{\vert m\vert}}e_{m,k}^{\xi_{m,k}}\prod_{k=D^{\cor}_{\vert m\vert}+1}^{D_{\vert m\vert}}e_{m,k}^{\xi_{h,m,k}} \mod G^{\ast} \Bigl)\in G'_{[\vert m\vert,1]}=G''_{[\vert m\vert,1]}.$$
	for all $m\neq \bold{0}.$
	Applying Corollary B.4 of  \cite{GTZ12}, we have that the sequence
	$$n\mapsto \Bigl(g'_{h}(\bold{0})^{-1}g'_{h}(n),id, \prod_{0<\vert m\vert\leq s}\Bigl(\prod_{k=1}^{D^{\cor}_{\vert m\vert}}e_{m,k}^{\xi_{m,k}}\prod_{k=D^{\cor}_{\vert m\vert}+1}^{D_{\vert m\vert}}e_{m,k}^{\xi_{h,m,k}}\Bigr)^{\binom{n}{m}} \mod G^{\ast}\Bigl)$$
	is a polynomial with respect to the filtration $G''_{\DR}$. By the Baker-Campbell-Hausdorff formula, (\ref{3:12.3}) and (\ref{3:12.4}), we have that 
	$$n\mapsto \prod_{0\leq \vert m\vert\leq s}\Bigl(\prod_{k=1}^{D^{\cor}_{\vert m\vert}}e_{m,k}^{\xi_{m,k}}\prod_{k=D^{\cor}_{\vert m\vert}+1}^{D_{\vert m\vert}}e_{m,k}^{\xi_{h,m,k}}\Bigr)^{\binom{n}{m}} \mod G^{\ast}$$
	differs from the sequence $n\mapsto g_{0}(n)\tilde{g}_{h}(n)$ by a sequence that is polynomial in the shifted pre-filtration $(\tilde{G}_{[d,r+1]})_{[d,r]\in\DR}$ (note that $\tilde{\phi}(e_{m,k})=id_{\tilde{G}}$ if $D_{\vert m\vert}^{\cor}+1\leq k\leq D_{\vert m\vert}^{\cor}+D_{\vert m\vert}^{\ind}$). Therefore, we have that (\ref{3:12.7}) is the product of a constant in $G'$ with a polynomial sequence with respect to the filtration $G''_{\DR}$, and thus is a polynomial sequence with respect to the pre-filtration $G'_{\DR}$ by Lemma \ref{3:pre3g}.

	On the other hand,
	it is obvious that $F'_{h}$ is invariant with respect to the action of the central group
	$$G'_{[s,r_{\ast}]}=\{(\phi(g),id,\tilde{\phi}(g))\colon g\in G^{\vec{D},d}_{[s,r_{\ast}]}\}.$$	
	By the above construction and by Lemma \ref{3:pre2g}, there exists $Q\in\N_{+}$ with $Q\leq O_{C,d,D,\e,r}(1)$ such that 
	we may  quotient $G'$ and $G''$ by $G'_{[s,r_{\ast}]}$  to get a representation of (\ref{3:12.6}) as a $Qp^{Q}$-rational nilsequence of degree-rank $<[s,r_{\ast}]$, whose complexity and dimension are clearly $O_{C,d,D,\e,r}(1)$. 
	So we have that  $\chi_{h}\otimes\overline{\tilde{\chi}}_{h}\in\Nil^{<[s,r_{\ast}];O_{C,d,D,\e,r}(1),O_{C,d,D,\e,r}(1)}_{\approx Qp^{Q}}(V_{p}(\tilde{M})^{h})$.

	Finally, since 
	$$\chi_{h}\otimes(\overline{\tilde{\chi}}_{h}\otimes\tilde{\chi}_{h})=(\chi_{h}\otimes\overline{\tilde{\chi}}_{h})\otimes\tilde{\chi}_{h},$$  
	using the fact that 1 is an $O_{C,d,D,\e,r}(1)$-complexity linear combination of the components of $\overline{\tilde{\chi}}_{h}\otimes\tilde{\chi}_{h}$, we deduce that $\chi_{h}$ is an $O_{C,d,D,\e,r}(1)$-complexity linear combination of the components $\tilde{\chi}_{h}\otimes \psi'_{h}$ for some $\psi'_{h}\in\Nil^{<[s,r_{\ast}];O_{C,d,D,\e,r}(1),O_{C,d,D,\e,r}(1)}_{\approx Qp^{Q}}(V_{p}(\tilde{M})^{h})$.
\end{proof}

\textbf{Step 4.} We now realize $(h,n)\mapsto\tilde{\chi}_{h}(n)$ as a nilcharacter of multi-degree $(1,s)$. The construction is similar to Section 12 of \cite{GTZ12}. 
Since it is convenience for us to treat each term $\{\alpha_{m,k,i}\cdot h\}\beta_{m,k,i}$ in the summation of $\xi_{h,m,k}$ separately,
a difference between of our construction and the one in \cite{GTZ12} is that instead of taking the skew product of $G$ with $\tilde{G}_{\ped}$ (the ``pedal" component of $\tilde{G}$), we take the skew product of $\tilde{G}$ with $\tilde{G}_{\ped}^{L}$ for some large power $L$.
To be more precise, let $\tilde{G}_{\ped}$ be the subgroup of $\tilde{G}$ generated by any iterated commutators of
$\tilde{e}_{m_{1},k_{1}},\dots,\tilde{e}_{m_{\ell},k_{\ell}}$ for all $\ell\geq 1$,  $1\leq \vert m_{i}\vert\leq s, 1\leq i\leq \ell$ such that there exists $1\leq i_{0}\leq \ell$ with $k_{i_{0}}\geq D^{\cor}_{\vert m_{i_{0}}\vert}+D^{\ind}_{\vert m_{i_{0}}\vert}+1$ and $k_{i}\leq D^{\cor}_{\vert m_{i}\vert}$ for all $i\neq i_{0}$. 
Let $I$ denote the collection of all $(m,k,i)\in\mathbb{N}^{d}\times\mathbb{N}\times\mathbb{N}$ with $1\leq \vert m\vert\leq s$, $D_{\vert m\vert}^{\cor}+D_{\vert m\vert}^{\ind}+1\leq k\leq D_{\vert m\vert}$ and $1\leq i\leq K$.
Using the identities
$$z^{-1}[x,y]z=[z^{-1}xz,z^{-1}yz] \text{ and } z^{-1}xz=x[x,z],$$
it is not hard to see that $\tilde{G}_{\ped}$ is a rational abelian normal subgroup of $\tilde{G}$.
Therefore, $\tilde{G}$ acts on $\tilde{G}_{\ped}$ by conjugation, leading to the semidirect product on $\tilde{G}\ltimes \tilde{G}_{\ped}^{\vert I\vert}$ given by
$$(g,(g_{m,k,i})_{(m,k,i)\in I})\ast (g',(g'_{m,k,i})_{(m,k,i)\in I}):=(gg',(g_{m,k,i}^{g'}g'_{m,k,i})_{(m,k,i)\in I}),$$
where $a^{b}:=b^{-1}ab$.

 Let $R$ be the commutative ring of tuples $t=(t_{m,k,i})_{(m,k,i)\in I}$ with $t_{m,k,i}\in \R$, which we endow with the pointwise product and addition. 
 Define an action $\rho$ of $R$ (viewed now as an additive group) on $\tilde{G}\ltimes \tilde{G}_{\ped}^{\vert I\vert}$ by 
 $$\rho(t)(g,(g_{m,k,i})_{(m,k,i)\in I}):=\Bigl(g\prod_{(m,k,i)\in I}g_{m,k,i}^{t_{m,k,i}},(g_{m,k,i})_{(m,k,i)\in I}\Bigr).$$
 It is not hard to see that this is a well-defined action.  
 We can then define the semi-direct product $G':=R\ltimes_{\rho}(\tilde{G}\ltimes \tilde{G}_{\ped}^{\vert I\vert})$ by setting
 \begin{equation}\nonumber
 \begin{split}
&\quad (t,(g,(g_{m,k,i})_{(m,k,i)\in I}))\ast'(t',(g',(g'_{m,k,i})_{(m,k,i)\in I}))
\\&=(t+t',(\rho(t')(g,(g_{m,k,i})_{(m,k,i)\in I}))\ast (g',(g'_{m,k,i})_{(m,k,i)\in I})).
 \end{split}
 \end{equation}
 This is a Lie group. We can endow $G'$ with an $\N^{2}$-filtration $(G'_{(s_{1},s_{2})})_{(s_{1},s_{2})\in\N^{2}}$ as follows:
 \begin{itemize}
 	\item If $s_{1}>1$, then $G'_{(s_{1},s_{2})}=\{id\}$.
 	\item If $s_{1}=1$ and $s_{2}>0$, then $G'_{(1,s_{2})}$ consists of the elements $(0,(g,id^{\vert I\vert}))$ with $g\in \tilde{G}_{s_{2}}\cap \tilde{G}_{\ped}$.
 	\item If $s_{1}=1$ and $s_{2}=0$, then $G'_{(1,0)}$ consists of the elements $(t,(g,id^{\vert I\vert}))$ with $t\in R$ and $g\in \tilde{G}_{\ped}$.
 	\item  If $s_{1}=0$ and $s_{2}>0$, then $G'_{(0,s_{2})}$ consists of the elements $(0,(g,(g_{m,k,i})_{(m,k,i)\in I}))$ with $g\in \tilde{G}_{s_{2}}$ and $g_{m,k,i}\in  \tilde{G}_{s_{2}}\cap \tilde{G}_{\ped}$.
 	\item If $s_{1}=s_{2}=0$, then $G'_{(0,0)}=G'$.
 \end{itemize}	
 On can verify that this is indeed an $\N^{2}$-filtration of degree $\leq (1,s)$.   Let $\Gamma'$ be the subgroup of $\tilde{G}$ consisting of tuples $(t,(g,(g_{m,k,i})_{(m,k,i)\in I}))$ with $g\in \tilde{\Gamma}, g_{m,k}\in \tilde{G}_{\ped}\cap \tilde{\Gamma}$ and with all the coefficients of $t$ being integers. It is not hard to see that $\Gamma'$ is a cocompact subgroup of $G'$, and that the above $\N^{2}$-filtration of $G'$ is rational with respect to $\Gamma'$. Therefore, $G'/\Gamma'$ is a nilmanifold, and is clearly of complexity $O_{C,d,D,\e,r}(1)$.
 
 For $(m,k,i)\in I$, let $g_{m,k,i}\in\poly(\Z^{d}\to \tilde{G}_{\N})$ be given by
 $$g_{m,k,i}(n):=\tilde{\phi}(e_{m,k})^{\beta_{m,k,i}\binom{n}{m}}.$$
 Then
  \begin{equation}\label{3:121e2}
  \tilde{g}_{h}(n)=\prod_{(m,k,i)\in I}g_{m,k,i}(n)^{\{\alpha_{m,k,i}\cdot h\}}.
  \end{equation}
 
 Consider 
 the map $f\colon (\Z^{d})^{2}\to G'$ given by
 $$f(h,n):=(\bold{0},(g_{0}(n),(g_{m,k,i}(n))_{(m,k,i)\in I}))\ast'(\alpha h,(id,id^{\vert I\vert})),$$
 where
 $$\alpha h:=(\alpha_{m,k,i}\cdot h)_{(m,k,i)\in I}.$$
 It is not hard to check that $f\in\poly((\Z^{d})^{2}\to (G')_{\N^{2}})$  
and that
 \begin{equation}\label{3:121e3}
 f'(h,n)\Gamma'=\Bigl(\{\alpha h\},\Bigl(g_{0}(n)\prod_{(m,k,i)\in I}g_{m,k,i}(n)^{\{\alpha_{m,k,i} h\}},(g_{m,k,i}(n))_{(m,k,i)\in I}\Bigr)\Bigr)\Gamma'.
 \end{equation}

 Recall that for all $h\in H'$, each component $\{\alpha_{m,k,i}h\}$ of $\{\alpha h\}$ lies in an interval $I_{m,k,i}$ of length at most $1/10$. Let $2I_{m,k,i}$ be the interval of twice the length and with the same center as $I_{m,k,i}$, and let $\psi_{m,k,i}\colon\R\to\R$ be a smooth cutoff function supported on $2I_{m,k,i}$ and taking value 1 on $I_{m,k,i}$. We then define a function $F'\colon G'/\Gamma'\to \C^{O_{C,d,D,\e,r}(1)}$ by setting
 \begin{equation}\label{3:121e4}
 F'((t,(g,(g_{m,k,i})_{(m,k,i)\in I}))\Gamma'):=\Bigl(\prod_{(m,k,i)\in I}\psi_{m,k,i}(t_{m,k,i})\Bigr)\tilde{F}(g\Gamma)
 \end{equation}
 for all $(g,(g_{m,k,i})_{(m,k,i)\in I})\in G\ltimes\tilde{G}_{\ped}^{\vert I\vert}$ and
 $t=(t_{m,k,i})_{(m,k,i)\in I}$ with $t_{m,k,i}\in 2I_{m,k,i}$, and set $F'$ to be zero whenever such an representation of $(t,(g,(g_{m,k,i})_{(m,k,i)\in I}))$ does not exist. One can easily verify that $F'$ is a well defined  Lipschitz function with Lipschitz norm bounded by $O_{C,d,D,\e,r}(1)$. 
Since
 $\tilde{F}$ has vertical frequency $\tilde{\eta}$, $F'$ has vertical frequency $\eta'\colon G'_{(1,s)}\to\R$ defined by 
 $$\eta'(0,(g,id^{\vert I\vert})):=\tilde{\eta}(g)$$
 for all $g\in \tilde{G}_{s}$, which is of complexity  $O_{C,d,D,\e,r}(1)$. Combining  (\ref{3:121e1}),  (\ref{3:121e2}),  (\ref{3:121e3}) and  (\ref{3:121e4}), we have that 
 \begin{equation}\label{3:121e5}
 \tilde{\chi}_{h}(n)=F'(f'(h,n)\Gamma')
 \end{equation}
 for all $h\in H'$ and $n\in\Z^{d}$.

\textbf{Step 5.}
Unfortunately $f'$ is not necessarily periodic and $F'$ does not take values in $\mathbb{S}^{O_{C,d,D,\e,r}(1)}$.
To complete  the proof of Proposition \ref{3:srthm}, we need to substitute expression (\ref{3:121e5}) with a periodic nilcharacter.
We use the approximate results obtained in Appendix \ref{3:s:AppB}.
By
 (\ref{3:longlongago}), (\ref{3:121e5}), Lemma \ref{3:L12.1}, and the Pigeonhole Principle,
  for all $h\in H'$,   
 there exists a  scalar-valued nilsequence  $\alpha_{h}\in\Nil_{\approx Qp^{Q}}^{<[s,r_{\ast}];O_{C,d,D,\e,r}(1),1}(\Z^{d})$ for some $Q\in\N_{+}$ with $Q\leq O_{C,d,D,\e,r}(1)$ such that
\begin{equation}\nonumber%\label{3:longlongago2}
\E_{h\in H'}\Bigl\vert\E_{n\in V_{p}(\tilde{M})^{h}}f(n+h)\overline{f}(n)\chi_{0}(h,n)\otimes F'(f'(h,n)\Gamma')\alpha_{h}(n)\psi_{h}(n)\Bigr\vert\gg_{C,d,D,\e,r} 1.
\end{equation} 
Since  $F'(f'(\cdot)\Gamma')\in\Nil^{(1,s);O_{C,d,D,\e,r}(1),O_{C,d,D,\e,r}(1)}((\Z^{d})^{2})$,
by Lemmas \ref{3:r2p}, \ref{3:LE.4}, \ref{3:LE.5}  and the Pigeonhole Principle, enlarging $Q$ if necessary, we deduce that there exist some $\chi'\in\Xi^{(1,s);O_{C,d,D,\e,r}(1),O_{C,d,D,\e,r}(1)}_{p}((\Z^{d})^{2})$, $\alpha'_{h}\in\Nil_{\approx Qp^{Q}}^{<[s,r_{\ast}];O_{C,d,D,\e,r}(1),1}(\Z^{d})$  and $\psi'_{h}\in\Nil_{p}^{s-1;O_{C,d,D,\e,r}(1),1}(\Z^{d})$   for each $h\in H'$ such that 
\begin{equation}\nonumber
\E_{h\in H'}\Bigl\vert\E_{n\in V_{p}(\tilde{M})^{h}}f(n+h)\overline{f}(n)\chi_{0}(h,n) \otimes\chi'(h,n)\alpha'_{h}(n)\psi'_{h}(n)\Bigr\vert\gg_{C,d,D,\e,r} 1.
\end{equation}
By the Pigeonhole Principle, there exists a subset $H''$ of $H'$ with $\vert H''\vert\gg_{C,d,D,\e,r}p^{rd}$ such that 
\begin{equation}\nonumber
\Bigl\vert\E_{n\in V_{p}(\tilde{M})^{h}}f(n+h)\overline{f}(n)\chi_{0}(h,n) \otimes\chi'(h,n)\alpha'_{h}(n)\psi'_{h}(n)\Bigr\vert\gg_{C,d,D,\e,r} 1
\end{equation}
for all $h\in H''$.  
By Corollary \ref{3:LE.6} and the Pigeonhole Principle, there exists $\chi'_{h}$ in the set $\Xi^{<[s,r_{\ast}];O_{C,d,D,\e,r}(1),O_{C,d,D,\e,r}(1)}_{\approx Qp^{Q}}(\Z^{d})$  for all $h\in H''$ such that 
\begin{equation}\nonumber
\Bigl\vert\E_{n\in V_{p}(\tilde{M})^{h}}f(n+h)\overline{f}(n)\chi_{0}(h,n)\otimes \chi'(h,n)\otimes\chi'_{h}(n)\psi'_{h}(n)\Bigr\vert\gg_{C,d,D,\e,r} 1
\end{equation}
for all $h\in H''$.

Finally, we need to upgrade $\chi'_{h}$ to a $p^{Q}$-periodic nilcharacter. Indeed, by Lemma \ref{3:r2p}, 
enlarging $Q$ if necessary, we may assume without loss of generality that 
  $\chi'_{h}\in\Xi^{<[s,r_{\ast}];O_{C,d,D,\e,r}(1),O_{C,d,D,\e,r}(1)}_{Qp^{Q}}(\Z^{d})$.
By the Pigeonhole Principle and Lemma \ref{3:countingh}, it is not hard to see that there exist a subset $H'''$ of $H''$ with  $\vert H'''\vert\gg_{C,d,D,\e,r}p^{rd}$ and some $r\in[Q]^{d}$ such that 
\begin{equation}\label{3:ppddrre}
\Bigl\vert\E_{n\in V_{p}(\tilde{M})^{h}}f((Qn+r)+h)\overline{f}(Qn+r)\chi_{0}(h,Qn+r)\otimes \chi'(h,Qn+r)\otimes\chi'_{h}(Qn+r)\psi'_{h}(Qn+r)\Bigr\vert\gg_{C,d,D,\e,r} 1
\end{equation}
for all $h\in H'''$. 
Let $Q^{\ast}$ be any integer with $Q^{\ast}Q\equiv 1\mod p^{r}\Z$. Then $Q(Q^{\ast}(n-r))+r\equiv n \mod p^{Q}\Z^{d}$ for all $n\in\Z^{d}$. Since $f, \chi_{0},\chi', \chi'_{h}(Q\cdot+r)$ and $\psi'_{h}$ are $p^{Q}$-periodic (we may pick $Q\geq r$),   and since the map $n\mapsto Q^{\ast}(n-r) \mod p^{Q}\Z^{d}$ is a bijection from $[p^{Q}]^{d}$   to itself, it follows from (\ref{3:ppddrre}) that 
\begin{equation}\nonumber%\label{3:ppddrre}
\Bigl\vert\E_{n\in V_{p}(\tilde{M})^{h}}f(n+h)\overline{f}(n)\chi_{0}(h,n)\otimes \chi'(h,n)\otimes\chi'_{h}(Q(Q^{\ast}(n-r))+r)\psi'_{h}(n)\Bigr\vert\gg_{C,d,D,\e,r} 1
\end{equation}
for all $h\in H'''$, where clearly $\chi'_{h}(Q(Q^{\ast}(\cdot-r))+r)\in \Xi^{<[s,r_{\ast}];O_{C,d,D,\e,r}(1),O_{C,d,D,\e,r}(1)}_{p^{Q}}(\Z^{d})$.
This completes the proof of Proposition \ref{3:srthm} by setting $\chi:=\chi_{0}\otimes\chi''$.

\section{The symmetric argument for $s=1$}\label{3:s:b7}

In Sections  \ref{3:s:b7} and \ref{3:s:b72}, we complete the proof of $\SGI(s+1)$.
It is convenient to identify two nilcharacters when their difference is a nilsequence of lower degree. Therefore, we introduce the following notion:

\begin{defn}[An equivalence relation for nilcharacters]\label{3:deneq}
	Let $k\in\N_{+}$, $C>0$, $p$ be a prime, $\Omega$ be a subset of $\F_{p}^{k}$, $I$ be the degree, multi-degree,  or  degree-rank ordering with $\dim(I)\vert k$ and let $s\in I$. 	For $\chi,\chi'\in\Xi^{s}_{p}(\Omega)$, we write 
	$\chi\sim_{C}\chi' \mod\Xi^{s}_{p}(\Omega)$
	if $\chi\otimes \overline{\chi}\in\Nil^{\prec  s;C}(\Omega)$.\footnote{Note that if $\chi$ and $\chi'$ to are  $p$-periodic $s$-step nilcharacters, then so is $\chi\otimes \overline{\chi}$. However, if $\chi\otimes \overline{\chi}$ coincides with a $\prec s$-step nilsequence, it is unclear whether  $\chi\otimes \overline{\chi}$ is $p$-periodic as a $\prec s$-step nilsequence. In our definition, if $\chi\sim_{C}\chi' \mod\Xi^{s}_{p}(\Omega)$, then we do not require   $\chi\otimes \overline{\chi}$ to be $p$-periodic as a $\prec s$-step nilsequence.} We write $\chi\sim\chi' \mod\Xi^{s}_{p}(\Omega)$ if $\chi\sim_{C}\chi' \mod\Xi^{s}_{p}(\Omega)$ for some $C>0$. 
\end{defn}

We summarized some properties for the equivalence relation $\sim$ in Appendix \ref{3:s:AppC} for later uses.

	Recall that we assume that
	$\SGI(s)$ holds for some $s\geq 1$ and we wish to prove $\SGI(s+1)$. 
By Lemma \ref{3:changeh} and a change of variables,
	it is not hard to check that we only need to prove $\SGI(s+1)$ for the case when $M$ is a pure quadratic form. So we assume that $M$ is pure throughout Sections  \ref{3:s:b7} and \ref{3:s:b72}.
	By Example 6.11 of \cite{GTZ12}, every connected, simply-connected nilpotent Lie group $G$ of degree $s$ induces a filtration of degree-rank $[s,s]$. So $\SGI(s)$ and (\ref{3:ini10.5}) implies the initial hypothesis of Proposition \ref{3:srthm} for the case $r_{\ast}=s$ holds (with $r=1$). Since $\Xi_{\DR}^{<[s,1]}(\Z^{d})\subseteq \Nil^{s-1}(\Z^{d})$ and $d\geq N(s)$,
	we may then use  Proposition \ref{3:srthm} inductively combined with Corollary \ref{3:pppap} and the Pigeonhole Principle to conclude that there exist $H\subseteq\V$ with $\vert H\vert\gg_{d,\e}p^{d}$, some $\chi\in\Xi^{(1,s);O_{d,\e}(1),O_{d,\e}(1)}_{p}((\V)^{2})$, and some $\psi_{h}\in \Nil_{p}^{s-1;O_{d,\e}(1),1}(\V)$ for all $h\in\V$ such that   
	\begin{equation}\label{3:midpoint0}
	\begin{split}
	\vert\E_{n\in V(M)^{h}}\Delta_{h}f(n)\chi(h,n)\psi_{h}(n)\vert\gg_{d,\e} 1
	\end{split}
	\end{equation} 
	for all $h\in H$, where we translate everything back to the $\F_{p}$-setting since all the relevant objects are $p$-periodic. Therefore,
 	\begin{equation}\label{3:midpoint}
 	\begin{split}
 	\E_{h\in\V}\vert\E_{n\in V(M)^{h}}\Delta_{h}f(n)\chi(h,n)\psi_{h}(n)\vert\gg_{d,\e} 1.
 	\end{split}
 	\end{equation} 
	By Lemma \ref{3:iiddpp}, we may assume without loss of generality that all the elements in $H$ are $M$-non-isotropic.
	
	By Lemma \ref{3:LE8} (vii) and Theorem \ref{3:E.10}, there exists $\tilde{\chi}\in\Xi_{p}^{(1,\dots,1);O_{d,\e}(1),O_{d,\e}(1)}((\V)^{s+1})$ (with 1 repeated $s+1$ times) which is symmetric in the last $s$ variables   such that
	\begin{equation}\label{3:ini2}
\begin{split}
\chi(h,n)\sim_{O_{d,\e}(1)}\tilde{\chi}(h,n,\dots,n)^{\otimes (s+1)}\mod \Xi_{p}^{s+1}((\V)^{2}).
\end{split}
\end{equation}

 Our goal is to first use (\ref{3:midpoint}) to deduce some symmetric property for $\tilde{\chi}$. Then we use it to show that in (\ref{3:midpoint}), up to a nilsequence of lower degree, the term $\chi(h,n)$ can be ``replaced" by $\tilde{\chi}(n+h,\dots,n+h)\otimes\overline{\chi}(n,\dots,n)$, which can be then absorbed by the term $\Delta_{h}f$. Finally, we use the Cauchy-Schwartz inequality to complete the proof of  $\SGI(s+1)$. 
	
The proof for the case $s=1$ and the case $s\geq 2$ are slightly different. We prove $\SGI(2)$ in Section \ref{3:s:b7} and $\SGI(s+1), s\geq 2$ in Section \ref{3:s:b72}.

\subsection{A symmetric property for the case $s=1$}\label{3:s:smp1}
We start with the case $s=1$. We  first use (\ref{3:midpoint0}) to deduce some symmetric property on  $\tilde{\chi}$.
Since $s=1$, we may set $\psi_{h}(n)\equiv 1$ in (\ref{3:midpoint}).
Denote
$$\Lambda_{1}:=\{(h,n,m)\in(\V)^{3}\colon n,m,n+h,m+h\in V(M)\}.$$
Then $\Lambda_{1}$ 
is a consistent $M$-set of total co-dimension 4.
Since $d\geq 9$, by the (vector-valued) Cauchy-Schwartz inequality and Theorem \ref{3:ct},  we deduce from (\ref{3:midpoint}) that 
\begin{equation}\label{3:ff1}
\begin{split}
&\quad 1\ll_{d,\e}  \E_{h\in\V}\vert\E_{n\in V(M)^{h}}\Delta_{h}f(n)\chi(h,n)\psi_{h}(n)\vert^{2}
\\&=\vert\E_{h\in \V}\E_{n,m\in V(M)^{h}}f(n+h)\overline{f}(m+h)\overline{f}(n)f(m)\chi(h,n)\otimes\overline{\chi}(h,m)\vert
\\&=\vert\E_{(h,m,n)\in \Lambda_{1}}f(n+h)\overline{f}(m+h)\overline{f}(n)f(m)\chi(h,n)\otimes\overline{\chi}(h,m)\vert+O(p^{-1/2}).
\end{split}
\end{equation}
Denote
$$\Lambda_{2}:=\{(z,n,m)\in(\V)^{3}\colon n,m,z-n,z-m\in V(M)\},$$
which is a consistent $M$-set of total co-dimension 4. Let $\Lambda_{2}'$ denote the set of  $(z,m)\in(\V)^{2}$ with $m,z-m\in V(M)$.
For $z\in\V$, let $\Lambda_{2}(z)$ denote the set of $n\in\V$ with $n,z-n\in V(M)$.
Replacing $h$ with $z=n+m+h$ and using Theorem \ref{3:ct}, we have that the right hand side of (\ref{3:ff1}) is bounded by
\begin{equation}\label{3:ff2}
\begin{split}
&\quad\vert\E_{(z,n,m)\in \Lambda_{2}}f(z-m)\overline{f}(z-n)\overline{f}(n)f(m)\chi(z-n-m,n)\otimes\overline{\chi}(z-n-m,m)\vert+O(p^{-1/2})
\\&=\vert\E_{(z,m)\in\Lambda'_{2}}f(z-m)f(m)
\E_{n\in\Lambda_{2}(z)}\overline{f}(z-n)\overline{f}(n)\chi(z-n-m,n)\otimes\overline{\chi}(z-n-m,m)\vert+O(p^{-1/2})
\\&\leq\E_{(z,m)\in\Lambda'_{2}}
 \vert\E_{n\in\Lambda_{2}(z)}\overline{f}(z-n)\overline{f}(n)\chi(z-n-m,n)\otimes\overline{\chi}(z-n-m,m)\vert+O(p^{-1/2}).
\end{split}
\end{equation}
Denote 
$$\Lambda_{3}:=\{(z,m,n_{1},n_{2})\in(\V)^{4}\colon m,n_{1},n_{2},z-m,z-n_{1},z-n_{2}\in V(M)\},$$
which is a consistent $M$-set of total co-dimension 6.
Using a similar method, the square of the right hand side of (\ref{3:ff2}) is bounded by $O_{d,\e}(1)$ times
\begin{equation}\label{3:ff3}
\begin{split}
&\quad\E_{(z,m)\in \Lambda_{2}'}\vert\E_{n\in \Lambda_{2}(z)}\overline{f}(z-n)\overline{f}(n)\chi(z-n-m,n)\otimes\overline{\chi}(z-n-m,m)\vert^{2}+O(p^{-1/2})
\\&=\Bigl\vert\E_{(z,m)\in \Lambda_{2}'}\E_{n_{1},n_{2}\in \Lambda_{2}(z)}
\\&\qquad\prod_{i=1}^{2}\mathcal{C}^{i+1}(f(z-n_{i})f(n_{i}))\bigotimes_{i=1}^{2}\mathcal{C}^{i}(\chi(z-n_{i}-m,n_{i})\otimes\overline{\chi}(z-n_{i}-m,m))\Bigr\vert+O(p^{-1/2})
\\&=\Bigl\vert\E_{(z,m,n_{1},n_{2})\in \Lambda_{3}}
 \\&\qquad\prod_{i=1}^{2}\mathcal{C}^{i+1}(f(z-n_{i})f(n_{i}))\bigotimes_{i=1}^{2}\mathcal{C}^{i}(\chi(z-n_{i}-m,n_{i})\otimes\overline{\chi}(z-n_{i}-m,m))\Bigr\vert+O(p^{-1/2})
 \\&\leq\E_{(z,n_{1},n_{2})\in \Lambda_{2}}
 \Bigl\vert\E_{m\in \Lambda_{2}(z)}\bigotimes_{i=1}^{2}\mathcal{C}^{i}(\chi(z-n_{i}-m,n_{i})\otimes\overline{\chi}(z-n_{i}-m,m))\Bigr\vert+O(p^{-1/2}).
\end{split}
\end{equation}

Denote
$$\Lambda_{4}:=\{(z,n_{1},n_{2},m_{1},m_{2})\in(\V)^{5}\colon n_{1},n_{2},m_{1},m_{2},z-n_{1},z-n_{2},z-m_{1},z-m_{2}\in V(M)\}.$$
which is a consistent $M$-set of total co-dimension 8.
Similarly, since $d\geq 17$, the square of the right hand side of (\ref{3:ff3}) is bounded by $O_{d,\e}(1)$ times
\begin{equation}\nonumber
\begin{split}
&\quad\E_{(z,n_{1},n_{2})\in \Lambda_{2}}
\Bigl\vert\E_{m\in \Lambda_{2}(z)}\bigotimes_{i=1}^{2}\mathcal{C}^{i}(\chi(z-n_{i}-m,n_{i})\otimes\overline{\chi}(z-n_{i}-m,m))\Bigr\vert^{2}+O(p^{-1/2})
\\&=\Bigl\vert\E_{(z,n_{1},n_{2})\in \Lambda_{2}}\E_{m_{1},m_{2}\in\Lambda_{2}(z)}
\bigotimes_{i,j\in\{1,2\}}\mathcal{C}^{i+j}(\chi(z-n_{i}-m_{j},n_{i})\otimes\overline{\chi}(z-n_{i}-m_{j},m_{j}))\Bigr\vert+O(p^{-1/2})
\\&=\Bigl\vert\E_{(z,n_{1},n_{2},m_{1},m_{2})\in \Lambda_{4}}
\bigotimes_{i,j\in\{1,2\}}\mathcal{C}^{i+j}(\chi(z-n_{i}-m_{j},n_{i})\otimes\overline{\chi}(z-n_{i}-m_{j},m_{j}))\Bigr\vert+O(p^{-1/2}).
\end{split}
\end{equation}
In conclusion, we have that
\begin{equation}\label{3:sse1}
\begin{split}
\Bigl\vert\E_{(z,n_{1},n_{2},m_{1},m_{2})\in \Lambda_{4}}
\bigotimes_{i,j\in\{1,2\}}\mathcal{C}^{i+j}(\chi(z-n_{i}-m_{j},n_{i})\otimes\overline{\chi}(z-n_{i}-m_{j},m_{j}))\Bigr\vert\gg_{d,\e} 1
\end{split}
\end{equation}
if $p\gg_{d,\e} 1$.  
By Lemma \ref{3:LE80}, the map
$$(z,n_{1},n_{2},m_{1},m_{2})\mapsto\bigotimes_{i,j\in\{1,2\}}\mathcal{C}^{i+j}(\chi(z-n_{i}-m_{j},n_{i})\otimes\overline{\chi}(z-n_{i}-m_{j},m_{j}))$$
belongs to $\Xi^{2;O_{d,\e}(1),O_{d,\e}(1)}_{p}((\V)^{5})$. On the other hand,
 Note that $\Lambda_{4}$ is the set $\Omega_{1}$ in Example %\ref{1:mainex}
 B.4 of \cite{SunA}, which is a nice and consistent $M$-set of total co-dimension 8. 
 Since $d\geq 33$,
it follows from Lemma  \ref{3:LE.11} and (\ref{3:sse1}) that
\begin{equation}\label{3:cnm1}
\begin{split}
\bigotimes_{i,j\in\{1,2\}}\mathcal{C}^{i+j}(\chi(z-n_{i}-m_{j},n_{i})\otimes\overline{\chi}(z-n_{i}-m_{j},m_{j}))\sim_{O_{d,\e}(1)}1 \mod \Xi^{2}_{p}(\Lambda_{4}).
\end{split}
\end{equation}
By Lemmas \ref{3:eqqeqq}, \ref{3:LE8} (iv) and \ref{3:L13.2},
it is not hard to simplify (\ref{3:cnm1}) as  
\begin{equation}\label{3:cnm2}
\begin{split}
\chi(n_{1}-n_{2},m_{1}-m_{2})\sim_{O_{d,\e}(1)}\chi(m_{1}-m_{2},n_{1}-n_{2})\mod\Xi^{2}_{p}(\Lambda_{4}).
\end{split}
\end{equation}

In order to better apply (\ref{3:cnm2}) to study (\ref{3:ff1}), we need to do a change of variable with $h_{1}=m_{1}-m_{2}$ and $h_{2}=n_{1}-n_{2}$, and add an extra variable $x$.   let $L\colon (\V)^{6}\to(\V)^{5}$ be the linear transformation given by
$$L(x,y_{1},y_{2},y_{3},h_{1},h_{2}):=(y_{1},y_{2}+h_{1},y_{2},y_{3}+h_{2},y_{3}).$$
Note that $L^{-1}(\Lambda_{4})$ consists of $(x,y_{1},y_{2},y_{3},h_{1},h_{2})\in (\V)^{6}$ with 
$$y_{2}+h_{1},y_{2},y_{3}+h_{2},y_{3},y_{1}-y_{2}-h_{1},y_{1}-y_{2},y_{1}-y_{3}-h_{2},y_{1}-y_{3}\in V(M).$$
By (\ref{3:cnm2}) and Lemma \ref{3:LE8} (vi),   we have that 
\begin{equation}\label{3:cnm33}
\begin{split}
\chi(h_{1},h_{2})\sim_{O_{d,\e}(1)}\chi(h_{2},h_{1})\mod\Xi^{2}_{p}(L^{-1}(\Lambda_{4}))
\end{split}
\end{equation}
(where the sequences are viewed as in the variables $(x,y_{1},y_{2},y_{3},h_{1},h_{2})$ in $L^{-1}(\Lambda_{4})$).
By  (\ref{3:ini2}), (\ref{3:cnm33}), Lemmas \ref{3:eqqeqq} and  \ref{3:LE.13}, we have that
\begin{equation}\label{3:cnm33k}
\begin{split}
\tilde{\chi}(h_{1},h_{2})\sim_{O_{d,\e}(1)}\tilde{\chi}(h_{2},h_{1})\mod\Xi^{2}_{p}(L^{-1}(\Lambda_{4})).
\end{split}
\end{equation}
Let $\Lambda_{5}$ be the set of $(x,y_{1},y_{2},y_{3},h_{1},h_{2})\in L^{-1}(\Lambda_{4})$  
such that
$(x,h_{1},h_{2})\in \Gow_{2}(V(M)).$ 
By  (\ref{3:cnm33k}) and Lemma \ref{3:LE8} (iv), we have that  
\begin{equation}\label{3:cnm3}
\begin{split}
\tilde{\chi}(h_{1},h_{2})\sim_{O_{d,\e}(1)}\tilde{\chi}(h_{2},h_{1})\mod\Xi^{2}_{p}(\Lambda_{5}).
\end{split}
\end{equation} 
By the description of $L^{-1}(\Lambda_{4})$, it is not hard to compute that $\Lambda_{5}$ is a consistent $M$-set of total co-dimension 12.

\begin{rem}\label{3:sbr1}
In  Section 13 of \cite{GTZ12}, the authors only kept track of the domain of  the parameters $h_{1},h_{2}$, and treated all other parameters  as constants.  In our setting,  if we treated $x,y_{1},y_{2},y_{3}$ as constants, then there will be a significant loss of information in expressions such as (\ref{3:cnm3}).
This is why in our method  we must  keep track of the domain of all the parameters $x,y_{1},y_{2},y_{3},h_{1},h_{2}$ (which makes the computation more complicated). 
The same remark applies to the proof for the case $s\geq 2$ in Section \ref{3:s:ss03}.
\end{rem}

\subsection{Completion of the proof of $\SGI(2)$}\label{3:s:ss02}

In Section 13 of \cite{GTZ12}, the authors deduced a conclusion similar to (\ref{3:cnm3}), and then used it to conclude that $\chi(h,n)$ can be written as $\tilde{\chi}(n+h,\dots,n+h)\otimes\overline{\tilde{\chi}}(n,\dots,n)$ for some nilcharacter $\Theta$ of degree $\leq s+1$. However, this method does not apply to our case. This is because that  $\Lambda_{5}$ is a sparse subset of $(\V)^{6}$, and that equation (\ref{3:cnm3}) does not provide us any information on whether $\tilde{\chi}(h_{1},h_{2})$ is equivalent to $\tilde{\chi}(h_{2},h_{1})$ outside $\Lambda_{5}$. So we need to use a more sophisticated method.

\begin{conv}\label{3:d4n}
	Throughout Sections \ref{3:s:b7} and \ref{3:s:b72} only, $n\in\V$ is treated as a special variable. If $F(n,m_{1},\dots,m_{k})$ is a map with one of the variables being $n$, then we write $$\Delta_{h}F(n,m_{1},\dots,m_{k}):=F(n+h,m_{1},\dots,m_{k})\otimes \overline{F}(n,m_{1},\dots,m_{k})$$   for all $h\in\V$ (i.e. the difference is taken with respect to the special variable $n$). For example, the expression $\Delta_{h}\chi(n+m,m+y,n+y)$ is understood as $\chi(n+h+m,m+y,n+h+y)\otimes \overline{\chi}(n+m,m+y,n+y)$.
\end{conv}

 \textbf{Step 1:  reformulate  (\ref{3:midpoint0}).}  
Under a change of variable, we may rewrite the right hand side of (\ref{3:ff1}) as
\begin{equation}\label{3:ff5}
\begin{split}
 \vert\E_{(n,h_{1},h_{2})\in \Gow_{2}(V(M))}\Delta_{h_{2}}\Delta_{h_{1}}\tilde{f}(n)\otimes \beta(h_{1},h_{2},n) \vert+O(p^{-1/2}),
\end{split}
\end{equation} 
where $\tilde{f}(n):=f(n)\tilde{\chi}(n,n)$ and 
$$\beta(h_{1},h_{2},n):=\Delta_{h_{2}}(\chi(h_{1},n)\otimes\overline{\tilde{\chi}}(n+h_{1},n+h_{1})\otimes \tilde{\chi}(n,n)).$$

We may write $\Lambda_{5}=V(\mathcal{J})$ for some independent $(M,6)$-family $\mathcal{J}\subseteq \F_{p}[n,y_{1},y_{2},y_{3},h_{1},h_{2}]$ of total dimension 12. Let $(\mathcal{J}',\mathcal{J}'')$ be an $\{n,h_{1},h_{2}\}$-decomposition of $\mathcal{J}$.
For convenience denote $\h:=(h_{1},h_{2})$ and $\y:=(y_{1},y_{2},y_{3})$. 
 It is not hard to compute that $(n,\y,\h)\in V(\mathcal{J}')$ if and only if $(n,\h)\in \Gow_{2}(V(M))$. For $(n,\h)\in (\V)^{3}$, let 
  $\Lambda_{5}(n,\h)$ denote the set of $\y\in(\V)^{3}$ such that $(n,\y,\h)\in V(\mathcal{J}'')$. 
Since   $\Lambda_{5}$  
is a consistent $M$-set of total co-dimension 12,
   by Theorem \ref{3:ct}, (\ref{3:ff1}) and (\ref{3:ff5}),
\begin{equation}\label{3:mmdd1}
\begin{split}	 
 &\quad 1\ll_{d,\e}  
 \vert\E_{(n,h_{1},h_{2})\in \Gow_{2}(V(M))}\Delta_{h_{2}}\Delta_{h_{1}}\tilde{f}(n)\otimes \beta(h_{1},h_{2},n) \vert+O(p^{-1/2})
 \\&=
 \vert\E_{(n,\h)\in \Gow_{2}(V(M))}\E_{\y\in \Lambda_{5}(n,\h)} \Delta_{h_{2}}\Delta_{h_{1}}\tilde{f}(n)\otimes\beta(\h,n)\vert+O(p^{-1/2})
\\&=\vert\E_{(n,\y,\h)\in\Lambda_{5}} \Delta_{h_{2}}\Delta_{h_{1}}\tilde{f}(n)\otimes\beta(\h,n)\vert+O(p^{-1/2}),
\end{split}
\end{equation}

\textbf{Step 2:  use (\ref{3:ini2}) to convert $\beta$ into a lower degree nilsequence.} 
Using (\ref{3:ini2}), (\ref{3:cnm3}) and Lemma \ref{3:L13.2},
it is not hard to compute that 
\begin{equation}\nonumber
\begin{split}
 \beta(h_{1},h_{2},n)\sim_{O_{d,\e}(1)}\tilde{\chi}(h_{1},h_{2})\otimes\overline{\tilde{\chi}}(h_{2},h_{1})\sim_{O_{d,\e}(1)} 1 \mod \Xi^{2}_{p}(\Lambda_{5}).
\end{split}
\end{equation} 
Therefore, $\beta(h_{1},h_{2},n)\in\Nil^{1;O_{d,\e}(1)}(\Lambda_{5})$. Then $\beta(h_{1},h_{2},n)\in\Nil^{J;O_{d,\e}(1)}(\Lambda_{5})$, where
$$J:=\{(a,b_{1},b_{2},b_{3},c_{1},c_{2})\in\N^{5}\colon a+b_{1}+b_{2}+b_{3}+c_{1}+c_{2}\leq 1\}.$$
For $1\leq i\leq 2$, let 
$$J_{i}:=\{(a,b_{1},b_{2},b_{3},c_{1},c_{2})\in J\colon c_{i}=0\}.$$
Then $J=J_{1}\cup J_{2}$. Note that any sequence in $\Nil^{J_{i}}(\Lambda_{5})$ is independent of $h_{i}$. By (\ref{3:mmdd1}), Lemma \ref{3:LE.4} and the Pigeonhole Principle, for $1\leq i\leq 2$, there exists a scalar valued sequence $\psi_{i}(n,\y,\h)$ bounded by 1 which is independent of $h_{i}$ such that
\begin{equation}\nonumber
\begin{split}
\Bigl\vert\E_{(n,\y,\h)\in \Lambda_{5}}\Delta_{h_{2}}\Delta_{h_{1}}\tilde{f}(n)\prod_{i=1}^{2}\psi_{i}(n,\y,\h)\Bigr\vert\gg_{d,\e} 1
\end{split}
\end{equation}    
if $p\gg_{d,\e} 1$.  By the definition of $\tilde{f}$ and the Pigeonhole Principle,  there exists
$\phi'\in\Nil^{2;O_{d,\e}(1),1}_{p}(\V)$
  such that writing $f'(n):=f(n)\phi(n)$, we have that 
\begin{equation}\label{3:1pprr0}
\begin{split}
\Bigl\vert\E_{(n,\y,\h)\in \Lambda_{5}}\Delta_{h_{2}}\Delta_{h_{1}}f'(n)\prod_{i=1}^{2}\psi_{i}(n,\y,\h)\Bigr\vert\gg_{d,\e} 1.
\end{split}
\end{equation}

\textbf{Step 3:  use the Cauchy-Schwartz inequality to remove $\psi_{i}$.} 	
For $r\leq 3$, we say that \emph{Property-$r$} holds if there exist a scalar valued function $\psi_{i}(n,\y,\h,h'_{1},$ $\dots,h'_{r-1})$   bounded by 1 and independent of $h_{i}$ for all $1\leq i\leq r$  such that if $p\gg_{d,\e} 1$, then
\begin{equation}\label{3:1pprr}
\begin{split}
\Bigl\vert\E_{(n,\y,\h,h'_{1},\dots,h'_{r-1})\in \Lambda_{5,r}}\Delta_{h_{2}}\Delta_{h_{1}}f'(n)\prod_{i=r}^{2}\psi_{i}(n,\y,\h,h'_{1},\dots,h'_{r-1})\Bigr\vert\gg_{d,\e} 1,
\end{split}
\end{equation} 	
where $\Lambda_{5,r}$ is the set of all $(n,\y,\h,h'_{1},\dots,h'_{r-1})\in (\V)^{r+5}$ such that $$(n-h'_{1}-\dots-h'_{r-1},\y,h'_{1}+\e_{1}h_{1},\dots,h'_{r-1}+\e_{r-1}h_{r-1},h_{r},\dots,h_{2})\in \Lambda_{5}$$
for all $\e_{1},\dots,\e_{r-1}\in \{0,1\}$
 (for $r=3$, the term $\prod_{i=r}^{2}\psi_{i}(n,\y,\h,h'_{1},\dots,h'_{r-1})$ is considered to be constant 1).
Since $M$ is pure, $\Lambda_{5}$ is a pure and consistent $M$-set. 
By
 Propositions \ref{3:yy3} and \ref{3:yy33}, $\Lambda_{5,r}$ is a pure and consistent $M$-set of totally co-dimension at most 36.\footnote{The total co-dimension of $\Lambda_{5,r}$ can be computed explicitly, but here we do not care about its precise value.}

Clearly, Property-1 holds by (\ref{3:1pprr0}). Now suppose that Property-$r$ holds for some $1\leq r\leq 2$. 
For convenience denote $\x:=(n,\y,h_{1},\dots,h_{r-1},h_{r+1},\dots,h_{2},h'_{1},\dots,h'_{r-1})$, and let $(\x;h)$ be the vector obtained by inserting $h$ in the middle of $h_{r-1}$ and $h_{r+1}$.
Let $\tilde{\Lambda}_{5,r}$ denote the set of $(\x,h'_{r},h''_{r})\in (\V)^{r+6}$ such that $(\x;h'_{r}),(\x;h'_{r})\in \Lambda_{5,r}$. Since $\Lambda_{5,r}$ is a pure and consistent $M$-set, by Propositions \ref{3:yy3} and \ref{3:yy33}, $\tilde{\Lambda}_{5,r}$  is a pure and consistent $M$-set of total co-dimension  at most 45.
 	We may then write $\Lambda_{5,r}=V(\mathcal{J})$ for some consistent $(M,r+6)$-family $\mathcal{J}\subseteq \F_{p}[n,y_{1},y_{2},y_{3},h_{1},h_{2},h'_{1},\dots,h'_{r-1}]$ of total dimension at most 45. Let $(\mathcal{J}',\mathcal{J}'')$ be a  $\{n,y_{1},y_{2},y_{3},h_{1},\dots,h_{r-1},h_{r+1},\dots,h_{2},h'_{1},\dots,h'_{r-1}\}$-decomposition of $\mathcal{J}$. Let $\Lambda'_{5,r}$ be the set of $\x\in(\V)^{r+4}$  such that $(\x;h_{r})\in V(\mathcal{J}')$ for all $h_{r}\in \V$. For $\x\in(\V)^{r+4}$, let  $\Lambda_{5,r}(\x)$ denote the set of $h_{r}\in \V$ such that $(\x;h_{r})\in V(\mathcal{J}'')$.
Then
by pulling out the $h_{r}$-independent term $$\Delta_{h_{2}}\dots\Delta_{h_{r+1}}\Delta_{h_{r-1}}\dots\Delta_{h_{1}}\overline{f}'(n)\psi_{r}(n,\y,\h,h'_{1},\dots,h'_{r-1})$$ in (\ref{3:1pprr}) and then apply the Cauchy-Schwartz inequality and Theorem \ref{3:ct}, we have that 
\begin{equation}\nonumber
\begin{split}
&\quad 1\ll_{d,\e}
\E_{\x\in \Lambda'_{5,r}}
\Bigl\vert\E_{h_{r}\in \Lambda_{5,r}(\x)}\Delta_{h_{2}}\dots\Delta_{h_{r+1}}\Delta_{h_{r-1}}\dots\Delta_{h_{1}}f'(n+h_{r})\prod_{i=r+1}^{2}\psi_{i}(\x;h_{r})\Bigr\vert^{2}
\\&=\Bigl\vert\E_{(\x,h'_{r},h''_{r})\in \tilde{\Lambda}_{5,r}}\Delta_{h_{2}}\dots\Delta_{h_{r+1}}\Delta_{h_{r-1}}\dots\Delta_{h_{1}}(f'(n+h''_{r})\overline{f}'(n+h'_{r}))\prod_{i=r+1}^{2}\psi_{i}(\x;h''_{r})\overline{\psi}_{i}(\x;h'_{r})\Bigr\vert+O(p^{-1/2})
\\&=\Bigl\vert\E_{(\x;h'_{r}),(\x;h'_{r}+h_{r})\in \Lambda_{5,r}}\Delta_{h_{2}}\dots\Delta_{h_{r+1}}\Delta_{h_{r}}\Delta_{h_{r-1}}\dots\Delta_{h_{1}}f'(n+h'_{r})\prod_{i=r+1}^{2}\psi_{i}(\x;h'_{r}+h_{r})\overline{\psi}_{i}(\x;h'_{r})\Bigr\vert+O(p^{-1/2}).
\end{split}
\end{equation}

For all $r+1\leq i\leq 2$, since $\psi_{i}$ is independent of $h_{i}$ and bounded by 1,
$\psi_{i}(\x;h'_{r}+h_{r})\overline{\psi}_{i}(\x;h'_{r})$ is also independent of $h_{i}$ and bounded by 1.
On the other hand, note
  that $(\x;h'_{r}),(\x;h'_{r}+h_{r})\in \Lambda_{5,r}$ if and only if
$$(n-h'_{1}-\dots-h'_{r-1},\y,h'_{1}+\e_{1}h_{1},\dots,h'_{r-1}+\e_{r-1}h_{r-1},h'_{r},h_{r+1},\dots,h_{2})\in \Lambda_{5}$$
and 
$$(n-h'_{1}-\dots-h'_{r-1},\y,h'_{1}+\e_{1}h_{1},\dots,h'_{r-1}+\e_{r-1}h_{r-1},h'_{r}+h_{r},h_{r+1},\dots,h_{2})\in \Lambda_{5}$$
for all $\e_{1},\dots,\e_{r-1}\in \{0,1\}$. 	Writing $n'=n+h'_{r}$, this is equivalent of saying that 
$$(n'-h'_{1}-\dots-h'_{r},\y,h'_{1}+\e_{1}h_{1},\dots,h'_{r}+\e_{r}h_{r},h_{r+1},\dots,h_{2})\in \Lambda_{5}$$
for all $\e_{1},\dots,\e_{r-1}\in \{0,1\}$, or equivalently, $(n',h_{1},h_{2},h'_{1},\dots,h'_{r})\in  \Lambda_{5,r+1}$. So we have that Property-$(r+1)$ holds. 

Inductively, we have that Property-3 holds. Since $\Lambda_{5,3}$ is a pure and consistent $M$-set of total co-dimension at most 45, we may write 
 $\Lambda_{5,3}=V(\tilde{\mathcal{J}})$ for some consistent $(M,8)$-family $\tilde{\mathcal{J}}\subseteq \F_{p}[n,y_{1},y_{2},y_{3},h_{1},h_{2},h'_{1},h'_{2}]$ of total dimension at most 45. 
 For convenience denote $\h':=(h'_{1},h'_{2})$.
 Let $(\tilde{\mathcal{J}}',\tilde{\mathcal{J}}'')$ be an $\{n,h_{1},h_{2}\}$-decomposition of $\tilde{\mathcal{J}}$. Let $\Lambda''_{5,3}$ denote the set of $(n,\h)\in(\V)^{3}$ such that $(n,\y,\h,\h')\in V(\tilde{\mathcal{J}}')$ for all $(\y,\h')\in(\V)^{5}$.
It is not hard to compute that $V(\tilde{\mathcal{J}}')=\Gow_{2}(V(M))$. For $(n,\h)\in (\V)^{3}$, let 
  $\Lambda''_{5,3}(n,\h)$ denote the set of $(\y,\h')\in(\V)^{5}$ such that $(n,\y,\h,\h')\in V(\tilde{\mathcal{J}}'')$.
It follows from Property-3 and Theorem \ref{3:ct} that
\begin{equation}\nonumber
\begin{split}
&\quad 1\ll_{d,\e}
\vert\E_{(n,\y,\h,\h')\in \Lambda_{5,3}}\Delta_{h_{2}}\Delta_{h_{1}}f'(n)\vert
=\vert\E_{(n,\h)\in \Gow_{2}(V(M))}\Delta_{h_{2}}\Delta_{h_{1}}f'(n)\E_{(\y,\h')\in\Lambda''_{5,3}(n,\h)}1\vert+O(p^{-1/2})
\\&=\vert\E_{(n,\h)\in \Gow_{2}(V(M))}\Delta_{h_{2}}\Delta_{h_{1}}(f\phi')(n)\vert+O(p^{-1/2})
=\Vert f\phi'\Vert^{4}_{U^{2}(V(M))}+O(p^{-1/2}).
\end{split}
\end{equation}
Since $\SGI(1)$ holds by assumption, if $p\gg_{d,\e} 1$, then there exists $\phi''\in \Nil_{p}^{1;O_{d,\e}(1),1}(\V)$ such that 
$$\vert\E_{n\in V(M)}f(n)\phi'(n)\phi''(n)\vert\gg_{d,\e} 1.$$
This completes the proof of $\SGI(2)$ since $\phi'\phi''\in \Nil_{p}^{2;O_{d,\e}(1),1}(\V)$.

\section{The symmetric argument for $s\geq 2$}\label{3:s:b72}
\subsection{A symmetric property for the case $s\geq 2$}\label{3:s:ss03}
In this section, we prove $\SGI(s+1)$ for $s\geq 2$. Again we first use (\ref{3:midpoint0}) to deduce some symmetric property on  $\tilde{\chi}$.
By (\ref{3:midpoint0}), Corollary \ref{3:LE.6} and the Pigeonhole Principle,  for all $h\in \V$, there exists  $\varphi_{h}\in\Xi_{p}^{s-1;O_{d,\e}(1),O_{d,\e}(1)}(\V)$  such that
\begin{equation}\label{3:ff0}
\begin{split}
\vert\E_{n\in V(M)^{h}}\Delta_{h}f(n)\chi(h,n)\otimes\varphi_{h}(n)\vert\gg_{d,\e} 1
\end{split}
\end{equation} 
for all $h\in H$. For convenience denote $h_{4}:=h_{1}+h_{2}-h_{3}$.
  By Proposition \ref{3:inductioni2}, if $p\gg_{d,\e} 1$ and $d\geq 9$, then 
   the absolute value of the average of the sequence	
\begin{equation}\nonumber
\begin{split}
&\quad (n,h_{1},h_{2},h_{3})\mapsto \chi(h_{1},n)\otimes\chi(h_{2},n+h_{3}-h_{2})\otimes\overline{\chi}(h_{3},n)\otimes\overline{\chi}(h_{4},n+h_{3}-h_{2})
\\&\otimes \varphi_{h_{1}}(n)\otimes\varphi_{h_{2}}(n+h_{3}-h_{2})\otimes\overline{\varphi}_{h_{3}}(n)\otimes\overline{\varphi}_{h_{4}}(n+h_{3}-h_{2})
\end{split}
\end{equation}
along the set $\{(n,h_{1},h_{2},h_{3})\in(\V)^{4}\colon n\in V(M)^{h_{1},h_{3},h_{3}-h_{2}}\}$ is $\gg_{d,\e} 1$, where we set $\varphi_{h}\equiv 0$ for $h\notin H$.\footnote{Note that $\varphi_{h}$ is not a nilcharacter for $h\notin H$ since it does not have absolute value 1. But this does not affect the proof.}
By the change of variable $(h_{1},h_{2},h_{3},h_{4})=(h_{0}+a,h_{0}+b,h_{0}+a+b,h_{0})$, 
we have that 
\begin{equation}\label{3:ff6}
\begin{split}
\vert\E_{h_{0},a,b\in \V,n\in V(M)^{h_{0}+a,h_{0}+a+b,a}}\tau(h_{0},a,b,n)\otimes \varphi_{h_{0}+a}(n)\otimes\varphi_{h_{0}+b}(n+a)\otimes\overline{\varphi}_{h_{0}+a+b}(n)\otimes\overline{\varphi}_{h_{0}}(n+a)\vert
\end{split}
\end{equation}
is $\gg_{d,\e} 1,$
where
\begin{equation}\label{3:13.6}
\begin{split}
\tau(h_{0},a,b,n):=\chi(h_{0}+a,n)\otimes\chi(h_{0}+b,n+a)\otimes\overline{\chi}(h_{0}+a+b,n)\otimes\overline{\chi}(h_{0},n+a).
\end{split}
\end{equation}
Note that the square of (\ref{3:ff6}) is bounded by $O_{d,\e}(1)$ times
\begin{equation}\label{3:ff7}
\begin{split}
&\quad\vert\E_{h_{0},a\in \V}\E_{n\in V(M)^{h_{0}+a,a}}\E_{b\in V(M(h_{0}+a+n+\cdot))}\tau(h_{0},a,b,n)\otimes \varphi_{h_{0}+a}(n)\otimes\varphi_{h_{0}+b}(n+a)
\\&\qquad\otimes\overline{\varphi}_{h_{0}+a+b}(n)\otimes\overline{\varphi}_{h_{0}}(n+a)\vert^{2}+O(p^{-1/2})
\\&\ll_{d,\e} \E_{h_{0},a\in \V}\E_{n\in V(M)^{h_{0}+a,a}}\vert\E_{b\in V(M(h_{0}+a+n+\cdot))}\tau(h_{0},a,b,n)\otimes\varphi_{h_{0}+b}(n+a)\otimes\overline{\varphi}_{h_{0}+a+b}(n)\vert^{2}+O(p^{-1/2})
\\&\ll_{d,\e}\vert\E_{h_{0},a\in \V}\E_{n\in V(M)^{h_{0}+a,a}}\E_{b,b'\in V(M(h_{0}+a+n+\cdot))}\tau(h_{0},a,b,n)\otimes\overline{\tau}(h_{0},a,b',n)
\\&\qquad\otimes\varphi_{h_{0}+b}(n+a)\otimes\overline{\varphi}_{h_{0}+b'}(n+a)\otimes\overline{\varphi}_{h_{0}+a+b}(n)\otimes\varphi_{h_{0}+a+b'}(n)\vert+O(p^{-1/2})
\\&\ll_{d,\e}\vert \E_{h_{0},a,b,b'\in\V}\E_{n\in V(M)^{h_{0}+a,h_{0}+a+b,h_{0}+a+b',a}}\tau(h_{0},a,b,n)\otimes\overline{\tau}(h_{0},a,b',n)
\\&\qquad\otimes\varphi_{h_{0}+b}(n+a)\otimes\overline{\varphi}_{h_{0}+b'}(n+a)
\otimes\overline{\varphi}_{h_{0}+a+b}(n)\otimes\varphi_{h_{0}+a+b'}(n)\vert+O(p^{-1/2}), 
\end{split}
\end{equation}
where we used Theorem \ref{3:ct} and the fact that the sets 
$$\{(h_{0},a,b,n)\in(\V)^{4}\colon n\in V(M)^{h_{0}+a,h_{0}+a+b,a}\}$$
and
$$\{(h_{0},a,b,b',n)\in(\V)^{4}\colon n\in V(M)^{h_{0}+a,h_{0}+a+b,h_{0}+a+b',a}\}$$
are   consistent $M$-sets of total co-dimensions 4 and 5 respectively.
Substituting $c:=a+b+b'$ in (\ref{3:ff7}), we have that
\begin{equation}\label{3:13.9}
\begin{split}
\vert \E_{h_{0},c,b,b'\in\V}\E_{n\in V(M)^{h_{0}+c-b-b',h_{0}+c-b',h_{0}+c-b,c-b-b'}}\alpha(h_{0},c,b,b',n)\otimes\varphi'_{h_{0},b,c}(n)\otimes\overline{\varphi'}_{h_{0},b',c}(n)\vert\gg_{d,\e} 1,
\end{split}
\end{equation}
provided that $p\gg_{d,\e} 1$,
where 
\begin{equation}\label{3:13.10}
\begin{split}
\alpha(h_{0},c,b,b',n):=\tau(h_{0},c-b-b',b,n)\otimes\overline{\tau}(h_{0},c-b-b',b',n)
\end{split}
\end{equation}
and
%\begin{equation}\nonumber
%\begin{split}
$\varphi'_{h_{0},b,c}(n):=\varphi_{h_{0}+b}(n+c-b-b')\otimes\varphi_{h_{0}+c-b}(n).$
%\end{split}
%\end{equation}
Let $\Lambda_{1}$ be the set of $(h_{0},c,b_{1},b_{2},b',n)\in (\V)^{6}$ such that 
$$n,n+h_{0}+c-b',  n+h_{0}+c-b_{i}-b',n+h_{0}+c-b_{i},n+c-b_{i}-b'\in V(M) \text{ for } i=1,2,$$
and $\Lambda'_{1}$ be the set of $(h_{0},c,b_{1},b_{2},n)\in(\V)^{5}$ such that
$$n, n+h_{0}+c-b_{1}, n+h_{0}+c-b_{2}\in V(M).$$ For $(h_{0},c,b_{1},b_{2},n)\in (\V)^{5}$, let $\Lambda_{1}(h_{0},c,b_{1},b_{2},n)$ denote the set of $b'$ such that
$$n+h_{0}+c-b',  n+h_{0}+c-b_{i}-b',n+c-b_{i}-b'\in V(M) \text{ for } i=1,2.$$
Let $A$ be the matrix associated to $M$.
Note that $(h_{0},c,b_{1},b_{2},n,b')\in \Lambda_{1}$ if and only if 
\begin{itemize}
	\item $h_{0},c,b_{1}\in\V$;
	\item $((b_{1}-b_{2})A)\cdot h_{0}=0$;
	\item $M(n)=M(n+c+h_{0}-b_{1})=M(n+c+h_{0}-b_{2})=0$;
	\item $M(n+h_{0}+c-b')=M(n+h_{0}+c-b_{1}-b')=M(n+h_{0}+c-b_{2}-b')=M(n+c-b_{1}-b')=0$.\footnote{In particular, the condition $M(n+c-b_{2}-b')=0$ is a consequence of other restrictions by Lemma \ref{3:or}.}  
\end{itemize}	 
This implies that  $\Lambda_{1}$ is a consistent  $M$-set  of total co-dimension 8, and that $\Lambda_{1}'$ is a consistent $M$-set  of total co-dimension 4. By Theorem \ref{3:ct},
the square of left hand side of (\ref{3:13.9}) is bounded by $O_{d,\e}(1)$ times
\begin{equation}\label{3:ff8}
\begin{split}
&\quad \vert\E_{h_{0},c,b'\in\V}\E_{n\in V(M)^{h_{0}+c-b'}}\E_{b\in \V\colon n+h_{0}+c-b-b',n+h_{0}+c-b,n+c-b-b'\in V(M)}
\\&\qquad\alpha(h_{0},c,b,b',n)\otimes\varphi'_{h_{0},b,c}(n)\otimes\overline{\varphi'}_{h_{0},b',c}(n)\vert^{2}+O(p^{-1/2})
\\&\ll_{d,\e}\E_{h_{0},c,b'\in\V}\E_{n\in V(M)^{h_{0}+c-b'}}\vert\E_{b\in \V\colon n+h_{0}+c-b-b',n+h_{0}+c-b,n+c-b-b'\in V(M)}
\\&\qquad\alpha(h_{0},c,b,b',n)\otimes\varphi'_{h_{0},b,c}(n)\vert^{2}+O(p^{-1/2})
\\&\ll_{d,\e}\vert\E_{h_{0},c,b'\in\V}\E_{n\in V(M)^{h_{0}+c-b'}}\E_{b_{1},b_{2}\in \V\colon n+h_{0}+c-b_{i}-b',n+h_{0}+c-b_{i},n+c-b_{i}-b' \in V(M) \text{ for } i=1,2}
\\&\qquad \alpha(h_{0},c,b_{1},b',n)\otimes \overline{\alpha}(h_{0},c,b_{2},b',n)\otimes\varphi'_{h_{0},b_{1},c}(n)\otimes\overline{\varphi'}_{h_{0},b_{2},c}(n)\vert+O(p^{-1/2})
\\&\ll_{d,\e}\vert\E_{(h_{0},c,b_{1},b_{2},n)\in\Lambda'_{1}}\E_{b'\in\Lambda_{1}(h_{0},c,b_{1},b_{2},n)}
\\&\qquad \alpha(h_{0},c,b_{1},b',n)\otimes \overline{\alpha}(h_{0},c,b_{2},b',n)\otimes\varphi'_{h_{0},b_{1},c}(n)\otimes\overline{\varphi'}_{h_{0},b_{2},c}(n)\vert+O(p^{-1/2})
\\&\ll_{d,\e}\E_{(h_{0},c,b_{1},b_{2},n)\in\Lambda'_{1}}\vert\E_{b'\in\Lambda_{1}(h_{0},c,b_{1},b_{2},n)}
 \alpha(h_{0},c,b_{1},b',n)\otimes \overline{\alpha}(h_{0},c,b_{2},b',n)\vert+O(p^{-1/2}),
\end{split}
\end{equation}
where we used  the fact that  the set    
$$\{(h_{0},c,b,b',n)\in(\V)^{5}\colon n\in V(M)^{h_{0}+c-b-b',h_{0}+c-b',h_{0}+c-b,c-b-b'}\}$$
is a consistent $M$-set   of total co-dimensions 5.

Let $\Lambda_{2}$ denote the set of $(h_{0},c,b_{1},b'_{1},b_{2},b'_{2},n)\in(\V)^{7}$ such that $(h_{0},c,b_{1},b_{2},b'_{1},n)$ and $(h_{0},c,b_{1},b_{2},b'_{2},n)$ belong to $\Lambda_{1}$. Since $M$ is pure,
by Propositions \ref{3:yy3} and \ref{3:yy33}, $\Lambda_{2}$ is  a pure and consistent $M$-set and is of total co-dimension at most 28.
So by Theorem \ref{3:ct}, 
the square of the right hand side of (\ref{3:ff8}) is bounded by $O_{d,\e}(1)$ times
\begin{equation}\label{3:ff9}
\begin{split}
&\quad \E_{(h_{0},c,b_{1},b_{2},n)\in\Lambda'_{1}}\vert\E_{b'\in\Lambda_{1}(h_{0},c,b_{1},b_{2},n)}
\alpha(h_{0},c,b_{1},b',n)\otimes \overline{\alpha}(h_{0},c,b_{2},b',n)\vert^{2}+O(p^{-1/2})
\\&\ll_{d,\e}\vert\E_{(h_{0},c,b_{1},b_{2},n)\in\Lambda'_{1}}\E_{b'_{1},b'_{2}\in\Lambda_{1}(h_{0},c,b_{1},b_{2},n)}\alpha'(h_{0},c,b_{1},b'_{1},b_{2},b'_{2},n)\vert+O(p^{-1/2})
\\&\ll_{d,\e}\vert\E_{(h_{0},c,b_{1},b'_{1},b_{2},b'_{2},n)\in\Lambda_{2}}\alpha'(h_{0},c,b_{1},b'_{1},b_{2},b'_{2},n)\vert+O(p^{-1/2}),
\end{split}
\end{equation}
where 
\begin{equation}\label{3:13.12}
\begin{split}
&\quad\alpha'(h_{0},c,b_{1},b'_{1},b_{2},b'_{2},n)
\\&:=\alpha(h_{0},c,b_{1},b'_{1},n)\otimes \overline{\alpha}(h_{0},c,b_{2},b'_{1},n)\otimes \overline{\alpha}(h_{0},c,b_{1},b'_{2},n)\otimes \alpha(h_{0},c,b_{2},b'_{2},n).
\end{split}
\end{equation}

We have now successfully eliminated the $\varphi_{h}$ term in (\ref{3:ff0}), and we could use Lemma \ref{3:LE.11} to deduce from (\ref{3:ff9}) that $\alpha'$ coincides with a nilsequence of degree at most $s$ on $\Lambda_{2}$. However, in order to better use  (\ref{3:ff9}) to derive the information we need, we need to make some further transformations for  $\alpha'$ before applying Lemma \ref{3:LE.11}.

Firstly, we need to take sufficiently many derivatives of $\alpha$ in the variable $n$ in order to ``annihilate" the appearance of other variables.
Since $\Lambda_{2}$ is  a consistent $M$-set of total co-dimension at most 28, we may write $\Lambda_{2}=V(\mathcal{J})$ for some consistent $(M,7)$-family $\mathcal{J}\subseteq\F_{p}[h_{0},c,b_{1},b'_{1},b_{2},b'_{2},n]$ of total dimension at most 28. Let $(\mathcal{J}',\mathcal{J}'')$ be an $\{h_{0},c,b_{1},b'_{1},b_{2},b'_{2}\}$-decomposition of $\mathcal{J}$. Let $\Lambda'_{2}$ be the set of $\y:=(h_{0},c,b_{1},b'_{1},b_{2},b'_{2})\in(\V)^{6}$  such that $(\y,n)\in V(\mathcal{J}')$ for all $n\in\V$. For $\y\in(\V)^{6}$, let  $\Lambda_{2}(\y)$ denote the set of $n\in\V$ such that $(\y,n)\in V(\mathcal{J}'')$.
Since $\Lambda_{2}$ is a pure and consistent $M$-set of total co-dimension at most 28, so is $\Lambda'_{2}$ by Propositions \ref{3:yy3} and \ref{3:yy33}. For convenience write $\h'=(h_{1},\dots,h_{s-1})\in(\V)^{s-1}$.
 Let $\Lambda_{3}$ denote the set of $(\y,n,\h')\in (\V)^{s+6}$ such that $(\y,n+\e\cdot\h')\in \Lambda_{2}$ for all $\e\in\{0,1\}^{s-1}$, where 
 $$\e\cdot\h':=(\e_{1},\dots,\e_{s-1})\cdot (h_{1},\dots,h_{s-1})=\e_{1}h_{1}+\dots+\e_{s-1}h_{s-1}.$$
   By repeatedly using Propositions \ref{3:yy3} and \ref{3:yy33}, we have that $\Lambda_{3}$ is a pure and consistent $M$-set of total co-dimension at most $\binom{s+7}{2}$.
By Theorem \ref{3:ct},  the $2^{s-1}$-th power of the right hand side of (\ref{3:ff9}) is bounded by 
$O_{d,\e}(1)$ times
\begin{equation}\nonumber
\begin{split}
&\quad \vert\E_{\y\in\Lambda_{2}'}\E_{n\in \Lambda_{2}(\y)}\alpha'(\y,n)\vert^{2^{s-1}}+O(p^{-1/2})
\ll_{d,\e} \E_{\y\in\Lambda_{2}'}\vert\E_{n\in \Lambda_{2}(\y)}\alpha'(\y,n)\vert^{2^{s-1}}+O(p^{-1/2})
\\&\ll_{d,\e} \Bigl\vert\E_{\y\in\Lambda_{2}'}\E_{(n,\h')\in\Gow_{s-1}( \Lambda_{2}(\y))}\bigotimes_{\e\in\{0,1\}^{s-1}}\mathcal{C}^{\vert\e\vert}\alpha'(\y,n+\e\cdot \h')\Bigr\vert+O(p^{-1/2})
\\&\ll_{d,\e} \Bigl\vert\E_{(\y,n,\h')\in\Lambda_{3}}\bigotimes_{\e\in\{0,1\}^{s-1}}\mathcal{C}^{\vert\e\vert}\alpha'(\y,n+\e\cdot \h')\Bigr\vert+O(p^{-1/2}).
\end{split}
\end{equation}
Combining all the computations above, we have that 
\begin{equation}\label{3:ff12}
\begin{split}
\vert\E_{(\y,n,\h')\in \Lambda_{3}} \sigma(\y,n,\h')\vert\gg_{d,\e} 1,
\end{split}
\end{equation}
where 
\begin{equation}\label{3:119}
\sigma(\y,n,\h'):=\bigotimes_{\e\in\{0,1\}^{s-1}}\mathcal{C}^{\vert\e\vert}\alpha'(\y,n+\e\cdot\h').
\end{equation}

We next provide a better description for 
  the set $\Lambda_{3}$  under a change of variables.
Let $L\colon (\V)^{s+6}\to(\V)^{s+6}$ be the  bijective linear transformation 
which maps  $(h_{0},c,b_{1},$ $b_{2}, b'_{1},b'_{2},$ $n,\h')$ to
\begin{equation}\nonumber
\begin{split}
(n+c-b_{1}-b'_{1},n+c-b_{1}-b'_{1}+h_{0},n+c-b_{1}+h_{0},n+c-b'_{1}+h_{0},n,\h',b_{1}-b_{2},b'_{1}-b'_{2}).
\end{split}
\end{equation}
One can compute that 
\begin{equation}\label{3:L-1}
\begin{split}
&\quad L^{-1}(w_{1},w_{2},w_{3},w_{4},n,\h',h_{s},h_{s+1})
\\&=(w_{2}-w_{1},w_{1}-2w_{2}+w_{3}+w_{4}-n,w_{4}-w_{2},w_{4}-w_{2}-h_{s},w_{3}-w_{2},w_{3}-w_{2}-h_{s+1},n,\h').
\end{split}
\end{equation}

Denote $\sigma':=\sigma\circ L^{-1}$, $\Lambda_{4}:=L(\Lambda_{3})$. 
By Propositions \ref{3:yy3} and \ref{3:yy33},  $\Lambda_{4}$ is a pure and consistent $M$-set of total co-dimension at most $\binom{s+7}{2}$.

Then (\ref{3:ff12}) implies that
\begin{equation}\label{3:13.146}
\begin{split}
\vert\E_{(w_{1},w_{2},w_{3},w_{4},n,\h',h_{s},h_{s+1})\in \Lambda_{4}}
\sigma'(w_{1},w_{2},w_{3}.w_{4},n,\h',h_{s},h_{s+1})\vert\gg_{d,\e} 1.
\end{split}
\end{equation}

We need a precise description for the set $\Lambda_{4}$.
Note that $(h_{0},c,b_{1},b_{2},b'_{1},b'_{2},n,\h')\in \Lambda_{3}$ if and only if the following vectors are in $V(M)$:
$$n+\e'\cdot \h',n+\e'\cdot \h'+h_{0}+c-b'_{i},n+\e'\cdot \h'+h_{0}+c-b_{i},n+\e'\cdot \h'+h_{0}+c-b'_{i}-b_{j},n+\e'\cdot \h'+c-b'_{i}-b_{j}$$
for all $i,j\in\{1,2\}$ and $\e'\in\{0,1\}^{s-1}$.  
This is equivalent of saying that
$$n+\e'\cdot \h',w_{4}+\e'\cdot \h'+\e_{s+1} h_{s+1},w_{3}+\e'\cdot \h'+\e_{s}h_{s},w_{2}+\e'\cdot \h'+\e_{s}h_{s}+\e_{s+1}h_{s+1},w_{1}+\e'\cdot \h'+\e_{s}h_{s}+\e_{s+1}h_{s+1}$$
for all $\e'\in\{0,1\}^{s-1}$ and $\e_{s},\e_{s+1}\in\{0,1\}$, which is further equivalent of saying that
 $(w_{1},\h',h_{s},h_{s+1}),(w_{2},\h',h_{s},h_{s+1})\in \Gow_{s+1}(V(M)),$ $(w_{3},\h',h_{s})\in \Gow_{s}(V(M)),$ $(w_{4},\h',h_{s+1})\in \Gow_{s}(V(M))$ and $(n,\h')\in \Gow_{s-1}(V(M)).$

 In conclusion, $\Lambda_{4}$ is the set $\Omega_{2}$ in Example %\ref{1:mainex} 
 B.4 of \cite{SunA} (with the variable $w_{5}$ replaced by $n$), which is a nice and consistent $M$-set of total co-dimension $(s^{2}+11s+12)/2$.  By 
 Lemma \ref{3:LE80}, we have $\sigma,\sigma'\in\Xi_{p}^{s+1;O_{d,\e}(1),O_{d,\e}(1)}((\V)^{s+6})$.
 Since  $d\geq N(s)$ and $p\gg_{d,\e} 1$,
it follows from Lemma \ref{3:LE.11} and (\ref{3:13.146}) that 
\begin{equation}\nonumber
\begin{split}
\sigma'(w_{1},w_{2},w_{3},w_{4},n,\h',h_{s},h_{s+1})\sim_{O_{d,\e}(1)} 1\mod\Xi^{s+1}_{p}(\Lambda_{4}).
\end{split}
\end{equation}
Therefore, by Lemma \ref{3:LE8} (vi),
\begin{equation}\label{3:13.1455}
\begin{split}
\sigma(h_{0},c,b_{1},b'_{1},b_{2},b'_{2},n,\h')\sim_{O_{d,\e}(1)} 1\mod\Xi^{s+1}_{p}(\Lambda_{3}).
\end{split}
\end{equation}

\subsection{Calculation of the equivalence classes}
We now simplify (\ref{3:13.1455}) up to the equivalence relation $\sim$ by using the nilcharacters defined in (\ref{3:13.6}),  (\ref{3:13.10}), (\ref{3:13.12}) and (\ref{3:119}).
We first recapture the definitions:
\begin{equation}\label{3:13all}
\begin{split}
&\tau(h_{0},a,b,n):=\chi(h_{0}+a,n)\otimes\chi(h_{0}+b,n+a)\otimes\overline{\chi}(h_{0}+a+b,n)\otimes\overline{\chi}(h_{0},n+a),
\\&\alpha(h_{0},c,b,b',n):=\tau(h_{0},c-b-b',b,n)\otimes\overline{\tau}(h_{0},c-b-b',b',n),
\\&\alpha'(h_{0},c,b_{1},b'_{1},b_{2},b'_{2},n)
\\&:=\alpha(h_{0},c,b_{1},b'_{1},n)\otimes \overline{\alpha}(h_{0},c,b_{2},b'_{1},n)\otimes \overline{\alpha}(h_{0},c,b_{1},b'_{2},n)\otimes \alpha(h_{0},c,b_{2},b'_{2},n)
\\&\sigma(h_{0},c,b_{1},b'_{1},b_{2},b'_{2},n,\h'):=\bigotimes_{\e\in\{0,1\}^{s-1}}\mathcal{C}^{\vert \e\vert}\alpha'(h_{0},c,b_{1},b'_{1},b_{2},b'_{2},n+\e\cdot\h').
\end{split}
\end{equation}

By Lemma \ref{3:LE8} (iv), (vi), Lemma \ref{3:L13.2}, (\ref{3:ini2}) and (\ref{3:13all}), for all $i,j\in\{1,2\}$, we have that 
\begin{equation}\nonumber
\begin{split}
&\quad \tau(h_{0},c-b_{i}-b'_{j},b_{i},n)
\\&\sim_{O_{d,\e}(1)}\chi(h_{0}+c-b_{i}-b'_{j},n)\otimes\chi(h_{0}+b_{i},n+c-b_{i}-b'_{j})  
\\&\qquad\qquad \otimes\overline{\chi}(h_{0}+c-b'_{j},n)\otimes\overline{\chi}(h_{0},n+c-b_{i}-b'_{j}) 
\\&\sim_{O_{d,\e}(1)}\chi(b_{i},n+c-b_{i}-b'_{j})\otimes\overline{\chi}(b_{i},n)\mod\Xi^{s+1}_{p}(\Lambda_{3}),
\end{split}
\end{equation}
where all the sequences are viewed as nilcharacters in the variables  $(h_{0},c,b_{1},b'_{1},b_{2}$, $b'_{2}$, $n$, $\h')$.
Similarly, for all $i,j\in\{1,2\}$, we have that 
\begin{equation}\nonumber
\begin{split}
&\quad \alpha(h_{0},c,b_{i},b'_{j},n)
\\&\sim_{O_{d,\e}(1)}\tau(h_{0},c-b_{i}-b'_{j},b_{i},n)\otimes\overline{\tau}(h_{0},c-b_{i}-b'_{j},b'_{j},n) 
\\&\sim_{O_{d,\e}(1)}\chi(b_{i},n+c-b_{i}-b'_{j})\otimes\overline{\chi}(b_{i},n)\otimes\overline{\chi}(b'_{j},n+c-b_{i}-b'_{j})\otimes\chi(b'_{j},n)
\\&\sim_{O_{d,\e}(1)}\chi(b_{i}-b'_{j},n+c-b_{i}-b'_{j})\otimes\overline{\chi}(b_{i}-b'_{j},n)\mod\Xi^{s+1}_{p}(\Lambda_{3})
\end{split}
\end{equation}
and 
\begin{equation}\label{3:thisisap}
\begin{split}
&\quad\alpha'(h_{0},c,b_{1},b_{2},b'_{1},b'_{2},n)
\\&\sim_{O_{d,\e}(1)}	\bigotimes_{i,j\in\{1,2\}}\mathcal{C}^{i+j}(\chi(b_{i}-b'_{j},n+c-b_{i}-b'_{j})\otimes\overline{\chi}(b_{i}-b'_{j},n))
\\&\sim_{O_{d,\e}(1)}\bigotimes_{i,j\in\{1,2\}}\mathcal{C}^{i+j}\chi(b_{i}-b'_{j},n+c-b_{i}-b'_{j})\mod \Xi^{s+1}_{p}(\Lambda_{3}).
\end{split}
\end{equation}

\begin{lem}\label{3:ee33}	
	We have that 
	\begin{equation}\label{3:eee21}
	\begin{split}
	&\quad\sigma(h_{0},c,b_{1},b_{2},b'_{1},b'_{2},n,\h')
	\\&\sim_{O_{d,\e}(1)}\mathcal{C}^{s-1}(\tilde{\chi}(b'_{1}-b'_{2},b_{1}-b_{2},\h')\otimes\overline{\tilde{\chi}}(b_{1}-b_{2},b'_{1}-b'_{2},\h'))^{\otimes (s+1)!}\mod \Xi^{s+1}_{p}(\Lambda_{3}).
	\end{split}
	\end{equation}
\end{lem}	
\begin{proof}
 By (\ref{3:13all}) and (\ref{3:thisisap}), we have that
	\begin{equation}\label{3:ee32}
	\begin{split}
	&\quad\sigma(h_{0},c,b_{1},b_{2},b'_{1},b'_{2},n,\h')=\bigotimes_{\e\in\{0,1\}^{s-1}}\mathcal{C}^{\vert\e\vert}\alpha'(h_{0},c,b_{1},b_{2},b'_{1},b'_{2},n+\e\cdot\h')
	\\&\sim_{O_{d,\e}(1)}\bigotimes_{i,j\in\{1,2\}}\mathcal{C}^{i+j}\bigotimes_{\e\in\{0,1\}^{s-1}}\mathcal{C}^{\vert\e\vert}\chi(b_{i}-b'_{j},n+c-b_{i}-b'_{j}+\e\cdot\h')\mod \Xi^{s+1}_{p}(\Lambda_{3}).
	\end{split}
	\end{equation}
	
	We first claim that for all $i,j\in\{1,2\}$,  
	\begin{equation}\label{3:ee31}
	\begin{split}
	&\quad\bigotimes_{\e\in\{0,1\}^{s-1}}\mathcal{C}^{\vert\e\vert}\chi(b_{i}-b'_{j},n+c-b_{i}-b'_{j}+\e\cdot\h')^{\otimes 2}
	\\&\sim_{O_{d,\e}(1)}\mathcal{C}^{s-1} \Bigl(\tilde{\chi}(b_{i}-b'_{j},n+c-b_{i}-b'_{j},\h')^{\otimes 2} \bigotimes_{\ell=1}^{s-1}\tilde{\chi}(b_{i}-b'_{j},h_{\ell},\h')\Bigr)^{\otimes (s+1)!}\mod\Xi^{s+1}_{p}(\Lambda_{3}).
	\end{split}
	\end{equation}

	For convenience denote $h_{0}=n+c-b_{i}-b'_{j}$.
	By (\ref{3:ini2}), Lemma \ref{3:LE8} (iv) and (vi), Lemma \ref{3:L13.2}, and the fact that $\tilde{\chi}$ is symmetric in the last $s$ variables, 
		\begin{equation}\nonumber
		\begin{split}
		&\quad\bigotimes_{\e\in\{0,1\}^{s-1}}\mathcal{C}^{\vert\e\vert}\chi(b_{i}-b'_{j},n+c-b_{i}-b'_{j}+\e\cdot\h')^{\otimes 2}
		\\&\sim_{O_{d,\e}(1)} \bigotimes_{i_{0}+\dots+i_{s-1}=s}\Bigl(\tilde{\chi}(b_{i}-b'_{j},h_{0},\dots,h_{0},\dots,h_{s-1},\dots,h_{s-1})^{\otimes 2C_{i_{0},\dots,i_{s+1}}}\Bigr)^{\otimes (s+1)}\mod\Xi^{s+1}_{p}(\Lambda_{3})
		\end{split}
		\end{equation}
	for some $C_{i_{0},\dots,i_{s-1}}\in\Z$, where in the last $s$ variables of 
	$\tilde{\chi}(b_{i}-b'_{j},h_{0},\dots,h_{0},\dots,h_{s-1},\dots,h_{s-1})$, $h_{j}$ appears $i_{j}$ times for all $0\leq j\leq s+1$. Moreover,
	$C_{i_{0},\dots,i_{s-1}}$ are the unique integers such that the polynomial
	$$F(x_{0},\dots,x_{s-1}):=\sum_{\e=(\e_{1},\dots,\e_{s-1})\in\{0,1\}^{s-1}}(-1)^{\vert\e\vert}(x_{0}+\sum_{i=1}^{s-1}\e_{i}x_{i})^{s} \in \F_{p}[x_{0},\dots,x_{s-1}]$$
	can be written as
	$$F(x_{0},\dots,x_{s-1}):=\sum_{i_{0}+\dots+i_{s-1}=s} C_{i_{0},\dots,i_{s-1}}x_{0}^{i_{0}}\dots x_{s-1}^{i_{s-1}}.$$
In other words,
	$$C_{i_{0},\dots,i_{s-1}}=\sum_{\e=(\e_{1},\dots,\e_{s-1})\in E_{i_{0},\dots,i_{s-1}}}(-1)^{\vert\e\vert}\frac{s!}{i_{0}!\dots i_{s-1}!},$$
	where $E_{i_{0},\dots,i_{s-1}}$ is the set of $\e=(\e_{1},\dots,\e_{s-1})\in \{0,1\}^{s-1}$ such that $\e_{j}=1$ whenever $i_{j}\geq 1$ for all $1\leq j\leq s-1$.
	If one of $i_{1},\dots,i_{s-1}$ is 0, say $i_{\ell}=0$, then by separating those $(\e_{1},\dots,\e_{s-1})\in E_{i_{0},\dots,i_{s-1}}$ with $\e_{\ell}=0$ and those with $\e_{\ell}=1$ in the sum, we have that $C_{i_{0},\dots,i_{s-1}}=0$. So if  $C_{i_{0},\dots,i_{s-1}}\neq 0$, then all of $i_{1},\dots,i_{s-1}$ are at least 1. Since $i_{0}+\dots+i_{s-1}=s$, we have that the only nontrivial terms $C_{i_{0},\dots,i_{s-1}}$ are $C_{1,\dots,1}$, which equals to $(-1)^{s-1}s!$,
	and $C_{0,1,\dots,1,2,1,\dots,1}$ (where the coordinate 2 corresponds to the unique $i_{\ell}$ which equals to 2), which equals to $(-1)^{s-1}s!/2$.
	Since $\tilde{\chi}$ is symmetric in the last $s$ variables, this implies (\ref{3:ee31}) and proves the claim. 
	
	\

    After some simple computations,
    one can deduce (\ref{3:eee21}) from (\ref{3:ee32}) and (\ref{3:ee31}) by the claim and Lemmas \ref{3:LE8} (iv), \ref{3:L13.2} and \ref{3:LE.13}. 
	  We leave the details to the interested readers.
\end{proof}	

Since $\Lambda_{4}$ is a nice and consistent $M$-set of total co-dimension $(s^{2}+11s+12)/2$, so  is $\Lambda_{3}$ by Propositions \ref{3:yy3} (iii) and \ref{3:yy33}.
Combining Lemmas \ref{3:LE.13}, \ref{3:ee33} and equation (\ref{3:13.1455}),  we have that 
$$\tilde{\chi}(b'_{1}-b'_{2},b_{1}-b_{2},h_{1},\dots,h_{s-1})\sim_{O_{d,\e}(1)}\tilde{\chi}(b_{1}-b_{2},b'_{1}-b'_{2},h_{1},\dots,h_{s-1})\mod\Xi^{s+1}_{p}(\Lambda_{3}).$$
By Lemma \ref{3:LE8} (vi), we have that 
\begin{equation}\nonumber
\tilde{\chi}(h_{s},h_{s+1},h_{1},\dots,h_{s-1})\sim_{O_{d,\e}(1)}\tilde{\chi}(h_{s+1},h_{s},h_{1},\dots,h_{s-1})\mod\Xi^{s+1}_{p}(\Lambda_{4}),
\end{equation}
where all the sequences are viewed as nilcharacters in the variables  $(w_{1},\dots,w_{4},n,h_{1},\dots$, $h_{s+1})$.
Since $\tilde{\chi}$ is symmetric in the last $s$ variables, we have that 
\begin{equation}\label{3:keysim}
\tilde{\chi}(h_{1},\dots,h_{s+1})\sim_{O_{d,\e}(1)}\tilde{\chi}(h_{\sigma(1)},\dots,h_{\sigma(s+1)})\mod\Xi^{s+1}_{p}(\Lambda_{4})
\end{equation}
for any permutation $\sigma\colon \{1,\dots,s+1\}\to\{1,\dots,s+1\}$.

\subsection{Completion of the proof of $\SGI(s+1)$ for $s\geq 2$}\label{3:s:sss02}
	
Our final step is to use (\ref{3:keysim}) to complete the proof.

\textbf{Step 1:  reformulate (\ref{3:midpoint0}).}  
	Since every $h\in H$ is  $M$-non-isotropic,  $V(M)^{h}$ can be written as $V(M)\cap(V+c)$ for some affine subspace $V+c$ of $\V$ of co-dimension 1 such that $V\cap V^{\pp}=\{\bold{0}\}$. By Proposition \ref{3:iissoo}, $\rank(M\vert_{V+c})=d-1$. 
	So by (\ref{3:midpoint0}) and Proposition \ref{3:cinv},   
	$\Vert\Delta_{h}f\cdot\chi(h,\cdot)\Vert_{U^{s}(V(M)^{h})}\gg_{d,\e} 1$ for all $h\in H$ and thus
	\begin{equation}\nonumber
	\begin{split}
	\E_{h\in\V}\Vert\Delta_{h}f\cdot\chi(h,\cdot)\Vert_{U^{s}(V(M)^{h})}\gg_{d,\e} 1.
	\end{split}
	\end{equation}

	Denote $\tilde{f}(n):=f(n)\tilde{\chi}(n,\dots,n)$ for all $n\in\V$.
	Since $\Gow_{s+1}(V(M))$ 
 is an $M$-set of total co-dimension  $(s^{2}+3s+4)/2$ and $d\geq s^{2}+3s+5$,
 by Theorem \ref{3:ct},   	
 \begin{equation}\label{3:mmdd3}
	\begin{split}
	&\quad 1\ll_{d,\e}  \E_{h\in\V}\Vert \Delta_{h}f\cdot\chi(h,\cdot)\Vert_{U^{s}(V(M)^{h})}^{2^{s}}
	\\&=\vert\E_{(n,h_{1},\dots,h_{s+1})\in \Gow_{s+1}(V(M))}\Delta_{h_{s+1}}\dots\Delta_{h_{1}}\tilde{f}(n)\otimes\beta(h_{1},\dots,h_{s+1},n)\vert+O_{d,\e}(p^{-1/2}),
	\end{split}
	\end{equation} 
	where (using Convention \ref{3:d4n}) $$\beta(h_{1},\dots,h_{s+1},n):=\Delta_{h_{s+1}}\dots\Delta_{h_{2}}(\chi(h_{1},n)\otimes\overline{\tilde{\chi}}(n+h_{1},\dots,n+h_{1})\otimes \tilde{\chi}(n,\dots,n)).$$

\textbf{Step 2:  use (\ref{3:keysim}) to convert $\beta$ into a lower degree nilsequence.} 
	For convenience from now on we write the variables in $\Lambda_{4}$ as $(n,w_{2},\dots,w_{5},h_{1},\dots,h_{s+1})$ instead of  $(w_{1},\dots,$ $w_{4},n,h_{1},\dots,h_{s+1})$, and 
	 denote $\h:=(h_{1},\dots,h_{s+1})$ and $\w:=(w_{2},\dots,w_{5})$.

	Since $\Lambda_{4}$ is  a consistent $M$-set is of total co-dimension $(s^{2}+11s+12)/2$, we may write $\Lambda_{4}=V(\mathcal{J})$ for some consistent $(M,s+6)$-family $\mathcal{J}\subseteq\F_{p}[n,w_{2},\dots,w_{5},h_{1},\dots,h_{s+1}]$ of total dimension $(s^{2}+11s+12)/2$. Let $(\mathcal{J}',\mathcal{J}'')$ be an  $\{n,h_{1},\dots,h_{s+1}\}$-decomposition of $\mathcal{J}$. Let $\Lambda'_{4}$ be the set of $(n,\h)\in(\V)^{s+2}$  such that $(n,\w,\h)\in V(\mathcal{J}')$ for all $\w\in(\V)^{4}$.
	 For $(n,\h)\in(\V)^{s+2}$, let  $\Lambda_{4}(n,\h)$ denote the set of $\w\in(\V)^{4}$ such that $(n,\w,\h)\in V(\mathcal{J}'')$. 	By the description of $\Lambda_{4}$, we have that $\Lambda'_{4}=\Gow_{s+1}(V(M))$.
So
by Theorem \ref{3:ct} and (\ref{3:mmdd3}), we have
  		\begin{equation}\label{3:mmdd2}
		\begin{split}	 
			&\quad \vert\E_{(n,\w,\h)\in\Lambda_{4}} \Delta_{h_{s+1}}\dots\Delta_{h_{1}} \tilde{f}(n)\otimes\beta(\h,n)\vert
		\\& \gg_{d,\e}\vert\E_{(n,\h)\in \Gow_{s+1}(V(M))}\Delta_{h_{s+1}}\dots\Delta_{h_{1}}\tilde{f}(n)\otimes\beta(\h,n)\E_{\w\in \Lambda_{4}(n,\h)} 1\vert+O_{d,\e}(p^{-1/2})\gg_{d,\e} 1.
		\end{split}
		\end{equation} 
Using (\ref{3:ini2}), (\ref{3:keysim}), Lemma  \ref{3:LE8} (iii), (iv), (vi) and Lemma \ref{3:L13.2},
	 it is not hard to compute that (the details are left to the interested readers)
$$\beta(h_{1},\dots,h_{s+1},n)\sim_{O_{d,\e}(1)} 1\mod\Xi^{s+1}_{p}(\Lambda_{4}),$$
where now the variables in $\Lambda_{4}$ are labeled as $(n,w_{2},\dots,w_{5},h_{1},\dots,h_{s+1})$.
Therefore, $\beta(h_{1},\dots,h_{s+1},n)$ belongs to $\Nil^{s;O_{d,\e}(1),O_{d,\e}(1)}(\Lambda_{4})$ and thus belongs to  
$\Nil^{J;O_{d,\e}(1),O_{d,\e}(1)}(\Lambda_{4})$, where
  $$J:=\{(a_{1},\dots,a_{4},b,c_{1},\dots,c_{s+1})\in\N^{s+6}\colon a_{1}+\dots+a_{4}+b+c_{1}+\dots+c_{s+1}\leq s\}.$$
For all $1\leq i\leq s+1$, let 
$$J_{i}:=\{(a_{1},\dots,a_{4},b,c_{1},\dots,c_{s+1})\in J\colon c_{i}=0\}.$$
Then $J=\cup_{i=1}^{s+1}J_{i}$. Note that any sequence in $\Nil^{J_{i}}(\Lambda_{4})$ is independent of $h_{i}$. By (\ref{3:mmdd2}), Lemma \ref{3:LE.4} and the Pigeonhole Principle, for $1\leq i\leq s+1$,
there exists a scalar valued sequence $\psi_{i}(\w,n,\h)$ bounded by 1 which is independent of $h_{i}$ such that
	\begin{equation}\nonumber
	\begin{split}
	\Bigl\vert\E_{(n,\w,\h)\in \Lambda_{4}}\Delta_{h_{s+1}}\dots\Delta_{h_{1}}\tilde{f}(n)\prod_{i=1}^{s+1}\psi_{i}(n,\w,\h)\Bigr\vert\gg_{d,\e} 1
	\end{split}
	\end{equation}    
	since $p\gg_{d,\e} 1$.  By the definition of $\tilde{f}$ and the Pigeonhole Principle, there exists a scalar valued nilsequence $\phi'$ of step at most $s+1$ of complexity $O_{d,\e}(1)$ such that writing $f':=f\phi'$, we have that 
		\begin{equation}\label{3:pprr0}
		\begin{split}
		\Bigl\vert\E_{(n,\w,\h)\in \Lambda_{4}}\Delta_{h_{s+1}}\dots\Delta_{h_{1}}f'(n)\prod_{i=1}^{s+1}\psi_{i}(n,\w,\h)\Bigr\vert\gg_{d,\e} 1.
		\end{split}
		\end{equation}

\textbf{Step 3:  use the Cauchy-Schwartz inequality to remove $\psi_{i}$.} 	
	For $1\leq r\leq s+2$, we say that Property-$r$ holds if there exist a scalar valued function $\psi_{i}(n,\w,\h,h'_{1},\dots,h'_{r-1})$ bounded by 1 and independent of $h_{i}$ for all  $1\leq i\leq r$  such that  
	\begin{equation}\label{3:pprr}
	\begin{split}
	\Bigl\vert\E_{(n,\w,\h,h'_{1},\dots,h'_{r-1})\in \Lambda_{4,r}}\Delta_{h_{s+1}}\dots\Delta_{h_{1}}f'(n)\prod_{i=r}^{s+1}\psi_{i}(n,\w,\h,h'_{1},\dots,h'_{r-1})\Bigr\vert\gg_{d,\e} 1,
	\end{split}
	\end{equation} 	
	where $\Lambda_{4,r}$ is the set of all $(n,\w,\h,h'_{1},\dots,h'_{r-1})\in (\V)^{s+r+5}$ such that $$(n-h'_{1}-\dots-h'_{r-1},\w,h'_{1}+\e_{1}h_{1},\dots,h'_{r-1}+\e_{r-1}h_{r-1},h_{r},\dots,h_{s+1})\in \Lambda_{4}$$
	for all $\e_{1},\dots,\e_{r-1}\in \{0,1\}$, and the term $\prod_{i=r}^{s+1}\psi_{i}(n,\w,\h,h'_{1},\dots,h'_{r-1})$ is understood as 1 if $r=s+2$.
	
	Clearly, Property-1 holds by (\ref{3:pprr0}).
	Now suppose that Property-$r$ holds for some $1\leq r\leq s+1$. 
	For convenience denote $\x:=(n,\w,h_{1},\dots,h_{r-1},h_{r+1},\dots,h_{s+1},h'_{1},\dots,h'_{r-1})$, and let $(\x;h)$ be the vector obtained by inserting $h$ in the middle of $h_{r-1}$ and $h_{r+1}$.
Let $\tilde{\Lambda}_{4,r}$ denote the set of $(\x,h'_{r},h''_{r})\in(\V)^{s+r+6}$ such that $(\x;h'_{r}),(\x;h''_{r})\in\Lambda_{4,r}$.
	Since $\Lambda_{4}$  is a pure and consistent $M$-set, by Propositions \ref{3:yy3} and \ref{3:yy33}, $\Lambda_{4,r}$ and $\tilde{\Lambda}_{4,r}$ are  pure and consistent $M$-sets of total co-dimension at most $\binom{2s+8}{2}$.  
	We may then write $\Lambda_{4,r}=V(\mathcal{J})$ for some consistent $(M,s+r+5)$-family $\mathcal{J}\subseteq\F_{p}[n,w_{2},\dots,w_{5},h_{1},\dots,h_{s+1},h'_{1},\dots,h'_{r-1}]$ of total dimension at most $\binom{2s+8}{2}$.   Let $(\mathcal{J}',\mathcal{J}'')$ be an  $\{n,w_{2},\dots,w_{5},h_{1},\dots,h_{r-1},h_{r+1},\dots,h_{s+1},h'_{1},\dots,h'_{r-1}\}$-decomposition of $\mathcal{J}$. Let $\Lambda'_{4,r}$ be the set of $\x\in(\V)^{s+r+4}$  such that $(\x;h_{r})\in V(\mathcal{J}')$ for all $h_{r}\in \V$. For $\x\in(\V)^{s+r+4}$, let  $\Lambda_{4,r}(\x)$ denote the set of $h_{r}\in \V$ such that $(\x;h_{r})\in V(\mathcal{J}'')$.
	Since $d\geq (2s+7)(2s+6)+1$,
	by pulling out the $h_{r}$-independent term $$\Delta_{h_{s+1}}\dots\Delta_{h_{r+1}}\Delta_{h_{r-1}}\dots\Delta_{h_{1}}\overline{f'}(n)\psi_{r}(n,\w,h_{1},\dots,h_{s+1},h'_{1},\dots,h'_{r-1})$$ in (\ref{3:pprr}) and apply the Cauchy-Schwartz inequality and Theorem \ref{3:ct}, we have that 
	\begin{equation}\nonumber
	\begin{split}
	&\quad 1\ll_{d,\e}
	\E_{\x\in \Lambda'_{4,r}}
	\Bigl\vert\E_{h_{r}\in \Lambda_{4,r}(\x)}\Delta_{h_{s+1}}\dots\Delta_{h_{r+1}}\Delta_{h_{r-1}}\dots\Delta_{h_{1}}f'(n+h_{r})\prod_{i=r+1}^{s+1}\psi_{i}(\x;h_{r})\Bigr\vert^{2}
	\\&=\Bigl\vert\E_{(\x;h'_{r},h''_{r})\in \tilde{\Lambda}_{4,r}}\Delta_{h_{s+1}}\dots\Delta_{h_{r+1}}\Delta_{h_{r-1}}\dots\Delta_{h_{1}}(f'(n+h''_{r})\overline{f'}(n+h'_{r}))\cdot\prod_{i=r+1}^{s+1}\psi_{i}(\x;h''_{r})\overline{\psi}_{i}(\x;h'_{r})\Bigr\vert
	\\&=\Bigl\vert\E_{(\x;h'_{r}),(\x;h'_{r}+h_{r})\in \Lambda_{4,r}}\Delta_{h_{s+1}}\dots\Delta_{h_{r+1}}\Delta_{h_{r}}\Delta_{h_{r-1}}\dots\Delta_{h_{1}}f'(n+h'_{r})\cdot\prod_{i=r+1}^{s+1}\psi_{i}(\x;h'_{r}+h_{r})\overline{\psi}_{i}(\x;h'_{r})\Bigr\vert.
	\\&\qquad 
	\end{split}
	\end{equation}
	
	For all $r+1\leq i\leq s+1$, since $\psi_{i}$ is independent of $h_{i}$ and bounded by 1,
 $\psi_{i}(\x;h'_{r}+h_{r})\overline{\psi}_{i}(\x;h'_{r})$ is independent of $h_{i}$ and bounded by 1.
On the other hand, note that $(\x;h'_{r}),(\x;$ $h'_{r}+h_{r})\in \Lambda_{4,r}$ if and only if
	$$(n-h'_{1}-\dots-h'_{r-1},\w,h'_{1}+\e_{1}h_{1},\dots,h'_{r-1}+\e_{r-1}h_{r-1},h'_{r},h_{r+1},\dots,h_{s+1})\in \Lambda_{4}$$
	and 
	$$(n-h'_{1}-\dots-h'_{r-1},\w,h'_{1}+\e_{1}h_{1},\dots,h'_{r-1}+\e_{r-1}h_{r-1},h'_{r}+h_{r},h_{r+1},\dots,h_{s+1})\in \Lambda_{4}$$
	for all $\e_{1},\dots,\e_{r-1}\in \{0,1\}$. 	Writing $m=n+h'_{r}$, this is equivalent of saying that 
	$$(m-h'_{1}-\dots-h'_{r},\w,h'_{1}+\e_{1}h_{1},\dots,h'_{r}+\e_{r}h_{r},h_{r+1},\dots,h_{s+1})\in \Lambda_{4}$$
	for all $\e_{1},\dots,\e_{r-1}\in \{0,1\}$, or equivalently, $(m,\w,h_{1},\dots,h_{s+1},h'_{1},\dots,h'_{r})\in  \Lambda_{4,r+1}$. So we have that Property-$(r+1)$ holds. 
	
	Inductively, we have that Property-$(s+2)$ holds. 
	Since $\Lambda_{4}$  is a pure and consistent $M$-set, by Propositions \ref{3:yy3} and \ref{3:yy33}, $\Lambda_{4,s+2}=V(\tilde{\mathcal{J}})$ for some pure and consistent $(M,2s+7)$-family $\tilde{\mathcal{J}}\subseteq \F_{p}[n,w_{2},\dots,w_{5},h_{1},\dots,h_{s+1},h'_{1},\dots,h'_{s+1}]$ of total dimension at most $\binom{2s+8}{2}$. Let $(\tilde{\mathcal{J}}',\tilde{\mathcal{J}}'')$ be an  $\{n,h_{1},\dots,h_{s+1}\}$-decomposition of $\tilde{\mathcal{J}}$. 
	For convenience denote $\h'=(h'_{1},\dots,h'_{s+1})$. 
	Let $\Lambda''_{4,s+2}$ be the set of $(n,\h)\in(\V)^{s+2}$  such that $(n,\w,\h,\h')\in V(\tilde{\mathcal{J}}')$ for all $(\w,\h')\in (\V)^{s+5}$. 
	For $(n,\h)\in(\V)^{s+2}$, let  $\Lambda_{4,r}(n,\h)$ denote the set of $(\w,\h')\in (\V)^{s+5}$ such that $(n,\w,\h,\h')\in V(\tilde{\mathcal{J}}'')$. It is not hard to see that $\Lambda''_{4,s+2}=\Gow_{s+1}(V(M))$.
	By Property-$(s+2)$ and Theorem \ref{3:ct},	
	\begin{equation}\nonumber
	\begin{split}
	&\quad 1\ll_{d,\e}
	\vert\E_{(n,\w,\h,\h')\in \Lambda_{4,s+2}}\Delta_{h_{s+1}}\dots\Delta_{h_{1}}f'(n)\vert
	\\&=\vert\E_{(n,\h)\in \Gow_{s+1}(V(M))}\Delta_{h_{s+1}}\dots\Delta_{h_{1}}f'(n)\E_{(\w,\h')\in\Lambda_{4,s+2}(x,\h)}1\vert+O(p^{-1/2})
	\\&=\vert\E_{(n,\h)\in \Gow_{s+1}(V(M))}\Delta_{h_{s+1}}\dots\Delta_{h_{1}}(f\phi')(n)\vert+O(p^{-1/2})
	=\Vert f\phi'\Vert^{2^{s+1}}_{U^{s+1}(V(M))}+O(p^{-1/2}).
	\end{split}
	\end{equation}
	Since $\SGI(s)$ holds by assumption, there exists $\phi''\in \Nil_{p}^{s;O_{d,\e}(1),1}(\V)$   such that 
	$$\vert\E_{n\in V(M)}f(n)\phi'(n)\phi''(n)\vert\gg_{d,\e} 1.$$
This completes the proof of $\SGI(s+1)$ since $\phi'\phi''\in \Nil_{p}^{s+1;O_{d,\e}(1),1}(\V)$.

\appendix

\section{Results from previous papers}\label{3:s:AppA}

In this appendix, we collect some results proved in \cite{SunA,SunB} on quadratic forms which are used in this paper.

\subsection{Liftings of polynomials}

We refer the readers to Section \ref{3:s:defn} for the definition of liftings of polynomials.

 \begin{lem}[Lemma %\ref{1:ivie} 
 2.8 of \cite{SunA}]\label{3:ivie}
Let $d\in\N$, $p$ be a prime, and $f\in\poly(\Z^{d}\to \Z)$ be an integer valued polynomial of degree smaller than $p$. There exist integer valued polynomial $f_{1}\in\poly(\Z^{d}\to \Z)$ and integer coefficient polynomial $f_{2}\in\poly(\Z^{d}\to \Z)$ both having degrees at most $\deg(f)$ such that $\frac{1}{p}f=f_{1}+\frac{1}{p}f_{2}$.
\end{lem}

We refer the readers to Appendix \ref{3:s:AppA4} for the definition of $d$-integral  linear transformations. 

 \begin{lem}[Lemma %\ref{1:lifting2}
 9.1 of \cite{SunA}]\label{3:lifting2}
Every $d$-integral linear transformation  $\tilde{L}\colon(\Z^{d})^{k}\to (\frac{1}{p}\Z^{d})^{k'}$ induces a $d$-integral linear transformation  $L\colon(\V)^{k}\to (\V)^{k'}$. Conversely, every $d$-integral linear transformation  $L\colon(\V)^{k}\to (\V)^{k'}$ admits a regular lifting $\tilde{L}\colon(\Z^{d})^{k}\to (\frac{1}{p}\Z^{d})^{k'}$ which is a $d$-integral linear transformation.

Moreover, we have that $p\tilde{L}\circ \tau\equiv \tau\circ L \mod p(\Z^{d})^{k'}$.
\end{lem}

\subsection{The rank of quadratic forms}\label{3:s:AppA2}
 
 \begin{lem}[Lemma %\ref{1:or} 
 4.2 of \cite{SunA}]\label{3:or}
	Let $M\colon\V\to\F_{p}$ be a quadratic form associated with the matrix $A$, and $x,y,z\in \V$. Suppose that $M(x)=M(x+y)=M(x+z)=0$. Then $M(x+y+z)=0$ if and only if $(yA)\cdot z=0$.
\end{lem}

Let $M\colon\V\to\F_{p}$ be a quadratic form associated with the matrix $A$ and $V$ be a subspace of $\V$. Let $V^{\pp}$ denote the set of $\{n\in\V\colon (mA)\cdot n=0 \text{ for all } m\in V\}$.
		A subspace $V$ of $\V$ is \emph{$M$-isotropic} if $V\cap V^{\pp}\neq\{\bold{0}\}$. 
	We say that a tuple $(h_{1},\dots,h_{k})$ of vectors in $\V$ is \emph{$M$-isotropic} if the span of $h_{1},\dots,h_{k}$ is an $M$-isotropic subspace.
	We say that a subspace or tuple of vectors is \emph{$M$-non-isotropic} if it is not $M$-isotropic.

\begin{prop}[Proposition %\ref{1:iissoo} 
4.8 of \cite{SunA}]\label{3:iissoo}
Let $M\colon \V\to\F_{p}$ be a quadratic form and $V$ be a subspace of $\V$ of co-dimension $r$, and $c\in\V$.
	\begin{enumerate}[(i)]
		\item We have $\dim(V\cap V^{\pp})\leq \min\{d-\rank(M)+r,d-r\}$.
		\item The rank of $M\vert_{V+c}$ equals to $d-r-\dim(V\cap V^{\pp})$ (i.e. $\dim(V)-\dim(V\cap V^{\pp})$). 		
		\item The rank of $M\vert_{V+c}$ is at most $d-r$ and at least $\rank(M)-2r$.
		\item $M\vert_{V+c}$ is non-degenerate (i.e. $\rank(M\vert_{V+c})=d-r$) if and only if $V$ is not an $M$-isotropic subspace.
	\end{enumerate}	
\end{prop}	
 
  \begin{lem}[Lemma %\ref{1:iiddpp}
  4.11 of \cite{SunA}]\label{3:iiddpp}
	Let $d,k,K\in\N_{+}$, $p$ be a prime dividing $K$, and $M\colon\V\to\F_{p}$ be a non-degenerate quadratic form.
	\begin{enumerate}[(i)]
%		\item The number of tuples $(h_{1},\dots,h_{k})\in (\V)^{k}$ such that $h_{1},\dots,h_{k}$ are linearly  dependent is at most $kp^{(d+1)(k-1)}$.
	%	\item The number of $M$-isotropic tuples $(h_{1},\dots,h_{k})\in (\V)^{k}$ is at most $O_{d,k}(p^{kd-1})$. 
		\item %For $K\in\N_{+}$ with $K\in p\Z$, if $d\geq k$, 
		The number of tuples $(h_{1},\dots,h_{k})\in (\Vk)^{k}$ such that $\iota(h_{1}),\dots,\iota(h_{k})$ are  linearly  dependent is at most %$kK^{(d+1)(k-1)}(\frac{K}{p})^{d-k+1}\leq kK^{(d+1)(k-1)}$.
		$k\frac{K^{dk}}{p^{d-k+1}}$.
		\item %For $K\in\N_{+}$ with $K\in p\Z$, if $d\geq k$, 
		The number of tuples $(h_{1},\dots,h_{k})\in (\Vk)^{k}$ such that $\iota(h_{1}),\dots,\iota(h_{k})$ are $M$-isotropic is at most %$kK^{(d+1)(k-1)}(\frac{K}{p})^{d-k+1}\leq kK^{(d+1)(k-1)}$.
		$O_{d,k}(\frac{K^{dk}}{p})$.
		\footnote{Although  Lemma 4.11 of \cite{SunA} was stated for the case $K=p$, the general case can be deduced immidiately.}
	\end{enumerate}	
\end{lem}

\subsection{Some basic counting properties}
  
  We refer the readers to Section \ref{3:s:defn} for the notations used in this section.
 
\begin{lem}[Lemma %\ref{1:ns}
4.10 of \cite{SunA}]\label{3:ns}
Let $P\in\poly(\V\to\F_{p})$ be  of degree at most $r$.
	Then $\vert V(P)\vert\leq O_{d,r}(p^{d-1})$ unless $P\equiv 0$.
\end{lem}

 \begin{lem}[Corollaries %\ref{1:counting01}
 4.14 and   %\ref{1:countingh}
 C.7 of \cite{SunA}]\label{3:countingh}
Let $d\in \N_{+},r,s\in\N$ and $p$ be a prime. Let $M\colon\V\to\F_{p}$ be a quadratic form and $V+c$ be an affine subspace of $\V$ of co-dimension $r$. If either $\rank(M\vert_{V+c})$ or $\rank(M)-2r$ is at least $s^{2}+s+3$, then
	$$\vert \Gow_{s}(V(M)\cap (V+c))\vert=p^{(s+1)(d-r)-(\frac{s(s+1)}{2}+1)}(1+O_{s}(p^{-1/2})).$$
	
	In particular, If either $\rank(M\vert_{V+c})$ or $\rank(M)-2r$ is at least $3$, then
	$$\vert V(M)\cap (V+c)\vert=p^{d-r-1}(1+O(p^{-1/2})).$$
\end{lem}

\begin{lem}[Lemma %\ref{1:changeh} 
4.18 of \cite{SunA}]\label{3:changeh}
Let $s\in\N$, $M\colon\V\to\F_{p}$ be a quadratic form associated with the matrix $A$, and $V+c$ be an affine subspace of $\V$  of dimension $r$. For $n,h_{1},\dots,h_{s}\in\V$, we have that $(n,h_{1},\dots,h_{s})\in \Gow_{s}(V(M)\cap (V+c))$ if and only if 
	\begin{itemize}
		\item $n\in V+c$, $h_{1},\dots,h_{s}\in V$;
		\item $n\in V(M)^{h_{1},\dots,h_{s}}$;
		\item $(h_{i}A)\cdot h_{j}=0$ for all $1\leq i,j\leq s, i\neq j$.
	\end{itemize}	
	
	In particular, let $\phi\colon\F_{p}^{r}\to V$ be any bijective linear transformation, $M'(m):=M(\phi(m)+c)$. We have that $(n,h_{1},\dots,h_{s})\in \Gow_{s}(V(M)\cap (V+c))$ if and only if 
	$(n,h_{1},\dots,h_{s})=(\phi(n')+c,\phi(h'_{1}),\dots,\phi(h'_{s}))$
	for some $(n',h'_{1},\dots,h'_{s})\in \Gow_{s}(V(M'))$.
\end{lem}

 \subsection{$M$-families and $M$-sets}\label{3:s:AppA4}

We say that a linear transformation  $L\colon(\V)^{k}\to (\V)^{k'}$ is \emph{$d$-integral} if there exist $a_{i,j}\in\F_{p}$ for $1\leq i\leq k$ and $1\leq j\leq k'$ such that 
$$L(n_{1},\dots,n_{k})=\Bigl(\sum_{i=1}^{k}a_{i,1}n_{i},\dots,\sum_{i=1}^{k}a_{i,k'}n_{i}\Bigr)$$
for all $n_{1},\dots,n_{k}\in \V$. 
Let $L\colon(\V)^{k}\to \V$ be a $d$-integral linear transformation given by
$L(n_{1},\dots,n_{k})=\sum_{i=1}^{k}a_{i}n_{i}$ for some  $a_{i}\in\F_{p},1\leq i\leq k$. We say that $L$ is the $d$-integral linear transformation \emph{induced} by $(a_{1},\dots,a_{k})\in\F_{p}^{k}$.

Similarly, we say that a linear transformation  $L\colon(\Z^{d})^{k}\to (\frac{1}{p}\Z^{d})^{k'}$ is \emph{$d$-integral} if there exist $a_{i,j}\in\Z/p$ for $1\leq i\leq k$ and $1\leq j\leq k'$ such that 
$$L(n_{1},\dots,n_{k})=\Bigl(\sum_{i=1}^{k}a_{i,1}n_{i},\dots,\sum_{i=1}^{k}a_{i,k'}n_{i}\Bigr)$$
for all $n_{1},\dots,n_{k}\in \Z^{d}$.

Let $d\in\N_{+}$, $p$  be a prime, $M\colon\V\to\F_{p}$ be a quadratic form with $A$ being the associated matrix.  We say that a function $F\colon(\V)^{k}\to\F_{p}$ is an \emph{$(M,k)$-integral quadratic function} 
if 
\begin{equation}\label{3:thisisf2}
F(n_{1},\dots,n_{k})=\sum_{1\leq i\leq j\leq k}b_{i,j}(n_{i}A)\cdot n_{j}+\sum_{1\leq i\leq k} v_{i}\cdot n_{i}+u
\end{equation}
for some $b_{i,j}, u\in \F_{p}$ and $v_{i}\in\F_{p}^{d}$.
We say that an $(M,k)$-integral quadratic function $F$ is \emph{pure}
if $F$ can be written in the form of (\ref{3:thisisf2}) with $v_{1}=\dots=v_{k}=\bold{0}$. 
We say that  an $(M,k)$-integral quadratic function $F\colon(\V)^{k}\to\F_{p}$ is \emph{nice}
if 
\begin{equation}\nonumber
F(n_{1},\dots,n_{k})=\sum_{1\leq i\leq k'}b_{i}(n_{k'}A)\cdot n_{i}+u
\end{equation}
for some $0\leq k'\leq k$, $b_{i}, u\in \F_{p}$.
 
For $F$ given in (\ref{3:thisisf2}), denote 
$$v_{M}(F):=(b_{k,k},b_{k,k-1},\dots,b_{k,1},v_{k},b_{k-1,k-1},\dots,b_{k-1,1},v_{k-1},\dots,b_{1,1},v_{1},u)\in\F_{p}^{\binom{k+1}{2}+kd+1},$$
and
$$v'_{M}(F):=(b_{k,k},b_{k,k-1},\dots,b_{k,1},v_{k},b_{k-1,k-1},\dots,b_{k-1,1},v_{k-1},\dots,b_{1,1},v_{1})\in\F_{p}^{\binom{k+1}{2}+kd}.$$
Informally, we say that $b_{i,j}$ is the $n_{i}n_{j}$-coefficient, $v_{i}$ is the $n_{i}$-coefficient, and $u$ is the constant term coefficient for these vectors.

An \emph{$(M,k)$-family} is a collections of $(M,k)$-integral quadratic functions.
Let $\mathcal{J}=\{F_{1},\dots$, $F_{r}\}$ be an $(M,k)$-family.
\begin{itemize}
    \item We say that $\mathcal{J}$ is \emph{pure} if all of $F_{1},\dots,F_{r}$ are pure.
    \item We say that $\mathcal{J}$ is \emph{consistent} if $(0,\dots,0,1)$ does not belong to the span of $v_{M}(F_{1}),$ $\dots,$ $v_{M}(F_{r})$, or equivalently, there is no linear combination of $F_{1},\dots,F_{r}$ which is a constant nonzero function, or equivalently, for all $c_{1},\dots,c_{r}\in\F_{p}$, we have
    $$c_{1}v'_{M}(F_{1})+\dots+c_{r}v'_{M}(F_{r})=\bold{0}\Rightarrow c_{1}v_{M}(F_{1})+\dots+c_{r}v_{M}(F_{r})=\bold{0}.$$
    \item We say that $\mathcal{J}$ is \emph{independent} if $v'_{M}(F_{1}),\dots,v'_{M}(F_{r})$ are linearly independent, or equivalently, there is no nontrivial linear combination of $F_{1},\dots,F_{r}$ which is a constant function, or equivalently, for all $c_{1},\dots,c_{r}\in\F_{p}$, we have
    $$c_{1}v'_{M}(F_{1})+\dots+c_{r}v'_{M}(F_{r})=\bold{0}\Rightarrow c_{1}=\dots=c_{r}=0.$$
\item We say that $\mathcal{J}$ is \emph{nice} if there exist some bijective $d$-integral  linear transformation $L\colon(\V)^{k}\to(\V)^{k}$ and some $v\in(\V)^{k}$ such that $F_{i}(L(\cdot)+v)$ is nice for all $1\leq i\leq r$.
\end{itemize}

The dimension of the span of $v'_{M}(F_{1}),\dots,v'_{M}(F_{r})$ is called the \emph{dimension} of an $(M,k)$-family $\{F_{1},\dots,F_{r}\}$.

When there is no confusion, we call an $(M,k)$-family to be an \emph{$M$-family} for short.

We say that a subset $\Omega$ of $(\V)^{k}$ is an  \emph{$M$-set} if there exists an $(M,k)$-family $\{F_{i}\colon (\V)^{k}\to\V\colon 1\leq i\leq r\}$
 such that $\Omega=\cap_{i=1}^{r}V(F_{i})$. 
 We call either $\{F_{1},\dots,F_{r}\}$ or the ordered set $(F_{1},\dots,F_{r})$ an \emph{$M$-representation} of $\Omega$. 
 
  Let $\P\in\CP$.
We say that   $\Omega$  is $\P$  if one can choose the  $M$-family $\{F_{1},\dots,F_{r}\}$ to be $\P$.
 We say that the $M$-representation $(F_{1},\dots,F_{r})$ is $\P$ if $\{F_{i}\colon 1\leq i\leq r\}$ is $\P$.
  We say that $r$ is the \emph{dimension} of the $M$-representation $(F_{1},\dots,F_{r})$.
  The \emph{total co-dimension} of a consistent $M$-set $\Omega$, denoted by $r_{M}(\Omega)$, is the minimum of the dimension of the independent $M$-representations of $\Omega$. It was shown in Proposition %\ref{1:yy33} 
  C.8 of \cite{SunA} that $r_{M}(\Omega)$ is independent of the choice of the independent $M$-representation if $d$ and $p$ are sufficiently large.

  	We say that an independent $M$-representation $(F_{1},\dots,F_{r})$ of $\Omega$ is \emph{standard} if
	the matrix 
		$\begin{bmatrix}
		v_{M}(F_{1})\\
		\dots \\
		v_{M}(F_{r})
		\end{bmatrix}$ is in the reduced row echelon form (or equivalently, the matrix 
		$\begin{bmatrix}
		v'_{M}(F_{1})\\
		\dots \\
		v'_{M}(F_{r})
		\end{bmatrix}$ is in the reduced row echelon form).
	If $(F_{1},\dots,F_{r})$ is a standard $M$-representation of $\Omega$, then we may relabeling $(F_{1},\dots,F_{r})$ as $$(F_{k,1},\dots,F_{k,r_{k}},F_{k-1,1},\dots,F_{k-1,r_{k-1}},\dots,F_{1,1},\dots,F_{1,r_{1}})$$
  for some $r_{1},\dots,r_{k}\in \N$ such that $F_{i,j}$ is non-constant with respect to $n_{i}$ and is independent of $n_{i+1},\dots,n_{k}$.
	We also call $(F_{i,j}\colon 1\leq i\leq k, 1\leq j\leq r_{i})$	a \emph{standard $M$-representation} of $\Omega$. The vector $(r_{1},\dots,r_{k})$ is called the \emph{dimension vector} of this representation.  
	
	\begin{conv}\label{1:fpism}
	We allow $(M,k)$-families, $M$-representations and   dimension vectors to be empty. In particular, $(\V)^{k}$ is considered as a nice and consistent $M$-set with total co-dimension zero. 
	\end{conv}

\begin{prop}[Proposition %\ref{1:yy3}
B.3 of \cite{SunA}]\label{3:yy3}
Let $d,k,r\in\N_{+}$, $p$ be a prime, $M\colon\V\to\F_{p}$ be a quadratic form and $\mathcal{J}=\{F_{1},\dots,F_{r}\}, F_{i}\colon(\V)^{k}\to \F_{p}, 1\leq i\leq r$ be an $(M,k)$-family.
	\begin{enumerate}[(i)]
	\item For any subset $I\subseteq \{1,\dots,r\}$, $\{F_{i}\colon i\in I\}$ is an    $(M,k)$-family.
	\item Let $I\subseteq \{1,\dots,r\}$ be a subset such that all of $F_{i},i\in I$ are independent of the last $d$-variables. Then writing $G_{i}(n_{1},\dots,n_{k-1}):=F_{i}(n_{1},\dots,n_{k-1})$, we have that $\{G_{i}\colon i\in I\}$ is an $(M,k-1)$-family.  
	\item For any bijective $d$-integral linear transformation  $L\colon (\V)^{k}\to(\V)^{k}$ and any $v\in(\V)^{k}$, $\{F_{1}(L(\cdot)+v),\dots,F_{r}(L(\cdot)+v)\}, F_{i}\colon(\V)^{k}\to \F_{p}, 1\leq i\leq r$ is an $(M,k)$-family.
		\item Let $1\leq k'\leq k$ and $1\leq r'\leq r$ be such that $F_{1},\dots,F_{r'}$ are independent of $n_{k-k'+1},\dots,n_{r}$ and that every nontrival linear combination of $F_{r'+1},\dots,F_{r}$ dependents nontrivially on some of $n_{k-k'+1},\dots,n_{r}$. 
		Define $G_{i},H_{i}\colon (\V)^{k+k'}\to \F_{p}, 1\leq i\leq r$ by 
		$$\text{$G_{i}(n_{1},\dots,n_{k+k'}):=F_{i}(n_{1},\dots,n_{k})$ and $H_{i}(n_{1},\dots,n_{k+k'}):=F_{i}(n_{1},\dots,n_{k-k'},n_{k+1},\dots,n_{k+k'})$}.$$ Then $\{F_{1},\dots,F_{r'},G_{r'+1},\dots,G_{r},H_{r'+1},\dots,H_{r}\}$\footnote{Here we regard $F_{1},\dots,F_{r'}$ as functions of $n_{1},\dots,n_{k+k'}$, which acutally depends only on $n_{1},\dots,n_{k-k'}$.} is an $(M,k+k')$-family. 
	\end{enumerate}
Moreover, if $\mathcal{J}$ is nice, then so are the sets mentioned in Parts (i)-(iii);
if $\mathcal{J}$ is 
consistent/independent, then so are the sets mentioned in Parts (i)-(iv);  
if $\mathcal{J}$ is pure, then so are the sets mentioned in Parts (i), (ii), (iv), and so is the set mentioned in Part (iii) when $v=\bold{0}$.
\end{prop}

\begin{prop}[Proposition %\ref{1:yy33} 
C.8 of \cite{SunA}]\label{3:yy33}
Let $d,k,r\in\N_{+}$, $p$ be a prime, $M\colon\V\to\F_{p}$ be a quadratic form and $\Omega\subseteq (\V)^{k}$ be a consistent $M$-set. Suppose that $\Omega=\cap_{i=1}^{r}V(F_{i})$ for some consistent $(M,k)$-family $\{F_{1},\dots,F_{r}\}$. 
Let $1\leq k'\leq k$.
Suppose that $\rank(M)\geq 2r+1$ and $p\gg_{k,r} 1$.
\begin{enumerate}[(i)]
\item The dimension of all independent $M$-representations of $\Omega$ equals to $r_{M}(\Omega)$. 
\item We have that $r_{M}(\Omega)\leq r$, and that $r_{M}(\Omega)=r$ if the $(M,k)$-family $\{F_{1},\dots,F_{r}\}$ is independent.
\item For all $I\subseteq\{1,\dots,r\}$, the set $\Omega':=\cap_{i\in I}V(F_{i})$ is a consistent  $M$-set such that $r_{M}(\Omega')\leq r_{M}(\Omega)$. Moreover, if  the $(M,k)$-family $\{F_{1},\dots,F_{r}\}$ is independent, then $r_{M}(\Omega')=\vert I\vert$.
\item For any bijective $d$-integral linear transformation $L\colon(\V)^{k}\to(\V)^{k}$ and $v\in(\V)^{k}$, we have that $L(\Omega)+v$ is a consistent $M$-set and that $r_{M}(L(\Omega)+v)=r_{M}(\Omega)$.   
\item Assume that $\Omega$ admits a standard $M$-representation with dimension vector $(r_{1},\dots,$ $r_{k})$, Then for all $1\leq k'\leq k$, the set $$\{(n_{1},\dots,n_{k+k'})\in(\V)^{k+k'}\colon (n_{1},\dots,n_{k}), (n_{1},\dots,n_{k-k'},n_{k+1},\dots,n_{k+k'})\in\Omega\}$$ admits a standard $M$-representation with dimension vector $(r_{1},\dots,r_{k'},r_{k+1},\dots,r_{k'})$. 
\item Assume that $\Omega$ admits a standard $M$-representation with dimension vector $(r_{1},\dots,$ $r_{k})$.
For $I=\{1,\dots,k'\}$ for some $1\leq k'\leq k$, the $I$-projection $\Omega_{I}$ of $\Omega$ 
admits a standard $M$-representation with dimension vector $(r_{1},\dots,r_{k'})$. 
\item If $\Omega$ is a nice and consistent $M$-set or a  pure and consistent $M$-set, then $r_{M}(\Omega)\leq \binom{k+1}{2}$.
\end{enumerate}
\end{prop}

\subsection{Fubini's theorem for $M$-sets}\label{3:s:AppA5}		
Let $n_{i}=(n_{i,1},\dots,n_{i,d})\in\V$ denote a $d$-dimensional variable for $1\leq i\leq k$. For convenience we denote $\F^{d}_{p}[n_{1},\dots,n_{k}]$ to be the polynomial ring $\F_{p}[n_{1,1},\dots,n_{1,d},\dots,n_{k,1},\dots,n_{k,d}]$.
Let $J,J',J''$ be  finitely generated ideals of the polynomial ring $\F^{d}_{p}[n_{1},\dots,n_{k}]$ and $I\subseteq \{n_{1},\dots,n_{k}\}$. 
Suppose that $J=J'+J''$, $J'\cap J''=\{0\}$, all the polynomials in $J'$ are independent of  $\{n_{1},\dots,n_{k}\}\backslash I$, and all the non-zero polynomials in $J''$ depend nontrivially on $\{n_{1},\dots,n_{k}\}\backslash I$. Then we say that $J'$ is an  \emph{$I$-projection} of $J$ and that $(J',J'')$ is an \emph{$I$-decomposition} of $J$.
 It was proved in Proposition %\ref{1:ip}
 C.1 of \cite{SunA} that the $I$-projection of  $J$ exists and is unique.

Let $\mathcal{J}$, $\mathcal{J}'$ and $\mathcal{J}''$ be  finite subsets of $\F^{d}_{p}[n_{1},\dots,n_{k}]$, and $I\subseteq \{n_{1},\dots,n_{k}\}$. 
Let $J,J'$, and $J''$ be the ideals generated by $\mathcal{J}$, $\mathcal{J}'$ and $\mathcal{J}''$ respectively. If $J'$ is an $I$-projection of $J$ with $(J',J'')$ being an $I$-decomposition of $J$, then we say that $\mathcal{J}'$ is an \emph{$I$-projection} of $\mathcal{J}$ and that $(\mathcal{J}',\mathcal{J}'')$ is an \emph{$I$-decomposition} of $\mathcal{J}$.
Note that the $I$-projection $\mathcal{J}'$ of $\mathcal{J}$ is not unique. However, by Proposition %\ref{1:ip}
C.1 of \cite{SunA}, the ideal generated by $\mathcal{J}'$ and the set of zeros $V(\mathcal{J}')$ are unique.

If $I$ is a subset of $\{1,\dots,k\}$, for convenience we say that $J'$ is an \emph{$I$-projection} of $J$ if $J'$ is an $\{n_{i}\colon i\in I\}$-projection of $J$. Similarly, we say that $(J',J'')$ is an \emph{$I$-decomposition} of $J$ if $(J',J'')$ is an $\{n_{i}\colon i\in I\}$-decomposition of $J$. Here $J,J',J''$ are either ideals or finite subsets of the polynomial ring.

We now define the projections of $M$-sets.
Let $\mathcal{J}\subseteq \F_{p}[n_{1},\dots,n_{k}]$ be a consistent $(M,k)$-family and 
let $\Omega=V(\mathcal{J})\subseteq(\V)^{k}$.
Let $I\cup J$ be a partition of $\{1,\dots,k\}$ (where $I$ and $J$ are non-empty), and $(\mathcal{J}_{I},\mathcal{J}'_{I})$ be an $I$-decomposition of $\mathcal{J}$.
Let $\Omega_{I}$ denote the set of $(n_{i})_{i\in I}\in(\V)^{\vert I\vert}$ such that $f(n_{1},\dots,n_{k})=0$ for all $f\in \mathcal{J}_{I}$ and $(n_{j})_{j\in J}\in(\V)^{\vert J\vert}$. 
Note that all $f\in \mathcal{J}_{I}$ are independent of $(n_{j})_{j\in J}$, and that $\Omega_{I}$ is independent of the choice of the $I$-decomposition.
We say that $\Omega_{I}$ is an \emph{$I$-projection} of $\Omega$.

 For $(n_{i})_{i\in I}\in (\V)^{\vert I\vert}$, let $\Omega_{I}((n_{i})_{i\in I})$ be the set of $(n_{j})_{j\in J}\in(\V)^{\vert J\vert}$ such that $f(n_{1},\dots,n_{k})=0$ for all $f\in \mathcal{J}'_{I}$. By construction, for any $(n_{i})_{i\in I}\in \Omega_{I}$, we have that $(n_{j})_{j\in J}\in\Omega_{I}((n_{i})_{i\in I})$ if and only if $(n_{1},\dots,n_{k})\in\Omega$. So for all  $(n_{i})_{i\in I}\in \Omega_{I}$, $\Omega_{I}((n_{i})_{i\in I})$ is independent of the choice of the $I$-decomposition.

\begin{thm}[Theorem %\ref{1:ct}
C.3 of \cite{SunA}]\label{3:ct}
Let $d,k\in\N_{+}$, $r_{1},\dots,r_{k}\in\N$ and $p$ be a prime.  Set $r:=r_{1}+\dots+r_{k}$.   
	Let  $M\colon\V\to\F_{p}$ be a  quadratic form with $\rank(M)\geq 2r+1$, and $\Omega\subseteq (\V)^{k}$ be a consistent $M$-set admitting a  standard $M$-representation of $\Omega$ with dimension vector $(r_{1},\dots,r_{k})$. 
		\begin{enumerate}[(i)]
		\item We have $\vert\Omega\vert=p^{dk-r}(1+O_{k,r}(p^{-1/2}))$;
		\item If $I=\{1,\dots,k'\}$ for some $1\leq k'\leq k-1$, then 
		$\Omega_{I}$ is a consistent $M$-set admitting a  standard  $M$-representation with dimension vector $(r_{1},\dots,r_{k'})$. Moreover,
		for all but at most $(k-k')r'_{k'}p^{dk'+r'_{k'}-1-\rank(M)}$ many $(n_{i})_{i\in I}\in(\V)^{k'}$, we have that $\Omega_{I}((n_{i})_{i\in I})$ has a standard $M$-representation with dimension vector $(r_{k'+1},\dots,r_{k})$  and that $\vert\Omega_{I}((n_{i})_{i\in I})\vert=p^{d(k-k')-(r_{k'+1}+\dots+r_{k})}(1+O_{k,r}(p^{-1/2}))$, where $r'_{k'}:=\max_{k'+1\leq i\leq k}r_{i}$. 
		\item 
		If $k\geq 2$, and $I\cup J$ is a partition of $\{1,\dots,k\}$ (where  $I$ and $J$ are non-empty), then for any function $f\colon \Omega\to\C$ with norm bounded by 1,   we have that
		\begin{equation}\nonumber
		\E_{(n_{1},\dots,n_{k})\in\Omega}f(n_{1},\dots,n_{k})=\E_{(n_{i})_{i\in I}\in\Omega_{I}}\E_{(n_{j})_{j\in J}\in\Omega_{I}((n_{i})_{i\in I})}f(n_{1},\dots,n_{k})+O_{k,r}(p^{-1/2}).
		\end{equation}
	\end{enumerate} 
\end{thm}

 \subsection{Intrinsic definitions for polynomials}\label{3:s:AppA40}
 
   We refer the readers to Section \ref{3:s:defn} for the notations used in this section.

\begin{prop}[Proposition 7.12 of \cite{SunA}]\label{3:att3}
  		Let $d,s\in\N_{+}, k\in\N$, 
		$p\gg_{d} 1$ be a prime, $c\in\V$, 
		$V$ be a subspace of $\V$ of dimension $k$ with a basis $\iota(h_{1}),\dots,\iota(h_{k})$ for some $h_{1},\dots,h_{k}\in\Z^{d}$, 
		$M\colon\V\to\F_{p}$ be a non-degenerate quadratic form with $\rank(M\vert_{V^{\pp}})\geq s^{2}+s+3$ associated with a matrix $A$. Let $M_{0}\colon\V\to\F_{p}$ be the quadratic form given by $M_{0}(n):=(nA)\cdot n$.
	Then for any  $g\in\poly(\Z^{d}\to \Q)$  of degree at most $s$, the followings are equivalent:
 	\begin{enumerate}[(i)]
 		\item $\Delta_{n_{s}}\dots\Delta_{n_{1}}g(n_{0})\in \frac{1}{q}\Z$ for some $q\in\N,p\nmid q$ for all $(n_{0},n_{1},\dots,n_{s})\in \Gow_{p,s}(\iota^{-1}(V(M)\cap (V^{\pp}+c)))$;
  		\item we have  $$g=\frac{1}{q}g_{1}+g_{2}$$ for some $q\in\N,p\nmid q$,
 		some $g_{1}\in \poly(\iota^{-1}(V(M)\cap (V^{\pp}+c))\to \R\vert\Z)$ of degree at most $s$ and some $g_{2}\in\poly(\Z^{d}\to \Q)$ of degree at most $s-1$;
 		\item we have $$g=\frac{1}{q}g_{1}+g_{2}$$ for some $q\in\N,p\nmid q$,
 		some $g_{1}\in \poly(\iota^{-1}(V(M_{0})\cap V^{\pp})\to \R\vert\Z)$ of degree at most $s$, and some $g_{2}\in\poly(\Z^{d}\to \Q)$ of degree at most $s-1$.
 	\end{enumerate}	 
	 Moreover, if $g$ is a homogeneous polynomial of degree $s$, then we may require $g_{2}=0$ in (iii).
 \end{prop}

\subsection{Strong factorization property for nice $M$-sets}\label{3:s:AppA7}
  %A fundamental question in higher order Fourier analysis is to study the equidistribution property of polynomial sequences.  
  Let $\Omega$ be a non-empty subset of $\Z^{d}$ and $G/\Gamma$ be an $\N$-filtered nilmanifold. A sequence $\O\colon\Omega\to G/\Gamma$ is \emph{$\d$-equidistributed} on $G/\Gamma$ if for all Lipschitz function $F\colon G/\Gamma\to\C$, we have that
 	$$\limsup_{N\to\infty}\Bigl\vert\frac{1}{\vert \Omega\cap [N]^{d}\vert}\sum_{n\in\Omega\cap [N]^{d}}F(\O(n))-\int_{G/\Gamma}F\,dm_{G/\Gamma}\Bigr\vert\leq \d\Vert F\Vert_{\Lip},$$
	where  $m_{G/\Gamma}$ is the Haar measure of $G/\Gamma$ and the \emph{Lipschitz norm} is defined as 
	$$\Vert F\Vert_{\Lip}:=\sup_{x\in G/\Gamma}\vert F(x)\vert+\sup_{x,y\in G/\Gamma, x\neq y}\frac{\vert F(x)-F(y)\vert}{d_{G/\Gamma}(x,y)}.$$
	    We say that $\O$ is \emph{$\d$-totally equidistributed} on $G/\Gamma$ if for all $r\in\N_{+}$ with $r<\d^{-1}$ and $m\in\Z^{d}$, the sequence  $(\O(n))_{n\in\Omega\cap (r\Z^{d}+m)}$ is $\d$-equidistributed on $G/\Gamma$.

%  Let $\Omega$ be a non-empty finite set and $G/\Gamma$ be an $\N$-filtered nilmanifold. A sequence $\O\colon\Omega\to G/\Gamma$ is \emph{$\d$-equidistributed} on $G/\Gamma$ if for all Lipschitz function $F\colon G/\Gamma\to\C$, we have that
 %	$$\Bigl\vert\frac{1}{\vert \Omega\vert}\sum_{n\in\Omega}F(\O(n))-\int_{G/\Gamma}F\,dm_{G/\Gamma}\Bigr\vert\leq \d\Vert F\Vert_{\Lip},$$
%	where  $m_{G/\Gamma}$ is the Haar measure of $G/\Gamma$.
 
 %Let $g\in \poly(\Z^{d}\to G_{\N})$ be $Q$-periodic and $\Omega\subseteq \Z^{d}$ be a $Q$-periodic set. We say that $(g(n)\Gamma)_{\Omega}$ is \emph{$\e$-equidistributed} on $G/\Gamma$ if $(g(n)\Gamma)_{\Omega\cap [Q]^{d}}$ is $\e$-equidistributed on $G/\Gamma$.

 \begin{lem} [Corollary 8.6 of \cite{SunA}]\label{3:sftg}
  	Let $0<\d<1/2, d,k,m\in\N_{+},s,r\in\N$ with $d\geq r$, and $p\gg_{d,m} \d^{-O_{d,m}(1)}$ be a prime. Let $M\colon\V\to\F_{p}$ be a quadratic form  and $V+c$ be an affine subspace of $\V$ of co-dimension $r$. Suppose that $\rank(M\vert_{V+c})\geq s+13$.  Let $G/\Gamma$ be an $s$-step $\N$-filtered nilmanifold of dimension $m$, equipped with an $\frac{1}{\d}$-rational Mal'cev basis $\mathcal{X}$, and that  $g\in \poly(\Z^{d}\to G_{\N})$ be a rational polynomial sequence. Let  $\gamma\in G$ be an element of complexity at most 1, and let $g'\in \poly(\Z^{d}\to G'_{\N})$
	be the map given by
	 $g'(n):=\gamma^{-1}g(n)\gamma$ for all $n\in\Z^{d}$, where $G'_{i}:=\gamma^{-1} G_{i}\gamma$ for all $i\in\N$. Denote $\Gamma':=\gamma^{-1}\Gamma\gamma$.
	 If  $(g(n)\Gamma)_{n\in \iota^{-1}(V(M)\cap(V+c))}$ is not $\d$-equidistributed on $G/\Gamma$, then $(g'(n)\Gamma')_{n\in \iota^{-1}(V(M)\cap(V+c))}$ is not $C^{-1}\d^{C}$-equidistributed on $G'/\Gamma'$ for some $C=C(d,m)\geq 1$.
	   \end{lem}

\begin{thm}[Theorem %\ref{1:facf3}
11.2 of \cite{SunA}]\label{3:facf3}
  		Let $d,k\in\N_{+},r,s\in\N$, $C>0$, $\mathcal{F}\colon\R_{+}\to\R_{+}$ be a growth function, $p\gg_{C,d,\mathcal{F},k,r,s} 1$ be a prime, $M\colon\V\to \F_{p}$ be a quadratic form, and $\Omega\subseteq (\V)^{k}$ satisfying one of the following assumptions:
  		\begin{enumerate}[(i)] 
  			\item $r=0$ and $\Omega=(\V)^{k}$;
  			\item $k=1,r=0$ and $\Omega=V(M)\cap (V+c)$ for some affine subspace $V+c$ of $\V$ with $\rank(M\vert_{V+c})\geq s+13$;
  			\item $M$ is non-degenerate, $d\geq \max\{4r+1,4k+3,2k+s+11\}$ and $\Omega$ is a  nice and consistent $M$-set of total co-dimension $r$.
  		\end{enumerate}
  		Let $G/\Gamma$ be an $s$-step $\N$-filtered nilmanifold of  complexity at most $C$, and let $g$ be a polynomial sequence in $\poly_{p}(\iota^{-1}(\Omega)\to G_{\N}\vert\Gamma)$.	
  	There 		
  		exist some $C\leq C'\leq O_{C,d,\mathcal{F},k,r,s}(1)$, 
  		 a proper subgroup $G'$ of $G$ which is $C'$-rational relative to $\mathcal{X}$, and
  	 a factorization $$g(n)=\e g'(n)\gamma(n)  \text{ for all }  n\in (\Z^{d})^{k}$$  such that $\e\in G$ is of complexity $O_{C'}(1)$,
  		$g'\in \poly_{p}(\iota^{-1}(\Omega)\to G'_{\N}\vert\Gamma')$, $g'(\bold{0})=id_{G}$ and $(g'(n)\Gamma)_{n\in\iota^{-1}(\Omega)}$ is $\mathcal{F}(C')^{-1}$-equidistributed on $G'/\Gamma'$, where $\Gamma':=G'\cap \Gamma$, and that $\gamma\in\poly(\iota^{-1}(\Omega)\to G_{\N}\vert\Gamma)$, $\gamma(\bold{0})=id_{G}$.   		
  \end{thm}

%     \begin{defn}[Rational polynomial sequences]
 %      Let $d,Q,s\in\N_{+}$ and $G/\Gamma$ be a nilmanifold of step at most $s$, %with $\psi\colon G\to\R^{m}$ being its Mal'cev coordinate map, and 
 %      $\Omega$ be a subset of $\Z^{d}$ and $m\in\Omega$. We say that a rational polynomial sequence $g\in\poly(\Z^{d}\to G_{\N})$ is \emph{$(Q,\Omega,m)$-rational} if $g(m+Qn)\in\Gamma$ for all $n\in\Z^{d}$ with $m+Qn\in \Omega$.
 %      We say that  $g$ is \emph{$(Q,\Omega)$-periodic} if $g(m+Qn)^{-1}g(m)\in\Gamma$ for all $m,n\in \Z^{d}$ with $m+Qn, m\in\Omega$.
 %    % We say that $g$ is \emph{$Q$-power rational} if it is $Q^{r}$-rational for some $r\in\N$.
 %   \end{defn}
	
  %     Let $d\in\N_{+}$, $\d>0$, $p>\d^{-1}$ be a prime, $\Omega$ be a non-empty subset of $\V$ and $g\in\poly(\Z^{d}\to G_{\N})$ be rational for some nilmanifold $G/\Gamma$. We say that $(g(n)\Gamma)_{n\in\iota^{-1}(\Omega)}$ is \emph{$\d$-totally equidistributed on $G/\Gamma$} if for all $r\in\N_{+}$ with $r<\d^{-1}$ and $m\in\Z^{d}$, the sequence $(g(rn+m)\Gamma)_{n\in\iota^{-1}(r^{-1}(\Omega-m))}$ (or equivalently, the sequence $(g(n)\Gamma)_{n\in\iota^{-1}(\Omega)\cap (r\Z^{d}+m)}$) is $\d$-equidistributed on $G/\Gamma$.  
       
   \begin{thm}[Theorem %\ref{1:facf3}
11.5 of \cite{SunA}]\label{3:facf3r}
  		Let $d\in\N_{+},r,s\in\N$, $C>0$, $\mathcal{F}\colon\R_{+}\to\R_{+}$ be a growth function, $p\gg_{C,d,\mathcal{F},s} 1$ be a prime, $M\colon\V\to \F_{p}$ be a quadratic form, and $\Omega\subseteq \V$ satisfying one of the following assumptions:
  		\begin{enumerate}[(i)] 
  			\item $\Omega=\V$;
  			\item $\Omega=V(M)\cap (V+c)$ for some affine subspace $V+c$ of $\V$ with $\rank(M\vert_{V+c})\geq s+13$.
  			%\item $M$ is non-degenerate, $d\geq \max\{4r+1,4k+3,2k+s+11\}$ and $\Omega$ is a  nice and consistent $M$-set of total co-dimension $r$.
  		\end{enumerate}
  		Let $n_{\ast}\in\iota^{-1}(\Omega)$, $G/\Gamma$ be an $s$-step $\N$-filtered nilmanifold of  complexity at most $C$, and let $g\in \poly(\Z^{d}\to G_{\N})$ be rational.	
  	There exist some $C\leq C'\leq O_{C,d,\mathcal{F},s}(1)$, 
  		 a proper subgroup $G'$ of $G$ which is $C'$-rational relative to $\mathcal{X}$, and
  	 a factorization $$g(n)=\e g'(n)\gamma(n)  \text{ for all }  n\in \Z^{d}$$  such that $\e\in G$ is of complexity $O_{C'}(1)$,
  		$g'\in \poly(\Z^{d}\to G'_{\N})$ is  rational  %with $g'(\bold{0})=id_{G}$ 
		and $(g'(n)\Gamma)_{n\in\iota^{-1}(\Omega)}$  is $\mathcal{F}(C')^{-1}$-totally equidistributed on $G'/\Gamma'$, where $\Gamma':=G'\cap \Gamma$, and that $\gamma$ belongs to both $\poly_{\approx r,n_{\ast}}(\iota^{-1}(\Omega)\to G_{\N}\vert\Gamma)$ and $\poly_{r}(\iota^{-1}(\Omega)\to G_{N}\vert\Gamma)$ for 
	%	
	%	$\gamma\in\poly(\Z^{d}\to G_{\N})$ is $(r,\iota^{-1}(\Omega),n_{\ast})$-rational and $(r,\iota^{-1}(\Omega))$-periodic
	 for some $r\in\N_{+}$ with $r\leq O_{C,C',d,s}(1)$. %$\gamma(\bold{0})=id_{G}$.   	
  \end{thm}

 \subsection{Additive combinatorial properties for $M$-modules}\label{3:s:AppA8}

\begin{defn}[Almost linear function and Freiman homomorphism]\label{2:llfh}
Let $L\in\N_{+}$,  
$G$ be an additive group, $H$ be a subset of $G$, $R$ be a subset of $\R$, and $\xi\colon H\to\R$ be a map.
We say that $\xi$ is an \emph{$R$-almost linear function} of \emph{complexity} at most $L$ if 
there exist  $\alpha_{i}\in \widehat{G}$ and $\beta_{i}\in R$ for all $1\leq i\leq L$ such that for all $h\in H$,
$$\xi(h):=\sum_{i=1}^{L}\{\alpha_{i}\cdot h\}\beta_{i}.%\footnote{In particular, $\sum_{i=1}^{K}\{\alpha_{i}\cdot \tau(h)\}\beta_{i}$ is required to take values in $\Z/p$.}
$$
We say that an $R$-almost linear function $\xi$  is an \emph{$R$-almost linear Freiman homomorphism} if  for all $h_{1},h_{2},h_{3},h_{4}\in H$ with $h_{1}+h_{2}=h_{3}+h_{4}$, we have that $$\xi(h_{1})+\xi(h_{2})\equiv\xi(h_{3})+\xi(h_{4}) \mod \Z.$$

%Let $d\in\N_{+}$, $p$ be a prime, and $H$ be a subset of $\V$.
%We say that a map $\xi\colon H\to \F_{p}$ is an \emph{almost linear function/Freiman homomorphism} of complexity at most $L$ if there exists a $\Z/p$-almost linear function/Freiman homomorphism $\xi'$ with $\xi'(H)\subseteq\Z/p$ of complexity at most $L$ such that
%$\xi(h)=\iota(p\xi'(h))$ for all $h\in H$.

%Let $d,L\in\N_{+}$,  $p$ be a prime, $H$ be a subset of $\V$, and $\xi\colon H\to \Z/p$ be a map. 
%We say that $\xi$ is a \emph{$p$-almost linear function} of \emph{complexity} at most $L$ if 
%there exist  $\alpha_{i}\in (\Z/p)^{d}$ and $\beta_{i}\in\Z/p$ for all $1\leq i\leq L$ such that for all $h\in H$,
%$$\xi(h):=\sum_{i=1}^{L}\{\alpha_{i}\cdot \tau(h)\}\beta_{i}.\footnote{In particular, $\sum_{i=1}^{K}\{\alpha_{i}\cdot \tau(h)\}\beta_{i}$ is required to take values in $\Z/p$.}$$
%We say that a $p$-almost linear function $\xi$  is a \emph{$p$-almost linear Freiman homomorphism} if  for all $h_{1},h_{2},h_{3},h_{4}\in H$ with $h_{1}+h_{2}=h_{3}+h_{4}$, we have that $$\xi(h_{1})+\xi(h_{2})\equiv\xi(h_{3})+\xi(h_{4}) \mod \Z.$$
%
%
%We say that a map $\xi\colon H\to \F_{p}$ is an \emph{almost linear function/Freiman homomorphism} of complexity at most $K$ if there exists a $p$-almost linear function/Freiman homomorphism $\xi'\colon H\to \Z/p$ of complexity at most $K$ such that
%$\xi(h)=\iota(p\xi'(h))$ for all $h\in H$.

We say that a map $\xi\colon H\to \R[x_{1},\dots,x_{d}]$ %(resp. $\F_{p}[x_{1},\dots,x_{d}]$)
 is an \emph{$R$-almost linear function/$R$-Freiman homomorphism} %(resp. \emph{almost linear function/Freiman homomorphism}) 
 of complexity at most $L$ if $\xi_{m}$ is an $R$-almost linear function/$R$-Freiman homomorphism %(resp. almost linear function/Freiman homomorphism) 
 of complexity at most $L$ for all $m=(m_{1},\dots,m_{d})\in\N^{d}$, where $\xi_{m}(h)$ is the coefficient of the monomial $x_{1}^{m_{1}}\dots x_{d}^{m_{d}}$ of $\xi(h)$.
\end{defn}

 \begin{defn}[$M$-module for real polynomials]
		Let $d\in\N_{+}, p$ be a prime and $M\colon\V\to\F_{p}$ be a quadratic form associated with the matrix $A$. 
	We say that a subset $I$ of the polynomial ring $\R[x_{1},\dots,x_{d}]$ 
	is an \emph{$M$-module} if there exists a subspace $V$ of $\V$ such that 
	$$f\in I\Leftrightarrow f(n)\in\Z \text{ for all } n\in\Z^{d} \text{ with } (\iota(n)A)\cdot\iota(n)=0 \text{ and } (hA)\cdot \iota(n)=0 \text{ for all } h\in V$$ 
	(or equivalently, $f\in I\Leftrightarrow f(n)\in\Z \text{ for all } n\in \iota^{-1}(V(M_{0})\cap V^{\pp})$, where $M_{0}\colon\V\to\F_{p}$ is the quadratic form given by $M_{0}(n):=(nA)\cdot n$). 
	In this case we denote $I$ as $J^{M}_{V}$. For convenience we denote 
	$$J^{M}:=J^{M}_{\{\bold{0}\}} \text{ and } J^{M}_{h_{1},\dots,h_{k}}:=J^{M}_{\sp_{\F_{p}}\{h_{1},\dots,h_{k}\}}.$$ 
	for all $h_{1},\dots,h_{k}\in\V$.%\footnote{To simplify notations, we use the same   notation $J^{M}_{V}$ for $M$-modules over $\F_{p}$ and over $\R$. This will not cause confusions as the meaning of $J^{M}_{V}$ will be clear from the context.}
\end{defn}

 \begin{thm}[Theorem 1.5 of \cite{SunB}]\label{3:aadd}
	Let $d,K,r,s\in\N_{+}$, with $d\geq N(s)$, $\d>0$ and $p\gg_{\d,d,r}1$ be a prime dividing $K$,
	  $M\colon\V\to\F_{p}$ be a non-degenerate quadratic form,  $H\subseteq \Vk$ with $\vert H\vert>\d K^{d}$, and $\xi_{i}$ be a map from $H$ to $\st_{\Z/p^{r},d}(s)$ for $1\leq i\leq 4$.
	Suppose that there exists 
	  $U\subseteq \{(h_{1},h_{2},h_{3},h_{4})\in H^{4}\colon h_{1}+h_{2}=h_{3}+h_{4}\}$
	with $\vert U\vert>\d \vert H\vert^{3}$ such that for all $(h_{1},h_{2},h_{3},h_{4})\in U$, 
	\begin{equation}\nonumber
	\xi_{1}(h_{1})+\xi_{2}(h_{2})-\xi_{3}(h_{3})-\xi_{4}(h_{4})\in J^{M}_{\iota(h_{1}),\iota(h_{2}),\iota(h_{3})}.
	\end{equation}
	Then there exist   $H'\subseteq H$ with $\vert H'\vert\gg_{\d,d,r} K^{d}$, some $g\in\st_{\Z/p^{r},d}(s)$, and an $\Z/p^{r}$-almost linear Freiman homomorphism  $T\colon H'\to \st_{\Z/p^{r},d}(s)$  of complexity $O_{\d,d}(1)$.  % for some $Q\in\N_{+}, Q\leq O_{\d,d}(1)$.
	 such that 
	$$\xi_{1}(h)-T(h)-g\in J^{M}_{\iota(h)} \text{ for all $h\in H'$.}$$
\end{thm}

Let $V_{1},\dots,V_{k}$ be subspaces of $\V$.
We say that $V_{1},\dots,V_{k}$ are \emph{linearly independent} if for all $v_{i}\in V_{i}, 1\leq i\leq k$, $\sum_{i=1}^{k}v_{i}=\bold{0}$ implies that $v_{1}=\dots=v_{k}=\bold{0}$.
In this case we also say that   $(V_{1},\dots,V_{k})$ is a \emph{linearly independent tuple} (or a \emph{linearly independent pair} if $k=2$). 

Let $V_{1},\dots,V_{k}$ be subspaces of $\V$ and $v_{1},\dots,v_{k'}\in\V$. 
We say that $v_{1},\dots,v_{k'},V_{1},\dots,V_{k}$ are \emph{linearly independent} if for all $c_{i}\in\F_{p}, 1\leq i\leq k$ and $v'_{i'}\in V_{i'}, 1\leq i'\leq k'$, $\sum_{i=1}^{k}c_{i}v_{i}+\sum_{i'=1}^{k'}v'_{i'}=\bold{0}$ implies that $c_{1}=\dots=c_{k}=0$ and $v'_{1}=\dots=v'_{k'}=\bold{0}$.
In this case we also say that 
 $(v_{1},\dots,v_{k'},V_{1},\dots,V_{k})$ is a \emph{linearly independent tuple} (or a \emph{linearly independent pair} if $k+k'=2$).

 \begin{prop}[Propositions %\ref{2:gri}
 3.6 and %\ref{2:gr0}
 3.7 of \cite{SunB}]\label{3:gr0}
	Let $m,s\in\N$, $d,N,r\in\N_{+}$, $p\gg_{d} 1$ be a prime, $M\colon\V\to \F_{p}$ be a non-degenerate quadratic form, and $V,V_{1},\dots,V_{N}$ be  subspaces of $\V$ such that   $\dim(V)=m$ and $\dim(V_{i})\leq r, 1\leq i\leq N$.	Suppose that 
	either
	\begin{itemize}
	\item $V_{1},\dots,V_{N}$ are linearly independent and $N\geq s+m+1$; or
		\item $V,V_{1},\dots,V_{N}$ are linearly independent and $N\geq s+1$,	
	\end{itemize}	
	  and that
	either $\rank(M\vert_{V^{\pp}})\geq 2N(r-1)+7$ or $d\geq 2m+2N(r-1)+7$.
	Then for all $f\in\st_{\zp,d}(s)$, we have that
	$$f\in\cap_{i=1}^{N}J^{M}_{V+V_{i}}\Leftrightarrow f\in J^{M}_{V}.$$
\end{prop}

  \begin{prop}[Proposition 4.24 of \cite{SunB}]\label{3:coco1p}
Let $d,K\in\N_{+}$, $r,s\in\N$, $\d>0$, $p\gg_{\d,d,r} 1$ be a prime dividing $K$,   $H$ be a subset of $\Vk$ with $\vert H\vert>\delta K^{d}$, $\bold{0}\notin\iota(H)$ and $M\colon\V\to\F_{p}$ be a non-degenerate quadratic form.
	Let $F\colon H\to\st_{\Z/p^{r},d}(s)$ be a map such that
	\begin{equation}\nonumber
	F(x)\equiv F(y)\mod J^{M}_{\iota(x),\iota(y)}
	\end{equation}
	 for all $x,y\in H$.  If  $d\geq s+8$, then there exists $G\in\st_{\Z/p^{r},d}(s)$ such that 
	\begin{equation}\nonumber
	F(x)\equiv G\mod J^{M}_{\iota(x)}
	\end{equation}
	  for all $x\in H$. 
\end{prop}

%\begin{prop}[A special case of Proposition %\ref{2:coco1} 
%4.6 of \cite{SunB}]\label{3:coco1}
%Let $d\in\N_{+}$, $s\in\N$, $\d>0$, $p\gg_{\d,d} 1$ be a prime,   $H$ be a subset of $\V$ with $\vert H\vert>\delta p^{d}$ and $M\colon\V\to\F_{p}$ be a non-degenerate quadratic form.
%	Let $F\colon \V\to\st_{d}(s)$ be a map such that
%	\begin{equation}\nonumber
%	F(x)\equiv F(y)\mod J^{M}_{x,y}
%	\end{equation}
%	 for all $x,y\in H$.  If 
%	 $d\geq s+8$, then there exists $G\in\st_{d}(s)$ such that 
%	\begin{equation}\nonumber
%	F(x)\equiv G\mod J^{M}_{x}
%	\end{equation}
%	  for all $x\in H$. 
%\end{prop}

\section{Approximation properties for nilsequences}\label{3:s:AppB}

In this section, we provide some useful properties which allow us to approximate a nilsequence or nilcharacter by one with better properties.

The following lemma demonstrates the type of operations we can apply to nilcharacters. This is  a generalization of Lemma E.3 of \cite{GTZ12}.  Compared with Lemma E.3,
our statement needs to take the periodicity property of nilcharacters into consideration,  and our statement is more quantitative since we do not use the non-standard analysis approach as in \cite{GTZ12}.

\begin{lem}\label{3:LE80}
	Let $d,D,k,k'\in\N_{+}$, $C>0$, $I$ be the degree, multi-degree,  or  degree-rank ordering with $\dim(I)\vert k$, and $J$ be a finite downset of  $I$.	
	Let $\Omega\subseteq\F_{p}^{k}$ and $\chi,\chi'\in\Xi^{J;C,D}_{p}(\Omega)$. 
	Then for all $h\in \F_{p}^{k}$ and $q\in\F_{p}\backslash\{0\}$, we have that $\chi\otimes\chi'\in \Xi^{J;2C,D^{2}}_{p}(\Omega)$, $\chi(\cdot+h)\in \Xi^{J;C,D}_{p}(\Omega-h)$, $\chi(q\cdot)\in\Xi^{J;C,D}_{p}(q^{-1}\Omega)$ and $\overline{\chi}\in\Xi^{J;C,D}_{p}(\Omega)$.	
	
	If $I$ is the degree filtration, then 	for all    linear transformation $L\colon \F_{p}^{k'}\to\F_{p}^{k}$, we have that  $\chi\circ L\in\Xi^{J;C,D}_{p}(L^{-1}(\Omega))$.
\end{lem}	
\begin{proof}
Let $((G/\Gamma)_{I},g,F,\eta)$ and $((G'/\Gamma')_{I},g',F',\eta')$ be  $\Xi^{J}_{p}(\Omega)$-representations of $\chi$ and $\chi'$ respectively of complexities at most  $C$ and dimensions $D',D''$ respectively. 
Then $$\chi\otimes\chi'(n)=F\otimes F'(g(\tau(n))\Gamma,g'(\tau(n))\Gamma')$$ for all $n\in \Omega$. 
	Note that $(G\times G')/(\Gamma\times \Gamma')$ is an $I$-filtered nilmanifold of complexity at most $2C$ and degree $\subseteq J$, and $F\otimes F'$ is a function taking values in $\mathbb{S}^{D'D''}$ with  Lipschitz norm bounded by $2C$ and has a vertical frequency $(\eta,\eta')$ of complexity at most $2C$. Moreover, since $g\in \poly_{p}(\Z^{k}\to G_{I}\vert \Gamma)$ and $g'\in \poly_{p}(\Z^{k}\to G'_{I}\vert \Gamma')$,
	we have that 
	$(g(\cdot),g'(\cdot))\in \poly_{p}(\Z^{k}\to (G\times G')_{I}\vert \Gamma\times\Gamma')$. So  $\chi\otimes\chi'\in \Xi^{J;2C,D^{2}}_{p}(\Omega)$.

	Since $g\circ\tau\in \poly_{p}(\F_{p}^{k}\to G_{I}\vert\Gamma)$,
	by Proposition \ref{3:BB}, there exist $g',g''\in \poly_{p}(\F_{p}^{k}\to G_{I}\vert\Gamma)$ such that 
	$g'(n)\Gamma=g\circ\tau(n+h)\Gamma$ for all $n\in\Omega-h$ and that $g''(n)\Gamma=g\circ\tau(qn)\Gamma$ for all $n\in q^{-1}\Omega$. From this it is not hard to see that $\chi(\cdot+h)\in \Xi^{J;C,D}_{p}(\Omega-h)$ and $\chi(q\cdot)\in\Xi^{J;C,D}_{p}(q^{-1}\Omega)$.	
	We also have $\overline{\chi}\in\Xi^{J;C,D}_{p}(\Omega)$  since $\overline{\chi}(n)=\overline{F}(g(n)\Gamma)$ for all $n\in\Omega$.
	
	If $I$ is the degree filtration, then
	by Proposition \ref{3:BB}, there exists $g'''\in \poly_{p}(\F_{p}^{k}\to G_{I}\vert\Gamma)$ such that 
 $g'''(n)\Gamma=g\circ\tau\circ L(n)\Gamma$ for all $n\in L^{-1}(\Omega)$. From this it is not hard to see that   $\chi\circ L\in\Xi^{J;C,D}_{p}(L^{-1}(\Omega))$.	
	\end{proof}

Given any non-periodic nilsequence, the method of Manners \cite{Man14} allows us to approximate it with a $p$-periodic one. We summarize this approach in the following theorem, and tailor it for our purposes:

\begin{thm}[Manners' approximation theorem for degree and multi-degree filtrations]\label{3:papprox}
	Let $D,k\in\N_{+}, C>0$, $p$ be a prime, $I$ be the degree or multi-degree ordering with $\dim(I)\vert k$, $G/\Gamma$ be a nilmanifold with flirtation $G_{I}$ of complexity at most $C$, $g\in\poly(\Z^{k}\to G_{I})$ be a polynomial, and $F\in\Lip(G/\Gamma\to\C^{D})$ be of Lipschitz norm at most $C$. Let $\phi\colon\T^{k}\to [0,1]$ be a smooth function supported on a box of edge length $1/2$.
	Then there exist a nilmanifold $\tilde{G}/\tilde{\Gamma}$  with flirtation $\tilde{G}_{I}$ of complexity at most $O_{C,k}(1)$ and degree the same as $G/\Gamma$, a $p$-periodic polynomial $\tilde{g}\in\poly_{p}(\Z^{k}\to \tilde{G}_{I}\vert\tilde{\Gamma})$, and a function $\tilde{F}\in\Lip(\tilde{G}/\tilde{\Gamma}\to\C^{D})$  of Lipschitz norm at most $O_{C,k}(1)$ such that
	$$\tilde{F}(\tilde{g}(n)\tilde{\Gamma})=\phi(n/p)\otimes F(g(n \mod p\Z^{k})\Gamma)$$
	for all $n\in\Z^{k}$, where $n \mod p\Z^{k}$ is the representative of $n$ in $\{0,\dots,p-1\}^{k}$.	 
\end{thm}	
\begin{proof}
	Consider the polynomial group $\poly(\R^{k}\to G_{I})$ with the pointwise multiplication. By Lemma \ref{3:B.9}, there is an isomorphism $\poly(\Z^{k}\to G_{I})\simeq\poly(\R^{k}\to G_{I})$ given by restrictions.
	Similar to Lemma C.1 of \cite{GTZ12}, $\poly(\Z^{k}\to \Gamma_{I})$ is a discrete and cocompact subgroup of $\poly(\R^{k}\to G_{I})$, where $\Gamma_{I}$ is the filtration induced by $G_{I}$.

	Let $T\colon \R^{k}\to\poly(\R^{k}\to G_{I})$, $T(t,p):=p(\cdot+t)$ denote the shift action. Then $T$ induces a semi-product structure $\tilde{G}:=
	\R^{k}\ltimes_{T}\poly(\R^{k}\to G_{I})$ with the group operation given by
	$$(t,g)\ast(t',g'):=(t+t',T(t',g)\cdot g').$$
	Using the type-I Taylor expansion in Lemma \ref{3:B.9}, it is clear that the group $\poly(\R^{k}\to G_{I})$ is connected and simply connected. So
	$\tilde{G}$  is connected and simply connected.
	Similar to the argument in Theorem 1.5 of \cite{Man14}, $\tilde{\Gamma}:=\Z^{k}\ltimes_{T}\poly(\R^{k}\to \Gamma_{I})$ is a discrete and cocompact subgroup of $\tilde{G}.$

	We now define a pre-filtration $\tilde{G}_{I}$ as follows.  For $i\in I$, let $G^{+i}_{I}$ denote the pre-filtration given by $G^{+i}_{j}:=G_{i+j}$ for all $j\in I$. If $I$ is the multi-degree ordering, then set 
	$\tilde{G}_{0}=\tilde{G}$, $\tilde{G}_{i}=\R^{k}\ltimes_{T}\poly(\R^{k}\to G^{+i}_{I})$ for $\vert i\vert=1$ and $\tilde{G}_{i}=\{\bold{0}\}\ltimes_{T}\poly(\R^{k}\to G^{+i}_{I})$ for  $\vert i\vert\geq 2$. If  $I$ is the degree ordering, then set 
	$\tilde{G}_{0}=\tilde{G}$, $\tilde{G}_{i}=\R^{k}\ltimes_{T}\poly(\R^{k}\to G^{+i}_{I})$ for $i=1$ and $\tilde{G}_{i}=\{\bold{0}\}\ltimes_{T}\poly(\R^{k}\to G^{+i}_{I})$ for  $i\geq 2$.
%	If $I$ is the degree-rank ordering, then let $G_{\N}$ be the degree pre-filtration generated by $G_{[i,0]}$ and set $\tilde{G}'_{0}=\tilde{G}$, $\tilde{G}'_{i}=\R^{k}\ltimes_{T}\poly(\R^{k}\to G^{+i}_{\N})$ for $i=1$ and $\tilde{G}'_{i}=\{\bold{0}\}\ltimes_{T}\poly(\R^{k}\to G^{+i}_{\N})$ for  $i\geq 2$. We then set $\tilde{G}_{\DR}$ to be the degree-rank pre-filtration induced by $\tilde{G}'_{\N}$.
	In both cases, $\tilde{G}_{I}$ is an $I$-pre-filtration and defines a pre-nilmanifold $\tilde{G}/\tilde{\Gamma}$  of complexity $O_{C,k}(1)$ and having  the same degree as $G/\Gamma$.

	Let $h\in \poly(\R^{k}\to G_{I})$ be the rescaling of $g$ given by $h(n):=g(pn)$ (where we view $g$ as a polynomial in $\poly(\R^{k}\to G_{I})$). Let $\tilde{g}\colon\Z^{k}\to\tilde{G}$ be the map given by 	$$\tilde{g}(n):=	(\bold{0},h)\ast(n/p,id_{G}).$$
	It is not hard to compute that $\tilde{g}$ belongs to $\poly(\Z^{k}\to \tilde{G}_{I})$ in both cases. %(recall that $\poly(\Z^{k}\to G_{\DR})=\poly(\Z^{k}\to (G_{[i,0]})_{i\in\N})$ when $I$ is the degree-rank ordering). 
	Moreover, since
	$$\tilde{g}(n+pm)=(\bold{0},h)\ast(n/p+m,id_{G})=(\bold{0},h)\ast(n/p,id_{G})\ast(m,id_{G})\in\tilde{g}(n)\tilde{\Gamma}$$
	for all $m,n\in\Z^{k}$, we have that $\tilde{g}$ is $p$-periodic.
	
	Finally, we define $\tilde{F}$ on the fundamental domain $[0,1)^{k}\ltimes_{T}\poly(\R^{k}\to G_{I})/\poly(\Z^{k}\to \Gamma_{I})$ by
	$$\tilde{F}(t,g\cdot\poly(\Z^{k}\to \Gamma_{I})):=\phi(t)\otimes F(g(\bold{0})\Gamma)$$
	(which is obviously well defined), and extend periodically by $\tilde{\Gamma}$, namely
	$$\tilde{F}(t,g\cdot\poly(\Z^{k}\to \Gamma_{I})):=\phi(\{t\})\otimes F(g(-\lfloor t\rfloor)\Gamma)$$
	for $t\in\R^{k}$. It is clear that the Lipschitz norm of $\tilde{F}$ is $O_{C,k}(1)$. Then
	\begin{equation}
	\begin{split}
	&\quad \tilde{F}(\tilde{g}(n)\tilde{\Gamma})=\tilde{F}((n/p,h(\cdot+n/p))\tilde{\Gamma})
	\\&=\phi(\{n/p\})\otimes F(h(n/p-\lfloor n/p\rfloor))=\phi(\{n/p\})\otimes F(g(n \mod p\Z^{k})).
	\end{split}
	\end{equation}
We are done by Lemma \ref{3:pre2g}.
\end{proof}

We caution the readers that we do not know whether a result similar to Theorem \ref{3:papprox} holds for the degree-rank filtrations.
As a consequence of Theorem \ref{3:papprox}, %and \ref{3:papproxdr}, 
we have the following result which allows us to approximate   nilsequences by   periodic ones:

\begin{coro}[Approximating a nilsequence by  periodic ones]\label{3:pppap}
	Let $C>0$, $D,k\in\N_{+}$,  $I$ be the degree or multi-degree  ordering with $\dim(I)\vert k$,  $J$ be a down set of $I$, and $p$ be a prime. 
%	 If $I$ is the degree-rank ordering, then let $s_{\ast}$ denote the smallest positive integer with $[s_{\ast},s_{\ast}]\geq i$ for all $i\in J$. If $I$ is the degree or multi-degree ordering, denote $s_{\ast}:=1$.
	For any $\Omega\subseteq \Z^{k}$ and  $\phi\in\Nil^{J;C,D}(\Omega)$, there exist $\phi_{1},\dots,\phi_{20^{k}}\in \Nil^{J;O_{C,k}(1),D}_{p}(\Omega)$ such that for all $n\in \Omega$, we have $$\phi(n \mod p\Z^{k})=\sum_{m=1}^{20^{k}}\phi_{m}(n).$$	 
	
	Similarly, for any prime $p$, any subset $\Omega'\subseteq \F_{p}^{k}$ and any $\phi'\in\Nil^{J;C,D}(\Omega')$, there exist $\phi'_{1},\dots,\phi'_{20^{k}}\in \Nil^{J;O_{C,k}(1),D}_{p}(\Omega')$ such that for all $n\in\Omega'$, we have $$\phi'(n)=\sum_{m=1}^{20^{k}}\phi'_{m}(n).$$ 
\end{coro}	
\begin{proof}
Since any nilsequence defined on $\Omega$ naturally extends to a nilsequence defined on $\Z^{k}$, we may assume without loss of generality that $\Omega=\Z^{k}$. Similarly, we may assume without loss of generality that $\Omega'=\F_{p}^{k}$.

	We choose $20^{k}$ closed boxes $I_{1},\dots,I_{20^{k}}$ in $\T^{k}$, each of edge length exactly $1/10$, whose interiors cover $\T^{k}$. Let $\rho_{m}, 1\leq m\leq 20^{k}$ be a smooth partition of unity on $\T^{k}$ adapted to $I_{m}, 1\leq m\leq 20^{k}$.   Then each $\rho_{m}$ is supported on a box $J_{m}$ of edge 1/10. 
	
	Denote
	$$\phi_{m}(n):=\rho_{m}(\{n/p\})\otimes\phi(n \mod p\Z^{k})$$ for all $1\leq m\leq 20^{k}$. Then $$\sum_{m=1}^{20^{k}}\phi_{m}(n)=\phi(n \mod p\Z^{k}).$$
	By Theorem \ref{3:papprox}, we have that $\phi_{m}\in \Nil^{J;O_{C,k}(1),D}_{p}(\Z^{k})$. 
	
	For any $\phi'\in\Nil^{J;C,D}(\F_{p}^{k})$, we may write $\phi'=\phi\circ \tau$ for some $\phi\in\Nil^{J;C,D}(\Z^{k})$. So  there exist $\phi_{1},\dots,\phi_{20^{k}}\in \Nil^{J;O_{C,k}(1),D}_{p}(\Z^{k})$ such that $$\phi(n \mod p\Z^{k})=\sum_{m=1}^{20^{k}}\phi_{m}(n)$$ for all $n\in\Z^{k}$.	 
Then $\phi_{1}\circ \tau,\dots,\phi_{20^{k}}\circ \tau\in \Nil^{J;O_{C,k}(1),D}_{p}(\F_{p}^{k})$ and
$$\phi'(n)=\phi(\tau(n))=\sum_{m=1}^{20^{k}}\phi_{m}(\tau(n))=\sum_{m=1}^{20^{k}}\phi'_{m}(n)$$ for all $n\in\F_{p}^{k}$.
\end{proof}

The next lemma allows us to approximate nilsequences of degree $J\cup J'$ by the combinations of ($p$-periodic) nilsequences of degrees $J$ and $J'$. 

\begin{lem}[Splitting a nilsequence into nilsequences of smaller degrees]\label{3:LE.4}
	Let $k,Q\in\N_{+}$, $C,\e>0$, $I$ be the degree, multi-degree or degree-rank ordering with $\dim(I)\vert k$, $J,J'$ be finite downsets of $I$,
	 $p$ be a prime, and $\Omega\subseteq \Z^{k}$. 
	% If $I$ is the degree-rank ordering, then let $s_{\ast}$ denote the smallest positive integer with $[s_{\ast},s_{\ast}]\geq i$ for all $i\in J\cup J'$.If $I$ is the degree or multi-degree ordering, denote $s_{\ast}:=1$.
	 For any $\psi\in\Nil^{J\cup J';C,1}_{\approx Q}(\Omega)$, 	$$\Bigl\Vert\psi(n)-\sum_{j=1}^{K}\psi_{j}(n)\psi'_{j}(n)\Bigr\Vert_{\ell^{\infty}(\Omega)}<\e$$
	for some $K:=K(C,\e,J,J',k)\in\N_{+}$, $\psi_{j}\in\Nil^{J;O_{C,\e,J,J',k}(1),1}_{\approx Q}(\Omega)$ and  $\psi'_{j}\in\Nil^{J';O_{C,\e,J,J',k}(1),1}_{\approx Q}(\Omega)$ for all $1\leq j\leq K$.
	
	In addition, if 	$I$ is the degree or multi-degree  ordering, then we may instead require that  $\psi_{j}\in\Nil^{J;O_{C,\e,J,J',k}(1),1}_{p}(\Omega)$ and  $\psi'_{j}\in\Nil^{J';O_{C,\e,J,J',k}(1),1}_{p}(\Omega)$ for all $1\leq j\leq K$.
\end{lem}	
\begin{proof}
The idea of the proof is similar to Lemma E.4 of \cite{GTZ12}. However, our case is more intricate in the sense that we need  a more quantitative approximation, and that we need   to make the nilsequences to be $p$-periodic.

By Corollary \ref{3:pppap}, it suffices to prove this lemma under the milder restriction that $\psi_{j}\in\Nil^{J;O_{C,\e,J,J',k}(1),1}_{\approx Q}(\Omega)$ and  $\psi'_{j}\in\Nil^{J';O_{C,\e,J,J',k}(1),1}_{\approx Q}(\Omega)$, which does not require $p$-periodicity.
Let $((G/\Gamma)_{I},g,F)$ be a $\Nil^{J\cup J'}_{\approx Q}(\Omega)$-representation of $\psi$ of complexity at most $C$ and dimension 1. 
For each $j\in J\cup J'$, let $e_{j,1},\dots,e_{j,d_{j}}$ be a basis of generators for $\Gamma_{j}$. We may then lift $G$ to the \emph{universal nilpotent Lie group} that is formally generated by $e_{j,1},\dots,e_{j,d_{j}}$ under the only restriction that any iterated commutator of $e_{j_{1},i_{1}},\dots,e_{j_{r},i_{r}}$ is trivial if $j_{1}+\dots+j_{r}\notin J\cup J'$. Similarly, we can lift $\Gamma$, $F$ and $g$ (with the help of Corollary B.10 of \cite{GTZ12}). Therefore, since the universal nilpotent Lie group is $O(C)$-rational relative to $\mathcal{X}$, we may assume without loss of generality that $G$ is universal. 
	
	The degree $\subseteq J\cup J'$ nilmanifold $G/\Gamma$ projects down to the degree $\subseteq J$ nilmanifold $G/G_{>J}\Gamma$, which is $O_{C,J,J'}(1)$-rational relative to $\mathcal{X}$, where $G_{>J}$ is the group generated by the $G_{j}$ for all $j\in J'\backslash J$. Similarly, we have a projection from $G/\Gamma$ to the degree  $\subseteq J'$ nilmanifold $G/G_{>J'}\Gamma$ which is $O_{C,J,J'}(1)$-rational relative to $\mathcal{X}$. The algebras $\Lip(G/G_{>J}\Gamma\to \C)$ and $\Lip(G/G_{>J'}\Gamma\to \C)$ pull back to sub algebras of $\Lip(G/\Gamma\to \C)$. By the universality of $G$, $G_{>J}$ and $G_{>J'}$ are disjoint. So the union of these two algebras separate points in $G/\Gamma$. 
	
 We next approximate $F$ pointwise by the union of the pullbacks of $\Lip(G/G_{>J}\Gamma\to \C)$ and $\Lip(G/G_{>J'}\Gamma\to \C)$. We remark that we can not directly use Stone-Weierstrass theorem since we need this approximation to be quantitative.

	For convenience denote $X=G/\Gamma$, $Y_{1}=G/G_{>J}\Gamma$, $Y_{2}=G/G_{>J'}\Gamma$. 
	Let $\pi_{1}\colon X\to Y_{1}$ and $\pi_{2}\colon X\to Y_{2}$ denote the projection maps. 
	Denote $\pi=\pi_{1}\times\pi_{2}$ and $Y=Y_{1}\times Y_{2}$. We first claim that $\pi\colon X\to Y$ is an injection. Indeed, for all $x,x'\in X$, if $\pi_{i}(x)=\pi_{i}(x')$ for $i=1,2$, then $g(\pi_{i}(x))=g(\pi_{i}(x'))$ for all $g\in \Lip(Y_{i}\to \C)$. Since the union of the pullbacks of $\Lip(Y_{1}\to \C)$ and $\Lip(Y_{2}\to \C)$ separate points in $G/\Gamma$, we have that $x=x'$. This proves the claim.

Let $d_{X}$, $d_{Y_{1}}$ and $d_{Y_{2}}$ denote the metrices associated to $X,Y_{1}$ and $Y_{2}$ induced by their Mal'cev basis respectively. It is clear that there exists $L=O_{C,J,J'}(1)>1$ such that
$$L^{-1}d_{X}(x,x')\leq d_{Y_{i}}(\pi_{i}(x),\pi_{i}(x'))\leq L d_{X}(x,x')$$
for all $x,x'\in X$ and $i=1,2$.
    Let $S:=\pi(X)\subseteq Y$, and let $F'\colon S\to \C$ be given by $F'(y)=F(x)$ for all $x\in X, y\in Y$ with $\pi(x)=y$. Since $\pi$ is an injection, $F'$ is well defined. Moreover, for all $x,x'\in X, y,y'\in Y, y\neq y'$ with $\pi(x)=y$ and $\pi(x')=y'$, we have that 
    $$\frac{\vert F'(y)-F'(y')\vert}{d_{Y}(y,y')}=\frac{\vert F(x)-F(x')\vert}{d_{X}(x,x')}\cdot \frac{d_{X}(x,x')}{d_{Y}(\pi(x),\pi(x'))}\leq CL^{2}=O_{C,J,J'}(1),$$
    where $d_{Y}$ is the product metric of $d_{Y_{1}}$ and $d_{Y_{2}}$.
    By McShane's theorem \cite{Mic34}, writing 
    $$F''(y):=\sup_{z\in S}(F'(z)-CL^{2}d_{Y}(y,z))$$
    for all $y\in Y$, we have that $F''\vert_{S}=F'$ and that $\Vert F''\Vert_{\Lip(Y)}=O_{C,J,J'}(1)$.
   
   Let $\d>0$ to be chosen later. It is not hard to see that there exist $N=N(C,\d,J,J')\in\N_{+}, K\leq N$ and $y_{1,1},\dots,y_{1,K}\in Y_{1}$ such that writing  $U(y_{1,j})$ to be the open ball in $Y_{1}$ of radius $\d$ centered at $y_{1,j}$ for all $1\leq j\leq K$, we have that   $U(y_{1,1}),\dots,U(y_{1,K})$ cover $Y_{1}$, and that there exist a partition of unity $f_{1},\dots,f_{K}\colon Y_{1}\to[0,1]$  subordinate to this cover with $\Vert f_{j}\Vert_{\Lip(Y_{1})}\leq N$ for all $1\leq j\leq K$. Let
  $$g(x):=\sum_{j=1}^{K}f_{j}(\pi_{1}(x))F''(y_{1,j},\pi_{2}(x))$$
  for all $x\in X$. Then for all $x\in X$, writing $y_{1}=\pi_{1}(x)$ and $y_{2}=\pi_{2}(x)$, we have that
    \begin{equation}\nonumber
    \begin{split}
   &\quad\vert g(x)-F(x)\vert=\Bigl\vert \sum_{j=1}^{K}f_{j}(y_{1})(F''(y_{1,j},y_{2})-F''(y_{1},y_{2}))\Bigr\vert
   \\&\leq     \sum_{j=1}^{K}\bold{1}_{y_{1}\in U(y_{1,j})}f_{j}(y_{1})d_{Y_{1}}(y_{1},y_{1,j})\Vert F''\Vert_{\Lip(Y)}
    \leq\sum_{j=1}^{K}\bold{1}_{y_{1}\in U(y_{1,j})}f_{j}(y_{1})O_{C,J,J'}(\d)=O_{C,J,J'}(\d)<\e
   \end{split}
    \end{equation}
    if we set  $\d$ to be sufficiently small depending on $C,\e,J,J'$. Then $N=N(C,\d,J,J')=O_{C,\e,J,J'}(1)$.

 Finally,  for all $1\leq j\leq K$, recall that $\Vert f_{j}\Vert_{\Lip(Y_{1})}\leq N=N(C,\d,J,J')=O_{C,\e,J,J'}(1)$ and $\Vert F''(y_{1,j},\cdot)\Vert_{\Lip(Y_{2})}=O_{C,J,J'}(1)$. So $\Vert f_{j}\circ\pi_{1}\Vert_{\Lip(X)}$,   $\Vert F''(y_{1,j},\pi_{2}(\cdot))\Vert_{\Lip(X)}=O_{C,\e,J,J'}(1)$.
 	We are done by setting $\psi_{j}:=f_{j}\circ\pi_{1}\circ g$ and $\psi'_{j}:=F''(y_{1,j},\pi_{2}\circ g(\cdot))$.
\end{proof}

 The next lemma allows us to approximate a nilsequence with nilcharacters. The method we use is similar to the discussion in Section 6 and Proposition 8.3 of \cite{GTZ12}.

\begin{lem}[Approximating a nilsequence with nilcharacters]\label{3:LE.5}
	Let $D,k,Q\in\N_{+}$, $C,\e>0$, $I$ be the degree, multi-degree or degree-rank ordering with $\dim(I)\vert k$, $s\in I$, 
	 $p$ be a prime, and $\Omega\subseteq \Z^{k}$.  %If $I$ is the degree-rank ordering, then let $s_{\ast}$ denote the smallest positive integer with $[s_{\ast},s_{\ast}]\geq s$. If $I$ is the degree or multi-degree ordering, denote $s_{\ast}:=1$.
	For any $\psi\in\Nil^{s;C,D}_{\approx Q}(\Omega)$, there exist $N=N(C,D,\e,k,s)\in\N$, and for all $1\leq j\leq N$ 
	a linear transformation $T_{j}\colon\mathbb{C}^{O_{C,D,\e,k,s}(1)}\to \mathbb{C}^{D}$ of suitable dimensions of complexity $O_{C,D,\e,k,s}(1)$, a nilcharacter
	$\chi_{j}\in\Xi^{s;O_{C,D,\e,k,s}(1),O_{C,D,\e,k,s}(1)}_{\approx Q}(\Omega)$, and a nilsequence $\psi_{j}\in\Nil^{\prec s;O_{C,D,\e,k,s}(1),1}_{\approx Q}(\Omega)$ such that 
	$$\Bigl\Vert\psi-\sum_{j=1}^{N}T_{j}(\psi_{j}\otimes \chi_{j})\Bigr\Vert_{\ell^{\infty}(\Omega)}<\e.$$
	
	In addition, if $I$ is  the degree or multi-degree ordering, then we may instead require $\chi_{j}\in\Xi^{s;O_{C,D,\e,k,s}(1),O_{C,D,\e,k,s}(1)}_{p}(\Omega)$  and $\psi_{j}\in\Nil^{\prec s;O_{C,D,\e,k,s}(1),1}_{p}(\Omega)$.
\end{lem}	
\begin{proof}
	It suffices to show this for scalar valued nilsequences $\psi$. By Corollary \ref{3:pppap}, 
it suffices to prove Lemma \ref{3:LE.5} under the milder restriction that 	$\chi_{j}\in\Xi^{s;O_{C,D,\e,k,s}(1),O_{C,D,\e,k,s}}_{\approx Q}(\Omega)$, and a nilsequence $\psi_{j}\in\Nil^{\prec s;O_{C,D,\e,k,s}(1),1}_{\approx Q}(\Omega)$.
%	we may assume without loss of generality that $\psi\in\Nil^{s;C,1}_{p^{s_{\ast}}}(\Omega)$.  
	Let $((G/\Gamma)_{I},g,F)$ be a $\Nil^{s}_{\approx Q}(\Omega)$-representation of $\psi$  of complexity at most $C$ and dimension 1. For convenience denote $X=G/\Gamma$ and $X'=G/G_{s}\Gamma$.
		Let $\pi\colon X\to X'$ be the quotient map. Note that filtering $G/G_{s}$ with the groups $G_{i}/G_{s}$, we may view $X'$ as an $I$-filtered nilmanifold of degree $\prec s$, whose complexity is obviously $O_{C,s}(1)$. The fibers of the map $\pi$ are isomorphic to $T:=G_{s}/(G_{s}\cap\Gamma)$. Since $G_{s}$ is abelian, $T$ is a torus, and so $G/\Gamma$ is a torus bundle over $X'$ with the structure group $T$.
	
	Let $\d>0$ to be chosen later, and $\sum_{j=1}^{K}\varphi_{j}$ be a smooth partition of unity on $X'$, where each $\varphi_{j}\in\Lip(X'\to\C)$ is supported on an open ball $B_{j}$ of radius $\d$ (with respect to the metric induced by the $O_{C,s}(1)$-rational Mal'cev basis of $X'$). Then $K=O_{C,\d,s}(1)$. This induces a partition $\psi=\sum_{j=1}^{K}\psi_{j}$, where
	$$\psi_{j}(n):=F(g(n)\Gamma)\varphi_{j}(\pi(g(n)\Gamma)):=F_{j}(g(n)\Gamma) \text{ for all } n\in \Omega.$$
	Then $F_{j}$ is compactly supported in the cylinder $\pi^{-1}(B_{j})$ an has Lipschitz constant $O_{C,\d,s}(1)$.

	We now pick $\d>0$ in a way such that for each $1\leq j\leq K$, there is a smooth section
	$\iota_{j}\colon B_{j}\to G$ which partially inverts the projection from $G$ to $X'$. Then we can take $\d=O_{C,s}(1)$.
	Fix some $1\leq j\leq K$.
	We can then parametrize any element $x$ of $\pi^{-1}(B_{j})$ uniquely as $\iota_{j}(x_{0})t\Gamma$ for some $x_{0}\in B_{j}$ and $t\in T$ (note that $t\Gamma$ is well defined as an element of $G/\Gamma$). 
	
	Now fix $1\leq j\leq K$.
	We can now view the Lipschitz function $F_{j}$ as a compactly supported Lipschitz function in $B_{j}\times T$ (again with Lipschitz constant $O_{C,s}(1)$). 	
	For convenience denote  
 $r:=\dim(G_{s})$ and 
$R:=\dim(G)$ (then $T\simeq\T^{r}$). Then there exists a smooth function $F'_{j}\colon X\to\C$ with
$$\vert F_{j}(x)-F'_{j}(x)\vert\leq \e/2K \text{ for all } x\in X, \text{ and } \Vert F'_{j}\Vert_{\mathcal{C}^{2R}(X)}=O_{C,\e,s}(1).$$

Let $\phi\colon G\to \R^{R}$ be the Mal'cev coordinate map. For $h\in\Z^{r}$, let $\xi_{h}\colon G_{s}\to \R$ denote the vertical character  such that $\xi_{h}(t)=(\bold{0},h)\cdot\phi(t)$ for all $t\in G_{s}$. 
We may apply Fourier decomposition in the $T$ direction to write 
$$F'_{j,h}(\iota_{j}(x_{0})t\Gamma):=\int_{T}e(-\xi_{h}(u))F'_{j}(\iota_{j}(x_{0})tu\Gamma)\,d m_{T}(u(G_{s}\cap\Gamma))$$
for all $(x_{0},t)\in B_{j}\times T$ and  $u\in G_{s}$, where $m_{T}$ is the Haar measure of $T$. 
We write $F'_{j,h}(x):=0$ if $x$ can not be written in the form $\iota_{j}(x_{0})t\Gamma$.
It is not hard to see that $\Vert F'_{j,h}\Vert_{\Lip(X)}=O_{C,\e,s}(1)$ and  
\begin{equation}\label{3:fhx}
F'_{j,h}(\iota_{j}(x_{0})t\Gamma)=e(\xi_{h}(t))F'_{j,h}(\iota_{j}(x_{0})\Gamma) \text{ for all } x_{0}\in B_{j}, t\in G_{s}.
\end{equation}
 Moreover, since $\Vert F'_{j}\Vert_{\mathcal{C}^{2R}(X)}=O_{C,\e,s}(1)$, we have that $\Vert F'_{j,h}\Vert_{L^{\infty}}=O_{C,\e,s}((1+\vert h\vert)^{-2R})$ and that 
$$F'_{j}(\iota_{j}(x_{0})t\Gamma)=\sum_{h\in\Z^{r}}F'_{j,h}(\iota_{j}(x_{0})t\Gamma) \text{ for all } (x_{0},t)\in B_{j}\times T.$$
So there exists $m=O_{C,\e,s}(1)$ such that
$$\Bigl\vert F_{j}(x)-\sum_{h\in\Z^{r}, \vert h\vert\leq m} F'_{j,h}(x)\Bigr\vert<\e/2K \text{ for all } x\in X.$$

Fix also $h\in\Z^{r}, \vert h\vert\leq m$. 
Let $F''_{j,h}\colon X'\to\C$ be the function given by $$F''_{j,h}(y):=F'_{j,h}(\iota_{j}(x_{0})\Gamma)=\int_{T}e(-\xi_{h}(t))F'_{j}(\iota_{j}(x_{0})t\Gamma)\,dm_{T}(t(G_{s}\cap\Gamma))$$
if $y=\iota_{j}(x_{0})\Gamma$ for some $x_{0}\in B_{j}$, and $F''_{j,h}(y):=0$ if not such expression exists.
Then $F''_{j,h}$ is a compactly supported function on $X'$ with Lipschitz norm bounded by $O_{C,\e,s}(1)$. 
On the other hand, by an argument similar to pages 1253--1255 of \cite{GTZ12}, there exist $D'=O_{C,k,s}(1)$ and a vector valued function $f_{j,h}=(f_{j,h,1},\dots,f_{j,h,D'})\in\Lip(X\to\mathbb{S}^{D'})$ supported on $\pi^{-1}(B_{j})$ such that the map $\iota(x_{0})t\Gamma\mapsto e(u\cdot t), (x_{0},t)\in B_{j}\times T$ can be viewed as linear combination of the components of $f_{j,h}$ of complexity $O_{C,\e,j,s}(1)$, and that for all $1\leq j\leq D'$, 
$$f_{j,h,j}(\iota(x_{0})tu\Gamma)=e(\xi_{h}(u))f_{j,h,j}(\iota(x_{0})t\Gamma) \text{ for all } u\in G_{s}, (x_{0},t)\in B_{j}\times T.$$
	 It then follows from (\ref{3:fhx}) that 
	 $$F'_{j,h}(x)=T_{j,h}(F''_{j,h}(\pi(x))\otimes f_{j,h}(x))$$
for some linear transformation $T_{j,h}$ of suitable dimension of complexity $O_{C,\e,k,s}(1)$ for all $x\in X$. By the triangle inequality, we have that 
\begin{equation}\label{3:fhx2}
\Bigl\vert F(x)-\sum_{j=1}^{K}\sum_{h\in\Z^{r}, \vert h\vert\leq m}T_{j,h}(F''_{j,h}(\pi(x))\otimes f_{j,h}(x))\Bigr\vert<\e
\end{equation}
	 	 for all $x\in X$.  The conclusion follows by replacing $x$ with $g(n)\Gamma$ in (\ref{3:fhx2}).
\end{proof}	

As a consequence of Lemma \ref{3:LE.5}, we have

\begin{coro}\label{3:LE.6}
Let $D,D',k\in\N_{+}$, $C,\e>0$, $I$ be the degree, multi-degree, or degree-rank ordering with $\dim(I)\vert k$, $s\in I$,
 $p$ be a prime,
and 	$\Omega\subseteq \Z^{k}$.  %If $I$ is the degree-rank ordering, then let $s_{\ast}$ denote the smallest positive integer with $[s_{\ast},s_{\ast}]\geq s$. If $I$ is the degree or multi-degree ordering, denote $s_{\ast}:=1$.
	For all $\psi\in\Nil^{s;C,D}_{\approx Q}(\Omega)$ and functions $f\colon \Omega\to\C^{D'}$ with $\vert f\vert\leq 1$, if
	$$\vert\E_{n\in\Omega}f(n)\otimes\psi(n)\vert>\e,$$
	then there exists $\chi\in\Xi^{s;O_{C,D,D',\e,k,s}(1),O_{C,D,D',\e,k,s}(1)}_{\approx Q}(\Omega)$ such that 
	$$\vert\E_{n\in\Omega}f(n)\otimes\chi(n)\vert\gg_{C,D,D',\e,k,s} 1.$$
	
	In addition, if $I$ is  the degree or multi-degree ordering, then we may instead require $\chi\in\Xi^{s;O_{C,D,D',\e,k,s}(1),O_{C,D,D',\e,k,s}(1)}_{p}(\Omega)$.
\end{coro}

The proof of Corollary \ref{3:LE.6} is almost identical to the proof of Corollary E.6 of \cite{GTZ12}, except that we replace the applications of Lemmas E.3, E.4 and E.5 in  \cite{GTZ12} by that of Lemmas \ref{3:LE80}, \ref{3:LE.4} and \ref{3:LE.5} in this paper. We omit the details.

Finally, we provide an approximation property for a special type of nilcharacters, which is similar to Proposition 3.1 of \cite{GTZ24} in spirit.

\begin{prop}\label{3:newappr}
Let $d,D,s\in\N_{+}$, $C,\e>0$ and $p$ be a prime. Let $k\in\V$ and $\chi\in\Xi_{p}^{(1,s);C,D}((\Z^{d})^{2})$. For $k,h\in\Z^{d}$, denote $$\chi_{k,h}(n):=\chi(h+k,n)\otimes \overline{\chi}(h,n)$$ for all $n\in\Z^{d}$. For $1\leq j\leq D^{2}$, let $\chi_{k,h,j}$ denote the $j$-th component of $\chi_{k,h}$.
Then there exist $L=O_{C,d,D,\e,s}(1)$ and for each $1\leq i\leq L$, $1\leq j\leq D^{2}$ a
linear transformation $T_{k,i,j}\colon\mathbb{C}^{O_{C,d,D,\e,s}(1)}\to \mathbb{C}$ of suitable dimensions of complexity $O_{C,d,D,\e,s}(1)$, and some  $\psi_{k,i,j}\in\Xi_{p}^{s;O_{C,d,D,\e,s}(1),O_{C,d,D,\e,s}(1)}(\Z^{d}), \psi'_{k,i,j}\in\Nil_{p}^{(1,s-1);O_{C,d,D,\e,s}(1),1}((\Z^{d})^{2})$ for all  $1\leq j\leq D^{2}$ such that 
\begin{equation}\label{3:gporwpeofg}
\Bigl\Vert \chi_{k,h,j}-\sum_{i=1}^{K}T_{k,i,j}(\psi_{k,i,j})\psi'_{k,i,j}(h,\cdot)\Bigr\Vert_{\ell^{\infty}(\Z^{d})}<\e
\end{equation}
    for all $h\in\Z^{d}$ and $1\leq j\leq D^{2}$.
\end{prop}
\begin{proof}
For $k\in\Z^{d}$, denote $$\chi'_{k}(h,n):=\chi(h+k,n)\otimes \overline{\chi}(h,n)$$ for all $h,n\in\Z^{d}$.
For $1\leq j\leq D^{2}$, let $\chi'_{k,j}$ denote the $j$-th component of $\chi'_{k}$.
By Lemma \ref{3:LE8} (viii) (translated in the $\Z^{d}$-setting), for all $k\in\Z^{d}$, we have that  $$\chi'_{k,j}\in\Nil_{p}^{\prec (1,s);O_{C,D,s}(1),1}((\Z^{d})^{2})\subseteq \Nil_{p}^{J\cup J';O_{C,D,s}(1),1}((\Z^{d})^{2}),$$ where $J=\{(i,j)\in\N^{2}\colon (i,j)\preceq (0,s)\}$ and $J'=\{(i,j)\in\N^{2}\colon (i,j)\preceq (1,s-1)\}$.
By Lemma \ref{3:LE.4}, there exist $L=O_{C,d,D,\e,s}(1)$ and $\phi_{k,i,j}\in\Nil^{J;O_{C,d,D,\e,s}(1),1}((\Z^{d})^{2}), \phi'_{k,i,j}\in\Nil^{J';O_{C,d,D,\e,s}(1),1}((\Z^{d})^{2})$ for all $1\leq i\leq L$ such that
\begin{equation}\label{3:gporwpeofg2}
\Bigl\Vert \chi'_{k,j}-\sum_{i=1}^{L}\phi_{k,i,j}\phi'_{k,i,j}\Bigr\Vert_{\ell^{\infty}((\Z^{d})^{2})}<\e/2.
\end{equation} 
Note that $\chi_{k,h,j}=\chi'_{k,j}(h,\cdot)$. 
By applying Corollary \ref{3:pppap} and enlarging $K$ if necessary, we may upgrade $\phi_{k,i,j}$ to an element in $\Nil_{p}^{(0,s);O_{C,d,D,\e,s}(1),1}((\Z^{d})^{2})$, and  $\phi'_{k,i,j}$ to an element in $\Nil_{p}^{(1,s-1);O_{C,d,D,\e,s}(1),1}((\Z^{d})^{2})$.
Since $\phi_{k,i,j}$ is independent of $h$, we may write $\phi_{k,i,j}(h,n)$ as $\phi_{k,i,j}(n)$ and consider $\phi_{k,i,j}$ as an element in $\Nil_{p}^{s;O_{C,d,D,\e,s}(1),1}(\Z^{d})$. Then (\ref{3:gporwpeofg2}) implies that
\begin{equation}\label{3:gporwpeofg3}
\Bigl\Vert \chi_{k,h,j}-\sum_{i=1}^{L}\phi_{k,i,j}\phi'_{k,i,j}(h,\cdot)\Bigr\Vert_{\ell^{\infty}(\Z^{d})}<\e/2
\end{equation}
    for all $h\in\Z^{d}$ and $1\leq j\leq D^{2}$. Let $\d>0$ to be chosen later depending only on $C,d,D,\e,s$.
    By Lemma \ref{3:LE.5}, there exist $L'=O_{C,d,D,\e,s}(1)$, and for all $1\leq i'\leq L'$ 
	a linear transformation $T_{k,i,i',j}\colon\mathbb{C}^{O_{C,d,D,\e,s}(1)}\to \mathbb{C}$ of suitable dimensions of complexity $O_{C,d,D,\e,s}(1)$, a nilcharacter
	$\phi_{k,i,i',j}\in\Xi^{s;O_{C,d,D,\e,s}(1),O_{C,d,D,\e,s}(1)}_{p}(\Z^{d})$, and a nilsequence $\phi_{k,i,i',j}\in\Nil^{s-1;O_{C,d,D,\e,s}(1),1}_{p}(\Z^{d})$	such that 
	\begin{equation}\label{3:gporwpeofg4}
\Bigl\Vert\phi_{k,i,j}-\sum_{j'=1}^{L'}T_{k,i,i',j}(\phi_{k,i,i',j})\phi_{k,i,i',j}\Bigr\Vert_{\ell^{\infty}(\Z^{d})}<\d.
\end{equation}
   Note that $\phi_{k,i,i',j}$ can be viewed as an element in $\Nil^{(0,s-1);O_{C,d,D,\e,s}(1),1}_{p}((\Z^{d})^{2})$ by treating it as a sequence 	which is independent of the first coordinate. So (\ref{3:gporwpeofg}) follows from (\ref{3:gporwpeofg3}) and (\ref{3:gporwpeofg4}) by choosing $\d$ to be sufficiently small depending  only on $C,d,D,\e,s$.
\end{proof}

\section{Equivalence of nilcharacters}\label{3:s:AppC}

 In this appendix, we prove some properties for the equivalence relation on nilsequences defined in Definition \ref{3:deneq}. We recall here Definition \ref{3:deneq} for the convenience of the readers:

\begin{defn}[An equivalence relation for nilcharacters]
	Let $k\in\N_{+}$, $C>0$, $p$ be a prime, $\Omega$ be a subset of $\F_{p}^{k}$, $I$ be the degree, multi-degree,  or  degree-rank ordering with $\dim(I)\vert k$ and let $s\in I$.	For $\chi,\chi'\in\Xi^{s}_{p}(\Omega)$, we write 
	$\chi\sim_{C}\chi' \mod\Xi^{s}_{p}(\Omega)$
	if $\chi\otimes \overline{\chi}\in\Nil^{\prec s;C}(\Omega)$. We write $\chi\sim\chi' \mod\Xi^{s}_{p}(\Omega)$ if $\chi\sim_{C}\chi' \mod\Xi^{s}_{p}(\Omega)$ for some $C>0$. 
\end{defn}

 Most (but not all) of the results in this appendix are very similar to the results in Appendix E of \cite{GTZ12}, and their proofs are routine. However, since we are working in a completely different setting from \cite{GTZ12}, it is inevitable for us to write down the proofs in full details to insures the correctness of the statements.

We first show that $\sim$ is an equivalent relation.

\begin{lem}\label{3:eqqeqq}
	Let  $D,k\in\N_{+}$, $C,C'>0$, $p$ be a prime,
	$\Omega$ be a subset of $\F_{p}^{k}$, $I$ be the degree, multi-degree,  or  degree-rank ordering with $\dim(I)\vert k$, and let $s\in I$. Let $\chi_{1},\chi_{2},\chi_{3}\in\Xi^{I;C,D}_{p}(\Omega)$. Then
	\begin{enumerate}[(i)]
		\item $\chi_{1}\sim_{O_{C}(1)} \chi_{1} \mod \Xi^{s}_{p}(\Omega)$ and $\chi_{1}\otimes\overline{\chi}_{1} \sim_{O_{C}(1)} 1 \mod \Xi^{s}_{p}(\Omega)$;
		\item if $\chi_{1}\sim_{C'} \chi_{2} \mod \Xi^{s}_{p}(\Omega)$, then $\chi_{2}\sim_{C'} \chi_{1} \mod \Xi^{s}_{p}(\Omega)$;
		\item if $\chi_{1}\sim_{C'} \chi_{2} \mod \Xi^{s}_{p}(\Omega)$ and $\chi_{2}\sim_{C'} \chi_{3} \mod \Xi^{s}_{p}(\Omega)$, then $\chi_{1}\sim_{C''} \chi_{3} \mod \Xi^{s}_{p}(\Omega)$
		for some $C''=O_{C,C',D}(1)$.
	\end{enumerate}	
\end{lem}

\begin{proof}
	The proof   is very similar to Lemma E.7 of \cite{GTZ12}. 	We first  prove Part (i). Let 
	$((G/\Gamma)_{I},g,F,\eta)$ be a $\Xi^{s}_{p}(\Omega)$-representation $\chi_{1}$ of complexity at most $C$. Then 
	$$\chi_{1}\otimes\overline{\chi}_{1}(n)=F\otimes \overline{F}(g(\tau(n))\Gamma,g(\tau(n))\Gamma):=F'(g(\tau(n))\Gamma)$$
	for all $n\in\Omega$.	
	It is clear that $F\otimes \overline{F}$ and $F'$ are of complexities $O_{C}(1)$.
	Since $F'$ is invariant under $G_{s}$, we may quotient $G$ by $G_{s}$ and represent $F'(g(\tau(n))\Gamma)$ using a nilmanifold of degree $\prec s$ with all the relevant complexities being $O_{C}(1)$.
	So $\chi_{1}\otimes\overline{\chi}_{1}\in\Nil^{\prec s;O_{C}(1)}(\Omega)$. This proves Part (i).
	
	Part (ii) is obvious and so it remains to prove Part (iii). Note that each component of 
	$$(\chi_{1}\otimes\overline{\chi}_{2})\otimes (\chi_{2}\otimes\overline{\chi}_{3})=\chi_{1}\otimes (\overline{\chi}_{2}\otimes\chi_{2})\otimes\overline{\chi}_{3}$$
	belongs to $\Nil^{\prec s;O_{C,C',D}(1)}(\Omega)$. Since the trace of $\overline{\chi}_{2}\otimes\chi_{2}$ is 1, $\chi_{1}\otimes\overline{\chi}_{3}$ is an $O_{C,C',D}(1)$-complexity linear combination of the components of $\chi_{1}\otimes(\overline{\chi}_{2}\otimes\chi_{2})\otimes\overline{\chi}_{3}$. So $\chi_{1}\otimes\overline{\chi}_{3}$ belongs to $\Nil^{\prec s;O_{C,C',D}(1)}(\Omega)$. This proves Part (iii).
\end{proof}

\begin{rem}
	Since $\sim$ is  an equivalent relation, in \cite{GTZ12}, the authors defined the equivalence class of a nilcharacter $\chi$ under the relation $\sim$ as the \emph{symbol} of $\chi$. However, although $\sim$ is an equivalent relation, $\sim_{C}$ is not an   equivalent relation for any fixed $C>0$. Since we do not use the non-standard analysis approach, instead of working with the relation $\sim$, we use the more quantitative relation $\sim_{C}$. Therefore, we will not use the notion of symbol introduced in \cite{GTZ12}.
\end{rem}

Here are some basic properties for the relation $\sim$.

\begin{lem}\label{3:LE8}
	Let $D,k\in\N_{+}$,  $C>0$, $p$ be a prime,
	$\Omega$ be a subset of $\F_{p}^{k}$, $I$ be the degree, multi-degree, or degree-rank ordering with $\dim(I)\vert k$, and let $s\in I$. 
	\begin{enumerate}[(i)]
		\item If $\chi,\chi'\in\Xi^{s;C,D}_{p}(\Omega)$ and $\psi\in\Nil^{\prec s;C,D}(\Omega)$, and the components of $\chi'$ are $C$-complexity linear combinations of those of $\chi\otimes \psi$, then $\chi\sim_{O_{C,D}(1)}\chi'\mod\Xi^{s}_{p}(\Omega)$.  
		\item Conversely, if $\chi\sim_{C}\chi'\mod\Xi^{s}_{p}(\Omega)$ for some $\chi,\chi'\in\Xi^{s;C,D}_{p}(\Omega)$, then $\chi$ is an $O_{C,D}(1)$-complexity linear combinations 	of $\chi\otimes \psi$ for some $\psi\in\Nil^{\prec s;C,D^{2}}(\Omega)$.
		\item If $\chi,\chi',\chi''\in\Xi^{s;C}_{p}(\Omega)$ and $\chi\sim_{C}\chi'\mod\Xi^{s}_{p}(\Omega)$, then   $\chi\otimes\chi'' \sim_{O_{C}(1)}\chi'\otimes\chi''\mod\Xi^{s}_{p}(\Omega)$. 
		\item If $\chi\in\Xi^{s}_{p}(\Omega)$ and $\Omega'\subseteq \Omega$, then $\chi\vert_{\Omega'}\in\Xi^{s}_{p}(\Omega')$.  Moreover, if $\chi'\sim_{C}\chi\mod\Xi^{s}_{p}(\Omega)$, then $\chi'\vert_{\Omega'}\sim_{C}\chi\vert_{\Omega'}\mod\Xi^{s}_{p}(\Omega')$.
		\item If $\chi\in\Xi^{s;C}_{p}(\Omega)$, $I$ is either the degree or multi-degree ordering, and $q\in\Z, p\nmid q$, then $\chi(\iota(q)\cdot)\sim_{O_{C,q}(1)} \chi^{\otimes q^{\vert s\vert}}\mod \Xi^{s}_{p}(\iota(q)^{-1}\Omega)$. 
		\item   If $\chi\in\Xi^{s}_{p}(L(\Omega))$ for some $d$-integral linear transformation $L\colon (\V)^{k}\to(\V)^{k'}$ and $I$ is the degree filtration, then $\chi\circ L\in\Xi^{s}_{p}(\Omega)$. Moreover, if $\chi'\sim_{C}\chi\mod\Xi^{s}_{p}(L(\Omega))$, then $\chi'\circ L\sim_{C}\chi\circ L\mod\Xi^{s}_{p}(\Omega)$. 
		\item If $\chi\in\Xi^{s;C,D}_{p}(\Omega)$, $I$ is either the degree or multi-degree ordering, and $q\in\Z, p>q$, then there exists $\tilde{\chi}\in\Xi^{s;O_{C,q,s}(1),O_{D,q,s}(1)}_{p}(\Omega)$   such
		$\chi\sim_{O_{C,q}(1)}\tilde{\chi}^{\otimes q}\mod\Xi^{s}_{p}(\Omega)$.	Moreover, if $\Omega=\F_{p}^{k}$, then we may further require $\tilde{\chi}$ to satisfy to following property:  for any linear transformation $L\colon \F_{p}^{k}\to \F_{k}$, if $\chi\circ L(n)=\chi(n)$ for all $n\in\F_{p}^{k}$, then $\tilde{\chi}\circ L(n)=\tilde{\chi}(n)$ for all $n\in\F_{p}^{k}$.
		\item   If $\chi\in\Xi^{s;C,D}_{p}(\F_{p}^{k})$ and  $I$ is the degree or multi-degree filtration, then for all $h\in\F_{p}^{k}$,  the sequence $n\mapsto \chi(n+h)\otimes\overline{\chi}(n), n\in\V$ belongs to $\Nil_{p}^{\prec s;O_{C,D,s}(1),D^{2}}(\F_{p}^{k})$.
	\end{enumerate}
\end{lem}
\begin{proof}
 Throughout the proof we assume that $s\neq \bold{0}$ since otherwise there is nothing to prove.

	Proof of Part (i). Assume that $\chi'=\sum_{i,j}c_{i,j}(\chi\otimes \psi)_{i,j}$ for some vectors $c_{i,j}$ of complexity at most $C$ and dimension at most $D^{2}$, where $(\chi\otimes \psi)_{i,j}$ denotes the $(i,j)$-th entry of $\chi\otimes \psi$. Then
	$$\chi\otimes\overline{\chi}'=\sum_{i,j}c_{i,j}\chi\otimes (\overline{\chi}'\otimes \overline{\psi})_{i,j}.$$
	Since $\chi\otimes (\overline{\chi}'\otimes \overline{\psi})=(\chi\otimes \overline{\chi}')\otimes \overline{\psi}\in \Nil^{\prec s;O_{C,D}(1)}(\Omega)$, we have that $\chi\otimes\overline{\chi}'\in \Nil^{\prec s;O_{C,D}(1)}(\Omega)$. This proves Part (i).
	
	Proof of Part (ii). Note that 
	$$\chi\otimes(\overline{\chi'}\otimes\chi')=(\chi\otimes\overline{\chi'})\otimes\chi'.$$
	By Lemma  \ref{3:LE80}, $\chi\otimes\overline{\chi'}\in\Nil^{\prec s;2C,D^{2}}(\Omega)$. The conclusion follows from  the fact that 1 is an $O_{D}(1)$-complexity linear combination of the components of $\overline{\chi'}\otimes\chi'$.

	Proof of Part (iii). Note that $(\chi\otimes\chi'')\otimes \overline{(\chi'\otimes\chi'')}$ is an reordering of the entries of $(\chi\otimes\overline{\chi}')\otimes(\chi''\otimes\overline{\chi}'')$. Since $\chi\otimes\overline{\chi}'\in \Nil^{\prec s;C}(\Omega)$ and by Lemma \ref{3:eqqeqq} $\chi''\otimes\overline{\chi}''\in \Nil^{\prec s;O_{C}(1)}(\Omega)$, we have that $(\chi\otimes\overline{\chi}')\otimes(\chi''\otimes\overline{\chi}'')$ and thus $(\chi\otimes\chi'')\otimes \overline{(\chi'\otimes\chi'')}$ belongs to $\Nil^{\prec s;O_{C}(1)}(\Omega)$. This proves Part (iii).

	Proof of Part (iv). It follows from the definitions.   
	
	Proof of Part (v). Let $((G/\Gamma)_{I},g,F,\eta)$ be a $\Xi^{s}_{p}(\Omega)$-representation of $\chi$ of complexity at most $C$ and dimension $D$.	%By Lemma \ref{3:pre2g}, increasing the complexity from $C$ to $O(C)$ if necessary, we may assume without loss of generality that $g(\bold{0})=id_{G}$.
	With a slight abuse of notations, we use $q$ to denote both the integer and its image in $\F_{p}$ under $\iota$.
	Since $g\in \poly_{p}(\Z^{k}\to G_{I}\vert\Gamma)$, we have that $\chi(qn)=F(g(\tau(qn))\Gamma)=F(g(q\tau(n))\Gamma)$ for all $n\in q^{-1}\Omega$. 
	
 	We may define an $I$-pre-filtration of $G^{2}$ by setting $G^{2}_{\bold{0}}:=G^{2}$ and $G^{2}_{i}, i\neq \bold{0}$ to be the group generated by $G_{j}\times G_{j}$ for all $j\succ i$ and by $(g^{q^{\vert i\vert}},g), g\in G_{i}$. Similar to the proof of Lemma E.8 (v) of \cite{GTZ12}, this is an $I$-pre-filtration of $G^{2}$. Moreover, $G^{2}_{s}$ is in the center of $G^{2}$. Let $(\tilde{G}/\tilde{\Gamma})_{I}$ be the essential part of $(G^{2}/\Gamma^{2})_{I}$. %Then since  $g(\bold{0})=id_{G}$, 
	By the type-I Taylor expansion, it is not hard to see that the map $n\mapsto (g(qn),g(n))$ is equal to a constant in $G^{2}$ times a polynomial sequence with respect to the filtration $\tilde{G}_{I}$, and thus is a polynomial sequence with respect to the pre-filtration $(G^{2})_{I}$ by Lemma \ref{3:pre3g}. 
	Similar to the proof of Lemma E.8 (v) of \cite{GTZ12}, by using Lemma \ref{3:pre2g}, it is not hard to show that   
		$$\tilde{\chi}(n):=F(g(qn)\Gamma)\otimes \overline{F}(g(n)\Gamma)^{\otimes q^{\vert s\vert}} n\in \iota^{-1}(q^{-1}\Omega)$$
	is a nilsequence of step $\prec s$ and complexity at most $O_{C,q}(1)$, namely $\tilde{\chi}\in\Nil^{\prec s;O_{C,q}(1)}(\iota^{-1}(\Omega))$. Composing $\tilde{\chi}$ with $\tau$, we have that $\chi(q\cdot)\sim_{O_{C,q}(1)} \chi^{\otimes q^{\vert s\vert}}\mod \Xi^{s}_{p}(\iota(q)^{-1}\Omega)$.
	
	Proof of Part (vi). The fact that $\chi\circ L$ belongs to $\Xi^{s}_{p}(\Omega)$ follows from Lemma \ref{3:LE80} and the inclusion $L^{-1}(L(\Omega))\supseteq \Omega$.
	If $\chi'\sim_{C}\chi \mod \Xi^{s}_{p}(L(\Omega))$, then writing $\chi=\tilde{\chi}\circ \tau$ and $\chi'=\tilde{\chi'}\circ \tau$ for some $\tilde{\chi},\tilde{\chi'}\in\Xi_{p}^{s}(\iota^{-1}(L(\Omega)))$, we have that 
	$$\chi(n)\otimes\overline{\chi'}(n)=(\tilde{\chi}\circ\tau(n))\otimes(\overline{\tilde{\chi'}}\circ\tau(n))=\phi\circ\tau(n)$$ 
	for some $\phi\in \Nil^{<s;C}(\iota^{-1}(L(\Omega)))$ for all $n\in L(\Omega)$.
	By Lemma \ref{3:lifting2}, there exists a  linear transformation $\tilde{L}\colon \Z^{k}\to\Z^{k'}$ such that $\tilde{L}\circ\tau\equiv \tau\circ L \mod p\Z^{k'}$. 
	So
	\begin{equation}\nonumber
	\begin{split}
	&\quad (\chi\circ L(n))\otimes(\overline{\chi'}\circ L(n))
	=(\tilde{\chi}\circ\tau \circ L(n))\otimes(\overline{\tilde{\chi'}}\circ\tau\circ L(n))
	\\&=(\tilde{\chi}\circ\tilde{L}\circ\tau)\otimes(\overline{\tilde{\chi'}}\circ\tilde{L}\circ\tau)
	=\phi\circ \tilde{L}\circ \tau(n)
	\end{split}
	\end{equation}
	for all $n\in\Omega$. Since $\phi\in \Nil^{<s;C}(\iota^{-1}(L(\Omega)))$, we have  $\phi\circ \tilde{L}\in \Nil^{<s;C}(\iota^{-1}(\Omega))$. So 
	$\chi'\circ L\sim_{C}\chi\circ L\mod\Xi^{s}_{p}(\Omega)$.

	Proof of Part (vii). 	We present a  proof which is different from Lemma E.8 (vii) of \cite{GTZ12}.
	Let $((G/\Gamma)_{I},g,F,\eta)$ be a $\Xi^{s}_{p}(\Omega)$-representation of $\chi$ of complexity at most $C$ and dimension at most $D$.
	Let $q^{\ast}$ be the unique integer in $\{1,\dots,p-1\}$ such that $qq^{\ast}\equiv 1\mod p\Z$  (such $q^{\ast}$ exists since $p>q$).
	With a slight abuse of notations, we use $q,q^{\ast}$ to denote both the integer and their images in $\F_{p}$ under $\iota$.
	Since $g$ belongs to $\poly_{p}(\Z^{k}\to G_{I}\vert\Gamma)$, 
	so does $g(q^{\ast}\cdot)$. Writing $\tilde{\chi}(n):=F(g(q^{\ast}\tau(n))\Gamma)$ for all $n\in\F_{p}^{k}$, we have that
	$\tilde{\chi}\in\Xi^{s;C,D}_{p}(\Omega)$ and that
	\begin{equation}\label{3:tempeeqq1}
	\tilde{\chi}(qn)=F(g(q^{\ast}\tau(qn))\Gamma)=F(g(\tau(qq^{\ast}n))\Gamma)=F(g(\tau(n))\Gamma)=\chi(n)
	\end{equation}
	for all $n\in\Omega$.
	By Part (v), $\chi\sim_{O_{C,q}(1)}\chi'^{q^{\vert s\vert}}\mod\Xi^{s}_{p}(\Omega)$. The claim follows by setting $\tilde{\chi}:=(\chi')^{\otimes q^{\vert s\vert-1}}$ (whose complexity is $O_{C,q,s}(1)$ and dimension is $O_{D,q,s}(1)$ by Lemma \ref{3:LE80}).
	
	For the ``moreover" part, 
	note that for all $n\in \F_{p}^{k}$, it follows from (\ref{3:tempeeqq1}) that 
	$$\tilde{\chi}(n)=\tilde{\chi}(q^{\ast}qn)=\chi(qn).$$
	So
	$$\tilde{\chi}(L(n))=\chi(qL(n))=\chi(L(qn))=\chi(qn)=\tilde{\chi}(n).$$	
	
	Proof of Part (viii). Let $((G/\Gamma)_{I},g,F,\eta)$ be a $\Xi^{s}_{p}(\F_{p}^{k})$-representation of $\chi$ of complexity at most $C$ and dimension $D$. %By Lemma \ref{3:pre2g}, increasing the complexity from $C$ to $O(C)$ if necessary, we may assume without loss of generality that $g(\bold{0})=id_{G}$. 
	Write 
	$$\chi'(n):=F(g(n+\tau(h))\Gamma)\otimes \overline{F}(g(n)\Gamma) \text{ for all } n\in\Z^{k}.$$
	Since $g\in\poly_{p}(\Z^{k}\to G_{\N}\vert\Gamma)$, it is clear that 
	$$\chi'(\tau(n))=F(g(\tau(n)+\tau(h))\Gamma)\otimes \overline{F}(g(\tau(n))\Gamma)=F(g(\tau(n+h))\Gamma)\otimes \overline{F}(g(\tau(n))\Gamma)=\chi(n+h)\otimes \overline{\chi}(n)$$
	for all $n\in\F_{p}^{k}$. It suffices to show that $\chi'\in\Nil^{\prec s;O_{C,D,s}(1),D^{2}}_{p}(\Z^{k})$.

	We may define an $I$-pre-filtration of $G^{2}$ by setting $G^{2}_{\bold{0}}:=G^{2}$ and $G^{2}_{i}, i\neq \bold{0}$ to be the group generated by $G_{j}\times id_{G}$ for all $j\succ i$ and by $(g,g), g\in G_{i}$. Similar to the proof of Lemma E.8 (v) of \cite{GTZ12}, this is an $I$-pre-filtration of $G^{2}$. %(with $(G^{2})_{[i,0]}=(G^{2})_{[i,1]}$ {\color{} is it correct????}). 
	Moreover, $G^{2}_{s}$ is in the center of $G^{2}$. 
		Let $(\tilde{G}/\tilde{\Gamma})_{I}$ be the essential part of $(G^{2}/\Gamma^{2})_{I}$. 
		
			Since $(g(n+\tau(h)),g(n))=(\Delta_{\tau(h)}g(n),id_{G})\cdot (g(n),g(n))$, it is not hard to see that the map $n\mapsto (g(n+\tau(h)),g(n))$ takes values in $\tilde{G}$ and thus belongs to $\poly(\Z^{k}\to(\tilde{G}/\tilde{\Gamma})_{I})$.				%Then since  $g(\bold{0})=id_{G}$, 
			Similar to the proof of Lemma E.8 (iv) of \cite{GTZ12},   by using Lemma \ref{3:pre2g},  one can show that 
	  $\chi'$ belongs to $\Nil^{\prec s;O_{C,D,s}(1)}(\Z^{k})$ (we omit the details).  The fact that $\chi'$ is $p$-periodic follows from the fact that both $g$ and $g(\cdot+\tau(h))$ are $p$-periodic. Finally, it is clear that the dimension of $\chi'$ is at most $D^{2}$.
%	
%	For $i\in I$, let $G^{2}_{i}$ denote the subgroup of $G^{2}$ generated by $G_{i+j}\times id_{G}, j\in I\backslash\{\bold{0}\}$ and $(g,g), g\in G_{i}$. It is easy to check that $(G^{2})_{i\in I}$ is an $I$-filtered pre-filtration of $G^{2}$ which is rational with respect to $\Gamma^{2}$. Setting $G^{\square}:=(G^{2})_{0}$ and $\Gamma^{\square}:=\Gamma^{2}\cap G^{\square}$, we have that  $G^{\square}/\Gamma^{\square}$ is an $I$-filtered nilmanifold of degree at most $s$.
\end{proof}

The following is a generalization of Lemma 13.2 of \cite{GTZ12}:  
\begin{lem}[Multi-linearity lemma]\label{3:L13.2}
	Let $d,k,s\in\N_{+}$, $C>0$,  $p$ be a prime, $L_{1},\dots,L_{s}$, $L'_{1}\colon (\V)^{k}\to \V$ be $d$-integral linear transformations, and 
	let $\chi\in\Xi^{(1,\dots,1);C}_{p}((\V)^{s})$ (with 1 repeated $s$ times).
	We have that 
	\begin{equation}\nonumber
	\begin{split}
	&\quad \chi(L_{1}(n)+L'_{1}(n),L_{2}(n),\dots,L_{s}(n))
	\\&\sim_{O_{C}(1)}\chi(L_{1}(n),L_{2}(n),\dots,L_{s}(n))\otimes\chi(L'_{1}(n),L_{2}(n),\dots,L_{s}(n))\mod\Xi^{s}_{p}((\V)^{k}),
	\end{split}
	\end{equation}
	where $n\in(\V)^{k}$.
	A similar result holds for the other $s-1$ variables.	 
	\end{lem}
\begin{proof}
	By Lemma \ref{3:LE80} and Lemma \ref{3:LE8} (vi) (and by viewing $\chi$ as a degree $s$ nilsequence),  	it suffices to show that
	\begin{equation}\label{3:l13.2e1}
	\begin{split}
	\chi(h_{1}+h'_{1},h_{2},\dots,h_{s})
	\sim_{O_{C,L}(1)}\chi(h_{1},h_{2},\dots,h_{s})\otimes\chi(h_{1}',h_{2},\dots,h_{s})\mod\Xi^{s}_{p}((\V)^{s+1}),
	\end{split}
	\end{equation}
	where both sides of (\ref{3:l13.2e1}) are viewed as functions of $(h_{1},\dots,h_{s},h'_{1})\in (\V)^{s+1}$.
	We may assume that for all $(h_{1},\dots,h_{s})\in (\V)^{s}$, $$\chi(h_{1},\dots,h_{s})=F(g\circ\tau(h_{1},\dots,h_{s})\Gamma)$$ for some $\N^{s}$-filtered nilmanifold $G/\Gamma$ of degree $\leq (1,\dots,1)$ and complexity at most $C$   endowed with a  Mal'cev basis $\mathcal{X}$, some $F\in\Lip(G/\Gamma\to\mathbb{S}^{D})$ with Lipschitz norm at most $C$ and with a vertical frequency $\eta$ for some $D\in\N_{+}$, and some $g\in \poly_{p}((\Z^{d})^{s}\to G_{\N^{s}}\vert \Gamma)$. 
%	By Lemma \ref{3:pre2g}, increasing the complexity from $C$ to $O(C)$ if necessary, we may assume without loss of generality that $g(\bold{0})=id_{G}$.
	It suffices to show that the map
	\begin{equation}\label{3:e13.6}
	\begin{split}
	(h_{1},\dots,h_{s},h'_{1})\mapsto F(g(h_{1}+h'_{1},h_{2},\dots,h_{s})\Gamma)\otimes \overline{F}(g(h_{1},\dots,h_{s})\Gamma)\otimes\overline{F}(g(h'_{1},h_{2},\dots,h_{s})\Gamma)
	\end{split}
	\end{equation}
	belongs to $\Nil^{s-1;O_{C}(1)}((\Z^{d})^{s+1})$.

	We may  write (\ref{3:e13.6}) in the form
	$$\tilde{F}(\tilde{g}(h_{1},\dots,h_{s},h'_{1})\Gamma^{3}),$$
	where 
	$$\tilde{g}(h_{1},\dots,h_{s},h'_{1}):=(g(h_{1}+h'_{1},\dots,h_{s}),g(h_{1},\dots,h_{s}),g(h'_{1},\dots,h_{s}))$$
	and $\tilde{F}\in\Lip(G^{3}/\Gamma^{3}\to \C^{D^{3}})$ is given by
	$$\tilde{F}(x_{1},x_{2},x_{3}):=F(x_{1})\otimes\overline{F}(x_{2})\otimes\overline{F}(x_{3}).$$
	Then $\tilde{F}$ is a function taking values in $\mathbb{S}^{D^{3}}$ of Lipschitz norm at most $3C$ with a vertical frequency $(\eta,-\eta,-\eta)$.
	Since $g\in \poly_{p}((\Z^{d})^{s}\to G_{\N^{s}}\vert\Gamma)$,
	we have that $\tilde{g}\in \poly_{p}((\Z^{d})^{s+1}\to (G^{3})_{\N^{s}}\vert\Gamma^{3})$.
	
	By Lemma \ref{3:B.9} and the Baker-Campbell-Hausdorff formula, we may write
	$$g(h_{1},\dots,h_{s})=\prod_{i_{1},\dots,i_{s}\in \N^{d}, \vert i_{j}\vert\in\{0,1\}}g_{i_{1},\dots,i_{s}}^{\binom{h_{1}}{i_{1}}\dots \binom{h_{s}}{i_{s}}}$$
	for some $g_{i_{1},\dots,i_{s}}\in G_{(\vert i_{1}\vert,\dots,\vert i_{s}\vert)}$, where we insists the constant term $g_{\bold{0},\dots,\bold{0}}$ being the leading term. %Since $g(\bold{0})=id_{G}$, we have that $g_{\bold{0},\dots,\bold{0}}=id_{G}$.
	We may now give $G^{3}$ an $\N$-pre-filtration by setting $G^{3}_{0}:=G^{3}$ and $(G^{3})_{i}, i\neq 0$ to be the group generated by $G_{(j_{1},\dots,j_{s})}^{3}$ for all $j_{1},\dots,j_{s}\in\N$ with $j_{1}+\dots+j_{s}>i$ together with the elements $(g_{1}g_{2},g_{1},g_{2})$ for all $g_{1},g_{2}\in G_{(j_{1},\dots,j_{s})}$ for some $j_{1}+\dots+j_{s}=i$. From the Baker-Campbell-Hausdorff formula one can verify that this is a pre-filtration on $G^{3}$ of complexity $O_{C}(1)$. 
	It is also clear that $(G^{3})_{s}$ is in the center of $G^{3}$.
	
	Let $(\tilde{G}/\tilde{\Gamma})_{\N}$ be the essential part of $(G^{3}/\Gamma^{3})_{\N}$. Then $\tilde{G}_{\N}$ is an $\N$-filtration of complexity $O_{C}(1)$ of degree at most $s$. 
	By the type-I Taylor expansion, it is not hard to see that  $\tilde{g}$ is equal to a constant in $G^{3}$ times a polynomial sequence with respect to the filtration $\tilde{G}_{\N}$, and thus is a polynomial sequence with respect to the pre-filtration $(G^{2})_{I}$ by Lemma \ref{3:pre3g}. %Since $g_{\bold{0},\dots,\bold{0}}=id_{G}$, it follows from
	%  the type-I Taylor expansion  that $\tilde{g}$ is a polynomial with respect to the filtration $\tilde{G}_{\N}$.  
	Finally, as $F$ has a vertical frequency $(\eta,-\eta,-\eta)$, $F$ is invariant under $\tilde{G}_{s}=(G^{3})_{s}=\{(g_{1}g_{2},g_{1},g_{2})\colon g_{1},g_{2}\in G_{(1,\dots,1)}\}$. 
	We may then quotient $G^{3}$ and $\tilde{G}$ by $\tilde{G}_{s}$ and apply Lemma \ref{3:pre2g}  to conclude that (\ref{3:e13.6}) is a degree $<s$ nilsequence on $(\Z^{d})^{s+1}$ (note that the quotient of $\tilde{g}$ belongs  $\poly_{p}((\Z^{d})^{s+1}\to (G^{3}/G^{3}_{s})_{\N^{s}}\vert\Gamma^{3}/(G^{3}_{s}\cap \Gamma^{3}))$), whose complexity is certainly $O_{C}(1)$.
\end{proof}

We now turn to nilcharacters defined on a set admitting a partially periodic Leibman dichotomy, where the factorization theorem (Theorem \ref{3:facf3}) applies. 
The following is a generalization of Lemma E.13 of \cite{GTZ12}, which says that if $\chi^{\otimes q}\sim 1 \mod \Xi_{p}^{s}(\Omega)$ for some small $q$, then $\chi\sim 1 \mod \Xi_{p}^{s}(\Omega)$.

\begin{lem}[Torsion-free lemma]\label{3:LE.13}
		Let $d,D,k,q\in\N_{+},r,s\in\N$, $C>0$, $p\gg_{C,d,D,k,q,r,s} 1$ be a prime, $M\colon\V\to\F_{p}$ be a quadratic form, and  $\Omega$ 
		be a subset of $(\V)^{k}$ such that one of the three assumptions (i)--(iii) in Theorem \ref{3:facf3} holds (here $r$ is the parameter appearing in the statement of Theorem \ref{3:facf3}).
	For any 	$\chi\in \Xi^{s;C,D}_{p}(\Omega)$, if $\chi^{\otimes q}\in \Nil^{s-1;C}(\Omega)$, then $\chi\in\Nil^{s-1;O_{C,d,D,k,q,r,s}(1)}(\Omega)$. 
	
	In particular, if $\chi^{\otimes q}\sim_{C}1 \mod \Xi^{s}_{p}(\Omega)$, then $\chi\sim_{O_{C,d,D,k,q,r,s}(1)}1 \mod \Xi^{s}_{p}(\Omega)$.  
\end{lem}	
\begin{proof}	
	By Lemma \ref{3:LE80}, we have that
  $\chi^{\otimes q}\in \Nil^{s-1;C,D^{q}}(\Omega)\cap \Xi^{s;O_{C,q}(1),D^{q}}_{p}(\Omega)$. So
	$$\vert\E_{n\in\Omega}\chi(n)^{\otimes q}\otimes\tilde{\chi}(n)\vert\geq 1$$
	for some $\tilde{\chi}\in \Nil^{s-1;C,D^{q}}(\Omega)$ (we may set $\tilde{\chi}:=\overline{\chi}^{\otimes q}$). 	By Corollary \ref{3:pppap}  and the Pigeonhole Principle, there exist $\chi_{0}\in \Nil^{s-1;O_{C,d}(1),D^{q}}_{p}((\V)^{k})$ and $\e=\e(C,d,k)>0$ such that 
	$$\vert\E_{n\in\Omega}\chi(n)^{\otimes q}\otimes \chi_{0}(n)\vert\geq \e.$$
	
	Let $((G/\Gamma)_{\N},g,F,\eta)$ be a $\Xi^{s}_{p}(\Omega)$-representation of $\chi$ of complexity at most $C$ and dimension at most $D$, and $((G'/\Gamma')_{\N},g',F')$ be a $\Nil^{s-1}_{p}((\V)^{k})$-representation of $\chi_{0}$ of complexity at most $O_{C,d}(1)$ and dimension at most $D^{q}$.	
	Let $\mathcal{F}$ be a growth function to be chosen later depending only on $C,d,D,k,q,r,s$.
	By Theorem \ref{3:facf3}, there exists $C\leq C'\leq O_{C,d,D,\mathcal{F},k,q,r,s}(1)=O_{C,d,D,k,q,r,s}(1)$ such that if $p\gg_{C,d,D,\mathcal{F},k,q,r,s} 1$, then there exist  
	a proper subgroup $W$ of $G\times G'$ which is $C'$-rational relative to $\mathcal{X}\times\mathcal{X}'$, and
	a factorization $$(g(n),g'(n))=(t,t')\cdot(h(n),h'(n))\cdot(\gamma(n),\gamma'(n)) \text{ for all } n\in (\Z^{d})^{k}$$  such that $t\in G$ and $t'\in G'$ are of complexities $O_{C'}(1)$,
	$(h,h')\in \poly_{p}(\iota^{-1}(\Omega)\to W_{\N}\vert\Gamma_{W})$, $\Gamma_{W}:=W\cap(\Gamma\times\Gamma')$ with $(h(n)\Gamma,h'(n)\Gamma')_{n\in\iota^{-1}(\Omega)}$ being $\mathcal{F}(C')^{-1}$-equidistributed on $W/\Gamma_{W}$, and that $\gamma\in\poly(\iota^{-1}(\Omega)\to G_{\N}\vert\Gamma)$, $\gamma'\in\poly(\iota^{-1}(\Omega)\to G'_{\N}\vert\Gamma)$.

	Denoting 
	$\tilde{F}(x,y):=F(t x)^{\otimes q}\otimes F'(t' y)$ for $x\in G/\Gamma$ and $y\in G'/\Gamma'$, 
	we have that
	$$\vert\E_{n\in \Omega}\chi(n)^{\otimes q}\otimes\chi_{0}(n)\vert=\vert\E_{n\in \Omega}\tilde{F}(h(\tau(n))\Gamma,h'(\tau(n))\Gamma')\vert=\vert\E_{n\in \iota^{-1}(\Omega)}\tilde{F}(h(n)\Gamma,h'(n)\Gamma')\vert>\e.$$ 
	Since $F$ and $F'$ are bounded in Lipschitz norm  by $C$, 
	and  $t,t'$ are of complexities $O_{C'}(1)$, we have that $\tilde{F}$ is bounded in Lipschitz norm  by some $Q_{C,D,q}(C')$, where $Q_{C,D,q}$ is viewed as a function of $C'$. Denote $Y:=W/\Gamma_{W}$ and let $m_{Y}$ be the Haar measure on $Y$.
	So
	$$\Bigl\vert\E_{n\in \iota^{-1}(\Omega)}\tilde{F}(h(n)\Gamma,h'(n)\Gamma')-\int_{Y}\tilde{F}\,dm_{Y}\Bigr\vert\leq D^{q+1}Q_{C,D,q}(C')\mathcal{F}(C')^{-1}\leq \e/2$$
	provided that $\mathcal{F}$ grows sufficiently fast such that $\mathcal{F}(C')\geq 4\e^{-1}D^{q+1}Q_{C,D,q}(C')$. Clearly $\mathcal{F}$ depends only on $C,d,D,k,q,r,s$.  
	Thus
	$$\Bigl\vert\int_{Y}\tilde{F}\,dm_{Y}\Bigr\vert>\e/2.$$ 
	Since $F$ has a vertical frequency $\eta$, $\tilde{F}$ has a vertical frequency $(\eta^{q},id)$. This means that $(\eta^{q},id)$ must annihilate $W_{s}$. On the other hand, since $W$ is a subgroup of $G\times G'$, we have that $W_{s}$ is a subgroup of $G_{s}\times \{id_{G'}\}$. Assume that the projection of $W_{j}$ to the first coordinate is $\tilde{G}_{j}$ for all $j\in \N$. It is not hard to see that $\tilde{G}_{\N}$ is an $\N$-filtration of complexity $O_{C,C',d,D,k,q,r,s}(1)$.
	Since $W_{s}=\tilde{G}_{s}\times \{id_{G'}\}$, $\eta^{q}$  must annihilate $\tilde{G}_{s}$, which is a subgroup of $G_{s}$. 
	Since $\eta$ is a continuous homomorphism on the connected abelian Lie group $G$, $\eta$ itself must also annihilates $\tilde{G}_{s}$.
	Let $G'':=\tilde{G}/\tilde{G}_{s}$, $\Gamma'':=(\tilde{G}\cap\Gamma)/(\tilde{G}_{s}\cap\Gamma)$, and $\pi\colon \tilde{G}\to G''$ be the projection map.  
	Then we may write $F(t\cdot)=F''\circ \pi$ for some function $F''\colon G''/\Gamma''\to\mathbb{C}^{D'}$. Writing
	$h'':=\pi\circ h\in\poly_{p}(\iota^{-1}(\Omega)\to G''_{I}\vert\Gamma'')$, %\subseteq \poly((\Z^{d})^{k}\to G''_{I})$, 
	we have that 
	$$\chi(n)=F(g(\tau(n))\Gamma)=F(h(\tau(n))\Gamma)=F''(\pi\circ h(\tau(n))\Gamma)=F''(h''(\tau(n))\Gamma'')$$
	for all $n\in\Omega$.
	Since $\tilde{G}$ is $O_{C,C',d,D,k,q,r,s}(1)$-rational relative to the $C$-rational Mal'cev basis $\mathcal{X}$, so is $G''$. 
	Since $F$ is bounded in Lipschitz norm  by $C$ and $t$ is of complexity $O_{C'}(1)$,  $F''$  has Lipschitz norm bounded by $O_{C,C'}(1)$. 
Since $C'\leq O_{C,d,D,k,q,r,s}(1)$,
we have that 
	$\chi\in \Nil^{s-1;O_{C,d,D,k,q,r,s}(1)}(\Omega)$.
\end{proof}

The following lemma  generalizes of Lemma E.12 of \cite{GTZ12},
which says that if a degree-$s$ nilchracter $\chi$ has nontrivial correlation with a degree-$(s-1)$ nilsequence, then $\chi$ is also a degree-$(s-1)$ nilsequence.

\begin{lem}\label{3:LE.11}
		Let $d,D,k\in\N_{+},r,s\in\N$,  $C,\e>0$, $p\gg_{C,d,D,\e} 1$ be a prime, $M\colon\V\to\F_{p}$ be a quadratic form, and $\Omega$ 
		be a subset of $(\V)^{k}$ such that one of the three assumptions (i)--(iii) in Theorem \ref{3:facf3} holds (here $r$ is the parameter appearing in the statement of Theorem \ref{3:facf3}). 
	Let $\chi\in \Xi^{s;C,D}_{p}(\Omega)$  
	and $\psi\in \Nil^{s-1;C,D}(\Omega)$. If
	$$\vert\E_{n\in \Omega}\chi(n)\otimes \psi(n)\vert>\e,$$
	then $\chi\in \Nil^{s-1;O_{C,d,D,\e,k,r,s}(1),D}(\Omega)$. In particular, $\chi\sim_{O_{C,d,D,\e,k,r,s}(1)} 1 \mod \Xi^{s}_{p}(\Omega)$.	 
\end{lem}	
\begin{proof}
	By Corollary \ref{3:LE.6}, there exists $\psi'\in\Nil^{s-1;O_{C,d,D,\e,k,r,s}(1),O_{C,d,D,\e,k,r,s}(1)}_{p}(\Omega)$ and $\e'=\e'(C,d,$ $D,\e,k,r,s)>0$ such that 
	$$\vert\E_{n\in \Omega}\chi(n)\otimes \psi'(n)\vert>\e'.$$
	Let $((G/\Gamma)_{\N},g,F,\eta)$ be a $\Xi^{s}_{p}(\Omega)$-representation of $\chi$ of complexity at most $C$ and dimension at most $D$, and $((G'/\Gamma')_{\N},g',F')$ be a $\Nil^{s-1}_{p}(\Omega)$-representation of $\psi'$ of complexity  and dimension at most $O_{C,d,D,\e}(1)$.	
	Let $\mathcal{F}$ be a growth function to be chosen later depending only on $C,d,D,\e,k,r,s$.
	By Theorem \ref{3:facf3}, there exists $C\leq C'\leq O_{C,d,D,\e,\mathcal{F},k,r,s}(1)=O_{C,d,D,\e,k,r,s}(1)$ such that if $p\gg_{C,d,D,\e,\mathcal{F},k,r,s} 1$, then there exist 
	a proper subgroup $W$ of $G\times G'$ which is $C'$-rational relative to $\mathcal{X}\times\mathcal{X}'$, and
	a factorization $$(g(n),g'(n))=(t,t')\cdot(h(n),h'(n))\cdot(\gamma(n),\gamma'(n)) \text{ for all } n\in (\Z^{d})^{k}$$  such that $t\in G$ and $t'\in G'$ are of complexities at most $O_{C'}(1)$, 
	$(h,h')\in \poly_{p}(\iota^{-1}(\Omega)\to W_{I}\vert\Gamma_{W})$, $\Gamma_{W}:=W\cap(\Gamma\times\Gamma')$ with $(h(n)\Gamma,h'(n)\Gamma')_{n\in\iota^{-1}(\Omega)}$ being $\mathcal{F}(C')^{-1}$-equidistributed on $W/\Gamma_{W})$, and that $\gamma\in\poly(\iota^{-1}(\Omega)\to G_{I}\vert\Gamma)$, $\gamma'\in\poly(\iota^{-1}(\Omega)\to G'_{I}\vert\Gamma)$.

	Denoting 
	$\tilde{F}(x,y):=F(tx)\otimes F'(t' y)$ for $x\in G/\Gamma$ and $y\in G'/\Gamma'$, 
	we have that
	$$\vert\E_{n\in \Omega}\chi(n)\otimes\psi'(n)\vert=\vert\E_{n\in \Omega}\tilde{F}(h(\tau(n))\Gamma,h'(\tau(n))\Gamma')\vert=\vert\E_{n\in \iota^{-1}(\Omega)}\tilde{F}(h(n)\Gamma,h'(n)\Gamma')\vert>\e'.$$ 
	Since $F$ and $F'$ are bounded in Lipschitz norm  by $C$ and  $t,t'$ are of complexities $O_{C'}(1)$, we have that $\tilde{F}$ is bounded in Lipschitz norm  by some $Q_{C,D}(C')$, where $Q_{C,D}$ is viewed as a function of $C'$. Denote $Y:=W/\Gamma_{W}$ and let $m_{Y}$ be the Haar measure on $Y$.
	So
	$$\Bigl\vert\E_{n\in \iota^{-1}(\Omega)}\tilde{F}(h(n)\Gamma,h'(n)\Gamma')-\int_{Y}\tilde{F}\,dm_{Y}\Bigr\vert\leq D^{2}Q_{C,D}(C')\mathcal{F}(C')^{-1}\leq \e'/2$$
	provided that $\mathcal{F}'$ grows sufficiently fast such that $\mathcal{F}(C')\geq 2D^{2}Q_{C,D}(C')/\e'$. Clearly,  $\mathcal{F}'$ depends only on $C,d,D$, $\e,k,r,s$.  
	Thus
	$$\Bigl\vert\int_{Y}\tilde{F}\,dm_{Y}\Bigr\vert>\d/2.$$ 
	Since $F$ has a vertical frequency $\eta$, $\tilde{F}$ has a vertical frequency $(\eta,id)$. This means that $(\eta,id)$ must annihilate $W_{s}$. On the other hand, since $W$ is a subgroup of $G\times G'$, we have that $W_{s}$ is a subgroup of $G_{s}\times \{id_{G'}\}$. Assume that the projection of $W_{j}$ to the first coordinate is $\tilde{G}_{j}$ for all $j\in \N$. It is not hard to see that $\tilde{G}_{\N}$ is a filtration of complexity $O_{C,C',d,D,\e,k,r,s}(1)$.
	Since $W_{s}=\tilde{G}_{s}\times \{id_{G'}\}$, $\eta$  must annihilate $\tilde{G}_{s}$, which is a subgroup of $G_{s}$. 
	Let $G'':=\tilde{G}/\tilde{G}_{s}$, $\Gamma'':=(\tilde{G}\cap\Gamma)/(\tilde{G}_{s}\cap\Gamma)$, and $\pi\colon \tilde{G}\to G''$ be the projection map. 
	Then we may write $F(t\cdot)=F''\circ \pi$ for some function $F''\colon G''/\Gamma''\to\mathbb{C}^{D'}$. Writing
	$h'':=\pi\circ h\in\poly_{p}(\iota^{-1}(\Omega)\to G''_{\N}\vert\Gamma'')$, we have that 
	$$\chi(n)=F(g(\tau(n))\Gamma)=F(h(\tau(n))\Gamma)=F''(\pi\circ h(\tau(n))\Gamma)=F''(h''(\tau(n))\Gamma'')$$
	for all $n\in\Omega$.	
	Since $\tilde{G}$ is $O_{C,C',d,D,\e,k,r,s}(1)$-rational relative to the $C$-rational Mal'cev basis $\mathcal{X}$, so is $G''$. Since $F$  has Lipschitz norm bounded by $C$ and $t$ is of complexity $O_{C'}(1)$, $F''$ has Lipschitz norm bounded by $O_{C,C'}(1)$. 
Since  $C'\leq O_{C,d,D,\e,k,r,s}(1)$, we have that 
	$\chi\in \Nil^{s-1;O_{C,d,D,\e,k,r,s}(1),D}(\Omega)$.
\end{proof}

\begin{rem}
	It is natural to ask whether the assumptions in Lemma \ref{3:LE.11} imply that $\chi$ belongs to $\Nil_{p}^{s-1}(\Omega)$ (i.e. we further require $\chi$ is $p$-periodic as a degree-$(s-1)$ nilsequence). Unfortunately we do not know the answer to this question. This is because the factorization result (Theorem \ref{3:facf3}) we use in the proof of Lemma \ref{3:LE.11} does not ensure that $h,h'$ are $p$-periodic. 
\end{rem}

We conclude this section with an analog of Theorem E.10 of \cite{GTZ12}, which
asserts that every nilcharacter can be written as the composition  of a linear transformation and a nilcharacter of multi-degree $(1,\dots,1)$, up to a nilsequence of lower degree.

\begin{thm}\label{3:E.10}
	Let $d,D,k\in\N_{+}$, $s_{1},\dots,s_{k}\in\N_{+}$, $C>0$, $p$ be a prime and $\Omega\subseteq (\V)^{k}$. Denote $s:=(s_{1},\dots,s_{k})$ and $\vert s\vert:=s_{1}+\dots+s_{k}$. 
	For all $\chi\in\Xi^{s;C,D}_{p}(\Omega)$, if   $p\gg_{C,d,s} 1$, 	then there exists $\tilde{\chi}\in\Xi^{(1,\dots,1);O_{C,d,s}(1),O_{C,d,s}(1)}_{p}((\V)^{\vert s\vert})$ (with 1 repeated $\vert s\vert$ times)  such that writing
	$$\chi'(n_{1},\dots,n_{k}):=\tilde{\chi}(n_{1},\dots,n_{1},\dots,n_{k},\dots,n_{k}), n_{1},\dots,n_{k}\in\V,$$
	(with each $n_{i}$ repeated $s_{i}$ times), we have that $\chi'\in\Xi^{s;O_{C,d,s}(1),O_{C,d,s}(1)}_{p}((\V)^{k})$ and
	$\chi\sim_{O_{C,d,k,s}(1)}\chi'\mod\Xi^{\vert s\vert}_{p}(\Omega)$. Furthermore, one can select
	$$\tilde{\chi}(n_{1,1},\dots,n_{1,s_{1}},\dots,n_{k,1},\dots,n_{k,s_{k}})$$
	to be symmetric with respect to the permutations of $n_{i,1},\dots,n_{i,s_{i}}$ for all $1\leq i\leq k$. 
	\end{thm}

In order to prove Theorem \ref{3:E.10}, we need the following lemma:

\begin{lem}\label{3:sppsp}
	Let $d,D,k\in\N_{+}$, $s_{1},\dots,s_{k}\in\N_{+}$, $C>0$, and $p$ be a prime. Denote $s:=(s_{1},\dots,s_{k})$ and $\vert s\vert:=s_{1}+\dots+s_{k}$. Let $\chi\in\Xi^{(1,\dots,1);C,D}_{p}((\V)^{\vert s\vert})$  (with 1 repeated $\vert s\vert$ times).
	\begin{enumerate}[(i)]  
		\item 
		 Let $\sigma\colon\{1,\dots,\vert s\vert\}\to\{1,\dots,\vert s\vert\}$ be a permutation and denote  
		$$\chi'(n_{1},\dots,n_{\vert s\vert})
		 :=\chi(n_{\sigma(1)},\dots,n_{\sigma(\vert s\vert)})$$
		 for all $(n_{1},\dots,n_{\vert s\vert})\in (\V)^{\vert s\vert}$.
		 Then $\chi'\in\Xi^{(1,\dots,1);C,D}_{p}((\V)^{\vert s\vert})$.
		\item Let  
		$$\chi'(n_{1},\dots,n_{k})
		:=\chi(n_{1},\dots,n_{1},\dots,n_{k},\dots,n_{k})$$
		for all $(n_{1},\dots,n_{k})\in (\V)^{k}$, where each $n_{i}$ appears $s_{i}$ times in the above expression. Then $\chi'\in\Xi^{s;C,D}_{p}((\V)^{k})$.
	\end{enumerate}	 
\end{lem}	

 Lemma \ref{3:sppsp} can be proved directly by using the type-I Taylor expansion (Lemma \ref{3:B.9}) and the Baker-Campbell-Hausdorff formula. We leave the details to the interested readers.

\begin{proof}[Proof of Theorem \ref{3:E.10}]
	The outline of the proof is similar to Theorem E.10 of \cite{GTZ12}.
	In our setting, the proof is more difficult in the following two senses. The first is that we are dealing with multi-dimensional nilcharacters, which make the constructions more complicated. The second is that we need some additional efforts to ensure that the nilcharacters $\tilde{\chi}$ and $\chi'$ we construct are $p$-periodic.
	
	By Lemma \ref{3:LE8} (iv),
	extending the domain of $\xi$ from $\Omega$ to $(\V)^{k}$ if necessary, we may assume without loss of generality that $\Omega=(\V)^{k}$. By Lemmas \ref{3:eqqeqq}, \ref{3:LE8} (vii), \ref{3:LE.13} and \ref{3:sppsp}, it suffices to construct $\tilde{\chi}\in\Xi^{(1,\dots,1);O_{C,d,s}(1),O_{C,d,s}(1)}_{p}((\V)^{\vert s\vert})$ which is symmetric with respect to the permutations of $n_{i,1},\dots,n_{i,s_{i}}$ for all $1\leq i\leq k$ such that $$\chi'\sim_{O_{C,d,s}(1)}\chi^{\otimes s!}\mod\Xi^{\vert s\vert}_{p}((\V)^{\vert s\vert}),$$ where $s!:=s_{1}!\dots s_{k}!$. 
	Assume that $\chi(n)=\xi\circ \tau(n),n\in (\V)^{k}$ for some $\xi\in \Xi^{s;C,D}_{p}((\Z^{d})^{k})$.
	We may write $\xi(n)=F(g(n)\Gamma)$ for all $n\in(\Z^{d})^{k}$, where $G/\Gamma$ is a nilmanifold of multi-degree $s$ and complexity at most $C$, $g\in\poly_{p}((\Z^{d})^{k}\to G_{\N^{k}}\vert\Gamma)$, and $F\in \Lip(G/\Gamma\to\mathbb{S}^{D'})$ is of Lipschitz norm at most $C$ for some $D'\leq D$, and has a vertical frequency $\eta$ of complexity at most $C$.
 %	By Lemma \ref{3:pre2g}, increasing the complexity from $C$ to $O(C)$ if necessary, we may assume without loss of generality that $g(\bold{0})=id_{G}$.
	
	We first build a multi-degree $(1,\dots,1)$ nilpotent group $\tilde{G}$. Let $\log\tilde{G}$ be the Lie algebra given by the direct sum (as a real vector space)
	$$\log\tilde{G}:=\oplus_{J\subseteq\{1,\dots,\vert s\vert\}}\log G_{\Vert J\Vert},$$
	where $\Vert J\Vert\in\N^{k}$ is the vector
	$$\Vert J\Vert:=(\vert J\cap\{s_{1}+\dots+s_{i-1}+1,\dots,s_{1}+\dots+s_{i-1}+s_{i}\}\vert)_{1\leq i\leq k}.$$
	For each $J\subseteq \{1,\dots,\vert s\vert\}$, there is a natural vector space embedding $\iota_{J}\colon \log G_{\Vert J\Vert}\to \log\tilde{G}$. We then endow $\log\tilde{G}$ with a Lie bracket structure such that for all $J,J'\subseteq\{1,\dots,\vert s\vert\}$ and $x_{J}\in \log G_{\Vert J\Vert}$, $x_{J'}\in \log G_{\Vert J'\Vert}$,
	$$[\iota_{J}(x_{J}),\iota_{J'}(x_{J'})]=0$$
	if $J\cap J'\neq \emptyset$ and 
	$$[\iota_{J}(x_{J}),\iota_{J'}(x_{J'})]=\iota_{J\cup J'}([x_{J},x_{J'}])$$
	otherwise. As is shown in Theorem E.10 of \cite{GTZ12}, such a definition complies with the Lie bracket axioms. 
	
	We then define the $\N^{\vert s\vert}$-pre-filtration of $\log \tilde{G}$. Let $\log\tilde{G}_{\bold{0}}:=\log\tilde{G}$.
 For any $(a_{1},\dots,a_{\vert s\vert})\in\N^{\vert s\vert}\backslash\{\bold{0}\}$, let $\log \tilde{G}_{(a_{1},\dots,a_{\vert s\vert})}$ denote the sub Lie algebra of $\log \tilde{G}$ generated by $\iota_{J}(x_{J})$ for all $J$ with $\bold{1}_{J}(j)\geq a_{j}$ for all $1\leq j\leq \vert s\vert$, and for all $x_{J}\in G_{\Vert J\Vert}$.  As is shown in Theorem E.10 of \cite{GTZ12}, this is an  $\N^{\vert s\vert}$-pre-filtration of multi-degree $(1,\dots,1)$ and it induces an  $\N^{\vert s\vert}$-filtration of multi-degree $(1,\dots,1)$ of the Lie group $\tilde{G}$ by taking the exponential function. It is clear that $\tilde{G}_{(1,\dots,1)}$ is in the center of $\tilde{G}$.
	
	Let $\tilde{\Gamma}$ be the group generated by $\exp(W!\iota_{J}(\log\gamma_{j}))$ for some $W\in\N$ to be chosen latter for all $J\subseteq \{1,\dots,\vert s\vert\}$ and $\gamma_{j}\in\Gamma_{\Vert J\Vert}$. Again by Theorem E.10 of \cite{GTZ12}, $\tilde{G}/\tilde{\Gamma}$ is a pre-nilmanifold and that $\tilde{\Gamma}_{(1,\dots,1)}$ is contained in $\iota_{(1,\dots,1)}(\log \Gamma_{s})$ if $W\gg_{d,s} 1$. Fix such a $W$. It is not hard to see that $\tilde{G}/\tilde{\Gamma}$ is of complexity $O_{C,d,s}(1)$.
	Let $\tilde{\eta}$ be the vertical frequency on $\tilde{G}_{(1,\dots,1)}$ given by
	$$\tilde{\eta}(\exp(\iota_{(1,\dots,1)}(\log g_{s}))):=\eta(g_{s}).$$
	Since $\tilde{\Gamma}_{(1,\dots,1)}\subseteq\iota_{(1,\dots,1)}(\log \Gamma_{s})$ and $G_{(1,\dots,1)}$ is a central subgroup of $G$, $\tilde{\eta}$ is indeed a vertical frequency. Moreover, $\tilde{\eta}$ is of complexity $O_{C,d,s}(1)$.
	
	We then construct a function $\tilde{F}\in\Lip (\tilde{G}/\tilde{\Gamma}\to\mathbb{S}^{O_{C,d,s}(1)})$ with vertical frequency $\tilde{\eta}$. Such a function can be constructed using partitions of unity as in (6.3) of \cite{GTZ12}, and so we omit the details. Moreover, we can make the complexity and dimension of $\tilde{F}$ to be $O_{C,d,s}(1)$. 
		
	Next we define a polynomial sequence $\tilde{g}$ as follows. 
	For all $j=(j_{1},\dots,j_{k}), j_{i}\in\N^{d}$, denote $\tilde{j}:=(\vert j_{1}\vert,\dots,\vert j_{k}\vert)\in \N^{k}$.
	For convenience denote $n:=(n_{1},\dots,n_{k})$ for $n\in(\Z^{d})^{k}$.
	By Lemma \ref{3:B.9}  and the Baker-Campbell-Hausdorff formula, we may write
	$$g(n_{1},\dots,n_{k})=\prod_{j=(j_{1},\dots,j_{k})\in (\N^{d})^{k}}g_{j}^{n_{1}^{j_{1}}\dots n_{k}^{j_{k}}}$$
	for some $g_{j}\in G_{j}$, where
	$j=(j_{1},\dots,j_{k}), j_{i}\in\N^{d}$ ranges over all elements in $(\N^{d})^{k}$ such that   $\tilde{j}\preceq s$, and these elements are arranged in some arbitrary order except we insist that the constant term $g_{\bold{0},\dots,\bold{0}}$ being the leading term. %Since $g(\bold{0})=id_{G}$, we may set $g_{\bold{0},\dots,\bold{0}}=id_{G}$.  
	
	For each $J\subseteq\{1,\dots,\vert s\vert\}$ and $j=(j_{1},\dots,j_{k})=(j_{i,t})_{1\leq i\leq k, 1\leq t\leq d}$, let $P(J,j)$ be the set of tuples $(I_{i,t})_{1\leq i\leq k, 1\leq t\leq d}, I_{i,t}\subseteq J\cap\{s_{1}+\dots+s_{i-1}+1,\dots,s_{1}+\dots+s_{i}\}$ with $\vert I_{i,t}\vert=j_{i,t}$ such that $J\cap \{s_{1}+\dots+s_{i-1}+1,\dots,s_{1}+\dots+s_{i}\}=\cup_{t=1}^{d}I_{i,t}$. Clearly, if $P(J,j)$ is non-empty, then $\Vert J\Vert=\tilde{j}$.
	For $j=(j_{1},\dots,j_{k})=(j_{i,t})_{1\leq i\leq k,1\leq t\leq d}, j_{i,t}\in\N$, denote $j!:=\prod_{1\leq i\leq k,1\leq t\leq d}j_{i,t}!$. 
	We then write
	\begin{equation}\nonumber
	\begin{split}
	&\quad\tilde{g}(n_{1},\dots,n_{\vert s\vert})
	\\&:=\prod_{j=(j_{1},\dots,j_{k})\in (\N^{d})^{k}, \tilde{j}\preceq s}
	\exp\Bigl(j!\sum_{J\subseteq\{1,\dots,\vert s\vert\}\colon \Vert J\Vert=\tilde{j}}\sum_{(I_{i,t})_{1\leq i\leq k, 1\leq t\leq d}\in P(J,j)}\bigl(\prod_{i=1}^{k}\prod_{t=1}^{d}\prod_{\ell\in I_{i,t}}n_{\ell,t}\bigr)\iota_{J}(\log g_{j})\Bigr)
	\end{split}
	\end{equation}
	for all $n_{1},\dots,n_{\vert s\vert}\in\V$, where $n_{\ell,t}$ is the $t$-th entry of $n_{\ell}$. 
	Let $(\tilde{G}'/\tilde{\Gamma'})_{\N^{\vert s\vert}}$ be the essential component of $(\tilde{G}/\tilde{\Gamma})_{\N^{\vert s\vert}}$.
	Since eachmonomial
	$$(n_{1},\dots,n_{\vert s\vert})\mapsto
	\exp\Bigl(j!\sum_{J\subseteq\{1,\dots,\vert s\vert\}\colon \Vert J\Vert=\tilde{j}}\sum_{(I_{i,t})_{1\leq i\leq k, 1\leq t\leq d}\in P(J,j)}\bigl(\prod_{i=1}^{k}\prod_{t=1}^{d}\prod_{\ell\in I_{i,t}}n_{\ell,t}\bigr)\iota_{J}(\log g_{j})\Bigr)$$
	belongs to $\poly((\Z^{d})^{\vert s\vert}\to \tilde{G}')$ when $j\neq \bold{0}$,
	It follows from  Corollary B.4 of \cite{GTZ12} that  $\tilde{g}$ is equal to a constant in $\tilde{G}$ times a polynomial sequence with respect to the filtration $\tilde{G}'_{I}$, and thus is a polynomial sequence with respect to the pre-filtration $\tilde{G}_{\N^{3}}$ by Lemma \ref{3:pre3g}.

	 and since $g_{\bold{0},\dots,\bold{0}}=id_{G}$, it follows from 
	  Corollary B.4 of \cite{GTZ12} that $\tilde{g}$ belongs to $\poly((\Z^{d})^{\vert s\vert}\to \tilde{G}')$.
	
	Set 
	$$\tilde{\xi}(n_{1},\dots,n_{\vert s\vert}):=\tilde{F}(\tilde{g}(\tau(n_{1},\dots,n_{\vert s\vert}))\tilde{\Gamma}) \text{ for all } (n_{1},\dots,n_{\vert s\vert})\in (\Z^{d})^{\vert s\vert}$$
	and
	$$\xi'(n_{1},\dots,n_{k}):=\tilde{\xi}(n_{1},\dots,n_{1},\dots,n_{k},\dots,n_{k}) \text{ for all } (n_{1},\dots,n_{k})\in (\Z^{d})^{k},$$
	(with each $n_{i}$ repeated $s_{i}$ times).
	From the construction, we see that  $\tilde{\xi}$ is a nilcharacter on $(\Z^{d})^{\vert s\vert}$ of multi-degree $(1,\dots,1)$ (with $1$ repeated $\vert s\vert$ times) and  $\xi'$ is a nilcharacter on $(\Z^{d})^{k}$ of multi-degree $s$, and that both of them are of
	complexity and dimension at most  $O_{C,d,s}(1)$.	It is also clear that $\tilde{\xi}$ is symmetric with respect to the permutations of $n_{s_{1}+\dots+s_{i-1}+1},\dots,n_{s_{1}+\dots+s_{i}}$ for all $1\leq i\leq k$.
	However, we caution the readers that $\tilde{\xi}$ and $\xi'$ need not to be $p$-periodic.
	
	\textbf{Claim.}
	The map
	\begin{equation}\label{3:EE2}
	n\mapsto\xi(n)^{\otimes s!}\otimes \overline{\xi'}(n) \text{ for all } n\in (\Z^{d})^{k}
	\end{equation}
	belongs to $\Nil^{\prec s;O_{C,d,s}(1),O_{C,d,s}(1)}((\Z^{d})^{k})$. 
	
	It is not hard to see that 
	$\tilde{j}!=j!\vert P(J,j)\vert$ whenever $\Vert J\Vert=\tilde{j}$. So
	we may   expand (\ref{3:EE2}) as
	$$(F^{\otimes s!}\otimes \overline{\tilde{F}})\Bigl(\prod_{j=(j_{1},\dots,j_{k})\in (\N^{d})^{k}, \tilde{j}\preceq s}\Bigl(g_{j},\exp\Bigl(\tilde{j}!\sum_{J\subseteq \{1,\dots,\vert s\vert\}\colon \Vert J\Vert=\tilde{j}}\iota_{J}(\log g_{j})\Bigr)\Bigr)^{n_{1}^{j_{1}}\dots n_{k}^{j_{k}}}(\Gamma\times\tilde{\Gamma})\Bigr).$$
	The function $(F^{\otimes s!}\otimes \overline{\tilde{F}})$ is a Lipschitz function of complexity $O_{C,d,s}(1)$ on the nilmanifold $(G\times \tilde{G})/(\Gamma\times\tilde{\Gamma})$ of complexity $O_{C,d,s}(1)$. Let $G^{\ast}$ be the subgroup of $G\times \tilde{G}$ defined by
	$$G^{\ast}:=\{(g,\exp(s!\iota_{(1,\dots,1)}(\log g)))\colon g\in G_{s}\}\leq G_{s}\times \tilde{G}_{(1,\dots,1)}.$$
	This is a rational central subgroup of $G\times \tilde{G}$. As $F$ and $\tilde{F}$ have vertical frequencies $\eta$ and $\tilde{\eta}$ respectively, $F^{\otimes s!}\otimes \overline{\tilde{F}}$ is invariant under the group action $G^{\ast}$ and thus descends to a Lipschitz function $F'$ on the nilmanifold $G'/\Gamma'$, where $G':=(G\times \tilde{G})/G^{\ast}$ and $\Gamma'$ is the projection of $\Gamma\times \tilde{\Gamma}$ to $G$. We thus have
	\begin{equation}\label{3:EE3}
	\xi(n)^{\otimes s!}\otimes \overline{\xi'}(n)=F'\Bigl(\prod_{j=(j_{1},\dots,j_{k})\in (\N^{d})^{k}, \tilde{j}\preceq s}(g'_{j})^{n^{j}}\Gamma'\Bigr),
	\end{equation}
	where $g'_{j}$ is the projection of $\Bigl(g_{j},\exp\Bigl(\tilde{j}!\sum_{J\subseteq \{1,\dots,\vert s\vert\}\colon \Vert J\Vert=\tilde{j}}\iota_{J}(\log g_{j})\Bigr)\Bigr)$ to $G'$.
	
	We now give $G'$ an $\N^{k}$-pre-filtration as follows. Let $G'_{\bold{0}}:=G'$. 
 For all $u\in\N^{k}\backslash\{\bold{0}\}, u\preceq s$, let $G'_{u}$ be the group generated by
	$$\Bigl(h_{j},\exp\Bigl(\tilde{j}!\sum_{J\subseteq \{1,\dots,\vert s\vert\}\colon \Vert J\Vert=\tilde{j}}\iota_{J}(\log h_{j})\Bigr)\Bigr) \mod G^{\ast}$$
	for all $j\in (\N^{d})^{k}$ with $\tilde{j}=u$ and $h_{j}\in G_{\Vert J\Vert}$,
	and 
	$$(h_{j},id), (id, \exp(\iota_{J}(\log h_{j}))) \mod G^{\ast}$$
	for all $J\subseteq\{1,\dots,s\}$ with $\Vert J\Vert=\tilde{j}\succ u$ and $h_{j}\in G_{\Vert J\Vert}$. 
	By the Baker-Campbell-Hausdorff formula, one can show that this is a pre-filtration of degree $\prec s$. (Here we used the fact that for any $K\subseteq \{1,\dots,\vert s\vert\}$ with $\Vert K\Vert=\tilde{j}+\tilde{j'}$, then number of partitions $K=J\cup J'$ with $\Vert J\Vert=\tilde{j}$ and $\Vert J'\Vert=\tilde{j'}$ is $\frac{(\tilde{j}+\tilde{j})!}{\tilde{j}!\tilde{j'}!}$, which cancels the $u!$ term appearing in the definition of $G'_{u}$.) 
	Moreover, $g'_{j}\in G'_{j}$ by construction. So by Corollary B.4 of \cite{GTZ12}, Lemmas \ref{3:pre3g} and \ref{3:pre2g}, it is not hard to see that the right hand side of (\ref{3:EE3}) is a nilsequence of degree $\prec s$, whose complexity and dimension are certainly bounded by $O_{C,d,s}(1)$. This proves the claim.
	
	\

	Due to the fact that $\tilde{\xi}$ is not $p$-periodic, we can not directly apply the claim and set $\tilde{\chi}:=\tilde{\xi}\circ\tau$ to complete the proof. Our strategy is to approximate $\tilde{\xi}$ with $p$-periodic nilcharacters. For convenience denote $n:=(n_{1},\dots,n_{k})\in (\Z^{d})^{k}$.
	Since $\xi^{\otimes s!}\otimes\overline{\xi}'$ takes values in $\mathbb{S}^{O_{C,d,s}(1)}$, by the claim, there exists $\psi\in\Nil^{\prec s;O_{C,d,s}(1),O_{C,d,s}(1)}((\Z^{d})^{k})\subseteq \Nil^{\vert s\vert-1;O_{C,d,s}(1),O_{C,d,s}(1)}((\Z^{d})^{k})$  (one can take $\psi:=\overline{\xi}^{\otimes s!}\otimes\xi'$) such that 
	\begin{equation}\nonumber%\label{3:sskeqp}
	\begin{split}
	\vert\E_{n\in [p]^{dk}} \xi(n)^{\otimes s!}\otimes \overline{\tilde{\xi}}(n_{1},\dots,n_{1},\dots,n_{k},\dots,n_{k})\otimes \psi(n)\vert=1,
	\end{split}
	\end{equation}
where we denote $[p]:=\{0,\dots,p-1\}$.
	Note that $\tilde{\xi}$ belongs to $\Nil^{(1,\dots,1);O_{C,d,s}(1),O_{C,d,s}(1)}((\Z^{d})^{\vert s\vert})$.
	By Corollary \ref{3:LE.6} (applied to the set of $(n_{1,1},\dots,n_{1,s_{1}},\dots,n_{k,1},\dots,n_{k,s_{k}})\in(\V)^{\vert s\vert}$ such that $n_{i,j}=n_{i,j'}$ for all $1\leq i\leq k, 1\leq j,j'\leq s_{i}$), 
	there exists  $\tilde{\xi}'\in \Xi^{(1,\dots,1);O_{C,d,s}(1),O_{C,d,s}(1)}_{p}((\Z^{d})^{\vert s\vert})$  
	such that 
	\begin{equation}\label{3:sskeq1}
	\begin{split}
	\vert\E_{n\in [p]^{dk}} \xi(n)^{\otimes s!}\otimes \overline{\tilde{\xi}'}(n_{1},\dots,n_{1},\dots,n_{k},\dots,n_{k})\otimes \psi(n)\vert\gg_{C,d,s} 1.
	\end{split}
	\end{equation}

	Denote 
	\begin{equation}\nonumber
	\begin{split}
	 \tilde{\alpha}(n_{1,1},\dots,n_{1,s_{1}},\dots,n_{k,1},\dots,n_{k,s_{k}})
	 :=\bigotimes_{\sigma_{1},\dots,\sigma_{k}}\tilde{\xi}'(n_{1,\sigma_{1}(1)},\dots,n_{1,\sigma_{1}(s_{1})},\dots,n_{k,\sigma_{k}(1)},\dots,n_{k,\sigma_{k}(s_{k})})
	\end{split}
	\end{equation}
	for all $n_{i,j}\in\Z^{d}$, 
	where $\sigma_{1},\dots,\sigma_{k}$ ranges over all the permutations $\sigma_{i}\colon\{1,\dots,s_{i}\}\to\{1,\dots,s_{i}\}, 1\leq i\leq k$. 
	By Lemmas \ref{3:LE80} and \ref{3:sppsp},
	we have   $\tilde{\alpha}\in \Xi^{(1,\dots,1);O_{C,d,s}(1),O_{C,d,s}(1)}_{p}((\Z^{d})^{\vert s\vert})$. By Lemma \ref{3:LE8} (vii), there exists $\tilde{\beta}\in \Xi^{(1,\dots,1);O_{C,d,s}(1),O_{C,d,s}(1)}_{p}((\Z^{d})^{\vert s\vert})$ such that
	$$\tilde{\beta}^{\otimes s!}\sim_{O_{C,d,s}(1)}\tilde{\alpha} \mod \Xi^{(1,\dots,1)}_{p}((\Z^{d})^{\vert s\vert}).$$
	Since $\tilde{\alpha}$ is symmetric with respect to the permutations of $n_{s_{1}+\dots+s_{i-1}+1},\dots,n_{s_{1}+\dots+s_{i}}$ for all $1\leq i\leq k$ by construction, it follows from Lemma \ref{3:LE8} (vii) that we may also require $\tilde{\beta}$ to be symmetric with respect to these permutations.

	Set
	$$\alpha'(n_{1},\dots,n_{k}):=\tilde{\alpha}(n_{1},\dots,n_{1},\dots,n_{k},\dots,n_{k})$$
	and 
$$\beta'(n_{1},\dots,n_{k}):=\tilde{\beta}(n_{1},\dots,n_{1},\dots,n_{k},\dots,n_{k})$$
for all  $n_{1},\dots,n_{k}\in\Z^{d}.$ 
Since $\tilde{\alpha},\tilde{\beta}\in \Xi^{(1,\dots,1);O_{C,d,s}(1),O_{C,d,s}(1)}_{p}((\Z^{d})^{\vert s\vert})$,
by  Lemma \ref{3:sppsp}, we have that   $\alpha',\beta'\in \Xi^{s;O_{C,d,s}(1),O_{C,d,s}(1)}_{p}((\Z^{d})^{k})$. 
 By Lemma \ref{3:LE8} (vi), we have that  $${\beta'}^{\otimes s!}\sim_{O_{C,d,s}(1)} \alpha' \mod \Xi^{\vert s\vert}_{p}((\Z^{d})^{k}).$$
So
	\begin{equation}\nonumber
	\begin{split}
	&\quad\tilde{\xi}'(n_{1},\dots,n_{1},\dots,n_{k},\dots,n_{k})^{\otimes s!}
	\\&=\alpha'(n_{1},\dots,n_{k})\sim_{O_{C,d,s}(1)}\beta'(n_{1},\dots,n_{k})^{\otimes s!}\mod \Xi^{\vert s\vert}_{p}((\Z^{d})^{k}).
	\end{split}
	\end{equation}
	By Lemmas \ref{3:eqqeqq}, \ref{3:LE8} (ii), \ref{3:LE.13}, (\ref{3:sskeq1}) and the Pigeonhole Principle, there exists $\psi'\in\Nil^{\vert s\vert-1;O_{C,d,s}(1),O_{C,d,s}(1)}((\Z^{d})^{k})$ such that 
	$$\vert\E_{n\in[p]^{dk}} \xi(n)^{\otimes s!}\otimes \overline{\beta'}(n) \otimes\psi'(n)\vert\gg_{C,d,s} 1.$$
	Denoting $\tilde{\chi}:=\tilde{\beta}\circ\tau\in \Xi^{(1,\dots,1);O_{C,d,s}(1),O_{C,d,s}(1)}_{p}((\V)^{\vert s\vert})$ and $\chi':=\beta'\circ\tau\in \Xi^{s;O_{C,d,s}(1),O_{C,d,s}(1)}_{p}((\V)^{k})$, we have that
	$$\chi'(n_{1},\dots,n_{k})=\tilde{\chi}(n_{1},\dots,n_{1},\dots,n_{k},\dots,n_{k}) \text{ for all } n_{1},\dots,n_{k}\in\V$$  
	and that
	\begin{equation}\label{3:s1s1}
	\vert\E_{n\in(\V)^{k}} \xi(\tau(n))^{\otimes s!}\otimes \overline{\chi'}(n) \otimes\psi'(\tau(n))\vert\gg_{C,d,s} 1.
	\end{equation} 
Since $\chi=\xi\circ\tau$,
	By Lemma \ref{3:LE.11} and (\ref{3:s1s1}), we have that 
	$\chi^{\otimes s!}\sim_{O_{C,d,s}(1)}\chi'\mod\Xi^{\vert s\vert}_{p}((\V)^{\vert s\vert})$.
	Finally, since $\tilde{\beta}$ is symmetric with respect to the permutations of $n_{s_{1}+\dots+s_{i-1}+1},\dots,n_{s_{1}+\dots+s_{i}}$ for all $1\leq i\leq k$, so is $\tilde{\chi}$. This completes the proof. 
\end{proof}

\section{The converse of the spherical Gowers inverse theorem}\label{3:s:AppD}

In this appendix, we provide the proof of  the converse direction of $\SGI(s)$. 

\begin{prop}[Converse of $\SGI(s)$]\label{3:cinv}
	Let $r,s\in\N$, $d,D\in\N_{+}$, 
	and $C,\e>0$. There exist $\delta:=\d(C,d,D,\e)>0$ and $p_{0}:=p_{0}(C,d,D,\e)\in\N$ such that for every prime $p\geq p_{0}$, every quadratic form $M\colon\V\to\F_{p}$, every  affine subspace $V+c$ of $\V$ of co-dimension $r$, every function $f\colon\V\to \mathbb{C}^{D}$ bounded by 1,
	and every $\psi\in\Nil^{s;C,D}(V(M)\cap (V+c))$,
	if $\rank(M\vert_{V+c})\geq s^{2}+3s+5$ and
	\begin{equation}\label{3:csg0}
	\begin{split}
	\Bigl\vert\E_{n\in V(M)\cap (V+c)}f(n)\otimes \psi(n)\Bigr\vert>\e,
	\end{split}
	\end{equation}
	then $\Vert f\Vert_{U^{s+1}(V(M)\cap (V+c))}>\d$.  
\end{prop}

\begin{proof}%[Proof of Proposition \ref{3:cinv}] 
	The scheme of the proof is similar to Proposition 1.4 of \cite{GTZ11}. 
	There is nothing to prove when $s=0$ since degree 0 nilsequences are constants. Suppose now that the conclusion holds for some $s-1\in\N$ and we prove it for $s$. 
	By Corollary \ref{3:LE.6}, 	the induction hypothesis, (\ref{3:csg0}) and the Pigeonhole Principle, there exist $\psi'\in\Xi^{s;O_{C,d,D,\e}(1),O_{C,d,D,\e}(1)}_{p}(V(M)\cap (V+c))$ 
	such that 
	\begin{equation}\label{3:csgie1}
	\begin{split}
	\Bigl\vert\E_{n\in V(M)\cap (V+c)}f(n)\otimes \psi'(n) \Bigr\vert\gg_{C,d,D,\e} 1.
	\end{split}
	\end{equation}
	Let $((G/\Gamma)_{\N},g,F,\eta)$ be a $\Xi^{s}_{p}(V(M)\cap (V+c))$-representation of $\psi'$ of complexity and dimension at most $O_{C,d,D,\e}(1)$.	Write $\tilde{g}=g\circ \tau$.
	Let $\phi\colon\F_{p}^{d-r}\to V$ be any bijective linear transformation and set $M'(m):=M(\phi(m)+c)$. Then $(n,h)\in \Gow_{1}(V(M)\cap(V+c))$ if and only if $n=\phi(m)+c, h=\phi(h')$ for some $(m,h')\in V(M')$.
	By taking the square of (\ref{3:csgie1}), we have that 
	\begin{equation}\label{3:csgie2}
	\begin{split}
	&\quad 1\ll_{C,d,D,\e}
	\vert\E_{(n,h)\in \Gow_{1}(V(M)\cap(V+c))}\Delta_{h}f(n)\otimes F(\tilde{g}(n+h)\Gamma)\otimes\overline{F}(\tilde{g}(n)\Gamma)\vert
	\\&=\vert\E_{(m,h')\in \Gow_{1}(V(M'))}\Delta_{\phi(h')}f(\phi(m)+c)\otimes F(\tilde{g}(\phi(m+h')+c)\Gamma)\otimes\overline{F}(\tilde{g}(\phi(m)+c)\Gamma)\vert
	\\&=\E_{h'\in\F_{p}^{d-r}}\vert\E_{m\in V(M')^{h'}}\Delta_{\phi(h')}f(\phi(m)+c)\otimes F(\tilde{g}(\phi(m+h')+c)\Gamma)\otimes\overline{F}(\tilde{g}(\phi(m)+c)\Gamma)\vert+O(p^{-1/2}),
	\end{split}
	\end{equation}
	where 
	we used Theorem \ref{3:ct} in the last equality since $\rank(M')=\rank(M\vert_{V+c})\geq 5$ and $\Gow_{1}(V(M'))$ is an $M'$-set of total co-dimension 2.	
	So if $p\gg_{C,d,D,\e} 1$, then
	there exists $H'\subseteq \F_{p}^{d-r}$ of cardinality $\gg_{C,d,D,\e}p^{d-r}$ such that for all $h'\in H'$,
	$$\vert\E_{m\in V(M')^{h'}}\Delta_{\phi(h')}f(\phi(m)+c)\otimes F(\tilde{g}(\phi(m+h')+c)\Gamma)\otimes\overline{F}(\tilde{g}(\phi(m)+c)\Gamma)\vert\gg_{C,d,D,\e} 1.$$
	
	Let $A'$ be the $(d-r)\times (d-r)$ matrix associated with $M'$ and $H''$ be the set of $h'\in H'$ such that $h'A'\neq \bold{0}$. Then $\vert H'\backslash H''\vert\leq p^{d-r-\rank(A')}$.
	Since 
	$\rank(A')=\rank(M\vert_{V+c})\geq 1$, we have that $\vert H''\vert\gg_{C,d,D,\e}p^{d-r}$.		
	For $h'\in H''$, note that $V(M')^{h'}$ is the intersection of $V(M)$ with $U_{h'}:=\{m\in\F_{p}^{d-r}\colon M'(m+h')=M'(m)\}$. 
	Since $h'A'\neq \bold{0}$,  $U_{h'}$ is an affine subspace of $\F_{p}^{d-r}$ of co-dimension 1.
	By Proposition \ref{3:iissoo}, 
	$$\rank(M'\vert_{U_{h'}})\geq \rank(M')-2=\rank(M\vert_{V+c})-2\geq s^{2}+3s+3\geq (s-1)^{2}+3(s-1)+5.$$

	On the other hand, for all $h'\in H''$, by Proposition \ref{3:BB}, it is not hard to see that the map $$m\mapsto F(\tilde{g}(\phi(m)+c)\Gamma)$$ belongs to $\Xi^{s;O_{C,d,D,\e}(1),O_{C,d,D,\e}(1)}_{p}(\F_{p}^{d-r})$.
	By Lemma \ref{3:LE8} (viii), we have that the map $$m\mapsto F(\tilde{g}(\phi(m+h')+c)\Gamma)\otimes\overline{F}(\tilde{g}(\phi(m)+c)\Gamma)$$
	belongs to $\Nil^{s-1;O_{C,d,D,\e}(1),O_{C,d,D,\e}(1)}_{p}(\F_{p}^{d-r})$. By induction hypothesis,
	$$\Vert \Delta_{h'}f(\phi(\cdot)+c)\Vert_{U^{s}(V(M')^{h'})}\gg_{C,d,D,\e} 1$$
	for all $h'\in H''$. 
	Note that $\Gow_{s+1}(V(M'))$ is an $M'$-set of total co-dimension $\frac{(s+1)(s+2)}{2}+1$.
	If $p\gg_{C,d,D,\e} 1$, then
	\begin{equation}\nonumber
	\begin{split}
	&\quad\Vert f\Vert^{2^{s+1}}_{U^{s+1}(V(M)\cap(V+c))}
	\\&=\E_{(n,h_{1},\dots,h_{s+1})\in \Gow_{s+1}(V(M)\cap(V+c))}\prod_{\e=(\e_{1},\dots,\e_{s+1})\in\{0,1\}^{s+1}}\mathcal{C}^{\vert\e\vert}f\Bigl(n+\sum_{i=1}^{s+1}\e_{i}h_{i}\Bigr)
	\\&=\E_{(m,h'_{1},\dots,h'_{s+1})\in \Gow_{s+1}(V(M'))}\prod_{\e=(\e_{1},\dots,\e_{s+1})\in\{0,1\}^{s+1}}\mathcal{C}^{\vert\e\vert}f\Bigl(\phi\Bigl(m+\sum_{i=1}^{s+1}\e_{i}h'_{i}\Bigr)+c\Bigr)
	\\&=\E_{h'_{s+1}\in \F_{p}^{d-r}}\E_{(m,h'_{1},\dots,h'_{s})\in \Gow_{s}(V(M')^{h'_{s+1}})}\prod_{\e=(\e_{1},\dots,\e_{s})\in\{0,1\}^{s}}\mathcal{C}^{\vert\e\vert}\overline{\Delta_{h_{s+1}'}f}\Bigl(\phi\Bigl(m+\sum_{i=1}^{s+1}\e_{i}h'_{i}\Bigr)+c\Bigr)+O_{s}(p^{-\frac{1}{2}})
	\\&=\E_{h'_{s+1}\in \F_{p}^{d-r}}\Vert \overline{\Delta_{h'_{s+1}}f}(\phi(\cdot)+c)\Vert^{2^{s}}_{U^{s}(V(M')^{h'_{s+1}})}+O_{s}(p^{-\frac{1}{2}})\gg_{C,d,D,\e} 1,
	\end{split}
	\end{equation}
	where we used Lemma \ref{3:changeh} in the second equality, and Theorem \ref{3:ct} in the third equality (since $\rank(M')=\rank(M\vert_{V+c})\geq s^{2}+3s+5$).
	This finishes the proof.	
\end{proof}

\bibliographystyle{plain}
\bibliography{swb}
\end{document}